\documentclass[a4paper]{article}

\usepackage[english]{babel}

\usepackage[utf8]{inputenc}
\usepackage[utf8]{inputenc}
\usepackage{amsmath}
\usepackage{graphicx}
\usepackage{amssymb}
\usepackage{amsthm}
\usepackage{tikz-cd}
\usepackage{mathrsfs}
\usepackage[colorinlistoftodos]{todonotes}
\usepackage{enumitem}
\usepackage{yfonts}
\usepackage{ dsfont }
\usepackage{MnSymbol}
\usepackage{slashed}

\title{Thomas-Yau conjecture and holomorphic curves}

\author{Yang Li}

\date{\today}
\newtheorem{thm}{Theorem}[section]
\newtheorem{lem}[thm]{Lemma}

\theoremstyle{definition}
\newtheorem{eg}[thm]{Example}

\newtheorem{conj}[thm]{Conjecture}
\newtheorem{cor}[thm]{Corollary}
\newtheorem{claim}[thm]{Claim}
\newtheorem{rmk}{Remark}[section]
\newtheorem{prop}[thm]{Proposition}
\newtheorem{Def}[thm]{Definition}
\newtheorem{Question}{Question}

\newtheorem*{Acknowledgement}{Acknowledgement}

\newcommand{\ie}{\emph{i.e.} }
\newcommand{\cf}{\emph{cf.} }

\newcommand{\R}{\mathbb{R}}
\newcommand{\C}{\mathbb{C}}
\newcommand{\Z}{\mathbb{Z}}
\newcommand{\N}{\mathbb{N}}
\newcommand{\Q}{\mathbb{Q}}

\newcommand{\norm}[1]{\left\lVert#1\right\rVert}
\newcommand{\Lap}{\Delta}

\newcommand{\tarc}{\mbox{\large$\frown$}}
\newcommand{\arc}[1]{\stackrel{\tarc}{#1}}

\DeclareMathOperator{\Hom}{Hom}
\DeclareMathOperator{\End}{End}

\DeclareMathOperator{\Tr}{Tr}

\begin{document}
	\maketitle

\begin{abstract}
The main theme of this paper is the Thomas-Yau conjecture, primarily in the setting of exact, (quantitatively) almost calibrated, unobstructed Lagrangian branes inside Calabi-Yau Stein manifolds.  In our interpretation, the conjecture is that Thomas-Yau semistability is equivalent to the existence of special Lagrangian representatives. We clarify how holomorphic curves enter this conjectural picture, through the construction of bordism currents between Lagrangians, and in the definition of the Solomon functional. Under some extra hypothesis, we shall prove Floer theoretic obstructions to the existence of special Lagrangians, using the technique of integration over moduli spaces. In the converse direction, we set up a variational framework with the goal of finding special Lagrangians under the Thomas-Yau semistability asumption, and we shall make sufficient progress to pinpoint the outstanding technical difficulties, both in Floer theory and in geometric measure theory.

\end{abstract}

\tableofcontents

\section{Foreword}

The visionary proposal of Thomas and Yau \cite{Thomas}\cite{ThomasYau} is the philosophy that \emph{existence and uniqueness questions of certain special Lagrangian branes inside an almost Calabi-Yau manifold $(X,\omega,\Omega)$ should be governed by stability conditions in the derived Fukaya category} (\cf section \ref{ThomasYaupicture}). This proposal lies at the intersection of two major mathematical disciplines:
\begin{itemize}
\item From the viewpoint of \textbf{geometric measure theory}, the Thomas-Yau proposal promises a systematic method to prove existence theorems for special Lagrangians, which are certain absolute volume minimizers within prescribed homology classes. Currently, the known construction methods are based on perturbative techniques,  symmetry reductions, some special ansatzs, and integrable system techniques specific to low dimensions \cite{Joycecalibrated}. In contrast, Thomas and Yau suggests the difficult PDE questions may be reducible to largely topological and algebraic questions in Floer theory. 

\item 
From the viewpoint of \textbf{mirror symmetry}, the Thomas-Yau proposal is a window beyond homological mirror symmetry. Currently, most of the mathematical works on mirror symmetry are concerned with the duality between pure symplectic topology and pure complex geometry, and as such have a largely topological and algebraic flavour. As soon as one simultaneously consider the full K\"ahler geometric data on each side of the mirror, then questions of analytic nature become inevitable, and the Thomas-Yau proposal is a central component of this larger picture (\cf section \ref{Mirrorsymmetry}).

\end{itemize}

Two decades have lapsed since their proposal, but we feel only a small part of the mystery has been unveiled. Part of the problem is that the relevance of holomorphic curves to the Thomas-Yau proposal is insufficiently understood, despite their central role in defining the Fukaya category. One of the main goals of this paper is to propose the following picture, in the more specialized setting of exact Lagrangian branes inside (almost) Calabi-Yau Stein manifolds:

\begin{itemize}
\item  
Using $(n+1)$-dimensional universal families of holomorphic curves associated to certain $(n-1)$-dimensional moduli spaces, one can construct bordism currents between two $n$-dimensional Lagrangians in the same derived category class. More generally, one can also sometimes associate bordism currents between several Lagrangians, the notable example coming from distinguished triangles. This is a special case of the open-closed map well known to symplectic geometers, and some lower brow expositions are given in section \ref{LotayPacinirevisited}. 


\item  
Solomon \cite{Solomon} defined a functional on the universal cover of the space of Lagrangians within the same Hamiltonian deformation class, whose critical points are precisely special Lagrangians. In the exact setting, we will give a more homological formula for the Solomon functional, and propose that the bordism current produced from holomorphic curves allows one to extend the functional to Lagrangian objects in the same derived category class, which is well defined without the universal cover problem (\cf section \ref{Solomonrevisited}). A different perspective involving integration over moduli spaces is presented in section \ref{ModuliintegralSolomonsection}, and we suggest further how the Solomon functional may generalize to compact Calabi-Yau manifolds.

\item
We further specialize to almost calibrated Lagrangians. We propose that
given a distinguished triangle $L_1\to L\to L_2\to L_1[1]$,  then there are Floer theoretic necessary conditions for $L$ to be a special Lagrangian:
\[
\arg \int_{L_1} \Omega \leq \arg \int_L \Omega \leq \arg \int_{L_2}\Omega.
\]
Here $L_1, L_2$ are not required to be special Lagrangians, and the cohomological integral is the only quantitative way they enter into obstruction criterions. In other words, obstructions are of numerical nature.  We will prove this assuming further that the bordism current from holomorphic curves satisfies  automatic transversality and a \emph{positivity condition} (\cf section \ref{Automatictransversalitypositivity}, \ref{Positivitycondition}, \ref{Floertheoreticobstructions}). The main technique is integration over the $(n-1)$-dimensional moduli spaces of holomorphic curves. We also give generalizations to several Lagrangians, to make possible contact with Harder-Narasimhan decomposition.
Compared with the proposal of Joyce \cite{Joyceconj}, the Floer theoretic obstructions capture some features of the Bridgeland stability condition, but instead of tackling the entire derived Fukaya category $D^b Fuk(X)$, we restrict to almost calibrated Lagrangian objects, which morally correspond to the heart of a $t$-structure.

\begin{rmk}
While Joyce's prediction of a Bridgeland stability condition is a very important heuristic motivation for this paper, in our proposal we will carefully avoid assuming that the Bridgeland stability condition exists on $D^b Fuk(X)$, or that $D^b Fuk(X)$ is idempotent closed, or that the almost calibrated Lagrangians form an abelian subcategory of $D^bFuk(X)$, except when we make comparisons with Joyce's program. 
\end{rmk}

\item 
When there are no destabilizing distinguished triangles, we say the derived Fukaya category class of $L$ is \emph{Thomas-Yau semistable} (\cf Definition \ref{ThomasYausemistability}). This is similar to the original viewpoint of Thomas and Yau \cite{Thomas}\cite{ThomasYau}. Let $X$ be a Calabi-Yau Stein manifold, such that $(\omega, \Omega)$ satisfies the complex Monge-Amp\`ere equation, and the regularity scale of the Calabi-Yau manifold tends to infinity asymptotically. Let $\mathcal{L}$ be the geometric measure theoretic closure of the class of exact, quantitatively almost calibrated, unobstructed Lagrangians in the $D^bFuk(X)$ class. (Beware that we include immersed and singular Lagrangians as in Joyce \cite{Joyceconj}, so the definition of $\mathcal{L}$ is partially conjectural.) According to our interpretation of the Thomas-Yau conjecture, if $L_0\in \mathcal{L}$ is Thomas-Yau semistable, then  there exists a special Lagrangian current in $\mathcal{L}$  (\cf section \ref{ThomasYauconjecturesubsection} for the full version). Morever, it is desirable that the special Lagrangian current carries unobstructed brane structure in some formal sense, to represent the $D^bFuk(X)$ class or some weaker equivalence class.

\item
We will make a start on the variational approach to tackle the Thomas-Yau conjecture in the exact and quantitative almost calibrated setting. A summary of the variational strategy is in section \ref{Variationalstrategy}. We prove a number of uniform estimates independent of the Lagrangians, notably the potential clustering property, which implies a uniform bound on the energy of holomorphic curves, and that the Solomon functional deviates from an elementary functional by a uniformly bounded amount (\cf section \ref{Compactnessregularity}, \ref{Quantitativealmostcalibrated}). We are not able to complete the proof of the Thomas-Yau conjecture, but we try to identify the main technical difficulties to be overcome. To make contact with the a priori compactness theorems in geometric measure theory, it is natural to set up the variational program in terms of Lagrangian currents and varifolds. One then encounters both geometric measure theoretic difficulties, to do with the existence of enough Lagrangian competitors, and the difficulty to make sense of Floer theory for Lagrangian currents. We expect that many aspects of Floer theory will not be robust under passage to such weak limits of Lagrangians, but as a heuristic principle, we hope that the bordism currents relevant to our proposal are robust (\cf section \ref{Floertheoryweakregularity}).

\item
Assume there is a suitable compactness theory for Lagrangian objects, then we expect the existence of special Lagrangians to be essentially equivalent to the properness of the Solomon functional. The following heuristic principle would then explain why Thomas-Yau conjecture could be true: for a sequence of Lagrangians whose Solomon functional diverges to infinity, the underlying Lagrangian objects should break up into a bounded number of connected components $L_i$, such that the Lagrangian potential on each component has uniformly bounded oscillation, and the asymptotic behaviour of the Solomon functional is controlled by Floer theoretic data related to the cohomological integrals $\int_{L_i}\Omega$. We will give evidence for this picture (\cf section \ref{AsymptoticSolomon}), and explain a similar picture for Hermitian Yang-Mills connections (\cf section \ref{HYM}).

\item
Although we restrict to the exact setting in the above, we will leave the room open for more ambitious speculations in the setting of compact almost Calabi-Yau manifolds. The almost calibrated condition, on the other hand, is essential in our proposal, and cannot be dropped within this framework even with substantial efforts. 

\end{itemize}

Aside from the main goal of making the proposal and presenting evidence, this paper also contains a large amount of expository content, which aims to present both a (biased) overview and a critique of the current literature on the Thomas-Yau conjecture, and to explain how our picture fits into this body of works. The most relevant works are the original papers of Thomas and Yau \cite{Thomas}\cite{ThomasYau}, and the major update by Joyce \cite{Joyceconj}.
We will devote substantial attention to their main considerations (\cf section \ref{ThomasYaupicture}, \ref{ThomasYauevidence}, \ref{Lawlor}, \ref{Mirrorsymmetry}, \ref{HYM}, \ref{Extensionwallcrossing},  \ref{LMCF}), and the significance of the Thomas-Yau-Joyce proposal to mirror symmetry (\cf section \ref{Mirrorsymmetry}), and we will constantly draw comparisons between their pictures and ours. We also touch on a number of other relevant works:
\begin{itemize}
\item (\cf section \ref{Lawlor}) Joyce-Imagi-Santos \cite{JoyceImagi} gave a characterisation of Lawlor necks, in which holomorphic curves made an interesting, albeit somewhat technical appearance.

\item (\cf section \ref{HYM}, \ref{dHYM}) The Hermitian Yang-Mills equation \cite{DonaldsonHYM}\cite{DonaldsonHYM1}\cite{UhlenbeckYau} was a major inspiration to Thomas and Yau. The deformed Hermitian Yang-Mills equation \cite{Collins}\cite{CollinsJacobXie}\cite{CollinsJacobYau}\cite{Chen} has been advocated as a mirror version of special Lagrangians.

\item (\cf section \ref{Solomon}) Solomon \cite{Solomon}\cite{Solomon2}\cite{Solomon3} introduced a functional and a formal Riemannian metric on the infinite dimensional space of almost calibrated Lagrangians in a given Hamiltonian isotopy class. This space is too small for our variational purpose, but his functional is fundamental to our proposal.

\item (\cf section \ref{LotayPacini}) 
Lotay and Pacini \cite{Lotay}\cite{LotayPacini2} introduced a formal GIT picture for totally real submanifolds. Their notion of `geodesic' was our main inspiration for the bordism currents between Lagrangians.

\item (\cf section \ref{LMCF}, \ref{Quantitativealmostcalibrated}) Neves \cite{Neves}\cite{Neves2}\cite{Nevessurvey} pointed out the inevitability of finite time singularity in the Lagrangian mean curvature flow, which is fundamental to Joyce's proposal \cite{Joyceconj}. Some of his technical arguments are adapted to support our variational program in section \ref{Quantitativealmostcalibrated}.

\item (\cf section \ref{Continuitymethod}) Joyce's earlier work on wall crossing and special Lagrangian counting problems \cite{Joycecounting} is in some sense the elliptic cousin to his Lagrangian mean curvature flow proposal \cite{Joyceconj}.

\end{itemize}

\begin{rmk}
(Prerequisites)
While we have endeavoured to survey most of the previous works directly aimed at the Thomas-Yau conjecture, there is a very extensive literature on geometric measure theory and symplectic geometry in the background. We do not assume expertise on these matters, but some previous exposures such as F. Morgan's introductory book \cite{Morgan}, and the excellent surveys of Auroux \cite{Aurouxsurvey} and Smith \cite{Smithsurvey} would be useful. The most important background facts for our main purpose are also recalled in section \ref{Compactnessregularity} and the Appendix on the Fukaya category. While the brief summary therein is not completely sufficient for all arguments in this paper, we hope the casual reader could get the main gists, if not some sporadic remarks. The punctilious reader may wish to refer to the Floer degree and  sign convention summarized in the Appendix, which is different from e.g. Seidel's book \cite{Seidelbook}. The various allusions to K\"ahler geometry are mainly for motivational purposes, which can be skipped by readers less interested in these topics.
\end{rmk}

\begin{rmk}(Rigor)
This paper contains a somewhat unconventional mixture of proposals, expositions and proofs. We have attempted to indicate all speculative elements as `heuristic' or `conjecture'. Other arguments are either complete proofs, or sketches intended as expositions. 
\end{rmk}

\begin{Acknowledgement}
The author is currently an MIT CLE Moore Instructor and a Clay Research Fellow. This paper owes much intellectual debt to the proposal of Thomas-Yau and Joyce. The author thanks MIT for providing a stimulating research environment, and especially P. Seidel, S. Rezchikov and T. Collins for useful discussions. 
\end{Acknowledgement}




\section{Thomas-Yau conjecture backgrounds}

\subsection{The Thomas-Yau-Joyce picture}\label{ThomasYaupicture}

An \textbf{almost Calabi-Yau} manifold $(X,\omega, J, \Omega)$ is a complete $n$-dimensional K\"ahler manifold with a nowhere vanishing holomorphic volume form. It is called Calabi-Yau if the complex Monge-Amp\`ere equation
$
\omega^n=  \text{const} \Omega\wedge \overline{\Omega} 
$
holds, whence $X$ is Ricci flat.
A real $n$-dimensional compact submanifold of an almost Calabi-Yau manifold $\iota: L\to X$  is called \textbf{special Lagrangian} of constant phase angle $\hat{\theta}\in \R$, \footnote{Different lifts of the phase angle from $\R/\pi\Z$ to $\R$ shift the grading of the brane structure, so they are different as objects of the Fukaya category.} if
\begin{equation}\label{specialLagrangian}
\omega|_L=0, \quad \text{Im} (e^{-i\hat{\theta}}\Omega|_L)=0.
\end{equation}
Special Lagrangian submanifolds have unobstructed deformation theory.\footnote{However, obstructions may arise if brane structures are taken into account.}
Furthermore, if $X$ is Calabi-Yau, then a special Lagrangian is a minimal submanifold, and in fact minimizes the area within its homology class, thanks to the \textbf{calibration inequality} for submanifolds, saturated precisely by special Lagrangians:
\begin{equation}\label{calibration}
|\int_L \text{Re}( e^{-i\hat{\theta}}\Omega)|\leq |\int_L \Omega | \leq \int_L dvol(L)=\text{Vol}(L).
\end{equation}
We will always impose some compactness or convexity at infinity on $X$ to ensure the Fukaya category of $(X,\omega)$ makes sense.

\subsubsection*{Thomas-Yau's proposal}

In \cite{Thomas}\cite{ThomasYau} Thomas and Yau introduced the remarkable philosophy that \emph{existence and uniqueness questions of special Lagrangians inside an (almost) Calabi-Yau manifold $(X,\omega, \Omega)$ should be related to stability conditions in the Fukaya category}. Turning this intuition into precise mathematical predictions, is however not easy, for at least the following geometric reasons, in addition to the analytic difficulties related to minimal surfaces and mean curvature flows:
\begin{itemize}
\item  The Fukaya category as it currently stands is likely inadequate for the purpose: the special Lagrangian representatives of a given derived Fukaya category class in $D^b Fuk(X)$, should it exists, is by no means guaranteed to be smooth and embedded. Thus one would like to enlarge the objects of the Fukaya category to include immersed and possibly singular Lagrangians. From the viewpoint of symplectic topology, the lack of Lagrangian objects is a basic difficulty, which is usually hidden in the non-geometric step of taking the twisted complexes in the construction of $D^bFuk(X)$, and the  idempotent completions in the construction of $D^\pi Fuk(X)$ (\cf the Appendix \ref{immersedFukaya}).


\item The knowledge of the stability conditions is severely deficient. Thomas-Yau's paper predates the ingredient of Bridgeland stability\footnote{
What was available at the time, was the famous $\mu$-stability related to Hermitian Yang-Mills connections, of which Thomas and Yau were of course leading experts.} \footnote{There are current debates whether Bridgeland stability is the ultimately correct framework for formalising the physical intuition of stability conditions controlling BPS particle decay \cite{Douglas}. We regard Bridgeland stability as a working definition, to be modified in case future evidence arises.},  but the basic problem remains open: how to construct a stability condition on the derived Fukaya category $D^b Fuk(X)$ from the information of a holomorphic volume form $\Omega$ on $X$? In contrast, in analogous problems such as the existence of Hermitian-Yang-Mills connections, the stability condition is known a priori before the more serious endeavour to solve the PDE. One of the goals of this paper is to explain some modest progress on this issue, namely that there are  nontrivial Floer theoretic obstructions to the existence of special Lagrangians.

\end{itemize}

Thomas and Yau primarily restricted attention to the case of \textbf{almost calibrated} Lagrangians, meaning the Lagrangian angle function $\theta$ satisfies \footnote{The important thing is that the upper and lower bounds on $\theta$ differ by $\pi$. Shifting the interval by a constant is inconsequential.}
\[
-\frac{\pi}{2} <\theta< \frac{\pi}{2}.
\]
Notice this restriction removes any $\pi\Z$ ambiguity of the Lagrangian angle, so the Lagrangian is graded. 
Since we are focusing on compact Lagrangians, it makes sense to restrict to a more quantitative version, for some fixed small $\epsilon>0$:
\begin{equation}\label{quantitativealmostcalibrated}
-\frac{\pi}{2}+\epsilon \leq \theta \leq \frac{\pi}{2}-\epsilon.
\end{equation}
One immediate consequence is an a priori \textbf{volume bound}. 
\begin{lem}\label{Volumebound}
If $L$ satisfies (\ref{quantitativealmostcalibrated}), then $\text{Vol}(L)\leq \frac{1}{\sin \epsilon} \int_L\text{Re }\Omega.$
\end{lem}

\begin{proof}
From
$
e^{i\theta}dvol_L= \Omega|_L= \text{Re}\Omega|_L+ \sqrt{-1} \text{Im}\Omega|_L, 
$
we see $\text{Re }\Omega$ is an \textbf{orientation} form on $L$, and we arrive at
\begin{equation*}
\text{Vol}(L)= \int_L \frac{1}{\cos \theta} \text{Re }\Omega \leq \frac{1}{\sin \epsilon} \int_L\text{Re }\Omega
\end{equation*}	
as required. 
\end{proof}
\begin{rmk}
Another immediate consequence of the almost calibrated condition, is that the complex number $Z(L)=\int_L\Omega$ is nonzero, with $\phi(L)= \frac{1}{\pi}\arg Z(L)\in (-\frac{1}{2}, \frac{1}{2})$. 	
\end{rmk}

Keeping our narrative closer to the historical development, Thomas and Yau 
were inspired by $\mu$-stability for Hermitian Yang-Mills connections. They
assumed that the principal mechanism which can forbid the Hamiltonian isotopy class of $L$ from admitting a special Lagrangian representative, is related to a distinguished triangle in the Fukaya category, the primary geometric source being that $L$ is Hamiltonian isotopic to a graded Lagrangian connected sum $L=L_1\#L_2$,\footnote{It is very important for Thomas and Yau that the Lagrangian connected sum is asymmetric in $L_1$ and $L_2$. Our notation for the Lagrangian connected sum agrees with Thomas and Yau, but is opposite to Joyce \cite{JoyceSLsurvey} and a large number of symplectic geometry texts. Our convention is compatible with the distinguished triangle $L_1\to L_1\# L_2\to L_2\to L_1[1]$. } such that $\phi(L_1)\geq \phi(L_2)$. Based on this intuition, Thomas made an attempt to define a notion of stability (\cf \cite[Definition 5.1]{Thomas}, and Definition \ref{ThomasYausemistability} below).\footnote{We will not repeat their definition verbatim here because the author thinks its focus on the Hamiltonian isotopy class, instead of the Fukaya category class, is largely a limitation of its time. It is the spirit rather than the letter of their definition which matters.} Thomas then made the important prediction:

\begin{conj}
\cite[Conj 5.2]{Thomas} A graded Lagrangian has a special Lagrangian representative in its Hamiltonian class if and only if it is stable, and this special Lagrangian representative is unique.
\end{conj}

Thomas and Yau \cite{ThomasYau} analyzed the problem again from the Lagrangian mean curvature flow (LMCF) perspective, and made a somewhat more cautious prediction, which roughly amounts to the following. When
a destabilising decomposition into Lagrangian sums is forbidden, 
either by a smallness assumption on the oscillation of the phase function $\theta$, or because the volume of $L$ is smaller than $\text{Vol}(L_1)+\text{Vol}(L_2)$ for any putative decomposition, (notice both conditions are preserved under the flow), then they conjectured that the mean curvature flow starting from $L$ will converge into a special Lagrangian \cite[section 7]{ThomasYau}.

\subsubsection*{Bridgeland stability, Joyce's proposal}

To trace the subsequent development, and move beyond the almost calibrated setting, we need to recall
an important piece of homological algebra known as \textbf{Bridgeland stability conditions} on a triangulated category \cite{Bridgeland}.

\begin{Def}
A Bridgeland stability condition $(Z,\mathcal{P})$ on a triangulated category $\mathcal{D}$
consists of a group homomorphism  (`\textbf{central charge}') $Z$ from the (numerical) Grothendieck group $K(\mathcal{D})$  to $\C$,  and
full additive subcategories  $\mathcal{P}(\phi) \subset \mathcal{D}$ for each $\phi \in \R$, whose objects are called `\textbf{semistable objects} of phase angle $\pi \phi$',\footnote{In the convention of Bridgeland, the phase is $\phi$. We have instead opted to call $\pi \phi$ the phase, which is more naturally identified with the phase angle of special Lagrangians.} satisfying the following
axioms:
\begin{itemize}
\item (Phase) If $L \in \mathcal{P}(\phi)$ then $Z(L) = m(L) \exp(i\pi \phi)$ for some $m(L)>0$,
\item (Shift) For all $\phi \in \R$, $\mathcal{P}(\phi + 1) =\mathcal{P}(\phi)[1]$,
\item 
(Monotonicity)
 If $\phi_1>\phi_2$ and $L_j \in \mathcal{P}(\phi_j)$, $j=1,2$ then $Hom_{\mathcal{D}}(L_1,L_2) = 0$,
 \item 
 For each nonzero object $L \in \mathcal{D}$  there are a finite sequence of real numbers
$ \phi_1 > \phi_2 >\ldots > \phi_N$
 and a \textbf{Harder-Narasimhan decomposition}
 \begin{equation}\label{HarderNarasimhan}
 0=\mathcal{E}_0 \to \mathcal{E}_1\to \ldots \to \mathcal{E}_N=L,
 \end{equation}
 with distinguished triangles \[
 \mathcal{E}_{i-1}\to \mathcal{E}_{i}\to L_{i}\to \mathcal{E}_{i-1}[1]
 \]
such that $L_j\in \mathcal{P}(\phi_j)$.
\item (Calibration) For any fixed norm on the finite dimensional vector space $K(\mathcal{D})\otimes_\Z \R$, we have a uniform constant $C$, such that any semistable object $L$ satisfies
$
\norm{L} \leq C|Z(L)|.
$
\end{itemize}
\end{Def}


\begin{rmk}
For the heuristic but na\"ive geometric meaning,  one can imagine  that $\mathcal{D}$ is the derived Fukaya category, $K(\mathcal{D})$ is the sublattice of $H_n(X)$ generated by the Lagrangians, the central charge is $Z(L)=\int_L \Omega$, the subcategory $\mathcal{P}(\phi)$ is generated by the special Lagrangians of constant phase angle $\theta=\pi \phi$, the Harder-Narasimhan decomposition means a multiple Lagrangian connected sum with decreasing phase angles
\[
L\simeq L_1\#L_2\#\ldots \# L_N, 
\]
and the calibration property comes from the fact that the total mass of a special Lagrangian is equal to $|Z(L)|$, and the mass bounds any norm of $[L]\in H_n(X)$ using Poincar\'e duality, assuming $H_n(X)$ is finite dimensional.
\end{rmk}

\begin{rmk}
We write the subcategory generated by all the $\mathcal{P}(\phi)$ within the interval $\phi\in (\phi_0, \phi_1]$ as $\mathcal{P}(\phi_0<\phi\leq \phi_1)$. The important property is that $\mathcal{P}(\phi_0<\phi\leq \phi_0+1)$ is an abelian category, known as the heart of a bounded $t$-structure. The advantage is that while in general triangulated categories only distinguished triangles make sense, for abelian categories we can talk about exact sequences, subobjects and quotients.
\end{rmk}

With the hindsight of Bridgeland stability condition, and more than one decade of progress on the mean curvature flow as well as symplectic geometry, Joyce \cite{Joyceconj} produced a major update of the Thomas-Yau picture. We shall discuss more about the LMCF considerations in section \ref{LMCF}, 
but it suffices here to give away its main punchline (\cf \cite[Conj 3.34]{Joyceconj}):

\begin{itemize}
\item
Let $(X, \omega, \Omega)$ be a Calabi-Yau manifold \footnote{The complex Monge-Amp\`ere equation is convenient but not indispensable, see section \ref{LMCF}.}, either compact or Stein.
There should be an enlarged version of the derived Fukaya category $D^b Fuk(X)$, including classes of immersed or mildly singular Lagrangians, and a Bridgeland stability condition $(Z, \mathcal{P})$ on $D^b Fuk(X)$, whose central charge is 
\begin{equation}
Z(L)= \int_L \Omega.
\end{equation}

\item
Given a Lagrangian brane $L$ (with grading, orientation, relative spin structure, local system and bounding cochain data) such that the Floer cohomology is unobstructed, and suppose $L$ is generic in its Hamiltonian isotopy class. Then the LMCF $(L_t)_{t>0}$ (with brane structure) starting from $L$ exists for all time \emph{with surgeries} at a finite series of singular times. The nature of these surgeries have conjectural descriptions. The Lagrangian can change its Hamiltonian class at these surgeries, but maintains its derived category class. At $t\to \infty$, the Lagrangians $L_t$ converges in the geometric measure theory sense to a union of graded special Lagrangian currents $L_1, \ldots L_N$ of phase angle $\hat{\theta}_1>\hat{\theta}_2>\ldots > \hat{\theta}_N$ with multiplicities counted:
\begin{equation}\label{infinitetimelimit}
\lim_{t\to \infty} L_t= L_1+\ldots L_N.
\end{equation}

\item
If there is only one constant phase angle $\hat{\theta}$ appearing in the infinite time limit, then $L$ defines a semistable object of phase $\hat{\theta}$ in $ D^b Fuk(X)$. Otherwise $L$ is not semistable with respect to any phase angle, and the infinite time limit supposedly give rise to the Harder-Narasimhan decomposition. 
\end{itemize}


Joyce's program is much more ambitious in the sense that it attempts to capture the entire triangulated category $D^b Fuk(X)$. The original Thomas-Yau proposal, which is concerned only with almost calibrated Lagrangians, fits into the Joyce picture as an \emph{abelian subcategory}  $\mathcal{P}(-\pi/2<\theta\leq \pi/2)\subset D^b Fuk(X)$,\footnote{Here we ignore the `small' difference between $<$ and $\leq$. This would be justified if $\Omega$ is suitably generic so that there is no special Lagrangian of phase $\pi/2$, namely that the countably many numbers $\text{Arg} \int_L\Omega$ for $[L]\in H_n(X,\Z)$ miss the number $\pi/2$.
	} the heart of a certain $t$-structure. The main appeal of Joyce's mean curvature flow perspective, is that if the infinite time limiting currents $L_1,\ldots L_N$ can be given suitable brane structures to be admitted as objects of $D^bFuk(X)$, then the program gives a conjectural dynamical mechanism to explain the Harder-Narasimhan decomposition, which is the most nontrivial aspect of the Bridgeland stability. Its principal drawback is that Joyce offers no a priori information on the Bridgeland stability beyond the central charge, other than letting off the mean curvature flow to find its own destiny. This limits its practical applicability, for instance to existence questions of special Lagrangians in prescribed classes in $D^b Fuk(X)$.  See section \ref{TowardsBridgeland} for more discussions.

\subsubsection*{Why Thomas-Yau has predictive power}

At this moment an objection may arise: since the Thomas-Yau-Joyce picture is beset by some vagueness and plenty of technical difficulties, why is it useful as a guiding principle at all? Besides the supportive evidence that we shall soon discuss, the main answer is that the Thomas-Yau philosophy transforms a PDE problem (the existence of special Lagrangians) into a categorical framework, which if better understood is in principle checkable by algebraic means. \emph{A Bridgeland stability is the interplay between a category and a numerical property}. In the analogous problem of Hermitian Yang-Mills connections, the existence criterion is formulated by $\mu$-stability, which compresses the information of the K\"ahler form into only certain intersection numbers/cohomological integrals (\cf section \ref{HYM}). Likewise, the stability condition responsible for the existence of special Lagrangians, even though it is not adequately specified, is in principle a compression of the analytic data of a holomorphic volume form, and quite plausibly enters only through cohomological integrals of $\Omega$, as will be discussed more fully in Chapter \ref{Moduliholocurves}. As an indication of the possible predicative power of the Thomas-Yau picture, here is a sample question as food for thought:

\begin{Question}
Fix the holomorphic volume form $\Omega$.
Let $\omega_1, \omega_2$ be two generic K\"ahler forms, differing only by the differential of a compactly supported 1-form. Can we define a count of  special Lagrangian rational homology spheres, such that the numbers agree for $\omega_1$ and $\omega_2$? 
\end{Question}

\begin{rmk}
As we shall see, the main evidence of Thomas-Yau picture (\cf section \ref{ThomasYauevidence}) does not really use the complex Monge-Amp\`ere equation.\footnote{The almost Calabi-Yau setting is desirable not only for the sake of generality, but may be essential to achieve suitable genericity. As an additional motivation on the side of physics, the SCFT condition translates into a K\"ahler condition on the target space metric, which satisfies the Ricci-flatness only approximately \cite[section 14.2.4]{Claymirror}.} On the mirror side, as a consequence of the $\mu$-stability characterisation, the existence of Hermitian Yang Mills connections on a compact K\"ahler manifold does not depend on the choice of the K\"ahler form within a fixed K\"ahler class. 
\end{rmk}

\subsection{Principal evidence of Thomas-Yau}\label{ThomasYauevidence}

Thomas and Yau offered a number of arguments in support of their picture. We now describe the evidence from \emph{first principles},  and will later return to the evidence from \emph{analogies}.

The following elementary observation is a first indication about how special Lagrangian objects resemble stable objects of a Bridgeland stability condition:

\begin{prop}\cite[section 5.2]{ThomasYau}
Let $L$, $L'$ be unobstructed Lagrangian branes whose supports are compact special Lagrangian manifolds, with Lagrangian angles $\theta_L> \theta_{L'}$. Then $HF^k(L,L')=0$ for $k\leq 0$. 
\end{prop}

\begin{proof}
After $C^\infty$-small Hamiltonian perturbation we may assume any intersection point $p$ between $L$ and $L'$ is transverse. We can write the tangent planes of $L, L'$ inside $T_p X\simeq \C^n$ in the standard form
\[
T_p L= \R^n \subset \C^n, \quad T_p L'= (e^{i\phi_1}, \ldots e^{i\phi_n})\R^n\subset \C^n, \quad 0<\phi_i<\pi.
\]
The Floer degree of the intersection point $p$ is
\[
\mu_{L,L'}(p)=\frac{1}{\pi} (\theta_L- \theta_{L'}+ \sum_1^n \phi_i)>0,
\]
whence $HF^k(L,L')=0$ for $k\leq 0$.
\end{proof}

\begin{rmk}
As is clear from the proof, the special Lagrangian condition can be relaxed to $\inf_{L} \theta_L> \sup_{L'} \theta_{L'}$.
\end{rmk}

\begin{rmk}
Our Floer degree convention follows Joyce \cite{Joyceconj}, but is opposite to that of Seidel \cite{Seidelbook} and many other symplectic geometry texts. In our convention,
shifting the Lagrangian phase by $k\pi$ for $k\in \Z$ is equivalent to considering the shifted object $L[k]$. Thus a slight extension of the above is that if $\theta_L>\theta_{L'}+l$, then \[
HF^k(L,L')= HF^{k-l}(L, L'[l])=0, \quad k\leq l.
\]
\end{rmk}

The most compelling evidence due to Thomas and Yau is

\begin{thm}(Thomas-Yau uniqueness)
\cite{ThomasYau} Let $L, L'$ be unobstructed Lagrangian branes supported on embedded special Lagrangians of the same phase, which define isomorphic objects in $D^b Fuk(X)$, then their supports coincide.
\end{thm}

\begin{rmk}
The most general Thomas-Yau uniqueness, which includes immersed Lagrangians, seems to be due to Imagi \cite{Imagi}. The original Thomas-Yau theorem is phrased in terms of uniqueness in the Hamiltonian isotopy class, even though it can be cast in more general categorical terms. The categorical perspective is preferred, because it is closer to the spirit of homological mirror symmetry, and because one derived Fukaya category class may contain several Hamiltonian isotopy classes. If immersed Lagrangians are allowed, then isomorphism in $D^bFuk(X)$ would also identify certain embedded Lagrangian objects with immersed objects of different topologies. When this happens, an interesting corollary of Thomas-Yau uniqueness is that at most one of these Hamiltonian classes contains special Lagrangian branes.
\end{rmk}

\begin{proof}
(Sketch)
Assume the suppports do not coincide. 
 After small Hamiltonian perturbations, we can ensure the perturbed Lagrangians $\tilde{L}, \tilde{L}'$ have transverse intersections, and still define the same isomorphic objects in $D^b Fuk(X)$. Under the special Lagrangian assumption, through judicious choice of the Hamiltonian via Morse theory as in Thomas-Yau \cite[Thm 4.3]{ThomasYau}, or by using genericity arguments based on real analyticity as in Joyce-Imagi-Santos \cite[section 4.3]{JoyceImagi}, one can ensure there is no intersection point $\tilde{L}\cap \tilde{L}'$ of degree $0,n$, so in particular $HF^0(L,L')\simeq HF^0(\tilde{L}, \tilde{L}')= 0$. However, this implies the cohomological unit of $HF^*(L,L)$ is zero, so the Floer cohomology ring of $L$ is zero, namely $L$ is a zero object in $D^bFuk(X)$, contradiction.
\end{proof}

The Thomas-Yau argument reveals the relevance of the Fukaya category to special Lagrangian geometry. The role of holomorphic curves, which are a central ingredient in the Fukaya category, is however rather opaque in this argument; their only appearance is to make the Floer cohomology defined.

\subsection{Variant: uniqueness of the Lawlor neck}\label{Lawlor}

 A variant of the Thomas-Yau argument\footnote{Abouzaid and Imagi have mentioned in their talks some other interesting applications on the topology of special Lagrangians inside the cotangent bundle of a special Lagrangian with some fundamental group conditions, using another variant of the Thomas-Yau argument. We look forward to the appearance of their paper. } appears in the subsequent work of Joyce-Imagi-Santos \cite{JoyceImagi} on the uniqueness of Lawlor necks, where holomorphic curves and the algebraic structures of the Fukaya category appear in a more prominent, albeit somewhat technical way.


\subsubsection*{Lawlor necks}

We first recall some basics about Lawlor necks \cite{Lawlor}\cite{JoyceLeeTsui}, which are non-compact embedded exact special Lagrangians $L_{\phi,A}$ inside the standard Euclidean $\C^n$, asymptotic at infinity to the union of two planes
\[
\Pi_0= \R^n, \quad \Pi_\phi=(e^{i\phi_1}, \ldots e^{i\phi_n})\R^n, \quad 0<\phi_i<\pi, \quad \sum \phi_i=\pi. 
\]
Symplectic topologically, they can be viewed as a realisation of the Lagrangian handle $S^{n-1}\times \R$ that appears in the Lagrangian connected sum construction. This motivates the ansatz 
\begin{equation}
L_{\phi,A}= \{   (z_1(y)x_1, . . . , z_n(y)x_n) : y \in \R, x_k \in \R, x^2_1
+ \ldots+ x_n^2 =1      \}.
\end{equation}
The special Lagrangian condition translates into an ODE system on the functions $z_1(y),\ldots z_n(y)$, which can be solved exactly as follows.

Let $n > 2$ and $a_1, . . . , a_n > 0$, and define polynomials $p, P$ by
\[
p(x) = (1 + a_1x^2) \ldots (1 + a_nx^2) -1 , \quad  P(x) =
\frac{p(x)}{x^2}.
\]
Define real numbers $\phi_1, . . . , \phi_n$ and $A$ by
\[
\phi_k = a_k
\int_{-\infty}^{\infty}
\frac{dx}{
(1 + a_kx^2)\sqrt{P(x)}}
, \quad
 A =
\int_{-\infty}^{\infty}
\frac{dx}{
	2\sqrt{P(x)}}
\]
Clearly $\phi_k,A > 0$, and elementary integration shows $\sum \phi_i=\pi$. This yields a 1-1 correspondence between $n$-tuples $(a_1,\ldots, a_n)$ with $a_k > 0$,
and $(n+1)$-tuples $(\phi_1, \ldots , \phi_n,A)$ with $\phi_k\in (0,\pi)$, $\sum \phi_k=\pi$ and $A>0$. Setting
\[
z_k(y)= e^{i\psi_k(y)}
\sqrt{a_k^{-1}+y^2 },  \quad \text{where } \psi_k(y) = a_k
\int_{-\infty}^y \frac{dx}{
	(1 + a_kx^2)\sqrt{P(x)}},
\]
yields the solution $(z_1(y),\ldots, z_n(y))$, hence the \textbf{Lawlor necks} $L_{\phi,A}$.

For fixed asymptotic planes $\Pi_0, \Pi_\phi$, the Lawlor necks arise in a 1-parameter family, related to each other by the rescaling $\vec{z}\to \lambda\vec{z}$ in $\C^n$, and $A$ behaves like 2-dimensional area $A\to \lambda^2A$ under this scaling. One also observes that asymptotically near infinity, the Lawlor necks are graphs over $\Pi_0$ (resp. $\Pi_\phi$) of the differential $df$, where
\[
|f|=O( |\vec{x}|^{2-n }), \quad |\partial^k f|=O( |\vec{x}|^{2-n-k}  ).
\]
We say the Lawlor neck has asymptotic decay rate $2-n$. The upshot is that it approaches $\Pi_0\cup \Pi_\phi$ sufficiently fast.



\subsubsection*{Joyce-Imagi-Santos uniqueness theorem}

\begin{thm}
\cite{JoyceImagi} Let $n\geq 3$. \footnote{The theorem is also true in complex dimension 2, proved earlier by Lotay and Neves \cite{LotayNeves} by other means not involving Floer theory.} Assume $L$ is a smooth embedded exact special Lagrangian of phase zero, asymptotic at rate $<0$ to the union of the two planes $\Pi_0\cup \Pi_\phi$ with $\sum \phi_i=\pi$, then $L$ is  $L_{\phi,A}$ for some $A>0$.

\end{thm}

Here is a sketch of their arguments:
\begin{itemize}
\item  Using the asymptotic  assumption on the exact Lagrangian $L$, one can assign an analytic invariant $A(L)$ to $L$ as follows. Let $f_L:L\to \R$ be a primitive of the Liouville form $\lambda$, namely $df_L=\lambda|_L$, then $f_L$ converges to constants $c_0, c_\phi$ at the two asymptotic ends along $\Pi_0, \Pi_\phi$ respectively. Then one defines $A(L)=c_\phi-c_0$. If $L$ coincides with the Lawlor neck $L_{\phi,A}$, then $A(L)=A$.

\item  Partially compactify $\C^n$ into a Liouville manifold identified as the plumbing $M$ of two cotangent bundles $T^*S^n$ with $T^*S^n$. Here the two copies of $S^n$ arise topologically as one-point compactifications of $\Pi_0$ and $\Pi_\phi$ by adding the points at infinity $\infty_0$ and $\infty_\phi$, and topologically $M$ is the union of $\C^n$ and the two cotangent fibres over $\infty_0$ and $\infty_\phi$ respectively. Under suitably fast decay condition at infinity, the unknown special Lagrangian $L$ can be compactified into an exact graded embedded Lagrangian $\bar{L}$ inside $M$. One would like to compare this to the Lagrangian $\bar{L}_{\phi,A}$ obtained by the compactification of the standard Lawlor necks $L_{\phi,A}$ inside $M$.

\item  
By analyzing the intersection pattern with the two cotangent fibres at infinity, and using the classification results of Abouzaid and Smith \cite{AbouzaidSmith}, one shows that inside $D^bFuk(M)$, the Lagrangian object $\bar{L}$ is isomorphic to one of the two Lagrangian connected sums of the two $S^n$ with suitable gradings, and in fact the assumption on Floer degrees $\sum \phi_i=\pi$ singles out $\bar{L}\simeq \bar{L}_{\phi, A}\in D^bFuk(M)$, the opposite surgery  corresponding to $\sum \phi_i=(n-1)\pi$.
This step needs $n\geq 3$. 
 For contradiction, we assume $\bar{L}$ does not coincide with $\bar{L}_{\phi,A}$ for any choice of parameter $A>0$.

\item
By a modification of the Thomas-Yau argument, one shows that after a small Hamiltation perturbation $\bar{L}''$ of $\bar{L}_{\phi,A}$, we can ensure $\bar{L}''$ is transverse to $\bar{L}$, there is no degree $0, n$ intersection points in $\bar{L}''\cap \bar{L}$ inside $\C^n$, and there is precisely one intersection point $p$ and $q$ in $\bar{L}''\cap \bar{L}$ on each of the two cotangent fibres at infinity respectively. Morever, the $D^b Fuk(M)$ class and the analytic invariants of $\bar{L}''$ agree with that of $\bar{L}_{\phi,A}$.

\begin{rmk}
The subtlety at infinity prevents one from removing degree $0,n$ intersections outside the $\C^n$ region, so one does not reach an immediate contradiction as in the Thomas-Yau argument. This technical failure is necessary, because the Lawlor necks with fixed asymptotic planes are not unique, but do arise in a 1-parameter family. It is in overcoming this technical problem that holomorphic curves appear in \cite{JoyceImagi}.	
\end{rmk}

\item Now suppose the Lawlor neck is chosen with the parameter $A=A(L)$, which presumes $A(L)>0$. 
\begin{lem}\cite[Thm 2.15]{JoyceImagi}\label{JoyceImagicountinglemma}
Assume $J$ is a generic almost complex structure on $M$ compatible with the Liouville structure.
There exists a $J$-holomorphic strip $\Sigma$ with boundary on $\bar{L}$ and $\bar{L}''$ and two corners at $p$ and $q$ respectively.
\end{lem}

\begin{proof}
Consider the Floer cup product with mod 2 coefficients
\[
HF^0( \bar{L}'', \bar{L} )\otimes HF^0(\bar{L}, \bar{L}'')\to HF^0(\bar{L}, \bar{L}  ),
\]
which can be identified as the cup product
\[
H^0( \bar{L}, \bar{L} )\otimes H^0(\bar{L}, \bar{L})\to H^0(\bar{L}, \bar{L}  ),\quad 1_{\bar{L}}\cup 1_{\bar{L}}= 1_{\bar{L}},
\]
and thus must be nontrivial. However, at chain level this Floer product comes from the $A_\infty$ operation
\[
m_2:  CF^0(\bar{L}'', \bar{L}) \otimes  CF^0( \bar{L}, \bar{L}'' )\to CF^0(\bar{L}, \bar{L}  ),
\]
which must be nontrivial. The counting interpretation implies there are intersection points $p'\in CF^0(\bar{L}, \bar{L}'')$ and $q'\in CF^0(\bar{L}'', \bar{L})\simeq CF^n( \bar{L}, \bar{L}'' )^\vee$ and some holomorphic strip in between. Since degree $0,n$ intersection points cannot occur inside $\C^n$, they can only occur at infinity, so we must have $\{p,q \}=\{ p',q' \}$. 
\end{proof}

Now the area of the $J$-holomorphic curve can be computed cohomologically. Using the choice of parameter $A$,
\[
\int_{\Sigma}\omega=\int_{\partial \Sigma}\lambda= \int_{p'\to q'} df_{\bar{L}}+ \int_{q'\to p'} df_{\bar{L}''}= \pm( A(L)-A(L'')  ) =\pm(A-A)=0.
\]
This contradicts the positivity of area of the holomorphic curve, which proves $L$ must coincide with $L_{\phi,A}$.

\item Finally one needs to a priori justify $A(L)>0$. This relies on a slightly more complicated holomorphic polygon counting argument, and the main upshot is that one can produce a nontrivial holomorphic triangle from a distinguished triangle in $D^b Fuk(M)$, with the three edges on $\bar{L}$, $\Pi_0\cup\{\infty_0\}$ and $\Pi_\phi\cup \{ \infty_\phi \}$. Then one shows $A(L)$ has the interpretation as its area, so must be positive.

\end{itemize}

\subsubsection*{Ideal triangles}

In  \cite{JoyceImagi} the perturbations involved in the partial compactification and the genericity of the almost complex structure makes the holomorphic curves rather difficult to visualize.\footnote{A typical feature of Floer theory, is that completely realistic examples about holomorphic curves are also non-explicit.} We now present a heuristic way to see holomorphic triangles with the edges on $\bar{L}_{\phi,A}$, $\Pi_0\cup\{\infty_0\}$ and $\Pi_\phi\cup \{ \infty_\phi \}$, by restricting attention to $\C^n$ with the standard complex structure, and we imagine the two vertices $\infty_0, \infty_\phi$ as the intersection points at infinity.

We choose any $\vec{x}=(x_1,\ldots x_n)\in S^{n-1}$. Assume first that $x_k\neq 0$.  Then coordinatewise, we have a real curve in $\C$ swept out by $z_k(y)x_k$ as $y$ varies from $-\infty$ to $+\infty$, and two straight rays emanating from the origin defined by $\text{sign}(x_k) \R_+$ and $\text{sign}(x_k)e^{i\phi_k} \R_+$. Inside $\C$, these three real curves enclose a noncompact holomorphic triangle, with one vertex at the origin, and two idealized intersection points at the infinity of $\text{sign}(x_k)\R_+$ and $\text{sign}(x_k) e^{i\phi_k} \R_+$. 
In the product space $\C^n$, this gives rise to a holomorphic triangle $\Sigma_{\vec{x}}$ with boundary on $L_{\phi,A}, \Pi_0, \Pi_\phi$, and corners at $0, \infty_0, \infty_\phi$.
Now in case some $x_k= 0$, there is still a holomorphic triangle in the product space that makes sense; the $k$-th projection of this triangle is simply the origin. What happens when $x_k\to 0$, is simply that the $k$-th projection $\pi_k(\Sigma_{\vec{x}})$ becomes very thinly concentrated near the two rays $\text{sign}(x_k) \R_+$ and $\text{sign}(x_k)e^{i\phi_k} \R_+$, and its area shrinks to zero. Morever for any given $\epsilon$, the subset $\Sigma_{\vec{x}}\cap \pi_k^{-1}(|z|>\epsilon) $ disappears into the infinity of $\C^n$ as $x_k\to 0$. Thus when we restrict to any compact subset of $\C^n$, the holomorphic discs behave continuously as $x_k\to 0$.



\begin{eg}
In the most symmetric case $\phi_1=\phi_2=\ldots \phi_n=\frac{\pi}{n}$, the Lawlor neck is invariant under $SO(n,\R)$, and these holomorphic triangles are up to $SO(n,\R)$ rotation, simply the triangle inside the first coordinate line $\C\subset \C^n$ enclosed by the three Lagrangians. 
\end{eg}

Using any of the holomorphic triangles $\Sigma_{\vec{x}}$ parametrised by $S^{n-1}$, we can calculate its area cohomologically by 
\[
\int_{\Sigma}\omega=\int_{\partial \Sigma}\lambda= \int_{\partial \Sigma\cap L_{\phi,A}}  df_{L_{\phi,A}} =A,
\]
which is the intuitive explanation of why $A$ must be positive, an important ingredient of \cite{JoyceImagi}.

We want to draw attention also to a different aspect not explicit in \cite{JoyceImagi}: that these holomorphic
triangles naturally arise in an $(n-1)$-dimensional \emph{moduli}, rather than as isolated triangles. Consequently, the universal family of such holomorphic triangles is naturally  $(n+1)$-dimensional. A generic point on $\Pi_0\cup \Pi_\phi\cup L_{\phi,A}$ is swept out precisely once by some $\partial\Sigma_{\vec{x} }$. When the orientations are taken into account, then the total space of this universal family gives rise to an $(n+1)$-dimensional integration current, which
 provides a \emph{bordism current} between the integration cycles of $L_{\phi,A}$ and $\Pi_0\cup \Pi_\phi$. 
 Producing bordism currents via universal families of holomorphic curves will be essential to our proposals concerning the Thomas-Yau conjecture.

\subsection{Philosophy of open string mirror symmetry}\label{Mirrorsymmetry}

The Thomas-Yau conjecture is largely motivated by mirror symmetry, based on the analogy between special Lagrangian geometry on the $A$-side (`symplectic'), and stability conditions in the derived category of coherent sheaves on the $B$-side (`holomorphic'). It is worth emphasizing that mirror symmetry is quantum by nature, and only sometimes admits classical geometric interpretations. The mathematical predictions of open string mirror symmetry chiefly fall into the following three classes: 

\begin{itemize}
\item (\textbf{Homological mirror symmetry}, \ie the \textbf{categorical} approach) Given a (compact) Calabi-Yau manifold $X$ viewed as a symplectic manifold, one expects to find a (compact) Calabi-Yau manifold $X^\vee$ viewed as a complex manifold (sometimes over the Novikov field), such that the (derived, idempotent closed) Fukaya category $D^\pi Fuk(X)$ is equivalent as a triangulated category to the derived category of coherent sheaves $Coh(X^\vee)$. Frequently, there is a preferred $A_\infty$ enhancement on both sides, such that $D^\pi Fuk(X)\simeq D^bCoh(X^\vee)$ holds as an $A_\infty$-equivalence. The concrete implication is that Lagrangian branes $L$ correspond to objects $E$ of $D^b Coh(X^\vee)$, such that the Floer groups are identified with the $Ext$ groups, which explains the name `homological':
\[
HF^*(L,L')\simeq Ext^*(E,E').
\]
Homological mirror symmetry has many ramifications beyond Calabi-Yau manifolds, involving for instance Fano manifolds and Landau-Ginzburg models. As a general feature, homological mirror symmetry only uses pure symplectic topology (vis. a vis.  complex geometry) on either side of the mirror. It is the best understood aspect of mirror symmetry, with many special cases proven.

\item (\textbf{SYZ mirror symmetry}, \ie the \textbf{moduli space} approach) Taken in a somewhat generalised and very optimistic sense \footnote{The SYZ paper \cite{SYZ} focuses primarily on special Lagrangian torus fibrations, but hints at more general special Lagrangian branes. Some more discussions can be found in Fukaya's work \cite[section 8.5]{Fukaya}.}, the Strominger-Yau-Zaslow philosophy says that after taking into account the quantum instanton corrections, then the moduli space of (semi)stable Lagrangian branes within fixed homology classes acquires a modified complex structure, and can be identified with certain moduli spaces of (semi)stable objects from the mirror side. The stable Lagrangian branes are believed to be related to special Lagrangians via the Thomas-Yau conjecture. Most of the literature concentrates on the restricted case of (special) Lagrangian torus fibrations arising from the large complex structure limit, in which case the moduli space of Lagrangian branes supported on the fibres inside $X$ should recover the moduli space of points on $X^\vee$, namely the mirror manifold $X^\vee$ itself \footnote{This SYZ mirror construction has seen the intense research effort on the symplectic side by Auroux \cite{Auroux}, Abouzaid, Fukaya, among many others. The very recent progress \cite{Yuan} achieves a non-archimedean mirror from Fukaya categorical considerations, but the nature of singular fibres is still not adequately understood. The skeptics may reasonably doubt if special Lagrangians exist at all in the region  with large curvature inside the Calabi-Yau manifold. Even if they do exist, there is insufficient evidence that the fibration structure persists. This is maybe the weakest point of the SYZ conjecture.}.
 Beyond the torus fibration case, the moduli space of Lagrangian branes defined by the Mauer-Cartan equation typically has singularities, which is compatible with the fact that the moduli space of semistable objects in $D^b Coh(X^\vee)$ typically is very singular.

\item (\textbf{DT theory}, \ie the \textbf{counting} approach) In the special case of Calabi-Yau 3-folds, the moduli spaces/stacks of semistable objects have virtual dimension zero, so one can hope to extract enumerative invariants, which can be thought of as integrals over the moduli space, the most basic version being the Euler number. Donaldson-Thomas theory is a highly elaborate framework for extracting such numbers using virtual techniques, on the side of $D^b Coh(X^\vee)$. It is suggested (e.g in \cite[section 8.5]{Fukaya}) that one can assign DT invariants to the moduli space of semistable objects in the Fukaya category\footnote{This would presuppose suitable properness and algebraicity on the moduli space of special Lagrangian objects, which is a highly nontrivial claim related to the Thomas-Yau conjecture. For instance, in the SYZ special Lagrangian torus fibration case, compactifying the moduli space requires good understanding in the singular region, which is far beyond current knowledge.}, and these should be equivalent to the DT invariants on the mirror $X^\vee$. The role of Thomas-Yau conjecture is that in principle it transforms a problem about counting special Lagrangian objects, into a categorical framework involving Calabi-Yau $A_\infty$-categories and stability conditions, which then supposedly \footnote{modulo heroic efforts} feeds into the grand machinery of Kontsevich and Soibelman \cite{KontsevichSoibelman}\cite{KS2}\cite{KSsurvey1}\cite{KSsurvey2}. This DT perspective may be regarded as the ultimate goal of Thomas-Yau conjecture \footnote{As an analogy, the practical calculation of Donaldson's ASD instanton invariants depends largely on the solution of Hermitian-Yang-Mills equation on algebraic surfaces.}, and is close to the physical applications involving BPS states counting, although on the $A$-side it is the most removed from mathematical attempts.

\end{itemize}

It is worth emphasizing that taking stability questions into account requires one to go beyond homological mirror symmetry. While the Fukaya category (vis. a vis. $D^bCoh(X^\vee)$) depends only on the symplectic (holomorphic) data, the stability condition remembers information about the holomorphic volume form (polarisation class), which now sees the $B$-side ($A$-side). As such, neither the above strong version of the SYZ mirror symmetry, nor the DT mirror symmetry are consequences of homological mirror symmetry alone, as indicated already in the SYZ paper \cite{SYZ}.\footnote{In the words of SYZ, they required the full equivalence of the two type II string theories, including the full BPS spectrum.} A very general framework that encompasses all the three aspects of mirror symmetry is in Kontsevich and Soibelman \cite{KontsevichSoibelman}, but how to fit the $A$-side into this picture is wildly conjectural.

\subsubsection*{Caveats for differential geometers}

However, in the differential geometric literature on the Thomas-Yau conjecture, one often attempts to go beyond all these aspects of mirror symmetry, and tries to directly compare the space of Lagrangian submanifolds inside $X$, with the space of holomorphic vector bundles equipped with Hermitian metrics over $X^\vee$. This comparison must be treated with caution, because there is no general bijective correspondence between these two infinite dimensional spaces, and because mirror symmetry is properly a quantum phenomenon which is not fully captured by classical considerations. In fact, mirror symmetry only relates $Fuk(X)$ and $Coh(X^\vee)$ after taking the derived category, and there is no general reason for the heart of a preferred $t$-structure on the derived Fukaya category (e.g. the $t$-structure defined by some almost calibrated condition) to agree with $Coh(X^\vee)$. Having vaccinated the reader with these precautions, we shall regard such comparisons as useful formal analogies, and we consider those aspects with quantum interpretations as more reliable.

\subsection{Hermitian-Yang-Mills}\label{HYM}

Hermitian Yang-Mills (HYM) connections have long been established as the epitome of how \textbf{stability conditions} control the existence questions of geometric PDEs, but we shall attempt to say a few new words besides the customary hommage.
We consider Hermitian metrics $H$ on a holomorphic vector bundle $E\to X^\vee$ over a compact K\"ahler manifold $X^\vee$, \footnote{If we are interested in noncompact $A$-model target spaces, the mirror is in fact not a compact K\"ahler manifold. We hope the reader will excuse us on this issue, since mirror symmetry is only used as motivations in this paper.}  inducing the Chern connection $\nabla_H$ and the curvature $F_H\in \Omega^{1,1}(End(E))$. The HYM equation can be written as
\[
\omega_{X^\vee}^{n-1}\wedge \frac{\sqrt{-1}}{2\pi} F_H= \frac{ \mu(E)}{ \int_{X^\vee} \omega_{X^\vee}^n} Id_E\otimes \omega_{X^\vee}^n, \quad \mu(E)= \frac{ \int_{X^\vee} c_1(E)\wedge \omega_{X^\vee}^{n-1} }{  rk(E) }= \frac{ deg(E)}{rk(E) }.
\]
This implies the Yang-Mills inequation $\nabla_H^* F_H=0$, and in fact HYM connections are absolute minimizers of the Yang-Mills energy among all unitary connections on the Hermitian vector  bundle $E$. The $\mu$-stability (resp. semistability) means that for all proper coherent subsheaf $E'\subset E$, the slope $\mu(E')<\mu(E)$ (resp. $\mu(E')\leq \mu(E)$). The bundle $E$ is called $\mu$-polystable if it is a direct sum of $\mu$-stable bundles of the same slope.

The famous Donaldson-Uhlenbeck-Yau theorem says

\begin{thm}
On a compact K\"ahler manifold, the holomorphic bundle $E$ admits a HYM metric if and only if $E$ is $\mu$-polystable.
\end{thm}

\begin{rmk}
Certain parallels between HYM and special Lagrangians are already known to Thomas and Yau. The Chern connection $\nabla_H$ is analogous to the Lagrangian submanifold $L$ (with $U(1)$ local systems), the HYM equation as a first order equation on $\nabla_H$ is analogous to the special Lagrangian condition on $L$, the second order Yang-Mills equation is analogous to the minimal surface equation on $L$, and the Yang-Mills energy minimization is analogous to the volume minimization. If the Lagrangian $L$ is represented as the graph of an exact 1-form, then the potential function defining $L$ could be viewed as analogous to the Hermitian metric $H$. Finding a mirror analogue resembling the $\mu$-stability was one of Thomas and Yau's principal motivations.
\end{rmk}

\subsubsection*{$\mu$-stability and its wider context}

We now recall why the HYM equation implies $\mu$-semistability. Let $E'\subset E$ be a holomorphic subbundle (or more generally, a proper coherent subsheaf). A basic feature of holomorphic geometry is that pointwise curvature decreases in subbundles: the Chern curvature $F_{E',H}$ for the restricted Hermitian metric $H|_{E'}$ satisfies
\[
\sqrt{-1}F_{E',H} \leq \sqrt{-1}F_{E,H}|_{E'}.
\]
Wedging both sides with $\omega_{X^\vee}^{n-1}$, taking the trace, and integrating over ${X^\vee}$, we get
\[
\int_{X^\vee} c_1(E')\wedge \omega_{X^\vee}^{n-1} \leq \mu(E) rk(E'), 
\]
namely $\mu(E')\leq \mu(E)$, which is $\mu$-semistability. Although this argument is very transparent, we wish to summarize its key features:
\begin{itemize}
\item  Even though connections and curvatures make sense in a more general setting, we need the \textbf{integrability} of K\"ahler geometry to obtain pointwise \textbf{positivity}.\footnote{It would be interesting if the physicists can explain this positivity from supersymmetry, which is closely related to the K\"ahler condition.}

\item  To derive $\mu$-stability, one integrates over $X^\vee$, which can be interpreted as the \textbf{moduli space} of constant maps into $X^\vee$. \footnote{Path integrals on the topological B-model typically localizes to the moduli space of constant maps. This suggests a worldsheet interpretation, which will be more apparent on the mirror side.}

\item The input from complex geometry can be interpreted as a short exact sequence 
\[
0\to E'\to E\to E/E'\to 0,
\]
which has a categorical meaning in $D^b Coh(X^\vee)$ as a \textbf{distinguished triangle}.

\item 
The role of the K\"ahler form enters via \textbf{cohomological integrals}.

\item There is no need for the complex Monge-Amp\`ere equation. 

\end{itemize}
Most of these features are not specific to HYM connections, but similar arguments give rise to obstructions for a large class of PDEs involving holomorphic bundles, such as the deformed Hermitian Yang-Mills equation.

\begin{rmk}\label{Bridgelandflexible}
The $\mu$-semistability condition can be recast in terms of the central charge $Z(E)= -deg(E)+ \sqrt{-1}rk(E)$, as saying $\arg Z(E')\leq \arg Z(E)$ for all nonzero proper subsheaves $E'\subset E$.
It is worth emphasizing that except for the case of Riemann surfaces, $\mu$-stability does \emph{not} give rise to a  Bridgeland stability condition on $D^b Coh(X^\vee)$, since skyscrapper sheaves will generally have zero rank and zero degree, hence zero central charge. A similar but more subtle failure of Bridgeland stability happens in the context of the deformed Hermitian-Yang-Mills connections (\cf \cite[section 4]{Collins}). Such a failure does not spell doom for the PDE applications, nor for DT theoretic applications.\footnote{R. Thomas defined DT invariants for $\mu$-stability long before the insight of Bridgeland.}
Even though Bridgeland stability seems to be a plausible framework for special Lagrangians in the light of Joyce's proposal, it is probably advisable to maintain a more flexible attitude to stability conditions. 
\end{rmk}

\subsubsection*{Donaldson functional}

The reverse direction, that $\mu$-stability implies the existence of HYM connections, is the hard part of the subject, and a key ingredient is the \textbf{Donaldson functional}. We make the not very essential simplification that $c_1(E)=0$
. Donaldson \cite{DonaldsonHYM1}\cite{DonaldsonHYM} defined a functional $\mathcal{M}$ on the infinite dimensional space $\mathcal{H}$ of Hermitian metrics on the fixed bundle $E$
, by prescribing its first variation at any point $H\in \mathcal{H}$:
\begin{equation}\label{Donaldsonfunctional}
\delta \mathcal{M}=\int_{X^\vee} \Tr (\sqrt{-1}F_H \delta H H^{-1} )\wedge \omega_{X^\vee}^{n-1}.
\end{equation}
Obviously from the definition, the critical points are precisely the HYM metrics. Less obviously, this functional is well defined up to an additive constant fixed by the choice of a reference Hermitian metric $H_0$.  Different choices are related by
\begin{equation}\label{Donaldsonfunctionalchangeofreference}
\mathcal{M}_{H_0}(H)= \mathcal{M}_{H_0'}(H)+ \mathcal{M}_{H_0}(H_0').
\end{equation}

The space $\mathcal{H}$ can be formally assigned a Riemannian structure with non-positive curvature:
\begin{equation}\label{DonaldsonRiemstr}
|\delta H|^2= \int_{X^\vee} \Tr (\delta H H^{-1})^2 \wedge \omega_{X^\vee}^n.
\end{equation}
 The geodesics in $\mathcal{H}$ are given by $e^{tA}H$ where $A$ is self adjoint with respect to $H$, and the Donaldson functional is convex along geodesics.

\subsubsection*{Proof idea of Donaldson-Uhlenbeck-Yau theorem}

In very sketchy terms, one standard proof of the Donaldson-Uhlenbeck-Yau theorem (closest to Simpson's approach \cite{Simpson}) proceeds via the heat flow. It has two principal steps:

\begin{itemize}
\item Consider the HYM heat flow
\[
\partial_t H H^{-1} = -\sqrt{-1} \Lambda F_H .
\]
Using certain parabolic maximum principles, one proves long time existence by showing that all derivatives of $H$ remain bounded for any given finite time. Furthermore,  $\sup_{X^\vee} |\Lambda F|$ and $\norm{F}_{L^2}$ are non-increasing in time, so remain uniformly bounded for all time. This step does not use stability.

\item
The Donaldson functional is non-increasing in time almost by definition, so has a uniform upper bound for all time. Together with the pointwise bound on $\Lambda F$, which is like a Laplacian bound, one eventually shows that if $H(t)$ fails to be $L^\infty$ bounded for all time, then there exists an $L^2$-subbundle of $E$ with destabilizing properties, which is then interpreted algebraically as a subsheaf. (Roughly, the destabilizing sheaf comes from the eigensubspaces of $E$ corresponding to the small eigenvalues of $H(t)$ with respect to a fixed reference metric; compare the variational viewpoint below.) The stability condition rules out this case; one then shows that $H(t)$ actually converges smoothly at infinite time to a solution of the HYM equation.  
\end{itemize}

\begin{rmk}
An alternative approach by Donaldson \cite{DonaldsonHYM} in the projective manifold case, is also based on the flow method, but uses a dimensional induction in which the stability condition appears indirectly through alternative algebro-geometric characterisations. The approach of Uhlenbeck and Yau \cite{UhlenbeckYau} uses the continuity method instead, where the stability condition appears in a way similar to the above.

\end{rmk}

We now discuss the variational perspective to the HYM equation, even though no proofs have been constructed along such lines. One would try to compactify $\mathcal{H}$ in some weaker topology (which is not known),\footnote{What is known is how to compactify the space of K\"ahler potentials via psh functions, see Boucksom \cite{Boucksom} for its fantastic application to K\"ahler-Einstein metrics. In that context, the boundary at infinity is related to non-archimedean geometry.}  extend the Donaldson functional to this compactification, attempt to find a minimizer of the functional, and then prove its regularity. Since the Donaldson functional is convex, it is natural to expect the existence of minimizer is equivalent to the properness of $\mathcal{M}$, or roughly equivalently $\mathcal{M}$ should grow at the infinity of $\mathcal{H}$.

Suppose now that $E$ fits into an extension sequence
\[
0\to E_1\to E\to E_2\to 0.
\]
Take arbitrary Hermitian metrics $H_E$ and $H_{E_2}$ on $E$ and $E_2$ respectively,\footnote{The fact that the choice of Hermitian metrics will not ultimately matter is an expected feature, analogous to  the relation between K\"ahler potentials and non-archimedean potential theory.} and regard $H_{E_2}$ as a semi-Hermitian metric on $E$. We can then produce a 1-parameter family of Hermitian metrics on $E$, via $H(s)=e^{-s} H_{E}+ H_{E_2}$. For $s\gg 1$, the metric $H(s)$ can be understood as equal to $e^{-s} H_{E}$ when restricted to $E_1$, and almost equal to $H_{E_2}$ on the orthogonal complement of $E_1$. In terms of bundles with connections, in the limit $s\to \infty$ we get $E_1\oplus E_2$ with the Chern connection for $H_E|_{E_1}\oplus H_{E_2}$. As such, it is an easy exercise to show that to leading order
\[
\mathcal{M}(H(s)) \sim  -2\pi s \int_{X^\vee} c_1(E_1)\wedge \omega_{X^\vee}^{n-1}.
\]
Recall we have assumed $c_1(E)=0$ to simplify the definition of the Donaldson functional. Thus the subbundle $E_1$ destabilizes $E$, precisely when the degree of $E_1$ is positive, so $\mathcal{M}(H(s))$ goes to $-\infty$ along this 1-parameter family.

If one wants to turn the variational approach into an actual proof, one needs to further show that all possible ways to approach the infinity of (the metric completion of) $\mathcal{H}$ can be approximated by such 1-parameter families of algebraic origin. For our motivational purpose, it suffices to emphasize the following conceptual points:
\begin{itemize}
\item  The interesting limiting behaviour of the Donaldson functional occurs at an \emph{infinite distance} boundary of (the metric completion of) $\mathcal{H}$. This sits well with the non-positive Riemannian curvature of $\mathcal{H}$. 

\item
The \emph{stability condition controls the asymptotic behaviour} of the Donaldson functional near the boundary of $\mathcal{H}$. 
 
\end{itemize}

\begin{rmk}
The nonlinear analysis concerning special Lagrangians is more difficult than HYM. For instance, the finite time singularities of the LMCF are inevitable. The HYM equation can be viewed as a toy model which shares some, but by no means all, of the high level features with the special Lagrangian equation.
\end{rmk}

\subsubsection*{Infinite dimensional GIT picture, and possible lack of mirror analgoue}

The HYM equation famously fits into a formal geometric invariant theory (GIT) framework. For this, we slightly change viewpoint, and consider the bundle $E$  equipped with a fixed Hermitian structure, while the $(0,1)$-connection $\bar{\partial}_E$  encoding the holomorphic structure is allowed to vary. Each $\bar{\partial}_E$ uniquely determines the Chern connection $\nabla_E$. The group of complex gauge transformations $GL(E,\C)$ acts on the space of integrable $(0,1)$-connections, via $ \bar{\partial}_E\mapsto g\circ \bar{\partial}_E\circ g^{-1}$. This action is analogous to a complex reductive group action on a finite dimensional K\"ahler manifold. The subgroup $U(E)$ of unitary gauge transformations is analogous to the maximal compact subgroup. The space $GL(E,\C)/U(E)$ can be formally identified with the space of Hermitian metrics on $E$. The HYM equation arises naturally from considerations of the moment map, and the Donaldson-Uhlenbeck-Yau theorem can be formally motivated from this picture \cite{DonaldsonHYM}.

This kind of infinite dimensional GIT framework has successfully suggested the answer in many problems within K\"ahler geometry, so it is only natural that many people have attempted to find an analogue suitable for special Lagrangian geometry. One such attempt is as follows. Thomas \cite[section 3]{Thomas} considered the space 
\[
\mathcal{Z}=\{ (L,A): L\subset X \text{ is Lagrangian}, A \text{ is a flat $U(1)$-connection on } L \}
\]
(not up to gauge equivalence!). The tangent space is $Z^1(L)\oplus Z^1(L)$, where $Z^1(L)$ is the space of closed 1-forms on $L$. This suggests an almost complex structure on $\mathcal{Z}$ \[
J= \begin{pmatrix}
0 & 1 \\
-1 & 0
\end{pmatrix}.
\]
With some hesitation,\footnote{Thomas was aware of the possible objections, and did not use the GIT analogy as the principal basis of his proposal.}
Thomas attempted to complexify the Hamiltonian group action into a complex infinite dimensional group action. Unfortunately, there seems to be no natural way to do this in general, and $J$ is not quite an integrable complex structure. On the other hand, the moduli space of special Lagrangians with $U(1)$-local systems does have a natural complex structure induced from $J$, which is however na\"ive in the sense that the complex structure of the moduli space of Lagrangian branes is subject to further quantum corrections due to holomorphic curves, known also as `worldsheet instantons' \cite{Auroux}.

There seems to be no agreed interpretation, but in the author's view, this suggests the infinite dimensional GIT framework is itself inadequate for the purpose of special Lagrangian geometry.\footnote{Thomas's suggestion is not the only possible way to achieve a GIT analogy. However, the other proposals  \cite{DonaldsonSL}\cite{Lotay} do not exhibit the mirror analogy in the same intuitively plausible way.} As we explained in section \ref{Mirrorsymmetry}, the main predictions of mirror symmetry are quantum in nature, and one should be cautious about taking an overly classical perspective. From our perspective, one underlying reason why the infinite dimensional GIT framework is successful in K\"ahler geometry, is that on the B-side one does not see the worldsheet instanton effect directly. On the A-side we have no such luxury.

As a word of console, although GIT is very good at \emph{suggesting} the correct stability conditions for a PDE problem in K\"ahler geometry, it is almost never involved in the actual \emph{proofs} of existence and uniqueness results for PDEs.

\subsection{Deformed Hermitian Yang-Mills}\label{dHYM}

Much recent research concentrates on a more nonlinear cousin of HYM, known as the \textbf{deformed Hermitian Yang-Mills} equation (dHYM), expertly surveyed in \cite{Collins}. Although the techniques involved in dHYM are more akin to other areas of K\"ahler geometry, such as K\"ahler-Einstein metrics and the J-equation \footnote{Indeed, the recent breakthrough of Gao Chen \cite{Chen} on dHYM is largely based on the methods he developed for the J-equation. The work of Collins et al \cite{CollinsJacobYau} has strong analogy with the homogeneous complex Monge-Amp\`ere equation important for CSCK metrics.}, its motivation is in part to find an improved mirror analogue of special Lagrangians, with the distant goal of constructing the Bridgeland stability on the B-side of the mirror.

The equation can be motivated differential geometrically from \textbf{semiflat mirror symmetry} \cite{CollinsJacobXie}. To start, one imposes \emph{toric symmetry}, so that over a local $n$-dimensional base $D$, the $A$ and $B$ sides are respectively the $T^n$-bundles $T^*D/\Lambda$ and $TD/\Lambda^\vee$ equipped with the canonical symplectic/complex structure, where $\Lambda$ and $\Lambda^\vee$ are dual lattices, so that the torus fibres are naturally dual. A \emph{section} $L$ of $T^*D/\Lambda$ over $D$ fibrewise determines a flat $U(1)$-connection on the corresponding fibre of $TD/\Lambda$, and the Lagrangian condition on $L$ is equivalent to these fibrewise connections fitting into a connection $A$ on a line bundle $E$ over $TD/\Lambda$, with the integrability condition $F_A^{0,2}=0$. Now a Hessian metric on $D$ induces a K\"ahler structure on both sides. The condition for $L$ to be a special Lagrangian of phase $\hat{\theta}$, translates into the mirror condition on the holomorphic line bundle:
\[
\text{Im} ( e^{-i\hat{\theta}}(\omega_{X^\vee}+ F_A)^n )=0.
\]
Here the curvature $F_A$ of the $U(1)$-connection is an imaginary valued $(1,1)$-form. Up to rescaling $\omega_{X^\vee}$,  the global version is the dHYM equation for a Hermitian metric on a holomorphic line bundle $E$ over $X^\vee$, thought of as a PDE on the potential function $\phi$:
\begin{equation}
\text{Im} ( e^{-i\hat{\theta}}(\omega_{X^\vee}+ \sqrt{-1}\alpha_\phi)^n )=0, \quad [\alpha_\phi=\alpha+ \sqrt{-1}\partial \bar{\partial} \phi]\in -c_1(E).
\end{equation}
For an alternative view, we can pointwise simultaneously diagonalize $\omega_{X^\vee}$ and $\alpha_\phi$, to extract the eigenvalues $\lambda_i$,
\[
\omega_{X^\vee}= \frac{\sqrt{-1}}{2} \sum dz_i\wedge d\bar{z}_i, \quad \alpha_\phi=\frac{\sqrt{-1}}{2} \sum \lambda_i dz_i\wedge d\bar{z}_i.
\]
The dHYM equation is then
\begin{equation}
\Theta_\omega(\alpha)=\sum_i \arctan \lambda_i= \hat{\theta} \mod \pi\Z.
\end{equation}
Thus the equation actually separates into several discrete branches depending on the choice of $\hat{\theta}$, and changing $\hat{\theta}$ by $k\pi$ can drastically alter the behaviour of the equations. Replacing $\hat{\theta}$ by $-\hat{\theta}$ (so $E$ is replaced by its dual) does not change the problem, so without loss of generality $0\leq \hat{\theta}< \frac{n\pi}{2}$. There are two significant limiting cases:
\begin{itemize}
\item When $\lambda_i\gg 1$ for any $i$, so $\hat{\theta}\approx \frac{n\pi}{2}-\sum \frac{1}{\lambda_i}$, the equation becomes approximately the J-equation
\[
\alpha_\phi^{n-1} \wedge \omega_{X^\vee}= \frac{ \int \alpha^{n-1}\wedge \omega_{X^\vee} }{ \int \alpha^n  } \alpha_\phi^n.
\]
Most mathematical works on dHYM assume some lower bound such as $\hat{\theta}>\frac{(n-2)\pi}{2}$, which may be morally interpreted as being near this large phase limit.

\item 
When $|\lambda_i|\ll 1$ for any $i$, so $\hat{\theta}\approx \sum \lambda_i$, the equation is approximately the line bundle case of the HYM equation
\[
\alpha_\phi \wedge \omega_{X^\vee}^{n-1}= \frac{ \int \alpha\wedge \omega_{X^\vee}^{n-1} }{ \int \omega_{X^\vee}^n  } \omega_{X^\vee}^n.
\]
This is also known as the large volume limit, in the sense that the K\"ahler metric length scale is much larger than the curvature scale of the  bundle.
\end{itemize}

\subsubsection*{Towards a Bridgeland stability condition}

The existence of dHYM solutions has algebraic obstructions. Let $V$ be a $p$-dimensional subvariety of $X^\vee$, with $1\leq p<n$.
Under the large phase assumption $\hat{\theta}> \frac{(n-2)\pi}{2}$, a pointwise consideration of eigenvalues shows \cite[Prop. 3.4]{Collins}
\[
\text{Im} \big(     \int_V  e^{ -\sqrt{-1}(\hat{\theta}- (n-p)\frac{\pi}{2})  }  (\omega_{X^\vee}+ \sqrt{-1}\alpha_\phi)^n   \big)>0.
\]
More suggestively, denote
\begin{equation}
Z_{V}(E)= -\int_V e^{ -\sqrt{-1}\omega_{X^\vee} } ch(E) \in \R_{>0} e^{\sqrt{-1}\phi_V(E)}, \quad \phi_V(E)\in (  0   , \pi  ),
\end{equation}
so the algebraic obstruction can be rewritten as
\[
\text{Im} (\frac{Z_{V}(E) }{ Z_{X^\vee}(E) }) >0, \quad \ie \quad \phi_V(E)>\phi_{X^\vee}(E)= \hat{\theta}- \frac{(n-2)\pi}{2}.
\]
This bears some resemblance to a Bridgeland stability condition with central charge $Z_{X^\vee}(E)$
, even though on the technical level there are some discrepancies 
\cite[section 3.2]{Collins}. In the simplest understood examples, such as the blow-up of $\mathbb{CP}^2$ in a point, the obstruction criterion for dHYM seems to refine Bridgeland stability, in the sense that every dHYM stable object is Bridgeland stable, but not conversely. It is not entirely clear how to interpret this (\cf Remark \ref{Bridgelandflexible}).

\subsubsection*{Main achievements on dHYM}

Some of the most significant results on the dHYM equation are:

\begin{itemize}
\item There is a well developed formal GIT picture \cite[section 2]{CollinsJacobYau}\cite{Collins1}\cite{Collins2}, including a complexified group action on an infinite dimensional K\"ahler manifold, a moment map, a negatively curved symmetric space, and a Donaldson type functional which is convex along geodesics. 

\item The geodesic equation with prescribed boundary data have $C^{1,1}$-solutions and viscosity solutions under a large phase assumption \cite[Thm 2.7]{Collins}. This leads to nontrivial algebraic obstructions \cite[section 3]{Collins}.

\item Assume large phase $\hat{\theta}>\frac{(n-2)\pi}{2}$ and  the existence of a  subsolution in the suitable sense, then the dHYM equation can be solved \cite[Thm 5.2]{Collins}.

\item Assume large phase $\hat{\theta}>\frac{(n-2)\pi}{2}$, then the existence of dHYM solution can be formulated in terms of algebraic obstruction conditions \cite{Chen}.
\end{itemize}

\subsubsection*{Will dHYM lead to the mirror Bridgeland condition?}

Despite substantial progress, many essential difficulties still need to be overcome before  the dHYM equation can give rise to a Bridgeland stability condition on $D^b Coh(X^\vee)$:
\begin{itemize}
\item Can one relax the \emph{large phase assumption}? 
\item Is there a generalisation of dHYM equation to \emph{higher rank vector bundles}? (Currently, there is no well established PDE, let alone how to solve it.)
\item Can this story be extended to \emph{complexes} of vector bundles? 
\item How can one compare the answer with the A-side of the mirror?
\end{itemize}

Time will tell how far one can push in this program, but we would like to momentarily play the skeptic's advocate:
\begin{itemize}
\item The differential geometric motivation \footnote{There is an independent physics motivation for dHYM from the Dirac-Born-Infeld action.
The DBI action is however not an exact result, but depends on the assumption that certain derivative terms of the curvature can be ignored. Unlike the A-model side where the central charge $Z(L)=\int_L \Omega$ is believed to hold exactly, on the B-side the central charge formula $Z_{X^\vee}(E)=-\int_{X^\vee}e^{-\sqrt{-1}\omega_{X^\vee}}ch(E)$ is believed to be subject to worldsheet instanton corrections.
} of dHYM assumes semiflat ambient K\"ahler metrics. Such metrics only arise naturally if the manifold has toric symmetry, or as a good approximate description for degenerating Calabi-Yau metrics near the large complex structure limit/large volume limit. Near this limit, the dHYM equation may be an improvement on the HYM equation as the mirror version of special Lagrangians.
But far away from such limits, the instanton corrections cannot be ignored, and indeed a general K\"ahler metric has no toric symmetry.

\item
The reliance on the large phase assumption is a possible indication of a breakdown once we move too far from a large phase limit. In the closely related special Lagrangian graph equation, relaxing the phase condition is known to result in rather severe singularities for the viscosity solutions.

\item
Being the section of a torus fibration is a serious assumption on the topology of a Lagrangian.

\item 
The difficulty with higher rank vector bundles is a possible indication that $Coh(X^\vee)$ may be not the natural abelian subcategory of $D^bCoh(X^\vee)$ suited for dHYM. Furthermore, there is no a priori guarantee for arbitrary stability conditions to correspond to PDEs,\footnote{If a Bridgeland stability condition arises from a PDE, there is no a priori guarantee that its deformations also come from PDEs.} or to have any classical geometrical interpretation at all. In particular, the $B$-side mirror to the hypothetical special Lagrangian stability condition may well be an abstract stability condition with no particular PDE interpretation.

\item
Due to the lack of imagination, it is hard to see how PDEs can know about the degree shifts of a complex of vector bundles, and how it can detect homotopy equivalence of complexes.

\end{itemize}

If we take this skeptic view, then we do not expect an exact comparison between the  hypothetical Bridgeland stabilities on both sides of the mirror, without adding very substantial assumptions such as toric symmetry. But when HYM shares the same qualitative features as dHYM, which admit a mirror interpretation, it lends more plausibility for these features to appear also on the special Lagrangians.

\subsection{Extensions and wall crossing}\label{Extensionwallcrossing}

Part of Thomas and Yau's insights is that one should focus on the \emph{categorical} aspect of mirror symmetry when we look for an analogy between the A-side and the B-side. Their starting observation is that \emph{Lagrangian connected sums are analogous to bundle extensions} \cite[section 3,4]{Thomas}.

\subsubsection*{Extension bundle vs. Lagrangian connection sum}

An essential aspect of $Coh(X^\vee)$ is that new bundles can be constructed from extensions of known bundles $E_1, E_2$, namely
\[
0\to E_1\to E\to E_2\to 0.
\]
Such extension sequences are classified by the complex vector space $\text{Ext}^1(E_2, E_1)$. Due to the $\C^*$-scaling, the choice of $E$ is parametrised by the projective space $\mathbb{P}(  \text{Ext}^1(E_2, E_1) )$.
In general, extensions are \emph{not symmetric} in $E_1$ and $E_2$. The extensions
\[
0\to E_2\to E\to E_1\to 0
\]
are classified by $\text{Ext}^1(E_1, E_2)$, which is a quite different space. Extensions can also be viewed as distinguished triangles in $D^bCoh(X^\vee)$.

In the mirror picture, exact sequences do not make a priori sense, but one can talk about distinguished triangles, whose geometric sources are the graded Lagrangian connected sums $L_1\#L_2$, fitting into a distinguished triangle
\[
L_1\to L_1\#L_2\to L_2\to L_1[1].
\]
Such distinguished triangles are classified by $HF^1(L_2,L_1)$.\footnote{The caveat is that unlike bundles, the neck length of the Lagrangian connected sum cannot be arbitrarily large.} Again there is a scaling symmetry related to the neck size of the Lagrangian connected sum, and there is an asymmetry between $L_1\#L_2$ and $L_2\#L_1$.

This analogy is a prime example of \emph{homological mirror symmetry}. On either side, only pure complex geometry/pure symplectic geometry appears.

\subsubsection*{Wall crossing}

Wall crossing in the categorical context refers to the following phenomenon when the stability condition varies in a 1-parameter family, with central charges $Z_t$. For $t>0$, the objects $E_1, E_2, E$ in the extension sequence are all stable, so necessarily
\[
\arg Z_t (E_1) < \arg Z_t (E)< \arg Z_t (E_2), \quad Z_t(E)= Z_t(E_1)+ Z_t(E_2) .
\]
At $t=0$, the phase angles become equal, and for $t<0$, the phase angle inequality is reversed, and $E$ becomes unstable. This phase alignment occurs on a codimension one locus in the space of stability condition, and thus they are called walls. Every extension sequence potentially gives rise to a wall, and the walls can be dense in general.

A notable special case is $\mu$-stability of bundles for a 1-parameter family of K\"ahler classes $[\omega_t]$, and the slopes $\mu(E_1), \mu(E_2), \mu(E)$ become equal precisely for $t=0$. On one side of the wall, the HYM connections exist on $E_1, E_2, E$, and on the other side $E$ becomes unstable and no longer admits any HYM connection.

Thomas \cite{Thomas} interpreted a gluing construction of Joyce as the mirror analogue of the wall crossing phenomenon for bundles.\footnote{It is quite remarkable that Thomas and Yau knew before Bridgeland, that stability conditions  make sense categorically beyond $\mu$-stability, and the mirror of the hypothetical special Lagrangian stability condition does not need to be $\mu$-stability.} In the simplest case, let $n\geq 3$, fix a symplectic structure $\omega$, and vary the holomorphic volume form $\Omega_t$ in a 1-parameter family while keeping the almost Calabi-Yau condition. Let $L^1_t, L^2_t$ be smooth special Lagrangians with respect to $\Omega_t$ with phase $\theta_t^1$ and $\theta_t^2$, intersecting transversely at precisely one point $p_t$, with Floer degree $\mu_{L_t^2, L_t^1  }(p_t)=1$, such that $\theta_t^2-\theta_t^1$ increases past zero at $t=0$, and in particular $\theta_0^1=\theta_0^2$. Thus at $t=0$, the 
tangent planes can be put into the standard form inside $\C^n$:
\[
 T_p L_0^1=(e^{i\phi_1}, \ldots e^{i\phi_n}) \R^n,  \quad T_p L_0^2= \R^n, \quad  0<\phi_k<\pi,\quad \sum \phi_k=\pi,
\]
\[
\omega= \frac{\sqrt{-1}}{2} \sum dz_k\wedge d\bar{z}_k, \quad e^{-i\theta_0^1}\Omega= a^n dz_1\wedge\ldots  dz_n, \quad a>0.
\]
\footnote{The constant $a$ comes from the almost Calabi-Yau structure, and $a=1$ in the Calabi-Yau case.} 
This provides the appropriate framing data to topologically glue in a Lawlor neck $L_{\phi, A}$ (\cf section \ref{Lawlor}) to desingularize $L_t^1\cup L_t^2$ for small $t$, so that the glued Lagrangian has the topology $L_t^1 \#L_t^2$. Joyce \cite[Thm 9.10]{JoyceSLsurvey} shows that for each small $t>0$, the glued Lagrangian can be perturbed into a special Lagrangian of phase $\theta_t\approx \theta_0^1$. As $t\to 0$, these special Lagrangians converge as currents to $L_t^1\cup L_t^2$, while for $t<0$ this gluing strategy does not produce any new special Lagrangian.

 The central charges are 
\[
Z(L_t^i)= \int_{L_t^i} \Omega_t^i= R_t^i e^{\sqrt{-1}\theta_t^i}, \quad Z(L_t^1\#L_t^2)= Z(L_t^1)+ Z(L_t^2)= R_t e^{\sqrt{-1}\theta_t}.
\]
In particular $R_t^1 \sin (\theta_t^1-\theta_t)= - R_t^2\sin(\theta_t^2-\theta_t)$. The asymmetry between $t>0$ and $t<0$ in Joyce's gluing construction, comes from an approximate formula for the Lawlor neck parameter valid for small $t$ \footnote{For a heuristic short derivation see \cite[section 6]{Joycecounting}. Beware that Joyce's Lagrangian connected sum has the opposite convention.}
\[
R_t^1 \sin (-\theta_t^1+\theta_t)= a^m A>0.
\]
Thus $\theta_t^1<\theta_t$ is needed in the gluing construction. This is strongly reminiscent of a Bridgeland stability condition. What happens when $t$ crosses zero can be interpreted as wall crossing.

To summarize, the mirror analogy of wall crossing phenomenon is of categorical nature. However, it goes beyond homological mirror symmetry as soon as it involves the stability conditions, and the new problems involve analytical aspects, beyond purely topological issues. The similarity between the Joyce gluing and the bundle case is a strong motivation for Thomas and Yau. However, the Joyce analysis is only valid in a  perturbative regime, and what is missing here is a non-perturbative understanding of when the $D^b Fuk(X)$ class of the Lagrangian connected sum can admit special Lagrangians.

%



\subsection{Space of almost calibrated Lagrangians}\label{Solomon}

We now return to the A-side of the mirror, and describe the work of J. Solomon \cite{Solomon}\cite{Solomon2}, generally accepted as the canonical picture on the subject.
While dHYM is motivated by the dream of a correspondence (semiflat mirror symmetry) between the two sides of the mirror at the level of classical objects, Solomon is interested in the structural similarities between the infinite dimensional spaces involved in the mirrors, and especially in finding analogues for the HYM equation.\footnote{Solomon's work predates the substantial works on the dHYM equation.} The reader is thus invited to keep in mind the comparison with section \ref{HYM}.

Let $\mathcal{L}^+$ be the space of almost calibrated compact immersed Lagrangians in an almost Calabi-Yau manifold $(X,\omega, \Omega)$,
\[
\mathcal{L}^+= \{  \iota: L\to X | -\frac{\pi}{2}<\theta<\frac{\pi}{2}      \}.
\]
For the benefit of intuition, we shall loosely identify the Lagrangian immersion with its image.
Solomon \cite{Solomon2} considers the space $\mathcal{O}$ of Lagrangians which are exact isotopic (aka. local Hamiltonian isotopic)\footnote{Exact isotopies of immersed Lagrangians differ from global Hamiltonian isotopies, in that the Hamiltonian function on the Lagrangians may depend on the local sheets, so may not always extend smoothly to the ambient space. } within $\mathcal{L}^+$ to a fixed Lagrangian. It should be borne in mind that unlike the space $\mathcal{H}$ of Hermitian metrics on a bundle, the space $\mathcal{O}$ may have very nontrivial topology; we call its universal cover $\tilde{\mathcal{O}}$. The formal deformations of the Lagrangians are given by the Hamiltonian functions $h: L\to \R$, up to the ambiguity of an additive constant. Solomon proposes to fix the constant by the normalisation condition $\int_L h\text{Re}\Omega=0$, and assigns a formal Riemannian metric on 
\[
T_L \mathcal{O}= \{ h: L\to \R |   \int_L h\text{Re}\Omega=0        \},
\]
via the formula
\begin{equation}
\langle h, k\rangle = \int_L h k \text{Re}\Omega.
\end{equation}
This is positive definite because 
 $\text{Re}\Omega$ is a volume form on $L$ by almost calibratedness.
The main result of \cite{Solomon2} is a computation on the Riemannian curvature of $\mathcal{O}$, which is found to be \emph{non-positively curved}, similar to the HYM setting. A further paper \cite{Solomon3}  computes the geodesic equation with respect to this formal metric, and reinterprets a geodesic between $L_1,L_2\in \mathcal{O}$ in terms of a 1-parameter family of special Lagrangians (of phase $\frac{\pi}{2}$ instead!) with boundary on $L_0\cup L_1$.

Another major aspect of Solomon's work is to look for an analogue of the Donaldson functional. The definition of this Solomon functional does not really require the almost calibrated condition, even though some of the main properties do. Choose some appropriate $\hat{\theta}\in (-\frac{\pi}{2}, \frac{\pi}{2})$ so that
 $\int_L \text{Im}(e^{-i\hat{\theta}}\Omega) =0$. Now take a 1-parameter family of Lagrangians $L_t$ in $\mathcal{O}$ with associated Hamiltonian functions $h_t: L_t\to \R$. Solomon defines
\begin{equation}
\mathcal{S}= \int_0^1 dt \int_{L_t} h_t \text{Im}(e^{-i\hat{\theta}}\Omega)   .
\end{equation}
His main theorem \cite{Solomon} is 

\begin{thm}\cite{Solomon}
The functional $\mathcal{S}$ is independent of Hamiltonian deformations of the path of Lagrangians fixing the two ends. In particular, by fixing the starting Lagrangian $L_0$, we obtain a functional $\mathcal{S}$ of the endpoint Lagrangian, which is well defined on the universal cover $\tilde{\mathcal{O}}$.
\end{thm}

Like the Donaldson functional, this well definition is nontrivial. It is however obvious that the critical points in $\mathcal{O}$ are precisely special Lagrangians of phase $\hat{\theta}$.
Furthermore,

\begin{thm}
\cite{Solomon} Assume $\hat{\theta}=0$. Then the second variation of $\mathcal{S}$ at a critical point is positive semidefinite. Furthermore, along a geodesic with respect to Solomon's formal Riemannian metric, the functional $\mathcal{S}$ is convex.
\end{thm}

The analogy with Donaldson's picture in section \ref{HYM} should be quite clear.

\subsubsection*{Limitations}

Unlike Thomas and Yau who based their bet primarily on the Floer theoretic or categorical aspects, which are closer to the quantum world of topological field theories, Solomon's picture is predominantly classical, and its chief limitation comes from fixing the topological type of the Lagrangian:

\begin{itemize}
\item There is no appearance of the brane structure, or the role of holomorphic curves.

\item Solomon works with exact isotopic Lagrangians, but the Thomas-Yau argument suggests it is more natural to work in a derived Fukaya category class.

\item The Solomon functional is only  well defined by passing to a highly nontrivial universal cover. In the very special case where $\omega$ and $\text{Im}(e^{i\hat{\theta}}\Omega)$ are exact forms on $X$, Solomon gave a formula \cite[Thm 1.3]{Solomon} that shows his functional is well defined on $\mathcal{O}$. We view the exactness on $\text{Im}(e^{i\hat{\theta}}\Omega)$ as too strong an assumption for applications.

\item The infinite dimensional Riemannian structure is incomplete in a much more severe way compared to the B-side analogues. This means that in non-pathological examples,
 we can reach the boundary of the exact isotopy class within finite distance in the Solomon metric, such that the Solomon functional remains finite.

This is geometrically very significant. In the LMCF approach, this would strongly suggest the formation of finite time singularity, which is a major difference with the HYM case. In the variational viewpoint, this incompleteness would negate all the favourable arguments from the convexity of the functional and the non-positivity of curvature, and suggest instead that the exact isotopy class is not an adequate framework for finding special Lagrangians. We will discuss later that a more promising variational framework needs to incorporate Lagrangians from the same \emph{derived Fukaya category class}, not just the same exact isotopy class.

\end{itemize}

\subsubsection*{Exact isotopy class versus derived category class}

The example below is closely related to the most symmetric case of the Lawlor necks (\cf section \ref{Lawlor}). It is also morally related to the Lawlor neck pinching singularity in the Joyce program \cite[section 3.5]{Joyceconj}.

\begin{eg}\label{Lagsumexample}
Consider two almost calibrated Lagrangians $L^1, L^2$ with a unique intersection point $p$, and there is a Darboux chart around $p$ modelled on $B_1\subset \C^n$, such that inside the chart $L^1, L^2$ the local setup is
\[
 L^1= e^{i\pi/n}\R^n, \quad L^2= \R^n, \quad \omega=\frac{\sqrt{-1}}{2} \sum dz_k\wedge d\bar{z}_k, \quad \Omega\approx \prod dz_k.
\]
Let $(L_t)_{0<t\ll 1}$ be a 1-parameter family of Lagrangian connected sums with neck length $O(t)$, which all agree with $L_1\cup L_2$ except in a compact subset in $B_1$. Inside $B_1$, we take the ansatz
\[
L_t=\{  ( t\gamma(s)x_1,\ldots, t\gamma(s) x_n        ) |x_1^2+\ldots x_n^2=1  \},
\]
where the curve $\gamma(s): \R\to \C$ can be chosen so that $L_t$ is almost calibrated and agrees with $L_1\cup L_2$ outside $B_{1/2}$. Clearly, $L_t$ are related by scaling inside $B_1$. The Hamiltonian vector field along $L_t$, which is really a section of $(TX/TL_t)|_{L_t}$, agrees with
$
 (\gamma(s)x_1,\ldots, \gamma(s) x_n)  
$
inside $B_1$ and is zero outside.\footnote{This is consistent because near the boundary of $B_1$, the position vector is a tangent vector of $L$, so vanishes in the quotient $TX/TL_t$.} The corresponding Hamiltonian function $h_t$ is $t^2$ times a smooth function of one variable $s$; a small caveat is that $h_t$ converges to two generally different constants along $L_1$ and $L_2$. Thus it takes finite distance in the Solomon metric to reach the limit $t\to 0$, and the Solomon functional remains finite, but the topology changes from $L_t$ to $L^1\cup L^2$.

\end{eg}

\begin{rmk}
The Hamiltonian functions $h_t$ along $L_t$ can be extended to global functions on $X$, with \[
\norm{h_t}_{C^0}=O(t^2), \quad \norm{dh_t}_{C^0} =O(t), \quad \norm{\nabla^2 h_t}_{L^\infty} \leq C.
\]
But in the $t\to 0$ limit, the second derivatives fail to be continuous at the origin, and indeed $L_t$ changes topology in the limit.

\end{rmk}

In this example the essential failure is the breakdown of smoothness. In view of Joyce's program, this suggests that the remedy is to allow for (Floer theoretically unobstructed) almost calibrated Lagrangians connected to each other not just by exact isotopies, but also surgeries such as Lagrangian connected sums. In these transitions the derived category class of the Lagrangian is unchanged, and the Thomas-Yau argument suggests the $D^bFuk(X)$ class is a natural framework to look for special Lagrangian representatives. The following fundamental question is thus relevant for the compatibility between the geometric and the categorical perspectives:

\begin{Question}\label{geometricconnectedness}
When are two almost calibrated unobstructed Lagrangian branes isomorphic in $D^b Fuk(X)$ connected by exact isotopies with surgeries?
\end{Question}

\begin{rmk}
The Joyce program suggests that running the LMCF would result in a sequence of exact isotopies and surgeries, to connect the initial Lagrangian to its infinite time limit, which one hopes to be the unique representative of the Harder-Narasimhan decomposition. In the almost calibrated case, there is no `collapsing zero object' in this process for homological reasons, so the surgeries should be continuous in the geometric measure theory sense. Since any two such Lagrangians within the same $D^bFuk(X)$ class are expected to flow to the same limit, they are supposedly connected to each other through a continuous family of unobstructed Lagrangians.
\end{rmk}

\subsection{Totally real geometry}\label{LotayPacini}

There is another interesting framework due to Lotay and Pacini \cite{Lotay} \cite{LotayPacini2}, which makes the holomorphic curves appear on the forefront, and exhibits good analogy with the classical GIT picture, by enlarging the space of Lagrangians into the space of \textbf{totally real submanifolds}.

Let $(M,\omega, J)$ be a K\"ahler manifold. A real $n$-dimensional submanifold $L$ is called totally real, if at every point $T_p L$ is transverse to $JT_pL$. Lotay and Pacini introduce a formal principal bundle $\mathcal{P}\to \mathcal{T}$, where $\mathcal{P}$ is the space of totally real immersions $\iota: L\to X$ isotopic to a given immersion, and $\mathcal{T}$ is formally its quotient by the orientation preserving diffeomorphism group $Diff^+(L)$ of $L$.\footnote{Lotay and Pacini did not worry about analytical issues involving the quotient; their picture is entirely formal.} There is a horizontal distribution, which at each point $\iota\in \mathcal{P}$ assigns the transverse bundle $JT\iota(L)$, so gives a way to lift tangent vectors from $\mathcal{T}$ to $\mathcal{P}$. They then define a \textbf{geodesic} to be a curve $\iota_t$ in $\mathcal{P}$, such that the tangent vector field $\partial_t\iota_t$ is parallel with respect to the horizontal distribution. As a caveat, here the word `geodesic' does not suggest a Riemannian metric. An alternative characterisation \cite[Lem 2.2]{Lotay} of a geodesic, is a 1-parameter family of totally real submanifolds $\iota_t: L\to X$, such that there is a fixed vector field $Y\in \Gamma(L,TL)$, with
\[
\frac{d}{dt} \iota_t= J\iota_{t*}Y, \quad [\iota_{t*} Y, J\iota_{t*} Y  ]=0.
\]
Geometrically, one can imagine $(\iota_t)_{0\leq t\leq1}$ sweeps out an $(n+1)$-dimensional submanifold with boundary, foliated into complexified integral curves of $Y$, which are \textbf{holomorphic curves} inside $X$.

Lotay and Pacini also define the J-volume functional on $\mathcal{T}$. 
Pointwise on a totally real submanifold $L$, a real cotangent vector of $L$ corresponds to a $(1,0)$-form in $\Lambda^{(1,0)}X|_{L}$, so taking the $n$-th wedge power, we have a canonical isomorphism between $\Lambda^nL\otimes_\R \C$ and $K_X|_L$. Now the K\"ahler structure on $X$ induces a Hermitian metric on $K_X$, and pointwise on $L$ an element of $K_X$ with unit Hermitian norm uniquely specifies a volume form $dvol_J$ on $L$. Their \textbf{$J$-volume functional} is
\[
\text{Vol}_J(L)=\int_L dvol_J.
\]
It is easy to show $\text{Vol}_J$ provides a lower bound $\text{Vol}_J(L)\leq \text{Vol}(L)$ to the Riemannian volume of $L$, with equality precisely when $L$ is Lagrangian.

In case $X$ is Calabi-Yau, this construction is particularly transparent. The totally real condition is equivalent to the pointwise non-vanishing of $\Omega$ when restricted to $L$, and 
\[
|\int_L \Omega|\leq 
\text{Vol}_J(L)=\int_L e^{-i\theta}\Omega \leq \text{Vol}(L),
\]
where the phase factor is chosen to make $e^{-i\theta}\Omega$ an orientation form on $L$. Consequently, any totally real submanifold in a Calabi-Yau manifold with $\theta=\hat{\theta}$ constant, with no need for the Lagrangian condition, is an absolute minimizer of the $\text{Vol}_J$ functional. This enormous space of critial points is closely related to the non-ellipticity of the critical point equation. The situation is somewhat better in a negative K\"ahler-Einstein ambient space, where the critical points coincide with minimal Lagrangians  \cite[section 5.5]{Lotay}.

One of their main results is 
\begin{prop}\label{LotayPaciniconvexity}
\cite[Thm 5.10]{Lotay} In a K\"ahler-Einstein ambient manifold with non-positive Ricci curvature, the $J$-volume functional is convex along the geodesics.

\end{prop}

There is also a formal GIT picture \cite[section 6]{Lotay}: to some extent the space $\mathcal{P}$ can be viewed as the infinitesimal complexification of $Diff^+(L)$, with the space of orbits $\mathcal{T}$. The J-volume functional is formally a K\"ahler potential on $\mathcal{P}$, and the $Diff^+(L)$-moment map gives rise to critical points of the J-functional.

\subsubsection*{Limitations}

From the viewpoint of special Lagrangian geometry, the limitations of the Lotay-Pacini picture  are:
\begin{itemize}
\item The Lagrangians are largely relegated to the back stage. As clear from the above, in the Calabi-Yau case the space of critical points is too enormous. Their framework may be more useful in the negative K\"ahler-Einstein case, but the lack of ellipticity is a severe obstacle.

\item 
There is no attempt to link up with Floer theory.
As such, they lack satisfactory existence criterions for the holomorphic curves, despite making some limited progress in special cases. In fact, Lotay and Pacini's geodesics seem too oversimplified from the Floer theoretic viewpoint: one needs to address transversality questions in general, and holomorphic disc breaking should occur in moduli spaces of dimension $\geq 1$.

\end{itemize}

\subsection{Analogy with tunneling effect}\label{Tunneling}

\footnote{This section is meant to be purely inspirational, and its aim is to present some analogies and comparisons with no claim to physical accuracy.} Notice the Thomas-Yau argument is a little mysterious from the following classical perspective: how could two Lagrangians in possibly different Hamiltonian isotopy classes communicate with each other? This situation seems conceptually similar to the phenomenon of tunneling in quantum mechanics: there may be several minima of the classical potential function separated by potential wells, but there is a nontrivial quantum amplitude for the particles to move in between, thereby removing the ground state degeneracy. In the Thomas-Yau setup, we have two putative special Lagrangians, which are analogous to the energy minima points,  and the Thomas-Yau uniqueness statement is similar to the removal of degeneracy, resulting in a unique ground state.

The physicists tell us that branes are dynamical objects and can fluctuate. If so, it might make sense to ask about the amplitude for a brane to start with a given configuration and end with another. Now if Lagrangian branes behave like classical particles moving on an infinite dimensional space such as Solomon's space $\mathcal{O}$, one might expect Solomon's formal picture to be relevant for describing this amplitude, and the incompleteness of $\mathcal{O}$ would suggest a nontrivial amplitude to tunnel outside $\mathcal{O}$ to  another Hamiltonian isotopy class. As a conflicting viewpoint, the 
use of Floer theory in the Thomas-Yau argument suggests the Lagrangian branes communicate by the strings stretched between them, so one might expect the amplitudes to be computed in terms of worldsheet integrals. Lotay and Pacini's geodesics fit this viewpoint better, and have the major advantage of making sense outside a given Hamiltonian isotopy class.




Tunneling effects made a famous appearance in Witten's interpretation of Morse theory \cite{Witten} in terms of supersymmetric quantum mechanics, with an eye towards applications in quantum field theory. In Witten's  context, the tunneling amplitude between two critical points of a Morse function is to leading order proportional to $e^{-S}$, where $S$ is proportional to the difference of the two critical values.  Now Solomon's functional has some similarity with a Morse function whose only critical points are minima. In sections \ref{Solomonrevisited}, \ref{ModuliintegralSolomonsection} we will \emph{unify} the Solomon functional with the Lotay-Pacini geodesics, in the special case of exact Lagrangians. Could the Solomon functional have any physical interpretation in terms of the logarithm of the tunneling amplitudes? \footnote{As a grain of salt, Witten's interpretation concerns supersymmetric QFT, while the Thomas-Yau picture concerns the dynamics of Lagrangian branes, which is a target space perspective on string theory, which lacks a fundamental path integral formulation. As such our suggested physcial interpretation is not logically rigorous, but only a guess.}


\section{Moduli space of holomorphic curves}\label{Moduliholocurves}

The theme of this Chapter is the \emph{bordism currents} produced from the $(n+1)$-dimensional universal family over the $(n-1)$-dimensional moduli spaces of holomorphic curves. 
This theme unifies the search for the Floer theoretic obstructions, and the problem of extending the Solomon functional to the derived Fukaya category setting. One of our central techniques is \emph{integration over the moduli spaces}, to prove both identities and inequalities. The primary setting of this Chapter is exact graded compact Lagrangians inside Calabi-Yau Stein manifolds, with occasional comments on the compact Calabi-Yau case.

The \emph{Floer theoretic obstructions} are necessary conditions to the existence of special Lagrangians in the exact and almost calibrated setting. We will give a large number of a priori heuristic principles to place very stringent constraints on what kind of obstructions we are looking for, then prove the Floer theoretic obstructions under the extra hypotheses of automatic transversality and a positivity condition (\cf section \ref{Floertheoreticobstructions}), and then compare our picture to Joyce's proposal on Bridgeland stability (\cf section \ref{TowardsBridgeland}). The converse direction, in search of \emph{sufficient conditions} for the existence of special Lagrangians, will be laid out in Chapter \ref{Variational}.

We also present two perspectives on extending the Solomon functional to unobstructed Lagrangians within the same $D^bFuk(X)$ class. The more elementary perspective (\cf section \ref{Solomonrevisited}) works only in the exact setting, and implies the first variation formula under Hamiltonian isotopies in a rather straightforward manner. The second perspective (\cf section \ref{ModuliintegralSolomonsection}) represents the Solomon functional as an integral over the moduli space of holomorphic curves, which we hope can be generalized to the compact Calabi-Yau setting. Section 
\ref{Variantsandobservations} contains
more applications of the moduli integral technique.

\begin{rmk}
The two \emph{assumptions} of automatic transversality (which roughly means that no perturbation of almost complex structure is needed, \cf section \ref{Automatictransversalitypositivity}) and the positivity condition (\cf section \ref{Positivitycondition}) will feature prominently in the major results of this Chapter. Correspondingly, we will not dwell on the detail of the perturbation schemes aimed at overcoming the transversality issues, which are painstakingly carried out in \cite{JoyceAkaho}\cite{Woodward1}\cite{Woodward2}\cite{CharestWoodward}\cite{CharestWoodward2}. The expert readers may convert to their favourite schemes as they prefer. For more on our clockwise conventions such as gradings and signs, which differ from many symplectic texts, see the Appendix.

\end{rmk}


\subsection{Lotay-Pacini picture revisited}\label{LotayPacinirevisited}

From a Floer theoretic perspective, the main insight of Lotay and Pacini (\cf section \ref{LotayPacini}) is that given two Lagrangians $L, L'$ decorated  with suitable brane structures, one should look for a \emph{family} of holomorphic curves  with boundary on $L$ and $L'$,  which pass through any generic point of $L$ and $L'$. In their formal picture, such families are called `geodesics'. What Lotay and Pacini did not provide is a good existence criterion for their geodesics. Now, even though we will soon specialize to a much simpler setting, we wish to explain how their geodesics fit into the Thomas-Yau-Joyce picture.

In the setup of Bridgeland stability, the central charge is defined as a homomorphism from the Grothendieck group of a triangulated category to $\C$, which factorizes through a finitely generated lattice, viewed as a numerical Grothendieck group. In view of the application to special Lagrangians, the triangulated category is $D^b Fuk(X)$, and the numerical Grothendieck group should be a subgroup of the homology group modulo torsion $H_n(X, \Z)/\text{tors}$. In particular, 
\begin{itemize}
\item   If two Lagrangian branes $L,L'$ define the same object in $D^b Fuk(X)$, then this picture predicts them to lie in the same homology class in $H_n(X)$.

\item If $L_1\to L\to L_2\to L_1[1]$ is a distinguished triangle, then $[L_1]+[L_2]=[L]$ in $H_n(X)$. 
\end{itemize}

 This means Floer theory must provide a bordism current between $L, L'$ (resp. $L$ and $L_1\cup L_2$), namely an $(n+1)$-dimensional integration current $\mathcal{C}$ with $\partial \mathcal{C}= L-L'$ (resp. $\partial \mathcal{C}=L-L_1-L_2$).\footnote{The boundary of integration currents do not detect contributions from supports of small enough dimension. In this respect they are similar to pseudocycles, although when we generalize to Lagrangians with very weak regularity later, the language of currents may be more natural.} In Floer theory, cycles are constructed from moduli spaces of holomorphic curves and evaluation maps, so this current $\mathcal{C}$ should come from families of holomorphic curves, possibly with highly sophisticated perturbations and virtual techniques. This suggests the Lotay-Pacini geodesics, construed in this homological sense, should be part of the Fukaya category foundation
necessary for the Thomas-Yau-Joyce program.


\subsubsection*{High brow viewpoint}

The claim that \emph{the $K$-theory of the derived Fukaya category of a symplectic Calabi-Yau manifold (exact with suitable convexity at infinity, or compact) factorizes through homology $H_n(X)$ modulo torsion} seems well known to symplectic topology experts, although a precise reference seems rather difficult to find. We now sketch a high brow viewpoint explained to the author by P. Seidel and S. Rezchikov, and will later explain in more detail a more pedestrian approach in the exact setting. The claim is a formal consequence of the existence of maps
\[
K_0(D^bFuk(X)) \xrightarrow{ch_0} HH_0(D^bFuk(X))\]
\[
HH_0(D^bFuk(X)) 
\xrightarrow{OC} QH^n(X; \Lambda_{nov})\simeq H_n(X)\otimes \Lambda_{nov}.
\]
Here $HH_0(D^bFuk(X))$ is the Hochschild homology in degree zero. The first map sends the K-theory class of an unobstructed  Lagrangian brane $L$ (compact, graded, oriented, with spin and bounding cochain structure) to the unit $1_L\in HF^0(L,L)\to HH_0(D^bFuk(X)  )$; the well definition of this map is an essentially algebraic fact. Suppose $L, L'$ are isomorphic in $D^bFuk(X)$, then there are closed morphisms $\alpha\in CF^0(L,L')$ and $\beta\in CF^0(L',L)$ whose derived category compositions are equal to $1_L\in HF^0(L,L)$ and $1_{L'}\in HF^0(L',L')$ in cohomology. The Hochschild differential of $\beta\otimes \alpha$ exhibits $1_L-1_{L'}$ as a coboundary in the Hochschild chain complex, so $1_L=1_{L'}$ in $HH_0(D^b Fuk(X))$. Some additional calculation shows the compatibility with distinguished triangles.

The second map is a special case of the open-closed string map, and in general requires working over the Novikov field. 
 One then needs the claim that $1_L$ is sent to the homology class $[L]\in H_n(X)$, without quantum correction. The intuitive meaning of the open-closed string map is to consider holomorphic discs with boundary on $L$ with an unconstrained boundary marked point, and find the cycle in $X$ traced out by an interior marked point. The claim amounts to saying that the only contribution comes from constant maps. Unfortunately, the author is unable to locate a general reference. Granted this claim, we would get by composition a map from $K_0(D^bFuk(X))$ to $H_n(X; \Lambda_{nov})$ which sends the K-theory class of $L$ to the homology class $[L]$.

\subsubsection{The exact embedded Lagrangian case}

We specialize to the setting of Stein manifolds, and all Lagrangians are assumed to be exact, graded and compact, and in particular carry an orientation (\cf the Appendix for some basic Floer theory). The local systems have holonomy in $\R$, $\Q$ or $\Z$. We consider two transverse embedded Lagrangians $L,L'$ in the \emph{same derived Fukaya category class}. By definition, we have closed morphisms $\alpha\in CF^0(L,L')$ and $\beta\in CF^0(L', L)$; we sometimes view the Lagrangian intersections in $\beta$ as degree $n$ outputs. Morever, in terms of the product structure on cohomology
\[
\begin{cases}
HF^0(L,L')\otimes HF^0(L',L)\to HF^0(L',L'),\\
 HF^0(L',L)\otimes HF^0(L,L')\to HF^0(L,L),
\end{cases}
\]
the composition $\alpha\circ \beta=1_{L'}$ and $\beta\circ \alpha=1_{L}$. These conditions completely characterize isomorphism in $D^b Fuk(X)$. Our goal is to explain

\begin{prop}\label{bordism}
There is a bordism current $\mathcal{C}$ such that $\partial \mathcal{C}=L-L'$ in the sense of currents.
\end{prop}

The bordism current will be constructed from \emph{universal families of (perturbed) holomorphic strips} with boundary on $L$ and $L'$, and with ends at $\alpha$ and $\beta$ (meaning that the ends of the strip converge to intersection points of $L,L'$ of degree $0$ and $n$ respectively, and $\alpha, \beta$ encode the weighting factors to the contribution of these intersection points). We will assume all the usual transversality assumptions in Floer theory are satisfied, so the compactified moduli space $\overline{\mathcal{M}(\alpha,\beta)/\R}$ of perturbed holomorphic strips up to domain translation  is a smooth manifold with boundary and corners, of dimension $(n-1)$. The notation really stands for a formal sum of many moduli spaces, coming from the summands of $\alpha,\beta$. The universal family is fibred over this moduli space $\overline{\mathcal{M}(\alpha,\beta)/\R}$, whose fibres are the solutions to the Cauchy-Riemann equation (with domain dependent perturbations of the almost complex structure), which we call perturbed holomorphic curves. The fibres over the boundary of the moduli space are broken holomorphic curves. The orientation on the universal family is induced from the complex orientation on $\Sigma$ and the orientation on the moduli space, up to an extra minus sign (\cf Example \ref{FLoerproduct2} for conventions). Upon evaluation to $X$ we obtain an $(n+1)$-dimensional current $\mathcal{C}$.



\begin{rmk}
If the Fukaya category is defined over $\Z$, then all the weighting factors to the various universal families are all \emph{integers}, and $\mathcal{C}$ is naturally an \emph{integral current}. If we use Fukaya categories over $\Q$ or $\R$ instead, then $\mathcal{C}$ is only guaranteed to be a finite $\Q$ (resp. $\R$) linear combination of integral currents.
\end{rmk}

Our main task is to understand the boundary of $\mathcal{C}$. There are two  sources of boundaries:
\begin{itemize}
\item The holomorphic curves themselves have boundary along $L\cup L'$. This boundary contribution is always supported on $L\cup L'$. 

\item The compactified moduli spaces have boundary due to holomorphic strip breaking.
\end{itemize}

In schematic notation, the boundary of the moduli space is described by
\begin{equation}
\partial (\overline{ \mathcal{M}(\alpha,\beta)/\R})= \bigcup_r \overline{\mathcal{M}(\alpha,r)/\R} \times \overline{\mathcal{M}(r, \beta)/\R}.
\end{equation}
Here $r$ can range from all intersection points of degree between $1$ and $n-1$. The notation $\mathcal{M}(\alpha, \beta)/\R$ stands  for a weighted sum of moduli spaces, with weighting coming from the holonomy factors of the local systems.

Next comes a crucial observation. Although $2\leq \deg r\leq n-2$ give rise to boundaries of the moduli space $\overline{\mathcal{M}(\alpha,\beta)/\R}$, their contributions to $\partial \mathcal{C}$ are contained in the universal families associated to $\overline{\mathcal{M}(\alpha,r)/\R}$ and $\overline{\mathcal{M}(r,\beta)/\R}$. These smaller moduli spaces have dimension at most $n-3$, and the corresponding universal families have dimension at most $n-1$. 
By rectifiability considerations, the $n$-dimensional current $\partial \mathcal{C}$ cannot receive contributions from at most $(n-1)$-dimensional supports, so such disc breakings do not contribute to $\partial \mathcal{C}$.

Now for $\deg r=1$, the moduli spaces 
$\overline{\mathcal{M}(\alpha,r)/\R}$ are zero dimensional, so their presence merely amounts to some counting factors.  The condition for $\alpha$ to be \emph{closed} in $CF^0(L,L')$ is equivalent to the weighted count of $\overline{\mathcal{M}(\alpha,r)/\R}$ being zero. This weighted sum appears as the coefficient of the $n$-dimensional current defined by the universal family over $\overline{\mathcal{M}(r,\beta)/\R}$. Thus we see that $r\in CF^1(L,L')$ does not contribute to $\partial \mathcal{C}$. Similarly, the condition for $\beta$ to be closed in $CF^0(L',L)$ implies that $r\in CF^1(L',L)$ (alternatively viewed as degree $n-1$ intersections from $L$ to $L'$) does not contribute to $\partial \mathcal{C}$. 
In summary, $\partial \mathcal{C}$ must be an integration cycle supported on $L\cup L'$.

Since $\partial \mathcal{C}$ is itself the boundary of a current, it must be closed. This explains why $\partial \mathcal{C}$ is a constant linear combination of the integration cycle of $L$ and $L'$, instead of some nontrivial function times these cycles. The constant coefficients can be pinned down by counting the number of holomorphic strips passing through a given generic point on $L$ (resp. $L'$), and the choice of the generic point does not matter. Such counts are precisely the geometric interpretation of the Floer product $HF^0(L,L')\otimes HF^0(L',L)\to HF^0(L',L'),$ and $
HF^0(L',L)\otimes HF^0(L,L')\to HF^0(L,L)$ (\cf Example \ref{FLoerproduct}). When the moduli space orientations are taken into account,  we obtain $\partial \mathcal{C}=L-L'$ (\cf Example \ref{FLoerproduct2} for an exposition on signs).

\begin{rmk}
(Homological uniqueness of the bordism current)
Some auxiliary perturbation data goes into the construction of $\mathcal{C}$ due to the need to ensure transversality. If we fix $L,L'$, but change the domain dependent almost complex structures, then the difference of two bordism currents $\mathcal{C}-\mathcal{C}'$ has zero boundary in the sense of currents. Recall that Stein manifolds have the homotopy type of a CW complex of dimension $\leq n$, and thus $H_k(X)=0$ for $k\geq n+1$, so $\mathcal{C}-\mathcal{C}'$ must be the boundary of an $(n+2)$-dimensional current. For an alternative viewpoint, this  $(n+2)$-dimensional current can be concretely provided by  parametrized families of pseudoholomorphic curves (\cf Remark \ref{3Lag}).

\end{rmk}

\subsubsection{Immersed case}\label{LotayPaciniimmersed}

Still working in the exact setting, we now allow $(L,b),(L',b')$ to be unobstructed immersed Lagrangians with transverse self intersections. Assume they intersect transversally, and define isomorphic objects in $D^bFuk(X)$. We now wish to explain why Proposition \ref{bordism} should continue to hold even in the immersed setting, without delving too deep into the specifics of the perturbation schemes and transversality issues. For some background on the immersed Floer theory, see the Appendix \ref{immersedFukaya}.

The isomorphism condition gives us closed morphisms $\alpha\in CF^0(L,L')$ and $\beta\in CF^0(L',L)$ whose cohomological compositions give the identities. At the chain level,
\[
m_2^{b,b'}(\beta, \alpha)= 1_L - m_1^{b}(\gamma),\quad m_2^{b',b}( \alpha, \beta)= 1_{L'} + m_1^{b'}(\gamma')
\]
where $1_L, 1_{L'}$ stand for the geometric units (represented by a sum of local maximum points of Hamiltonian functions on $L,L'$ respectively), and $\gamma, \gamma'$ are elements in $CF^{-1}(L,L)$, $CF^{-1}(L',L')$ respectively. Notice in the almost calibrated case, $\gamma, \gamma'$ would be both zero, since there are no self intersections of degree $-1$.

As before, the bordism current shall be constructed from the universal family of (perturbed) holomorphic curves with boundary on $L$ and $L'$. But instead of working only with holomorphic strips, we need holomorphic polygons with corners not only at intersection points in $\alpha, \beta$, but also at points in $b, b'$. In addition to the holomorphic strip moduli space $\mathcal{M}(\alpha, \beta)/\R$, we also need the moduli space of polygons $\mathcal{M}(b, \ldots b, \alpha, b',\ldots b', \beta )$, and $\mathcal{M}(b,\ldots, b, \gamma)$, $\mathcal{M}(b',\ldots b', \gamma')$. 
The notation here is a shorthand for a weighted sum of many moduli spaces of polygons. Since the bounding cochain elements have Floer degrees one, these moduli spaces all have  dimension $n-1$. The energy of the polygons satisfies the topological formula (\ref{topologicalenergy2}), so by the Novikov positivity requirement of bounding cochains, there is a uniform a  priori energy bound once $\alpha,\beta,\gamma,\gamma'$ are given, whence there are in fact only finitely many moduli spaces involved. Each moduli space provides a universal family of holomorphic curves, and the sum of all the contributions defines an $(n+1)$-dimensional current $\mathcal{C}$. For sign conventions, see the Appendix \ref{immersedFukaya}, and Example \ref{FLoerproduct2}.

The boundary of $\mathcal{C}$ comes from two sources: the boundary of the individual holomorphic curves which lie on $L\cup L'$, and the boundary of the compactified moduli spaces. In the exact setting, there are no sphere bubbles. As in the embedded case, for support dimension reasons, the boundaries of the compactified moduli space that can contribute to $\partial \mathcal{C}$, is caused by curve breaking into two pieces arising in $0$ and $n-2$ dimensional moduli spaces. The cancellation of these contributions is very similar to the standard argument for the Floer differential to square to zero (\cf the Appendix \ref{immersedFukaya}):
\begin{itemize}
\item For breakings at a nodal point mapping to $CF^1(L,L')$ (resp. $CF^1(L',L)$), the contributions vanish due to the closedness condition $m_1^{b,b'}(\alpha)=0$ (resp. the closedness of $\beta$).

\item For breakings at degree 2 self intersection point on $L$ (resp. $L'$), the contributions vanish due to the Mauer-Cartan equation on $b$ (resp. $b'$).

\item A new way of disc breaking/splitting, is at a degree zero self intersection $r$ on $L$ (the $L'$ case being entirely similar). The discs of $\mathcal{M}(b, \ldots b, \alpha, b',\ldots b', \beta )$ can break into a virtual dimension zero disc with input at $b,\ldots, \alpha, b',\ldots \beta,b,\ldots b$ and output at $r$, and a virtual dimension $n-1$ disc with input corners at $b,\ldots b, r$. On the other hand, the discs of $\mathcal{M}(b,\ldots, b, \gamma)$ can break into a virtual dimension zero disc with input at $b,\ldots, \gamma,b,\ldots b$ and output at $r$, and a virtual dimension $n-1$ disc with input corners at $b,\ldots b, r$. These two effects cancel out.

\end{itemize}

After the cancellation of all moduli space boundaries, the only contributions $\partial \mathcal{C}$ are supported in $L\cup L'$.
As in the embedded case, $\partial \mathcal{C}$ is locally a constant multiple of the underlying cycles of $L,L'$.
The interpretation of the geometric unit pins down $\partial \mathcal{C}=L-L'$ as in the embedded case.

\begin{eg}
If $L$ and $L'$ are disjoint Lagrangian branes which both define the zero object in $D^bFuk(X)$, then $\alpha, \beta$ are both zero, and $\mathcal{C}$ comes entirely from the $\gamma,\gamma'$ contributions. Of course, zero Lagrangian objects have zero homology class, which cannot happen in the almost calibrated case.
\end{eg}

\begin{rmk}
If there are degree $-2$ self intersections of $L$, then the choice of $\gamma$ is only unique up to $m_1^b$ of some element in $CF^{-2}(L,L)$. The corresponding choice of $\mathcal{C}$ would be ambiguous by the boundary of an $(n+2)$-dimensional integration current. As a closely related issue, our conditions on $\alpha, \beta$ are merely cohomological, so in general we can adjust $\alpha$ and $\beta$ by coboundary terms, which would affect $\mathcal{C}$ also by the boundary of an $(n+2)$ dimensional current. If we impose $L,L'$ to be almost calibrated, then there are no $CF^{-1}(L,L')$ elements to begin with, and these phenomena do not happen.

On the other hand, $\mathcal{C}$ still depends on the choice of local systems and bounding cochains, which may contribute nontrivial holonomy factors. Gauge equivalent choices affect $\mathcal{C}$ by the boundary of an $(n+2)$-dimensional current. One may naturally ask:

\begin{Question}\label{gaugerepresentative}
Up to gauge equivalence of bounding cochains and local systems, is there an optimal representative of $\mathcal{C}$?
\end{Question}

\begin{Question}\label{preferredbordism}
Given an exact isotopy with surgery between $L$ and $L'$ among unobstructed Lagrangians, is there a preferred choice of $\mathcal{C}$ (\cf Question \ref{geometricconnectedness})?
\end{Question}


\end{rmk}

\subsubsection*{Distinguished triangles}

In our convention, an immersed Lagrangian can be made up of \emph{several connected components}. A prototypical situation is when $L'$ is the union of two immersed Lagrangians $L_1$ and $L_2$, with some degree one intersections in $CF^1(L_2,L_1)$ arising as part of the bounding cochain data of $L'$. When the brane structure is taken into account, we can view $L'$ as a twisted complex built from $L_1,L_2$ (with bounding cochains $b_1,b_2$ suppressed in the notation) and a closed morphism $\gamma\in CF^1(L_2,L_1)=CF^0(L_2[-1],L_1)$. Inside $D^bFuk(X)$,
\[
L\simeq L'\simeq \left( \begin{matrix}
(L_2, b_2) & \\
\gamma & (L_{1}, b_{1}) 
\end{matrix} \right).
\]
We have a distinguished triangle
\[
L_2[-1]\xrightarrow{\gamma} L_1\to \text{Cone}(\gamma) \xrightarrow{} L_2,
\]
and $L\simeq L'\simeq \text{Cone}(\gamma)$. Rotating the triangles, we get another \textbf{distinguished triangle}
\[
L_1\to L\to L_2\xrightarrow{\gamma} L_1[1].
\] 
The bordism current between $L$ and $L'$ is an $(n+1)$-dimensional integration current, with
$
\partial \mathcal{C}=L- L_1-L_2.
$
In particular, this explains that the Grothendieck group of $D^b Fuk(X)$ should factorize through $H_n(X)$.

Here is a more geometric perspective on the bordism currents arising from distinguished triangles, which is very close to Thomas and Yau's original viewpoint, where the fundamental phenomenon is \textbf{Lagrangian breaking}. In the simplest case, we can imagine $L$ is  isomorphic in $D^b Fuk(X)$ to the Lagrangian connected sum $L_1\#L_2$ (beware our  convention for $L_1\#L_2$ is the same as Thomas-Yau \cite{Thomas} but different from many symplectic texts), so that we can construct a bordism current between $L$ and $L_1\#L_2$. Now when $L_1\#L_2$ deforms, the Lagrangian handle part can shrink, and in the limit  $L_1\#L_2$ can break into two components $L_1\cup L_2$ (\cf Example \ref{Lagsumexample}). The bordism current between $L$ and $L_1\cup L_2$ should simply be the limit of the sequence of bordism currents. This picture illustrates that even when the topology of the Lagrangians can change under non-smooth convergence, the bordism currents should persist in a continuous way.


One can proceed with the case of \emph{many Lagrangians}, namely we take the immersed Lagrangian $L'$ to be the \textbf{twisted complex} (\cf the Appendix section \ref{immersedFukaya})
\begin{equation}\label{twistedcomplex}
\left(
\begin{matrix}
(L_N, b_N) &  & \\
b_{N,N-1} & (L_{N-1}, b_{N-1}) & 
\\
\ldots
\\
b_{N,1} & b_{N-1,1} &\ldots & b_{2,1}  & (L_1,b_1)
\end{matrix} \right).
\end{equation} 
In this case, assuming $L$ is isomorphic to $L'$ in $D^b Fuk(X)$, the bordism current between $L$ and $L'$ amounts to a bordism current between $L$ and $L_1\cup L_2\ldots \cup L_N$.

This multi-Lagrangian situation is built out of many distinguished triangles: for $1\leq k\leq N$, let $\mathcal{E}_k$ be the immersed Lagrangian $L_1\cup \ldots \cup L_k$  corresponding to the twisted complex 
\begin{equation}\label{HNEk}
\left(
\begin{matrix}
(L_k, b_k) &  & \\
b_{k,k-1} & (L_{k-1}, b_{k-1}) & 
\\
\ldots
\\
b_{k,1} & b_{k-1,1} &\ldots & b_{2,1}  & (L_1,b_1)
\end{matrix} \right).
\end{equation}
Then (suppressing bounding cochains in the notation) we have a sequence in $D^b Fuk(X)$,
 \begin{equation*}
 0=\mathcal{E}_0 \to \mathcal{E}_1\to \ldots \to \mathcal{E}_N\simeq L,
 \end{equation*}
 with distinguished triangles \[
 \mathcal{E}_{i-1}\to \mathcal{E}_{i}\to L_{i}\to \mathcal{E}_{i-1}[1].
 \]
The morphism from $L_i$ to $\mathcal{E}_{i-1}$ comes from $b_{i,j}$ for $j<i$. This setup should be reminiscent of Harder-Narasimhan decompositions  (\ref{HarderNarasimhan}), although at this stage we have not yet brought in stability conditions, which shall be discussed further in section \ref{TowardsBridgeland}.

\subsection{Solomon functional revisited}\label{Solomonrevisited}

Let $(X,\omega,\Omega)$ be an almost Calabi-Yau Stein manifold, and $L$ be an  exact  Lagrangian brane, with $\int_{L} \text{Im}(e^{-i\hat{\theta}}\Omega)=0$ for some suitable $\hat{\theta}\in (-\frac{\pi}{2}, \frac{\pi}{2})$. Our goal is to suggest how the Solomon functional may be well defined without the universal cover issue, and extended beyond a given exact isotopy class. Both issues are essential for the variational approach to the Thomas-Yau conjecture, see Chapter \ref{Variational}.

\subsubsection*{Homological nature of the Solomon functional}

Write the Liouville 1-form as $\lambda$, so $d\lambda=\omega$, and the potential of the immersed Lagrangian $L$ as $f_L$, so $df_L=\lambda|_L$.  We consider the potential as part of the brane data, so adding a constant to $f_L$ is viewed as a different Lagrangian brane. We shall consider a path of such Lagrangians $L_t$, with associated Hamiltonian functions $h_t$, so there is a preferred way to parallel transport $f_L$, as recalled below.

\begin{lem}\label{Solomonfunctionallem1}
\begin{equation}\label{Solomonhomological}
\int_0^1 dt\int_{L_t} h_t \text{Im}(e^{-i\hat{\theta}}\Omega)= \int_{L_t} f_{L_t}\text{Im}(e^{-i\hat{\theta}}\Omega) |^{t=1}_{t=0} - \int_{\cup_t L_t} \lambda\wedge \text{Im}(e^{-i\hat{\theta}}\Omega).
\end{equation}	
\end{lem}

\begin{proof}
Let $X_t$ be the Hamiltonian vector field along $L_t$ associated to $h_t$, namely $dh_t=\omega(X_t,\cdot)$. We calculate the time derivative of $f_{L_t}$: along $L_t$
\[
\mathcal{L}_X \lambda= \iota_X d\lambda+ d(\iota_X\lambda)= \iota_X\omega+d(\iota_X\lambda)= d(h_t+ \iota_X\lambda),
\]
so there is a preferred parallel transport of $f_L$ along the path $L_t$,
\[
\partial_t f_{L_t}= h_t+ \iota_X\lambda.
\]
Hence
\[
\partial_t \int_{L_t} f_{L_t}\text{Im}(e^{-i\hat{\theta}}\Omega) = \int_{L_t} (h_t+ \iota_X\lambda)\text{Im}(e^{-i\hat{\theta}}\Omega)+ \int_{L_t} f_{L_t} \mathcal{L}_X  \text{Im}(e^{-i\hat{\theta}}\Omega).
\]
Now by the Cartan formula and the closedness of $\Omega$,
\[
\mathcal{L}_X  \text{Im}(e^{-i\hat{\theta}}\Omega)= d\iota_X \text{Im}(e^{-i\hat{\theta}}\Omega),
\]
so after integration by part, \[
\int_{L_t} f_{L_t} \mathcal{L}_X  \text{Im}(e^{-i\hat{\theta}}\Omega)= -\int_{L_t} df_t \wedge \iota_X \text{Im}(e^{-i\hat{\theta}}\Omega)= - \int_{L_t} \lambda \wedge \iota_X \text{Im}(e^{-i\hat{\theta}}\Omega).
\]
Combining the above,
\[
\partial_t \int_{L_t} f_{L_t}\text{Im}(e^{-i\hat{\theta}}\Omega)= \int_{L_t} h_t\text{Im}(e^{-i\hat{\theta}}\Omega)+ \int_{L_t} \iota_X( \lambda\wedge  \text{Im}(e^{-i\hat{\theta}}\Omega)  ).
\]
Integrating in $t$ gives the result.
\end{proof}

We observe that by the K\"ahler condition $\omega\wedge \Omega=0$, so
\[
d( \lambda\wedge \text{Im}(e^{-i\hat{\theta}}\Omega  ) )=\omega\wedge  \text{Im}(e^{-i\hat{\theta}}\Omega)=0.
\]
This means the term $\int_{\cup_t L_t} \lambda\wedge \text{Im}(e^{-i\hat{\theta}}\Omega)$ is a homological quantity, in the sense that 
 we can replace $\cup_t L_t$  by any compactly supported $(n+1)$-current $\mathcal{C}$ with $\partial \mathcal{C}=L_1-L_0$, which would automatically satisfy $[\mathcal{C}-\cup_t L_t]=0\in H_{n+1}(X)$, since $H_{n+1}(X)=0$ for Stein manifolds. In particular, this explains Solomon's theorem that his functional is invariant under Hamiltonian deformations of the path of Lagrangians. The advantage of our homological interpretation is to allow more general currents $\mathcal{C}$, which in particular can come from families of holomorphic curves.

\subsubsection*{Proposed extension of the Solomon functional}

Taking the homological interpretation of (\ref{Solomonhomological}) as starting point, a natural way to extend the Solomon functional is to make use of the bordism current $\mathcal{C}$ between an unobstructed Lagrangian $L$ and a fixed unobstructed reference Lagrangian $L_0$. We have $\partial \mathcal{C}= L-L_0$ as currents, and $\mathcal{C}$ comes from the universal family of holomorphic curves. Our proposed formula is
\begin{equation}\label{Solomonfunctionalextension}
\mathcal{S}(L)= \int_L f_L \text{Im}(e^{-i\hat{\theta}} \Omega) -\int_{L_0} f_{L_0} \text{Im}(e^{-i\hat{\theta}} \Omega)- \text{Im} \int_{\mathcal{C}} \lambda\wedge e^{-i\hat{\theta}} \Omega.
\end{equation}
A few conceptual points are in order:
\begin{itemize}
\item We emphasize that this depends not only on the underlying Lagrangian, but also on the potential $f_L$.

\item There is no need to pass to any universal cover in the space of Lagrangians, as in Solomon's work (\cf section \ref{Solomon}).

\item The topology of $L$ is no longer fixed, and in particular the Hamiltonian isotopy class may change. 

\item We view (\ref{Solomonfunctionalextension}) as a unification of the very different viewpoints of Solomon and Lotay-Pacini. In section \ref{Tunneling} we suggested that this extended Solomon functional may be relevant for quantum tunneling amplitudes between the branes $L_0$ and $L$.

\item Suppose we vary the Lagrangian $L$ within a 1-parameter exact isotopy family of unobstructed Lagrangians $L_t$. 
The bordism currents $\mathcal{C}_t$ between $L_t$ and $L_0$ satisfy 
\[
\mathcal{C}_{t_2}= \mathcal{C}_{t_1} + \cup_{t_1\leq t\leq t_2} L_t  \text{  modulo exact $(n+1)$-dim currents},
\]
then the computation in Lem \ref{Solomonfunctionallem1} proves the \textbf{first variation formula for the Solomon functional}
\begin{equation}\label{Solomonfirstvariation}
\frac{d}{dt} \mathcal{S}(L_t)= \int_{L_t} h_t \text{Im}(e^{-i\hat{\theta}} \Omega)
\end{equation}
which is of course the defining feature of the Solomon functional. Consequently, the formula (\ref{Solomonfunctionalextension}) extends Solomon's definition in our exact setting, and fixes the multivaluedness problem (\ie the need to pass to universal covers) in Solomon's work.
\end{itemize}




\subsubsection*{Change of reference Lagrangian}

The definition of the Solomon functional depends on the reference Lagrangian $L_0$, and we write $\mathcal{S}_{L_0}(L)$ when we wish to emphasize this dependence. The following feature of the Solomon functional resembles the Donaldson functional in the HYM context (\cf (\ref{Donaldsonfunctionalchangeofreference})):

\begin{prop}\label{SolomonchangeofreferenceProp}
Under the change of reference Lagrangians, 
\begin{equation}\label{SolomonfunctionalchangeL0}
\mathcal{S}_{L_0} (L)= \mathcal{S}_{L_0'}(L)+ \mathcal{S}_{L_0}(L_0').
\end{equation} 
\end{prop}

\begin{proof}
We shall use the homological nature of the Solomon functional and the fact that $H_{n+1}(X)=0$. We pick $\mathcal{C}_1,\mathcal{C}_2, \mathcal{C}_3$ such that
\[
\partial \mathcal{C}_1=L-L_0', \quad \partial \mathcal{C}_2=L_0'-L_0, \quad \partial \mathcal{C}_3= L-L_0.
\]
Then $\mathcal{C}_1+\mathcal{C}_2$ is homologous to $\mathcal{C}_3$, so we can replace $\mathcal{C}_3$ by $\mathcal{C}_1+\mathcal{C}_2$ to compute $\mathcal{S}_{L_0} (L)$, whence  (\ref{SolomonfunctionalchangeL0}) follows.
\end{proof}

\begin{rmk}\label{3Lag}
A more Floer theoretic argument that $\mathcal{C}_1+\mathcal{C}_2$ is homologous to $\mathcal{C}_3$, which does not appeal to $H_{n+1}(X)=0$ directly, can be sketched as follows. We assume $L_0,L_0',L$ are three unobstructed Lagrangians mutually isomorphic in $D^bFuk(X)$, and $HF^{-1}(L_0,L_0)=0$. Of course, the self Floer cohomologies of $L_0,L_0',L$ are all isomorphic, and $HF^{-1}=0$ is a necessary condition if the $D^bFuk(X)$ class admits any almost calibrated representative at all.
We consider $\alpha\in CF^0(L_0,L_0'),\beta\in CF^0(L_0',L),\gamma\in CF^0(L,L_0)$ representing the generators of $HF^0$, such that at the level of Floer cohomology
\[
\gamma\circ \beta\circ \alpha=1_{L_0}, \quad \alpha\circ \gamma\circ \beta=1_{L_0'},\quad  \beta\circ \alpha\circ \gamma=1_{L}.
\]
For simplicity we first assume almost calibratedness, so that $CF^{-1}=0$, and there is no ambiguity for these generators. Notice the compositions $\beta\circ\alpha, \gamma\circ\beta, \alpha\circ \gamma$ provide generators of $HF^0(L_0,L)$, $HF^0(L_0',L_0)$, $HF^0(L,L_0')$.
Consider the $n$-dimensional moduli spaces $\tilde{\mathcal{M}}$ of holomorphic discs with corners at $\alpha,\beta,\gamma$ and the self intersection points corresponding to the bounding cochains. The corresponding universal family $\tilde{\mathcal{C}}$ provides an $(n+2)$-dimensional current, whose boundary comes from disc bubbling and disc breaking. Most of the boundary contributions are eliminated by the Mauer-Cartan equation of the bounding cochains, the closedness of $\alpha,\beta,\gamma$, and support dimension reasons, and only three boundary contributions survive. These are the $(n+1)$-dimensional bordism currents between $L_0,L_0'$ (resp. $L_0',L$ and $L,L_0$) constructed from the universal family of holomorphic curves associated to the generators $\alpha, -\gamma\circ \beta$ (resp. $\beta, -\alpha\circ \gamma$ and $\gamma, \beta\circ\alpha$). We can identify these as $\mathcal{C}_2, \mathcal{C}_1, -\mathcal{C}_3$. The upshot is that  
Floer theory explicitly provides the $(n+2)$-dimensional current that exhibits the homological relation between $\mathcal{C}_1+\mathcal{C}_2$ and $\mathcal{C}_3$.

In general without assuming almost calibratedness, then $CF^{-1}$ can be nonzero.
Then we need some extra $n$-dimensional moduli spaces to account for the non-uniqueness of cohomological representatives of $HF^0$, an issue quite similar to section \ref{LotayPaciniimmersed}.
A subtle new issue is that the moduli space $\tilde{\mathcal{M}}$ receives new boundary contributions involving the $m_3^b$ products (this shorthand notation indicates the presence of bounding cochain elements, \cf (\ref{Ainftytwisted})) of $\alpha,\beta,\gamma$. The three cyclic permutations of $\alpha,\beta,\gamma$ produce three $m_3^b$ products, which are elements in $CF^{-1}(L_0,L_0), CF^{-1}(L_0',L_0')$ and $CF^{-1}(L,L)$ respectively, and the $(n-1)$-dimensional moduli of polygons with one corner at the $CF^{-1}$ intersections and the other corners at bounding cochain elements contribute to $\partial \tilde{\mathcal{C}}$. Now by the $A_\infty$ relation, and the closedness of $\alpha,\beta,\gamma$,
\[
m_1^b(m_3^b(\gamma,\beta,\alpha))+ m_2^b(\gamma, m_2^b(\beta,\alpha))- m_2^b(m_2^b(\gamma,\beta), \alpha)=0.
\]
Writing $m_2^b(\gamma, m_2^b(\beta,\alpha))=1_{L_0}+ m_1^b(\delta_1)$ and $m_2^b(m_2^b(\gamma,\beta),\alpha)= 1_{L_0}+ m_1^b(\delta_2)$, we see $m_3^b(\gamma,\beta,\alpha)+\delta_1-\delta_2$ is $m_1^b$-closed, so by the assumption that $HF^{-1}(L_0,L_0)=0$, it is in fact $-m_1^b(\epsilon_1)$ for some $\epsilon_1\in CF^{-2}(L_0,L_0)$. We can then produce an $n$-dimensional moduli space, from polygons with a corner at $\epsilon_1$, and other corners at the bounding cochain elements. Completely analogously, one can produce two other $n$-dimensional moduli spaces from $\epsilon_2\in CF^{-2}(L_0',L_0')$ and $\epsilon_3\in CF^{-2}(L,L)$. Combining the $(n+2)$-dimensional universal families over the $n$-dimensional moduli spaces, results in an explicit bordism current between $\mathcal{C}_1+\mathcal{C}_2$ and $\mathcal{C}_3$.
\end{rmk}


\subsection{Automatic transversality}\label{Automatictransversalitypositivity}

In our later applications, it is not enough to just have a bordism current between Lagrangians $L,L'$ constructed from perturbed pseudoholomorphic curves. Two additional conditions are desirable: \textbf{automatic transversality} and \textbf{positivity condition}. These are natural properties of the highly idealized picture of Lotay and Pacini (\cf section \ref{LotayPacini}), but may seem rather strong for Floer theorists.

 In this section we discuss various sufficient conditions for automatic transversality, which intuitively means that the bordism current $\mathcal{C}$ is constructed without perturbing the integrable complex structure. This requires that the (extended) linearized Cauchy-Riemann operator is surjective, namely the moduli space is regular. The next section will discuss the positivity condition. Complex integrability and the existence of holomorphic volume form $\Omega$ will be assumed throughout. All holomorphic curves are assumed to be nonconstant.


\begin{itemize}
\item (\textbf{Automatic transversality}) There exist a finite collection of $(n-1)$-dimensional \emph{smooth} moduli spaces of holomorphic curves $u:\Sigma\to X$ with respect to the \emph{integrable complex structure}, constructed from the inputs in $CF^0(L,L')$, $CF^0(L',L)\simeq CF^n(L,L')^\vee$ and the bounding cochain data as in section \ref{LotayPacinirevisited}, such that by taking the weighted sum of the $(n+1)$-dimensional universal families of holomorphic curves, we obtain a current $\mathcal{C}$ with $\partial \mathcal{C}=L-L'$. This bordism current $\mathcal{C}$ agrees with the bordism currents constructed from generically perturbed almost complex structures, up to the boundary of an $(n+2)$-dimensional current.

Morever, considering the boundary of holomophic curves $\partial \Sigma$ varying in the $(n-1)$ dimensional regular moduli spaces, we obtain $n$-dimensional universal families, sweeping out the cycle $L-L'$; we require the evaluation map from these $n$-dimensional universal family to $L\cup L'$ to be \emph{immersions}, except at the corner points of $\partial \Sigma$ mapping to the Lagrangian intersections, where the failure of immersion is `minimal' (see below for details). We say that the bordism current $\mathcal{C}$ consists purely of `automatically transverse curves'.

\item (\textbf{Automatic transversality, weak version)}  We can allow certain holomorphic curves $u:\Sigma\to X$ arising in $(n-1)$-virtual dimensional moduli spaces, which are not automatically transverse, subject to the following requirements on these extra bad curves:
\begin{enumerate}
\item 
When virtual perturbation theory is taken into account, $\partial \mathcal{C}=L-L'$ still holds.	
	
\item At any such bad curve $u:\Sigma\to X$, given any $n-1$ first order deformation vector fields, the 1-form $\Omega(\cdot, v_1,\ldots v_n)$ restricted to $\Sigma$ vanishes identically. Intuitively, this means $\Omega$ vanishes identically on the Zariski tangent space of the universal family at $u:\Sigma\to X$. As a caveat, these Zariski tangent spaces may be higher dimensional.

\item The boundary evaluation $u:\partial \Sigma\to L\cup L'$ for all such bad holomorphic curves is contained in some subset of $L\cup L'$ of Hausdorff dimension $\leq n-1$. As such, at almost every point on $L\cup L'$, only automatically transverse curves pass through it.

\item The Solomon functional can be computed by integrating only on the part of $\mathcal{C}$ consisting of automatically transverse curves.

\end{enumerate}
 
\end{itemize}

The automatic transversality assumption should be viewed as a higher dimensional generalization of the fact that on Riemann surfaces, the nontrivial holomorphic polygons are immersions up to the boundary (\cf \cite[Section 13 (b)]{Seidelbook}. The intuition for the weak version is that we sometimes need extra holomorphic curves to maintain $\partial \mathcal{C}=L-L'$, but for questions related to the Solomon functional and the boundary evaluation to the Lagrangians, these extra curves do not contribute.

\subsubsection*{Index theory preliminary}

For a pseudoholomorphic polygon $u:\Sigma\to X$ with inputs at $p_1, \ldots p_k$ and an output at $q$, arranged in clockwise order, the index is $\deg q-\sum_1^k \deg p_k$, where the degree convention is (\ref{Floerdegree}). The index amounts to a Maslov number computation, and an alternative topological description is as follows: take a section $s$ of the complex line bundle $\Lambda^n TM\to \Sigma$, which restricts on $\partial \Sigma$ to a section of the real line bundle $\Lambda^n TL$. (When several Lagrangians are involved, it is understood that $TL$ refers to the appropriate Lagrangian on the portion of $\partial \Sigma$.) We assume $s$ has isolated zeros up to the boundary and the corner (aka. strip like ends). We may also regard $s$ as a function $\Omega(s)$ on the polygon, by contraction with $\Omega$. Then
\begin{equation}\label{indexvanishingorder}
\text{Index}=2 \sum (\text{interior zeros}) + \sum (\text{boundary zeros}) + \sum (\text{excess corner zeros})+n.
\end{equation}
Here the order of zeros is computed from winding numbers, and for general sections $s$ may take positive and negative values. At a corner where $\partial \Sigma$ passes from $L_+$ to $L_-$ in the clockwise direction, we can put the tangent spaces $TL_\pm\subset TX$ into the standard form respecting the complex structure
\begin{equation}
L_+= \R^n, \quad L_-= (e^{i\phi_1},\ldots e^{i\phi_n})\R^n, \quad 0<\phi_i<\pi,\quad TX=\R^n\otimes \C,
\end{equation}
so if $\Omega(s)\sim z^\alpha$ in the complex coordinate of the upper half plane model, the excess vanishing order at the corner is $\frac{1}{\pi}(\alpha-\sum_1^n \phi_i)$. Formula (\ref{indexvanishingorder}) is equivalent to the standard index formula by a topological version of the Cauchy residue formula.

\subsubsection{Automatically transverse cases}

\subsubsection*{Holomorphic strip}

We first consider the holomorphic strip case with input $p$ and output $q$, and the \emph{integrability} of the complex structure will be important. The \emph{first order deformations} of the holomorphic strips $\Sigma$ are given by \emph{holomorphic sections} of $TX|_\Sigma$, which takes boundary value in $TL$ over $\partial \Sigma$, and decays at the corners.

\begin{lem}\label{immersionlem}
If $v_1,\ldots v_n$ are first order deformation vector fields, then either $\Omega(v_1,\ldots v_n)=0$ everywhere on $\Sigma$, or we must have $\deg q-\deg p\geq n$, and when the equality holds then $\Omega(v_1,\ldots v_n)$ only vanishes at the ends with excess vanishing order zero.
\end{lem}

\begin{proof}
We have a section of $\Lambda^n TX|_\Sigma$ given by $v_1\wedge \ldots v_n$, which takes boundary value in $\Lambda^n TL$ on $\partial \Sigma$. Since $v_1,\ldots v_n$ are all \emph{holomorphic}, so must be the function $\Omega(v_1,\ldots v_n)$. Assume this function is not identically zero. By holomorphicity, the zeros are isolated. We claim that the order of zeros must be nonnegative everywhere. This is clear for the interior and the boundary points. We analyze the ends of the strip as  the origin in the upper half plane model with holomorphic coordinate $z$, putting $TL_\pm$ in the standard form at the corner point. The deformation vector field has the leading asymptote
\[
v_k=( a_{k1} z^{\phi_1/\pi}, \ldots, a_{kn}z^{\phi_n/\pi} ) +O(z), \quad k=1,2,\ldots n,
\]
hence
\[
\Omega(v_1,\ldots v_n)=  z^{  (\sum \phi_k)/\pi} ( \det(a_{kj})  +o(1)),
\]
and the excess vanishing order is nonnegative. By the index formula (\ref{indexvanishingorder}), the index $\deg q-\deg p\geq n$, and when equality is achieved all vanishing orders must be zero. In particular $\det (a_{kj})\neq 0$ at the corners.
\end{proof}

\begin{cor}\label{automatictransversalitystrip}
(Automatic transversality, strip case)
Suppose $\deg q-\deg p=n$. 
If $v_1,\ldots v_n$ are $\R$-linearly independent at some point on $\partial \Sigma$ away from the two corners, then $v_1,\ldots v_n$ span the space of all first order deformations, the \emph{obstruction space vanishes}, and   \emph{the moduli space is smooth} at $u:\Sigma\to X$. Morever, the holomorphic strip is an \emph{immersion up to the boundary}.
\end{cor}

\begin{proof}
Since $v_1,\ldots v_n$ are $\R$-linearly independent at a point on $\partial \Sigma$, they span $TL$ at the point, so  $\Omega(v_1,\ldots v_n)\neq 0$. By the Lemma above $v_1,\ldots v_n$ are pointwise complex linearly independent as sections of the holomorphic vector bundle $TX$ over $\Sigma$, so any holomorphic first order deformation can be written as 
\[
v= f_1 v_1+\ldots f_n v_n.
\]
The functions $f_1,\ldots f_n$ are holomorphic on $\Sigma$ up to boundary, and even up to corners due to $\det(a_{kj})\neq 0$. Now subtracting a constant linear combination of $v_1,\ldots v_n$, we can ensure $v$ vanishes at any chosen point on $\partial\Sigma$. Then $\Omega(v, v_2,\ldots v_n)$ has a zero, so must be identically zero by the above Lemma, whence $f_1=0$ identically. Similar all $f_k=0$, so $v=0$. This proves that $v_1,\ldots v_n$ span all first order deformations. Since the index is $n$, and the first order deformation space is $n$-dimensional, we must have \emph{vanishing obstruction space}.

There is a special deformation vector field from $\R$ translation. The nonvanishing result then implies that the holomorphic strip is an immersion up to boundary. At the corners, the holomorphic strip is to leading order
\[
(a_1z^{\phi_1/\pi}+O(z), \ldots  a_n z^{\phi_n/\pi}+O(z) ), \quad a_k\neq 0, \forall k, \quad |z|\ll 1.
\]
By $\det(a_{kj})\neq 0$, this translation vector field cannot be $O(z)$ at the corner, so for at least one choice of $k$, we have $a_k\neq 0$. We say \emph{the failure of immersion at the corner is `minimal'}.
\end{proof}


\subsubsection*{Holomorphic polygon}

We now move on to holomorphic polygons with $k+1$ corner points for $k\geq 2$. The extended linearized Cauchy-Riemann equation  (\cf \cite[Chapter 9]{Seidelbook})
involves a vector field $v\in C^\infty(\Sigma, u^*TX)$ decaying at the ends, and $\rho\in \Omega^{0,1}(\Sigma, T\Sigma)$ representing a tangent vector of the Stasheff associahedron (\ie the deformation of Riemann surface structure on the domain $\Sigma$), satisfying
\[
\bar{\partial} v+ \frac{1}{2} J_X \circ du\circ \rho=0.
\]
where $J_X$ is the complex structure on $X$. Here $\rho$ can be taken to be compactly supported, so $\bar{\partial} v=0$ near the corners. It immediately follows that

\begin{lem}
Given   first order deformation vector fields $v_1,\ldots v_{n-1}$, then the (1,0)-form $\Omega(\cdot, v_1,\ldots v_{n-1})$  on $\Sigma$ is \textbf{holomorphic}.
\end{lem}

\begin{rmk}
Adding vector fields on $\Sigma$ valued in $T\Sigma$ to $v_1,\ldots v_{n-1}$ does not affect $\Omega(\cdot, v_1,\ldots v_{n-1})$ as a 1-form on $\Sigma$. Thus this 1-form is insensitive to how one represents the Riemann surface structures on the abstract polygon.
\end{rmk}

We impose that the input at one of the $k$ corners maps to an intersection point in $CF^0(L,L')$, and the other inputs map to degree one self intersections of $L$ or $L'$. The output maps to $q\in CF^*(L,L')$.

\begin{prop}\label{automatictransversalitypolygon}(Automatic transversality, polygon case)
Suppose $v_1,\ldots v_{n-1}$ are linearly independent first order deformation vector fields.  Either  $\Omega(\cdot, v_1,\ldots v_{n-1})$ vanishes identically as a 1-form on $\Sigma$, or we must have $\deg q\geq n$, and when the equality holds then $\Omega(\cdot, v_1,\ldots v_{n-1})$ only vanishes at corners. In this case, all first order defomation vector fields are spanned by $v_1,\ldots v_{n-1}$, the holomorphic polygon $u:\Sigma\to X$ is an immersion up to the boundary, the obstruction of the extended linearized operator vanishes, and the moduli space is smooth at $u:\Sigma\to X$. 
\end{prop}

\begin{proof}
By viewing the domain of the polygon as a strip with extra boundary punctures, we produce a holomorphic vector field $v_n$ as the $\R$-translation vector field. However, unlike in the strip case, at the degree one self intersection corners $v_n$ does not typically have the required decay to be admitted as a deformation vector field. Indeed, by thinking about such a corner point as the origin in the upper half plane model of $\Sigma$ with local coordinate $z$, then $zv_n$ decays at the corner, but not necessarily $v_n$ itself.

Now $v_1\wedge\ldots v_n$ is a section of $\Lambda^n TX$ with boundary value on $\Lambda^n TL$, and $\Omega(v_1,\ldots v_n)$ is a holomorphic function on $\Sigma$. We assume from now on that it is not identically zero. 
The index of the ordinary Cauchy-Riemann operator is 
\[
\deg q- \sum_1^k \deg p_k=\deg q-k+1.
\]
Invoking (\ref{indexvanishingorder}) this is computable from the vanishing orders of $\Omega(v_1,\ldots v_n)$:
\[
\deg q-k+1=2 \sum (\text{interior zeros}) + \sum (\text{boundary zeros}) + \sum (\text{corner zeros})+n.
\]
The interior and boundary vanishing orders are non-negative. Since the $v_k$ are holomorphic near the corners without correction, the proof of Lemma \ref{immersionlem} shows that
the excess vanishing order at the $CF^0(L,L')$ corner and the $q$ corner are both non-negative. At the degree one self intersection corners, the excess vanishing order of $\Omega(v_1,\ldots v_{n-1}, zv_n)$ is nonnegative by the same previous arguments, so $\Omega(v_1,\ldots v_{n-1}, v_n)$ itself has excess vanishing order $\geq -1$. Hence
$
\deg q-k+1 \geq n-k+1,
$
namely $\deg q\geq n$.

When the equality is achieved, then all bounds are saturated. In particular, $\Omega(v_1,\ldots v_n)$ \emph{can only vanish at the corners}, so $u: \Sigma\to X$ is an \emph{immersion up to boundary}. At the corners, the same arguments in Corollary \ref{automatictransversalitystrip} shows the failure of immersion is minimal.

If $v$ is the deformation vector field corresponding to an arbitrary kernel element of the extended linearized operator, then after subtracting off a constant linear combination of $v_1,\ldots v_{n-1}$, we may assume $v$ is tangent to $\Sigma$ at any chosen point on $\partial\Sigma$. The same argument in Corollary \ref{automatictransversalitystrip} shows $v$ is tangent to the image of $\Sigma$. The immersion property allows us to lift $v$ to the domain $\Sigma$. There is no room to deform the complex structure of $\Sigma$, nor is there any automorphism of $\Sigma$, so in fact $v$ vanishes identically. This shows that $v_1,\ldots v_{n-1}$ \emph{span all first order deformations}. But $\deg q=n$ implies that the index of the extended linearized operator is
\[
\deg q-\sum_1^k p_k +k-2= n-1
\]
Thus the cokernel dimension is zero, namely the obstruction space vanishes. Consequently, the moduli space of such holomorphic polygons is smooth.
\end{proof}

\begin{rmk}\label{cornerOmega}
Using the Floer degree formula  (\ref{Floerdegree}), 
the asymptotic behaviour of $\Omega(v_1,\ldots v_n)$ at the corners can be extracted from the above proof: at the $CF^0(L,L')$ corner point $p$
\[
\Omega(v_1,\ldots v_n)= a_p z^{(\theta_{L'}-\theta_{L})(p)/\pi} (1+O(z)), \quad a_p\neq 0,\quad  \arg a_p= \theta_L(p) \mod \pi \Z.
\]
At the degree one self intersections $p_l\in CF^1(L_+,L_-)$ on $L$ or $L'$, 
\[
\Omega(v_1,\ldots v_n)= a_l z^{ (\theta_{L_-}-\theta_{L_+} )(p_l)/\pi} (1+O(z)), \quad a_l\neq 0,\quad  \arg a_l= \theta_{L_+}(p_l) \mod \pi \Z.
\]
At the degree $n$ output $q$,
\[
\Omega(v_1,\ldots v_n)= a_q z^{(\theta_{L}-\theta_{L'})(q)/\pi} (1+O(z)), \quad a_q\neq 0,\quad  \arg a_q= \theta_{L'}(q) \mod \pi \Z.
\]
\end{rmk}

\subsubsection*{Weighted Sobolev space with exponential growth}

We now discuss solutions to linearized Cauchy-Riemann equations in weighted Sobolev spaces $W^{1,2;\mu}$ (\cf \cite[section 2]{SeidelLefV}). These spaces agree with their unweighted counterparts along the strip like input ends, but at the strip like output end $s\gg 0$, a vector field $v\in  W^{l,2;\mu}$ means that $\exp(-\mu s) v$ lies in $W^{l,2}$. Generally we choose $\mu$ to avoid a discrete set of indicial values. The main point of these weighted Sobolev spaces is that they allow for holomorphic vector fields with prescribed exponential growth along the output end, which is conceptually similar to allowing for meromorphic functions in Riemann surface theory. If we think of the strip like end $q$ as the infinity (resp. the origin) in the upper half plane model of $\Sigma$, then the natural coodinate is $z=e^{\pi (s+it)}$ (resp. $z=e^{-\pi (s+it)}$), and the exponential growth $o(e^{\mu s})$ becomes $o(|z|^{\mu/\pi})$ (resp. $o(|z|^{-\mu/\pi})$.

For larger $\mu$ more vector fields are included in the Sobolev space, and the index increases by one each time $\mu$ crosses an indicial value (counted with multiplicity). In our problem, the indicial values are
\[
\phi_1+\pi\Z, \quad  \phi_2+\pi\Z,\ldots, \phi_n+\pi\Z,
\]
where $\phi_1,\ldots \phi_n$ are the characterizing angles at the Lagrangian intersection point $q$ at the output end. Then the index for the linearized Cauchy-Riemann operator $W^{1,2;\mu}\to L^{2,\mu}$ is 
\begin{equation}\label{indexexpgrowth}
\deg q-\sum_1^k \deg p_i+ \text{number of indicial values between $0$ and $\mu$},
\end{equation}
where $k$ is the number of input ends. In particular, for holomorphic strips with $\deg p=\deg q$ (resp. $\deg q-\deg p=1$), then the index for $\mu=\pi$ is equal to $n$ (resp. $n+1$). In contrast, the ordinary index (for the $\mu=0$ case) is zero, and the moduli space obtained by taking $\R$-quotient has virtual dimension $-1$ (resp. zero). There are in fact sufficient conditions to rule out the negative dimension moduli spaces, and constrain the zero dimensional moduli spaces:

\begin{lem}\label{stripexpgrowth}
In the holomorphic strip case, 
assume $v_1,\ldots v_n$ are in the kernel of the linearized Cauchy-Riemann operator on $W^{1,2;\mu=\pi}$, such that $\Omega(v_1,\ldots v_n)$ does not vanish identically. Then $\deg q-\deg p\geq 1$. When the equality is achieved, the holomorphic strip is an immersion up to the boundary with minimal vanishing at the corner, and the zero dimensional moduli space is regular. 

\end{lem}

\begin{proof}
We modify the proof of Lemma \ref{immersionlem} and Cor. \ref{automatictransversalitystrip}. We think of the corner $q$ as the origin in the upper half plane model. Without loss of generality $v_n$ is the $\R$-translation vector field of the holomorphic strip.
 Then the leading order asymptotic is
\[
v_k= (a_{k1} z^{\phi_1/\pi-1}, \ldots, a_{kn}z^{\phi_n/\pi-1}) +O(1), \quad k=1,2,\ldots n-1,
\]
and 
\[
v_n=( a_{n1} z^{\phi_1/\pi}, \ldots, a_{nn}z^{\phi_n/\pi}) +O(z).
\]
hence
\[
\Omega(v_1,\ldots v_n)=  z^{  (\sum \phi_k)/\pi-n+1} ( \det(a_{kj})  +o(1)),
\]
The excess vanishing order is $\geq 1-n$, where negative order stands for poles. 
By the index formula (\ref{indexvanishingorder}) for the ordinary linearized Cauchy-Riemann equation, we have 
\[
\deg q-\deg p \geq 1-n+n=1,
\]
and equality forces $\Omega(v_1,\ldots v_n)$ to have no interior zero, no boundary zero, minimal zero at $p$, and $\det(a_{kj})\neq 0$ at $q$. The argument in Cor. \ref{automatictransversalitystrip} shows $v_1,\ldots v_n, \frac{v_n}{z}$ span the real vector space of first order deformations in $W^{1,2;\pi}$. In particular, the only first order deformation which decays at $q$ is the $\R$-translation vector field. Thus the cokernel to the ordinary linearized Cauchy-Riemann operator vanishes, and the moduli space is regular.
\end{proof}

A very analogous statement holds in the polygon case, and is left to the reader:

\begin{lem}\label{polygonexpgrowth}
In the holomorphic polygon case, 
assume $v_1,\ldots v_{n-1}$ are in the kernel of the extended linearized Cauchy-Riemann operator on $W^{1,2;\mu=\pi}$, such that  $\Omega(\cdot, v_1,\ldots v_{n-1})$ does not vanish identically as a 1-form on $\Sigma$. Then $\deg q-\sum_1^k\deg p_k+k-2\geq 0$. When the equality is achieved, the holomorphic polygon is an immersion up to the boundary with minimal vanishing at the corner, and the zero dimensional moduli space of holomorphic polygons is regular at $u:\Sigma\to X$. 
\end{lem}

A similar statement applies to teardrop curves:

\begin{lem}\label{teardropexpgrowth}
(Regularity of teardrops)
Let $u:\Sigma\to X$ be a teardrop curve with a unique output $q$ and no input ends.
Assume $v_1,\ldots v_{n-1}$ are in the kernel of the linearized Cauchy-Riemann operator on $W^{1,2;\mu=\pi}$, such that $\Omega(\cdot, v_1,\ldots v_{n-1})$ does not vanish identically as a 1-form on $\Sigma$. Then $\deg q\geq 2$. When the equality is achieved, the teardrop curve is an immersion up to the boundary with minimal vanishing at the corner, and the kernel of the ordinary Cauchy-Riemann operator is  spanned as a real vector space by the M\"obius vector fields on $\Sigma$ fixing the $q$ corner, and the cokernel vanishes. 
\end{lem}

\begin{proof}
We modify the proof of Lemma \ref{stripexpgrowth}. We think of the corner $q$ as the origin in the upper half plane model, and take $v_n$ instead to be the M\"obius vector field $z^2\partial_z$ on $\Sigma$. This has one higher order of vanishing:
\[
v_n= ( a_{n1} z^{\phi_1/\pi+1}, \ldots a_{nn}z^{\phi_n/\pi+1}) +O(z^2).
\]
This leads to 
\[
\Omega(v_1,\ldots v_n)=  z^{  (\sum \phi_k)/\pi-n+2} ( \det(a_{kj})  +o(1)),
\]
so the excess vanishing order at $q$ is $\geq 2-n$. The index of the ordinary linearized Cauchy-Riemann operator is
\[
\deg q= 2 \sum (\text{interior zeros}) + \sum (\text{boundary zeros}) + \sum (\text{excess corner zeros})+n,
\]
whence $\deg q\geq 2$.

 When the equality is achieved, then there is no interior or boundary zero, and $\det(a_{kj})\neq 0$ at the corner, hence the immersion claim. The argument in Cor. \ref{automatictransversalitystrip} shows that $v_1,\ldots v_n, \frac{v_n}{z}, \frac{v_n}{z^2}$ span the real vector space of first order deformations in $W^{1,2;\pi}$.
In particular, the only first order deformation which decays at $q$ are spanned by $v_n$ and $z^{-1}v_n$, namely the M\"obius generators. Since the index of the ordinary Cauchy-Riemann operator is two, the cokernel must have dimension zero, namely the obstruction vanishes.
\end{proof}

The above lemma describes the optimal case for teardrop curves. Such $\deg q=2$ teardrop curves arise in isolated zero dimensional moduli spaces after taking the $Aut(D^2,q)$ quotient, and the counting contribution to $m_0$ are $\pm 1$ depending on the spin structure and the orientation issues.

\subsubsection*{Structure of linearized Cauchy-Riemann equation}

Let $\Sigma$ be a holomorphic polygon with $k\geq 0$ input ends $p_i$, and one output end at $q$. The case $k=0$ corresponds to teardrops, and $k=1$ corresponds to strips.
We consider the ordinary linearized Cauchy-Riemann operator in weighted Sobolev spaces $W^{1,2;\mu}$, to classify the structure of the first order deformation theory. As usual, the complex structure is integrable.
Since $\bar{\partial}$ is elliptic, its cokernel in $L^2$ is finite dimensional, represented by holomorphic 1-forms on $\Sigma$, which must have finite order of vanishing at $q$. For large enough $\mu$, the dual space $L^{2;-\mu}$ for $L^{2;\mu}$ imposes an exponential decay condition $O(e^{-\mu s})$ at $q$, so the cokernel evantually vanishes for $\mu \gg 1$. Then the kernel dimension in $W^{1,2;\mu}$ is equal to the index, computed by (\ref{indexexpgrowth}). For convenience, we use $\mu\in \pi \N$, which avoids the indicial values. Then 
\begin{equation}\label{indexexpgrowth1}
\dim (\ker \bar{\partial}\subset W^{1,2;\mu})= \deg q-\sum_1^k \deg p_i+\frac{n\mu}{\pi}.
\end{equation}

It is convenient to view the domain $\Sigma$ of the holomorphic polygon as the upper half plane with coordinate $z$, with corners  $p_i$ on the real line and $q$ at infinity.

\begin{lem}
	If $v\in \ker \bar{\partial}\subset W^{1,2;\mu}$, then $v=fw$ for some real coefficient polynomial function $f$ on the upper half plane, such that $w$ is nonvanishing on $\R\setminus \{ p_1,\ldots p_k \} $, and vanishes minimally at $p_i$ (meaning $w$ is indivisible by $z-p_i$). 
\end{lem}

\begin{proof}
	If $v$ vanishes at any boundary point $a$ on $\R\setminus \{p_1,\ldots p_k \}$, or if $v$ vanishes at $p_i$ beyond minimal order, then $\frac{v}{z-a}$ (resp. $\frac{v}{z-p_i}$) is also a first order deformation with the same $TL$ boundary condition, subject to the growth constraints at infinity. Since the kernel dimension is finite, the divisions can only happen a finite number of times, producing the polynomial $f$. 
\end{proof}

Let $p\in \partial \Sigma\simeq \partial\mathbb{H}$, and let $K$ be the maximal number depending on $p$, such that there exist $\R$-linearly independent $v_1,\ldots v_K\in \ker\bar{\partial}$, satisfying
\begin{itemize}
\item In case $p$ is not a corner point, then $v_1(p),\ldots, v_K(p)$ are $\R$-linearly independent vectors,
\item In case $p=p_i$ is a corner point, then the nonzero elements in the $\R$-span of $v_1,\ldots v_K$ are vector fields vanishing minimally at $p$.
\end{itemize}

\begin{lem}
	Any $v\in \ker \bar{\partial}\subset W^{1,2;\mu}$ is of the form $g_1 v_1+ \ldots g_K v_K$ for some real coefficient rational functions $g_1,\ldots g_K$, nonsingular at $p$.
\end{lem}

\begin{proof}
Without loss of generality $p=0$.
Let $w_0$ be any first order deformation. By the maximality of $K$, we can choose real numbers $a_i$, such that the first order deformation $w_0- \sum_1^K a_i v_i$ vanishes at zero, so $w_0- \sum_1^K a_{i1} v_i= z^{k_1}w_1$ for some first order deformation $w_1$ which is nonzero at the origin. Finite dimensionality means this process can be repeated for only a finite number of times:
	\[
	\begin{cases}
	w_0= \sum a_{i1} v_i+ z^{k_1} w_1,
	\\
	w_1=\sum a_{i2}v_i+ z^{k_2} w_2, 
	\\ \ldots
	\\
	w_{N-1}=\sum a_{iN} v_i+ z^{k_N} w_N.
	\end{cases}
	\]
	We choose the smallest $N$ such that $v_1,\ldots v_K, w_0,\ldots w_N$ are $\R$-linearly dependent as vector fields; notice $v_1,\ldots v_K$ are linearly independent, so $N\geq 0$. We then get a linear relation
	\[
	f_0(z) w_N= \sum_1^K f_i(z) v_i,
	\]
	where $f_0,\ldots f_K$ are polynomials, and $f_0(0)\neq 0$. This implies the claim.
\end{proof}

We can also apply a M\"obius transform to make the output end $q$ lie at the origin. The growth condition translates to $o(|z|^{-\mu/\pi})$ at zero. Let $K$ be maximal, such that there are $\R$-linearly independent vector fields $z^{-\mu/\pi}v_1,\ldots, z^{-\mu/\pi}v_{K}\in \ker\bar{\partial}\subset W^{1,2;\mu}$, and any nonzero element in the $\R$-span of $v_1,\ldots v_K$ vanishes mimimally at $q=0$. As a caveat, this does not assume $v_1,\ldots v_{K}$ satisfy the growth constraints at infinity to lie in $\ker\bar{\partial}\subset W^{1,2;\mu}$. Minor adaptions give

\begin{lem}
Any $v\in \ker \bar{\partial}\subset W^{1,2;\mu}$ is of the form $z^{-\mu/\pi}(g_1 v_1+ \ldots g_K v_K)$ for some real coefficient rational functions $g_1,\ldots g_K$, nonsingular at $q$.
\end{lem}

\begin{cor}
The number $K$ is independent of the boundary and corner points on $\partial \Sigma$.
\end{cor}

We view the boundary $\partial\Sigma\simeq \mathbb{P}^1(\R)$. By the above lemmas, there is a real algebraic vector bundle $\mathcal{E}$ of rank $K$ over  $\mathbb{P}^1(\R)$ such that $v_1,\ldots v_K$ provide the basis of local sections.  By Grothendieck's classification of vector bundles,

\begin{prop}
$\mathcal{E}\simeq \oplus_1^K\mathcal{O}(n_i)$ for some $n_i\in \Z$. 
\end{prop}

Since the rank $K$ is nondecreasing in $\mu$, it evantually stabilizes for $\mu\gg 0$. Since around any given point, the same choice of $v_1,\ldots v_K$ is valid for all large $\mu$, the algebraic vector bundle $\mathcal{E}$ is independent of $\mu\gg 0$. The elements of $\ker \bar{\partial}\subset W^{1,2;\mu}$ can be interpreted as global sections of $\mathcal{E}\otimes \mathcal{O}(\frac{\mu}{\pi})$. Thus  for all large $\mu$,
\begin{equation}
\dim ( \ker \bar{\partial}\subset W^{1,2;\mu}  )= \dim \Gamma( \mathbb{P}^1(\R),  \mathcal{E}\otimes \mathcal{O}(\frac{\mu}{\pi}) ) = K(\frac{\mu}{\pi}+1)+ \sum_1^K n_i.
\end{equation}
Contrasting with the index formula (\ref{indexexpgrowth1}),

\begin{cor}
The rank $K=n$, and the degree $\sum_1^n n_i= \deg q-\sum_1^k \deg p_i -n$.
\end{cor}

The structure of $\mathcal{E}\simeq \oplus_1^K\mathcal{O}(n_i)$ provides meromorphic sections $v_1, \ldots v_n$ which are a basis of local sections on $\R\subset \mathbb{P}^1(\R)$, and have excess vanishing orders $n_1,\ldots n_n$ at $q$. Consider the function $\Omega(v_1,\ldots v_n)$. By construction, it has no boundary zero, and its excess corner vanishing order is $\sum_1^n n_i$.
Comparing with the index formula (\ref{indexvanishingorder}), 
\[
\begin{split}
\deg q-\sum_1^k p_i= 2 \sum (\text{interior zeros}) + \sum n_i +n.
\end{split}
\]
Since all interior vanishing orders are nonnegative by holomorphicity, 

\begin{cor}
We have $\Omega(v_1,\ldots v_n)\neq 0$ in the interior of $\Sigma$. 
\end{cor}

The significance is that the \textbf{algebraic vector bundle} structure on $\mathcal{E}\to \mathbb{P}^1(\R)$ now extends over the entire $\Sigma$. The $v_1,\ldots v_n$ now provide the basis of local sections for the vector bundle $u^*TX|_\Sigma$. One upshot is that an algebraic structure arises on $u^*TX\to \Sigma$ from solving the Cauchy-Riemann equation with Lagrangian boundary: 
\begin{equation}\label{canonicalLagboundary}
(u^*TX, u^*TL)\simeq (\oplus_1^n \mathcal{O}(n_i), \text{natural real structure}).
\end{equation}
In contrast, the Lagrangians are only assumed to be smooth, not necessarily real analytic.

To analyze obstructions, Serre duality motivates us to consider the dualized cokernel to the ordinary (unweighted, unextended) linearized Cauchy-Riemann operator. A dualized cokernel element $\eta$ is represented by a holomorphic 1-forms in $\Omega^{1,0}(\Sigma, u^*T^*X)$ with $L^2$ integrablity, and its $T^*X$ factor lies in the annilator of the $TL$ boundary condition.
Equivalently, for all test vector fields $v\in W^{1,2}(\Sigma, TX, TL)$,
\[
\int_{\Sigma} \langle \bar{\partial} v\wedge \eta \rangle =0,
\]
where $\langle,\rangle$ is the pairing of $TX$ with $T^*X$, and the wedge takes care of the forms on $\Sigma$. 
In the canonical form (\ref{canonicalLagboundary}), this dualized cokernel is isomorphic to
$
\Gamma(\mathbb{RP}^1, \oplus_1^n \mathcal{O}(-2-n_i) ).
$
In particular,

\begin{cor}
In the teardrop curve case $k=0$, the strip case $k=1$ and the triangle case $k=2$, the vanishing of cokernel is equivalent to $n_i\geq -1$ for all $i$.

\end{cor}

For $k\geq 3$ the deformation of the holomorphic polygons is governed instead by the extended Cauchy-Riemann equation, since the punctured Riemann surface structure on $\Sigma$ is allowed to vary. The dualized cokernel of the extended Cauchy-Riemann operator, is the subspace of the dualized cokernel of the ordinary Cauchy-Riemann operator, which pairs trivially with  $J_X\circ du\circ \rho$ for any $\rho$ representing some tangent vector of the Stasheff associahedron.

\subsubsection*{Hamiltonian deformations and transversality}

We now consider the parametrized moduli space of holomorphic curves over the infinite dimensional space of Hamiltonian deformations for the Lagrangian $L$. Infinitesimally around a holomorphic curve $\Sigma$, we have a Hamiltonian vector field $X_H$ defined by $\omega(X_H,\cdot)=dH$, viewed as a $TX$-valued vector field  over $\Sigma$. We are interested in whether the Hamiltonian deformation kills the cokernel of the ordinary Cauchy-Riemann operator. 
This question was first addressed by Oh \cite{Oh}. The following account follows a similar strategy but differs in details.

Recall the ordinary Cauchy-Riemann operator maps $W^{1,2}(\Sigma, u^*TX, u^*TL)$ to $L^2(\Sigma, u^*TX\otimes T^{*(1,0)}\Sigma)$. The effect of Hamiltonian deformation is to enlarge the domain of the $\bar{\partial}$ operator, by including the vector fields $X_H$ for all the allowed Hamiltonians $H$. 
The question is to analyze the pairing of $\bar{\partial} X_H$ with the dualized cokernel elements.

\begin{prop}
Let $u:\Sigma\to X$ be a holomorphic disc which is immersed near some point $z_0\in \partial \Sigma$ with the boundary injectivity property $u|_{\partial \Sigma}^{-1}(u(z_0))=\{z_0\}$. 
Let  $\eta$ be a nonzero dualized cokernel element for the ordinary linearized Cauchy-Riemann operator. Then there is a Hamiltonian $H$ supported in any prescribed small ball on $X$ containing $u(z_0)$, such that $\int_{\Sigma} \langle\bar{\partial}X_H \wedge \eta \rangle \neq 0$.
\end{prop}

\begin{proof}
Since $\eta$ is a holomorphic 1-form valued in $u^*T^*X$, Stokes theorem gives
\[
\int_{\Sigma} \langle\bar{\partial}X_H \wedge \eta \rangle = \int_{\partial \Sigma} \langle X_H, \eta\rangle,
\]
where $\langle,\rangle$ stands for the pairing between $TX$ and $T^*X$. On $\partial \Sigma$, we can write $\eta=\omega(\cdot, Y) ds$ for some vector field $Y$ valued in $u^*TX$, and $s$ is any local coordinate on $\partial \Sigma$. The cokernel element condition implies $\omega(v, Y)=0$ for any $v\in u^*TL$, so $Y$ must in fact be valued in the Lagrangian subbundle $u^*TL$. Thus
\[
\int_{\partial \Sigma} \langle X_H, \eta\rangle = \int_{\partial \Sigma} \omega( X_H, Y)ds = \int_{\partial \Sigma} dH( Y)ds.
\]
We suppose for contradiction, that this pairing vanishes identically for any $H$ supported in the prescribed ball.

By the holomorphicity of $\eta$, its zeros are isolated, so without loss of generality $Y$ does not vanish in the local portion of $\partial\Sigma$ where $u$ is injective and immersed. Suppose first that $Y$ is not tangent to the image of $\Sigma$. Then we find some local function $h$ on a small ball in $X$ with $dh(Y)= 1$ and $h=0$ on the local portion of $\partial \Sigma$, and another  cutoff function $h_2\geq 0$ with $dh_2(Y)=0$ along $\partial \Sigma$, supported in a small ball.
Taking $H=hh_2$, then 
\[
\int_{\partial \Sigma} dH( Y)ds= \int_{\partial \Sigma} h_2 ds\neq 0.
\]
This contradiction shows $Y$ is tangent to the image of $\Sigma$ in the local portion of $\partial \Sigma$. We can write $Y=f\partial_s$ for some local function $f$. Then requiring 
\[
\int_{\partial \Sigma} dH( Y)ds = \int_{\partial \Sigma} f \partial_s H ds= - \int_{\partial \Sigma} H \partial_s f ds
\]
for any compactly supported local function $H$, implies that $f$ is constant in the local portion of $\partial\Sigma$. Thus up to multiplying by a nonzero constant, locally
\begin{equation}\label{Hamtransversality1}
Y=\frac{\partial u}{\partial s} ds, \quad \eta= \omega(\cdot, \frac{\partial u}{\partial s}) ds.
\end{equation}

We now produce holomorphic vector fields on $\Sigma$. For holomorphic strips or polygons with $k+1\geq 3$ corners, we select one input end as $p$, and call the output $q$ as usual, and represent $\Sigma$ as a strip with $k-1$ boundary punctures. This perspective provides a natural translation vector field $\frac{\partial u}{\partial s}$, which have exponential decay along the $p, q$ ends, but may not be $L^2$ near the other $k-1$ ends. Instead, by thinking about the $k-1$ ends as the origin in the upper half plane model, we see
\[
\frac{\partial u}{\partial s}=O(|z|^{\alpha-1}), \quad \alpha= \min \{ \phi_1/\pi,\ldots \phi_n/\pi  \}
\]
for the characterizing angles $\phi_1,\ldots \phi_n$ at the Lagrangian intersection point.
The $T^{(1,0)}X$ part of $2J_X\frac{\partial u}{\partial s}$ is $J_X\frac{\partial u}{\partial s}+\sqrt{-1}\frac{\partial u}{\partial s}$. Contracting this with the $T^{*(1,0)}X\otimes T^{*(1,0)}\Sigma$ part of $\eta$
yields a 1-form on $\Sigma$
\[
\zeta=\eta( J_X \frac{\partial u}{\partial s} +  \sqrt{-1}\frac{\partial u}{\partial s} )
\]
which is also holomorphic, with boundary value along $\Sigma$
\begin{equation}\label{Hamtransversality2}
\zeta= \omega(J_X \frac{\partial u}{\partial s}+  \sqrt{-1}\frac{\partial u}{\partial s}   , Y    )ds = \omega(J_X \frac{\partial u}{\partial s}  , Y    )ds .
\end{equation}
Here $\omega( \frac{\partial u}{\partial s},Y  )=0$ since both vectors satisfy the $TL$ boundary condition. Notably, the boundary condition of $\zeta$ is real valued.
In the upper half plane model, the Schwartz reflection principle allows us to extend $\zeta$ meromorphically over $\mathbb{CP}^1$.

 At any of the 
 $k-1$ ends, since $\eta\in L^2$, we know by holomorphicity $|\eta|=O(|z|^\alpha)$, so $\zeta=O(|z|^{2\alpha-1})$ in the upper half plane model, hence has no pole. At the $p,q$ ends, by the decay of the holomorphic $\frac{\partial u}{\partial s}$ and $\eta$, we likewise infer that $\zeta$ has no pole in the upper half plane model. In conclusion, the extension of $\zeta$ over $\mathbb{CP}^1$ has no pole, so must in fact vanish. However, by (\ref{Hamtransversality1})(\ref{Hamtransversality2}), on a local portion of $\partial \Sigma$
 \[
 \zeta= \omega(J_X \frac{\partial u}{\partial s},   \frac{\partial u}{\partial s} ) ds \neq 0.
 \]
 This contradiction proves the Proposition in the $k\geq 1$ case.

Finally, for the teardrop curve case $k=0$, 
we replace the holomorphic vector field $\frac{\partial u}{\partial s}$ by the M\"obius vector fields vanishing at the corner, and the rest of the arguments are entirely similar.  
\end{proof}

The upshot is that by the Sard-Smale theorem, provided we can always ensure `somewhere boundary injectivity' for any holomorphic disc in a given moduli space, then generic Hamiltonian perturbation would be able to achieve regularity for the moduli space.

\begin{rmk}\label{boundaryinjectivity}
In the exact setting there is no closed holomorphic curve. 	
The failure of `somewhere boundary injectivity' is often associated with multiple cover issues, namely $u:\Sigma\to X$
may decompose into several domain components, each of which factorizes through a somewhere boundary injective holomorphic disc (\cf \cite{KwonOh} for the case of Lagrangian boundary with no corners).

In the simplest case, if $u$ factorizes through another disc, then the corner points would be repeated several times on $\partial \Sigma$. This phenomenon does not happen for the curves appearing in the bordism current $\mathcal{C}$, which involve only one corner at $CF^0(L,L')$ and one corner at $CF^0(L',L)$. Nor does this occur for teardrop curves, which have only one corner at a degree two self intersection point. This raises hope that the failure of `somewhere boundary injectivity' may be \emph{highly nongeneric}, or in certain situations can be ruled out altogether.

\end{rmk}

\subsubsection*{Further comments on automatic transversality}

We now comment on the gap between what we proved and the (weak version of) automatic transversality that we will later assume.

\begin{enumerate}
\item  Prop. \ref{automatictransversalitypolygon}, Lemma \ref{immersionlem} and Cor. \ref{automatictransversalitystrip} establish the dichotomy for holomorphic discs $u:\Sigma\to X$ arising in virtual dimension $n-1$ moduli spaces, that either $u$ is automatically transverse, or $\Omega(\cdot, v_1,\ldots v_{n-1})$ vanishes for any $(n-1)$ first order deformation vectors. This argument does not establish unperturbed regularity for the lower dimensional moduli spaces, so it is not completely clear if complex structure perturbations can be removed in the arguments for $\partial \mathcal{C}=L-L'$ in section \ref{LotayPacini}.

\item For the bad curves, $\Omega(\cdot, v_1,\ldots v_{n-1})$ vanishes identically as a 1-form on $\Sigma$, so at any point on the boundary, $v_1,\ldots v_{n-1}$ and the tangent vector to $\partial \Sigma$ are $\R$-linearly dependent. Suppose for the moment that the moduli spaces are regular, then the boundary evaluation to $L\cup L'$ for the bad curves arise in Hausdorff dimension at most $n-1$. Morever, since the Solomon functional is defined through $\int_{\mathcal{C}}\lambda\wedge \Omega$, and $\Omega$ vanishes around the bad curves, smoothness assumptions imply that the bad curves cannot contribute.

When regularity assumptions are dropped, one needs to appeal to virtual techniques, so these conclusions require further justification. One problem is that the standard virtual perturbation techniques based on Kuranishi structures do not necessarily produce virtual cycles inside the original moduli spaces, but only inside their small neighbourhoods. This perturbation step destroys the identical vanishing of $\Omega$, by a small amount corresponding to the size of the perturbation. As one shrinks the size of the perturbations, one needs uniform mass bound on the virtual chains to justify that the integral contribution to $\int_{\mathcal{C}}\lambda \wedge \Omega$ from the bad curves actually converges to zero.

\item Alternatively, one can hope to replace Lagrangians by arbitrarily small Hamiltonian perturbations to achieve transversality. This is mostly adequate for our purpose, except that one needs to justify the `somewhere boundary injectivity' property (\cf Remark \ref{boundaryinjectivity}).

\end{enumerate}

\subsection{Positivity condition}\label{Positivitycondition}

The $(n-1)$-dimensional moduli spaces contributing to the bordism current come with orientation signs and weighting factors. The \textbf{positivity condition} (\ie \textbf{no cancellation of signs}) means that around any automatically transverse curves, if the first order deformations $v_1,\ldots v_{n-1}$ form an oriented basis of the moduli space, and $v_0$ be a clockwise ordered vector field on $\partial \Sigma$, then upon boundary evaluation $v_0\wedge\ldots v_{n-1}$ agrees with the orientation of $L-L'$. For the other holomorphic curves, the question of orientation does not arise, because the boundary evaluation maps have degenerate differentials everywhere on $\partial \Sigma$.

The positivity condition forbids two curves passing through a generic point with the evaluation maps contributing opposite signs. Such a requirement is \emph{geometric rather than homological}, and if we go beyond the almost calibrated case, it also \emph{depends on the choice of the generators} $\alpha\in CF^0(L,L')$ and $\beta\in CF^0(L',L)$, rather than only their classes in $HF^0$.

\begin{Question}
Given two unobstructed Lagrangian objects $L,L'$ which are isomorphic in $D^b Fuk(X)$.
When can we make gauge choices for the local systems and the bounding cochain data, and choices of the Floer cohomology group generators, such that the bordism current $\mathcal{C}$ produced from the universal family of holomorphic curves satisfies the positivity condition?
\end{Question}


The positivity condition will arise in the applications as follows. We will write various quantities as integrals over the $(n-1)$-dimensional moduli spaces of holomorphic curves, and the positivity condition would in each case imply the \textbf{pointwise positivity of the integrand}. In the mirror analogy, this corresponds to the pointwise positivity of curvature integrands, which features for instance in the proof that the Hermitian Yang-Mills equation implies the semistability of bundles (\cf section \ref{HYM}).


\subsubsection*{Morse theory analogy}

The intuition of the positivity condition can be explained through the following \textbf{analogy with Morse theory}. Given a compact oriented manifold $M$ with a Morse-Smale function $H$, then
\begin{itemize}
\item The degree zero (resp. $n$) elements in the Morse cochain complex are generated by the local maxima (resp. local minima) of $H$ whose unstable submanifolds (resp. stable submanifolds) have \emph{preferred orientations}. 

\item The \emph{fundamental class} $1\in H^0(M)$ is represented by the sum of the local maxima with the preferred orientation. Similarly with the generator of $H^n(M)$.

\item  A generic point on $M$ lies on exactly one Morse flowlines which starts with one local maximum point and ends on one local minimum point. More formally, if we construct the universal family of Morse flowlines that start with some local maximum and ends on some local minimum, the evaluation map  would \emph{sweep out the fundamental cycle} of $M$ as an $n$-dimensional current, \emph{without any cancellation effect}.
\end{itemize}

Now for simplicity, if we start with an embedded Lagrangian brane $L$, and preform a small generic Hamiltonian deformation $\phi_H(L)$, then the Floer cohomology $HF^*(L, \phi_H (L))$ is computed as the Morse cohomology, so the Morse theory statements above would imply the positivity condition at least in such special cases. The intuition is that if the Lagrangian  $L'$ (together with its brane structure) is a sufficiently small deformation of $L$, then we expect the positivity condition to hold for the bordism current between $L$ and $L'$.

\subsubsection*{Positivity for individual moduli spaces}

In general the bordism current $\mathcal{C}$ receives contributions from many $(n-1)$-dimensional moduli spaces of holomorphic curves. We now focus on one moduli space by fixing the choice of the Lagrangian intersections $p_1,\ldots p_k, q$ and the homotopy type of $u: \Sigma\to X$, and consider a connected open subset of the moduli space which contains only automatically transverse curves. We observe:

\begin{itemize}
\item For a fixed automatically transverse holomorphic curve $u:\Sigma\to X$, along $\partial \Sigma$ between any two successive corners, by the nowhere vanishing of $\Omega(\cdot, v_1,\ldots v_n)$, the Jacobian of the boundary evaluation map cannot change sign. That is, either the orientation of the universal family agrees with the orientation of $L$ (resp. $-L'$) along $u:\partial \Sigma \to L$ at every point along the boundary portion of $\partial \Sigma$, or the two orientations disagree at every point.

\item The Lagrangians are graded by assumption, and the orientations are canonically determined by $e^{-i\theta}\Omega$. The corner behaviour (\cf Remark \ref{cornerOmega}) implies that at the degree one self intersections, $e^{-i\theta}\Omega(\cdot, v_1,\ldots v_{n-1})$ does not change sign. On the other hand, at the $CF^0(L,L')$ and the $CF^0(L',L)$ ends along $\partial \Sigma$, the 1-form $e^{-i\theta}\Omega(\cdot, v_1,\ldots v_{n-1})$ changes orientation sign.
Thus at a fixed automatically transverse curve, the orienation of the universal family and $L-L'$ either completely agree along every point of $\partial \Sigma$, or completely disagree.

\item As we deform among automatically transverse curves, the orientation signs cannot change. Thus either all these holomorphic curves contribute positively to $\partial \mathcal{C}$, or they all contribute negatively.

\end{itemize}

The above discussion also suggests the limitation of the positivity condition: if we encounter a holomorphic curve in the moduli space, which is \emph{not automatically transverse}, then it is possible to switch orientation signs.
For arbitrary exact immersed Lagrangians, it seems unreasonable to expect the positivity condition, and it is conceivable that counterexamples may arise from $h$-principle constructions. Whether counterexamples occur for more restrictive Lagrangians seems less clear, and we leave the following sample questions as food for thought:

\begin{Question}
How does the positivity condition behave under exact isotopy with surgery?
\end{Question}

\begin{Question}
Are there examples of exact Calabi-Yau manifolds such that the positivity condition is satisfied for bordism currents between all exact, almost calibrated, unobstructed immersed Lagrangians equipped with suitable brane structures? What if the Lagrangians are quantitatively almost calibrated (\cf (\ref{quantitativealmostcalibrated}))?
\end{Question}


\begin{rmk}
If the Fukaya category is defined over $\Z$, we can require all holonomy factors to be integer valued. The positivity condition requires all the contributions to $\partial \mathcal{C}=L-L'$ to have the same orientation sign. This has the amusing consequence that all holonomy factors associated with $(n-1)$-dimensional moduli spaces contributing to $\mathcal{C}$, must in fact all be $+1$. Intuitively, this means there is a unique such holomorphic curve through any generic point of $L$, and the boundary evaluation of universal family to $L$ is transverse. 
\end{rmk}

\subsection{Floer theoretic obstructions}\label{Floertheoreticobstructions}

\subsubsection*{General features of obstruction conditions}

Our goal is to look for obstructions to the existence of special Lagrangians within given $D^b Fuk(X)$ classes, which is the `easy direction' of the conjectural stability condition.
Before specializing to a technically oversimplified setup, we first explain the features we expect from these obstructions, which may hold in much more general contexts.  The mirror analogy (\cf our discussion on the $\mu$-stability in section \ref{HYM}) suggests:

\begin{itemize}
	\item  The obstructions are associated to certain \textbf{positivity} of signs, which essentially depend on the \textbf{integrability} of K\"ahler geometry.

	\item  The quantity involved in the obstruction can be expressed as an \textbf{integral over a moduli space} of worldsheet instantons (\ie holomorphic curves), and its sign comes from a \textbf{pointwise positivity} of the integrand on the moduli space.

	\item The input from Floer theory is associated to a \textbf{distinguished triangle} in $D^b Fuk(X)$, or possible generalisations to several Lagrangians.
	
	\item 
	The role of the holomorphic volume form enters via \textbf{cohomological integrals}.

	\item There is no need for the complex Monge-Amp\`ere equation. Only the \textbf{almost Calabi-Yau} condition is needed.
	
\end{itemize}

Furthermore, out of the many moduli spaces that may arise in Floer theory, we will only make use of certain $(n-1)$-dimensional moduli spaces of holomorphic curves, whose associated $(n+1)$-dimensional universal family provides bordism currents between the $n$-dimensional Lagrangians. Here are some a priori reasons why we restrict attention to these:
\begin{itemize}
\item   The holomorphic volume form is naturally integrated over $n$-cycles. This explains the dimension. 

\item  We need bordism currents canonically associated to the distinguished triangles. In Floer theory, the $A_\infty$-structure only becomes an invariant when considered as a whole, and individual $A_\infty$ products are not invariants, so invariance constrains how moduli spaces can enter into stability conditions. As mentioned in section \ref{LotayPacinirevisited}, the existence of the bordism current is the geometric manifestation of linear relations in the zeroth Hochschild homology $HH_0$ of the Fukaya category, which contains important invariant information.

\item  From a variational viewpoint which will be discussed more fully in Chapter \ref{Variational}, it is desirable to extend Floer theory to Lagrangians with much weaker regularity, in the varifold and current sense. We shall explain there that most of Floer cohomologies and $A_\infty$ products cannot be expected to pass to the limit when the Lagrangians degenerate in such weak topologies, and we hope that the bordism currents we use are among the few pieces of Floer theory that \emph{may} be well behaved under rather severe degenerations of Lagrangians.

\end{itemize}

These requirements are very stringent. We notice two other features:

\begin{itemize}
\item  We shall crucially rely on the almost calibrated condition.

\item  When we test the stability of $L$ via the distinguished triangle $L_1\to L\to L_2\to L_1[1]$, we do not wish to assume $L_1$ or $L_2$ is special Lagrangian. In our view, stability conditions should be expressed in Floer theoretic terms, without a priori knowledge of what special Lagrangians there are inside a given almost Calabi-Yau manifold. 
\end{itemize}

\subsubsection*{The Floer theoretic obstruction condition}

The following Floer theoretic obstruction will crucially require \textbf{complex integrability} and the \textbf{almost calibrated condition}. Assume $L_1\to L\to L_2\xrightarrow{\gamma} L_1[1]$ be a distinguished triangle of unobstructed exact immersed Lagrangian branes with bounding cochains, such that $L_1,L,L_2$ are all almost calibrated, and all intersections are transverse. In other words, the Lagrangian brane $L$ is isomorphic in $D^bFuk(X)$ to the immersed Lagrangian corresponding to the twisted complex (\cf section \ref{LotayPacinirevisited}, \ref{immersedFukaya})
\[
L'\simeq \left( \begin{matrix}
(L_2, b_2) & \\
\gamma & (L_{1}, b_{1}) 
\end{matrix} \right).
\]
We obtain a bordism current $\mathcal{C}$ with $\partial \mathcal{C}=L-L'=L-L_1-L_2$. In our generality, the domains of $L_1, L_2,L$ may have many connected components.


\begin{thm}\label{Floertheoreticobstruction1}
(Floer theoretic obstruction)
Assume the \textbf{automatic transversality} and the \textbf{positivity condition} hold for the bordism current $\mathcal{C}$. Assume the \textbf{destabilizing condition}
\[
\hat{\theta}_1= \arg \int_{L_1}\Omega> \hat{\theta}_2=\arg \int_{L_2}\Omega.
\]
Then the \textbf{Lagrangian phase angle} of $L$ has a lower bound on its oscillation:
\begin{equation}\label{phaseobstruction}
\sup_L \theta_L\geq \hat{\theta}_1,\quad \inf_L \theta_L\leq \hat{\theta}_2,
\end{equation}
and morever the \textbf{J-volume} of $L$ (\cf section \ref{LotayPacini}) has a nontrivial lower bound
\begin{equation}\label{volumeobstruction}
\text{Vol}_J(L)=\int_L e^{-i\theta}\Omega \geq |\int_{L_1}\Omega|+|\int_{L_2}\Omega|.
\end{equation}
\end{thm}

\begin{proof}
At a holomorphic polygon $\Sigma$ in the universal family $\mathcal{C}$, denote
 $v_1, \ldots v_{n-1}$ as the  first order deformation vector fields representing an oriented basis of tangent vectors to the moduli space. In the special case of holomorphic strips, the moduli space refers to the $\R$-translation quotient. We noted in section \ref{Automatictransversalitypositivity} that $\Omega(\cdot, v_1,\ldots, v_{n-1})$ restricts to a holomorphic 1-form on $\Sigma$, so can be written as the differential of a \textbf{holomorphic function} $F$ by the simply connectedness of $\Sigma$:
 \begin{equation}\label{holoF}
 dF= \Omega(\cdot, v_1,\ldots, v_{n-1}).
 \end{equation}

The corners on $\Sigma$ are arranged in \emph{clockwise} order with the following possibilities:
\begin{itemize}
\item In the primary case, we encounter some degree one self intersections on $L$ from bounding cochains, a corner $p \in CF^0(L, L_2)$, some degree one self intersections on $L_2$, a corner $r$ from $\gamma\in CF^1(L_2,L_1)$, some degree one self intersections on $L_1$, and a corner at $q\in CF^0(L_1, L)$. Notice the Lagrangian boundary follows $L, L_2, L_1$ in clockwise order, and we cannot go reversely from $L_1$ to $L_2$ instead.

\item In the secondary cases, the boundary data may miss either $L_1$ or $L_2$. For instance, we may encounter  some degree one intersections on $L$, a corner $p\in CF^0(L,L_2)$, some degree one intersections on $L_2$ and a corner at $q\in CF^0(L_2, L)$. The Lagrangian boundary follows $L, L_2$ in clockwise order. The alternative possibility of Lagrangian boundary along $L$ and $L_1$ is entirely similar.

\end{itemize}

In all cases, there is precisely one corner $p$ at $CF^0(L,L')$ and a corner $q$ at $CF^0(L',L)$. We can normalize $F(q)=0$ to fix the constant. In the primary case, there is a corner $r\in CF^1(L_2,L_1)$, which is absent in the secondary cases. In general, the bordism current $\mathcal{C}$ receives contributions from many moduli spaces, and all three cases may arise depending on the generators of $HF^0(L,L')$ and $HF^0(L',L)$.

We can now define \textbf{complex valued volume forms} on the $(n-1)$ dimensional moduli spaces of holomorphic curves. Recall $v_1,\ldots v_{n-1}$ represent the tangent vectors to the moduli spaces, and the holomorphic function $F$ depends on $v_1\wedge \ldots v_{n-1}$. In the primary case, we define 
\[
\begin{cases}
\tilde{\Omega}_{L}(v_1,\ldots v_{n-1}) = F(p),
\\
\tilde{\Omega}_{L_1}(v_1,\ldots v_{n-1})= F(r),
\\
\tilde{\Omega}_{L_2}(v_1,\ldots v_{n-1})= F(p)- F(r).
\end{cases}
\] 
In the secondary cases, if the Lagrangian boundary lies on $L$ and $L_1$, then
\[
\begin{cases}
\tilde{\Omega}_{L}(v_1,\ldots v_{n-1}) = F(p),
\\
\tilde{\Omega}_{L_1}(v_1,\ldots v_{n-1})= F(p),\\
\tilde{\Omega}_{L_2}(v_1,\ldots v_{n-1})=  0.
\end{cases}
\]
If the Lagrangian boundary lies on $L$ and $L_2$, then
\[
\begin{cases}
\tilde{\Omega}_{L}(v_1,\ldots v_{n-1}) = F(p),
\\
\tilde{\Omega}_{L_1}(v_1,\ldots v_{n-1})= 0,\\
\tilde{\Omega}_{L_2}(v_1,\ldots v_{n-1})=  F(p).
\end{cases}
\]
The values of $F$ should be understood as the integral of $dF$ on the appropriate portions of $\partial \Sigma$.
The key point is that since $\partial \mathcal{C}$ sweeps out the cycle $L-L_1-L_2$, we can write the \textbf{period integrals} as \textbf{integrals on the $(n-1)$-dimensional moduli spaces of holomorphic curves}:
\begin{equation}
\int_L\Omega= \int_{\mathcal{M}} \tilde{\Omega}_{L}, \quad \int_{L_i}\Omega= \int_{\mathcal{M}} \tilde{\Omega}_{L_i}, \quad i=1,2.
\end{equation}
where $\mathcal{M}$ is a shorthand for the weighted sum over contributions from all the $(n-1)$-dimensional moduli spaces involved in the construction of $\mathcal{C}$, \cf the Appendex \ref{immersedFukaya}.

Recall the \textbf{positivity condition} means that if $v_0$ stands for a \emph{clockwise oriented} tangent vector on $\partial \Sigma$, then $v_0\wedge v_1\ldots \wedge v_{n-1}$ agrees with the orientation on $L$, and is \emph{opposite} to the orientation on $L'$. The nonvanishing of $v_0\wedge v_1\ldots \wedge v_{n-1}$ is a consequence of the \emph{immersion property} from the automatic transversality (\cf Cor. \ref{automatictransversalitystrip}, Prop. \ref{automatictransversalitypolygon}).
The \textbf{almost calibrated} condition implies that $\text{Re}\Omega>0$ on the Lagrangians with respect to the orientation on $L$ and $L'$. 
Thus
\begin{claim}\label{Monotonicityclaim}(\textbf{Monotonicity})
	Clockwise along $\partial \Sigma$, the function $\text{Re }F$ is  \emph{increasing} on the $L$ boundary portion, but \emph{decreasing} on the $L'=L_1\cup L_2$ boundary portion. In particular,
	\[
	0=\text{Re }F(q)\leq \text{Re }F\leq  \text{Re } F(p).
	\]
	More intrinsically, the real part of the complex volume forms on the moduli spaces are \emph{nonnegative}.
\end{claim}

The holomorphic function $F$ maps $\Sigma$ into a bounded region in the complex plane. The behaviour at the corners is specified in Remark \ref{cornerOmega}. Since each vertical line intersects $\partial F(\Sigma)\subset \C$ at $\leq 2$ points by the monotonicity claim above, the boundary and corner local behaviours imply that

\begin{claim}\label{Imagecurveclaim}(\textbf{Image curve})
	The image $F(\Sigma)\subset \C$ lies \emph{above} its $L'$ boundary portion, and \emph{below} its $L$ boundary portion.
\end{claim}

We turn to the proof of the \textbf{Lagrangian phase angle inequality} (\ref{phaseobstruction}). For each curve that contributes nontrivially to $\tilde{\Omega}_{L}$, by the monotonicity claim we can find a unique point $r'$ on the $L$ boundary of $\partial \Sigma$, such that 
\[
\begin{cases}
\text{Re }F(r')= \text{Re }F(r), \quad & \text{primary case},
\\
r'=q,\quad & \text{secondary case, boundary on $L$ and $L_2$},
\\
r'=p,\quad & \text{secondary case, boundary on $L$ and $L_1$}.
\end{cases}
\]
From the image curve claim, we always have $\text{Im}F(r')\geq \text{Im}F(r)$ in the primary case. Integrating over the moduli space of holomorphic curves,
\[
\text{Re}\int_{\mathcal{M}} F(r') =\text{Re}\int_{\mathcal{M}} \tilde{\Omega}_{L_1}= \text{Re}\int_{L_1} \Omega,
\quad \text{Im}\int_{\mathcal{M}} F(r') \geq \text{Im}\int_{\mathcal{M}} \tilde{\Omega}_{L_1}= \text{Im}\int_{L_1} \Omega.
\]

We now introduce two almost everywhere defined functions $\chi_{A_1}, \chi_{A_2}$ on $L$. The recipe is that at any generic point $P\in L$, if an automatically transverse holomorphic curve in the universal family passes through $P$ on the boundary portion of $\partial \Sigma$ joining $q$ to $r'$ (resp. $r'$ to $p$), then it gives an additive contribution to $\chi_{A_1}(P)$ (resp. $\chi_{A_2}(P)$) equal to the weighting factor of the curve. Intuitively $\chi_{A_1}, \chi_{A_2}$ should be understood as the characteristic functions of weighted subsets $A_1, A_2\subset L$. 
The positivity condition gives $\chi_{A_i}\geq 0$, and $\partial\mathcal{C}=L-L'$ gives $\chi_{A_1}+\chi_{A_2}=1$. Intuitively $A_1, A_2$ give a (weighted) partition of $L$.

The moduli space integrals now have target space interpretations:
\[
\int_{\mathcal{M}} F(r')= \int_{L} \chi_{A_1} \Omega=: \int_{A_1} \Omega,\quad   \int_{\mathcal{M}}F(p)- F(r')= \int_{L} \chi_{A_2} \Omega=: \int_{A_2} \Omega.
\]
Since $L$ is homologous to $L_1+L_2$, we have $\int_L\Omega= \int_{L_1}\Omega+  \int_{L_2}\Omega$. Whence

\begin{claim}
There is a weighted partition $L=A_1+A_2$ such that
\[
\text{Re}\int_{A_i} \Omega = \text{Re}\int_{L_i} \Omega>0,\quad  \text{Im}\int_{A_2} \Omega \leq  \text{Im}\int_{L_2} \Omega, \quad \text{Im}\int_{A_1} \Omega \geq \text{Im}\int_{L_1} \Omega.
\]
Consequently $\arg \int_{A_2} \Omega \leq \arg \int_{L_2}\Omega=\hat{\theta}_2$ and $\arg \int_{A_1} \Omega \geq \arg \int_{L_1}\Omega=\hat{\theta}_1$, so in particular $\inf_L \theta_{L}\leq \hat{\theta}_2$ and $\sup_L \theta_{L}\geq \hat{\theta}_1$.
\end{claim}


Finally we deal with the \textbf{J-volume lower bound} (\ref{volumeobstruction}). By the triangle inequality,
\[
\int_L e^{-i\theta}\Omega= \int_L |\Omega| = \int_{A_1} |\Omega|+ \int_{A_2} |\Omega| \geq |\int_{A_1} \Omega|+ |\int_{A_2}\Omega|.
\]
The RHS is at least $|\int_{L_1} \Omega|+ |\int_{L_2}\Omega|$, due to an elementary numerical fact:

\begin{lem}\label{numericallemma1}
Let $z, w$ be complex numbers, with fixed real parts $0< \text{Re}(z)< \text{Re}(w)$. Then as a function of $\text{Im}(z)$, the function $|z|+|w-z|$ is decreasing when $\arg z\leq \arg w$, and increasing when $\arg z\geq \arg w$.

\end{lem}

This concludes the proof of (\ref{volumeobstruction}).
\end{proof}

A few remarks are in order to clarify the relevance to special Lagrangian geometry:

\begin{rmk}
Recall that for $L$ to be a special Lagrangian, then its phase angle is constant, and its J-volume is
\[
\text{Vol}_J(L)=| \int_L\Omega| < |\int_{L_1} \Omega|+ |\int_{L_2}\Omega|,
\]
using the triangle inequality, the homological relation $[L]=[L_1+L_2]\in H_n(X)$ and the assumption that $\hat{\theta}_1>\hat{\theta}_2$. Thus the conclusion of the theorem is a \textbf{quantitative obstruction} for $L$ to be special Lagrangian. In section \ref{TowardsBridgeland} we will discuss the relation to Joyce's LMCF program and the Bridgeland stability condition. 
\end{rmk}

\begin{rmk}
If $L_1,L_2$ are actually special Lagrangians, then the phase angle bounds (\ref{phaseobstruction}) would be evident from the Floer degree formula (\ref{Floerdegree}) applied to the intersection points $L\cap L'$. One main feature of the theorem is that we do not need a priori knowledge on the existence of special Lagrangians, and the holomorphic volume form enters the obstruction criterion only through cohomological information. 
\end{rmk}

\subsubsection*{Variant: twisted complex case}

The Floer theoretic obstruction for distinguished triangles can be easily generalized to involve many Lagrangians. Let $L'$ be an exact immersed Lagrangian with bounding cochain built from the data of a twisted complex (\ref{twistedcomplex}). We assume $L$ is isomorphic to $L'$ in $D^bFuk(X)$, so we obtain a bordism current $\mathcal{C}$ with $\partial \mathcal{C}= L-L'=L-\sum_1^N L_i$. As before, all Lagrangians are assumed to be almost calibrated, and all intersections are transverse.

\begin{thm}\label{Floertheoreticobstruction2}
(Floer theoretic obstruction, multiple Lagrangian case) 
Assume the \textbf{automatic transversality} and the \textbf{positivity condition}  hold for the bordism current $\mathcal{C}$. Assume the \textbf{destabilizing condition}
\[
\hat{\theta}_1>\hat{\theta}_2>\ldots >\hat{\theta}_N, \quad \hat{\theta}_i=\arg \int_{L_i}\Omega. 
\]
Then the \textbf{Lagrangian phase angle} of $L$ has a lower bound on its oscillation:
\begin{equation}\label{phaseobstruction2}
\sup_L \theta_L\geq \hat{\theta}_1,\quad \inf_L \theta_L\leq \hat{\theta}_N,
\end{equation}
and morever the \textbf{J-volume} of $L$ has a nontrivial lower bound
\begin{equation}\label{volumeobstruction2}
\text{Vol}_J(L)=\int_L e^{-i\theta}\Omega \geq \sum_1^N |\int_{L_i}\Omega|.
\end{equation}
\end{thm}

\begin{proof}
Since most parts of the proof are identical to the distinguished triangle case, we will only sketch the main difference.

The Lagrangian boundaries on $\partial \Sigma$ are arranged in the \emph{clockwise} order as $L,L_N, L_{N-1},\ldots L_1$. 
We construct the holomorphic function $F$ as in (\ref{holoF}), and use it to produce complex valued volume forms on the $(n-1)$-dimensional moduli spaces, such that for any $m=1,\ldots N$,
\[
\int_L\Omega= \int_{\mathcal{M}} \tilde{\Omega}_{L}, \quad \int_{L_i}\Omega= \int_{\mathcal{M}} \tilde{\Omega}_{L_i}, \quad i=1,2,\ldots N. 
\]
As before, the real part of these complex volume forms are all non-negative, as a consequence of the positivity condition. The claim on the image curve $F(\Sigma)$ holds verbatim. Similarly to the distinguished triangle case, we produce nonnegatively weighted subsets $A_1,\ldots A_N\subset L$ with $L=\sum A_i$, such that
\begin{equation}
\text{Re} \int_{A_m} \Omega= \text{Re} \int_{L_m} \Omega>0,\quad \text{Im} \sum_1^m \int_{A_i} \Omega\geq  \text{Im}\sum_1^m \ \int_{L_i} \Omega, \quad \int_L\Omega=\sum_1^N\int_{L_i}\Omega.
\end{equation}
This implies 
\[
\arg \int_{A_1} \Omega \geq \arg \int_{L_1}\Omega, \quad \arg \int_{A_N} \Omega \leq \arg \int_{L_N}\Omega,
\]
whence the phase inequality (\ref{phaseobstruction2}).

The J-volume can be bounded below by
\[
\text{Vol}_J(L)=\int_L |\Omega|= \sum_1^N \int_{A_i}|\Omega| \geq \sum_1^N |\int_{A_i}\Omega| \geq \sum_1^N |\int_{L_i}\Omega|.
\]
The last step uses the purely numerical Lemma \ref{numericallemma2} below.
\end{proof}

\begin{lem}\label{numericallemma2}
Let $z_1,\ldots z_N$ be complex numbers, and $a_1,\ldots a_N$ be fixed complex numbers with positive real parts, such that $\arg a_1>\arg a_2>\ldots > \arg a_N$. Assume
\[
\text{Re} (z_i)=\text{Re} (a_i), \quad \text{Im} \sum_1^m z_i\geq  \text{Im}\sum_1^m  a_i, \quad \sum_1^N z_i=\sum_1^N a_i.
\]
Then $\sum_1^N |z_i|\geq \sum_1^N |a_i|$. 
\end{lem}

\begin{proof}
We argue by induction. The $N=2$ case is implied by Lemma \ref{numericallemma1}. In general, we view $\sum_1^N |z_i|$ as a function of the imaginary parts of $z_1,\ldots z_N$ subject to the constraints. Clearly this function achieves its minimum for some $(z_i)$. If $z_1=a_1$, then we can conclude by induction. Otherwise $\text{Im} z_1>\text{Im} a_1$. If $\arg z_2<\arg z_1$, then we can fix $z_1+z_2$ and decrease $\sum_1^2 |z_i|$ by Lemma \ref{numericallemma1}, which would contradict minimality. Proceding with this argument, we are forced to have
\[
\arg z_1\leq \arg z_2\leq \ldots \leq \arg z_N,
\]
whence
\[
\arg \sum_1^N z_i \geq \arg z_1 > \arg a_1>\arg \sum_1^N a_i
\]
which contradicts $\sum_1^N z_i=\sum_1^N a_i$.
\end{proof}

\subsubsection*{What if we relax the positivity condition?}

We now discuss a weaker version that does not require the positivity condition on holomorphic curves. This amounts to dropping pointwise positivity of the integrand for the moduli space integral, which breaks some parts of our mirror analogy. 

As in Theorem \ref{Floertheoreticobstruction1}, we consider $L'$ built from two immersed Lagrangians $L_1,L_2$, and $L$ fits into the distinguished triangle. All Lagrangians are almost calibrated, and all intersections are transverse. We classify the automatically transverse holomorphic curves (\cf Cor. \ref{automatictransversalitystrip}, Prop. \ref{automatictransversalitypolygon}) into $\pm$ types according to whether the boundary evalation to $L$ agrees with the orientation on $L$ or its opposite. The complex volume forms $\tilde{\Omega}_{L_i}$ on the moduli space can be split into the sum of two parts according to whether $u:\Sigma\to X$ is of $\pm$ types:
\[
\tilde{\Omega}_{L_i}= \tilde{\Omega}_{L_i}^+ + \tilde{\Omega}_{L_i}^-,\quad \int_{\mathcal{M}} \tilde{\Omega}_{L_i}=\int_{L_i}\Omega.
\]
In particular the signed measure $\text{Re} \tilde{\Omega}_{L_i}$ is decomposed into its positive and negative parts, and $\text{Re}\int_{\mathcal{M}}  \tilde{\Omega}_{L_i}^-\leq 0  $. We define
\[
\hat{\theta}_i^+= \arg \int_{\mathcal{M}} \tilde{\Omega}_{L_i}^+, \quad \hat{\theta}_i^-=\arg ( - \int_{\mathcal{M}} \tilde{\Omega}_{L_i}^- ) ,\quad i=1,2.
\]
Here $\hat{\theta}_i^-$ is only defined when $\int_{\mathcal{M}} \tilde{\Omega}_{L_i}^-$ is nonzero, namely the case not covered already by the positivity condition.

\begin{thm}\label{Floertheoreticobstructionnopositivity}
(Floer theoretic obstruction, relaxing positivity condition) 
Assume the \textbf{automatic transversality}  holds for the bordism current $\mathcal{C}$ between $L$ and $L'$. 
Then the \textbf{Lagrangian phase angle} of $L$ has a lower bound on its oscillation:
\begin{equation*}
\sup_L \theta_L\geq \max\{  \hat{\theta}_1^+, \hat{\theta}_1^-\} ,\quad \inf_L \theta_L\leq \min\{  \hat{\theta}_2^+, \hat{\theta}_2^-   \}.
\end{equation*}

\end{thm}

\begin{proof}
We will only sketch the modifications. The positivity conditition enters through the monotonicity claim \ref{Monotonicityclaim}. Once we drop this, we would allow holomophic discs $u:\Sigma\to X$ that sweep out parts of $L\cup L'$ with the reversed orientation. For such curves, claim \ref{Monotonicityclaim} is modified to
\begin{claim}\label{Montonicityclaim2}
Clockwise along $\partial \Sigma$, the function $\text{Re }F$ is  \emph{decreasing} on the $L$ boundary portion, but \emph{increasing} on the $L'=L_1\cup L_2$ boundary portion. In particular,
\[
0=\text{Re }F(q)\geq \text{Re }F\geq  \text{Re } F(p).
\]
More intrinsically, the real part of the complex volume forms on the moduli spaces at such $u:\Sigma\to X$ are  \emph{nonpositive}.	
\end{claim}

The corresponding claim \ref{Imagecurveclaim} is modified to 
\begin{claim}\label{Imagecurveclaim2}
The image $F(\Sigma)\subset \C$ lies \emph{below} its $L'$ boundary portion, and \emph{above} its $L$ boundary portion.
\end{claim}

At almost every point on $L\cup L'$, only automatically transverse holomophic curves pass through it, since by assumption the boundary evaluation of the other holomorphic curves is contained in some subset of $L\cup L'$ with Hausdorff dimension $\leq n-1$. According to the $\pm$ types of the automatically transverse curves, we decompose the weighted characteristic functions $\chi_{A_i}$ on $L$ into its positive and negative parts $\chi_{A_i^+}\geq 0$ and $\chi_{A_i^-}\leq 0$.


 Then upon integration over the moduli space,
\[
\begin{split}
& \text{Re}\int_L \chi_{A_i^{+}}\Omega =\text{Re}\int_{\mathcal{M}} \tilde{\Omega}_{L_i}^+ \geq 0, \quad \text{Re}\int_L \chi_{A_i^{-}}\Omega =\text{Re}\int_{\mathcal{M}} \tilde{\Omega}_{L_i}^- \leq 0,
\\
& \text{Im}\int_L \chi_{A_2^{+}}\Omega \leq \text{Im}\int_{\mathcal{M}} \tilde{\Omega}_{L_2}^+,   \quad  \text{Im}\int_L \chi_{A_1^{+}}\Omega \geq \text{Im}\int_{\mathcal{M}} \tilde{\Omega}_{L_1}^+,
\\
& \text{Im}\int_L \chi_{A_2^{-}}\Omega \geq \text{Im}\int_{\mathcal{M}} \tilde{\Omega}_{L_2}^-,   \quad  \text{Im}\int_L \chi_{A_1^{-}}\Omega \leq \text{Im}\int_{\mathcal{M}} \tilde{\Omega}_{L_1}^-.
\end{split}
\]
In particular
\[
\arg \int_L \chi_{A_2^{+}}\Omega\leq 
\hat{\theta}_2^+ , \quad \arg (- \int_L \chi_{A_2^{-}} \Omega)\leq 
\hat{\theta}_2^-,
\]
and 
\[
\arg \int_L \chi_{A_1^{+}}\Omega\geq 
\hat{\theta}_1^+ , \quad \arg (- \int_L \chi_{A_1^{-}} \Omega)\geq 
\hat{\theta}_1^-.
\]
Now $\chi_{A_i^+}\geq 0$ and $\chi_{A_i^-}\leq 0$, and the special case where $\chi_{A_i^-}=0$ almost everywhere is already covered by the positivity condition.
The Theorem follows.
\end{proof}

\begin{rmk}
In the special case where $L_1,L_2$ are special Lagrangians of phase $\hat{\theta}_1,\hat{\theta}_2$, then clearly $\hat{\theta}_i^\pm=\hat{\theta}_i$. The conclusion in this case can be deduced easily from Floer degree considerations at Lagrangian intersections, similar to section \ref{ThomasYauevidence}.
\end{rmk}

\begin{rmk}
Recall $\hat{\theta}_i= \arg \int_{L_i} \Omega$. The caveat is that $\max\{ \hat{\theta}_i^\pm \}$ and $\min\{ \hat{\theta}_i^\pm  \}$ do not quite control $\hat{\theta}_i$, so the above Theorem \ref{Floertheoreticobstructionnopositivity} \emph{does not} imply the phase angle inequality (\ref{phaseobstruction}). In this sense the conclusion of Theorem \ref{Floertheoreticobstructionnopositivity} is weaker than Theorem \ref{Floertheoreticobstruction1}, illustrating the power of the positivity condition.

 On the other hand, we will heuristically argue in section \ref{TowardsBridgeland} that in the Thomas-Yau-Joyce picture, once we assume the existence of Joyce's Bridgeland stability condition, 
the phase angle inequality (\ref{phaseobstruction}) can be deduced without the positivity condition in Theorem \ref{Floertheoreticobstruction1}.

We think it is very interesting to either prove the positivity condition as a consequence of the other assumptions, or to find another Floer theoretic argument for (\ref{phaseobstruction}) that requires neither the positivity condition, nor the a priori knowledge of special Lagrangian representatives.

\end{rmk}

\subsection{Towards a Bridgeland stability condition}\label{TowardsBridgeland}

\subsubsection*{Joyce's proposal and Bridgeland stability condition revisited}

We now seek a better appreciation of the Bridgeland stability aspect of Joyce's proposal (\cf section \ref{ThomasYaupicture}). 
To specify the Bridgeland stability condition on $D^b Fuk(X)$, we need the central charge $Z(L)=\int_L\Omega$, and all the subcategories $\mathcal{P}( \phi_0<\phi <\phi_1  )$ for any interval $(\phi_0, \phi_1)$. Joyce's proposal \cite{Joyceconj} strongly suggests two claims:
\begin{itemize}
\item  If an unobstructed Lagrangian brane $L$ has phase angle function $\theta\in (\pi \phi_0, \pi \phi_1) $, then $L$ defines an object in the subcategory $\mathcal{P}( \phi_0<\phi <\phi_1  )$ generated by all stable objects with $\phi_0< \phi<\phi_1$. This claim is because the infinite time limit of the LMCF should provide the stable objects which generate $L$, and
by the monotonicity of the Lagrangian angle (\cf section \ref{LMCF} below), we  can predict a priori $\theta\in (\pi \phi_0, \pi \phi_1) $ for all these stable objects.

\item  Any object in  $\mathcal{P}( \phi_0<\phi <\phi_1  )$ can be generated by  unobstructed Lagrangian branes with phase angle function $\theta\in (\pi \phi_0, \pi \phi_1) $.

\end{itemize}

Thus one can simply \emph{define}  $\mathcal{P}( \phi_0<\phi <\phi_1  )$  to be the subcategory of the  derived Fukaya category (suitably enlarged to allow for immersed and singular objects) generated by all unobstructed Lagrangian branes $L$ with $\theta\in (\pi \phi_0, \pi \phi_1) $, and then reconstruct  $\mathcal{P}(\phi')$ as the intersection of all $\mathcal{P}(\phi_0<\phi<\phi_1)$ for all $\phi_0<\phi'<\phi_1$. Such a definition would make the Thomas-Yau proposal nearly tautological, and the difficult part of Joyce's proposal is to verify this indeed defines a Bridgeland stability. In fact, by the discussions in section \ref{ThomasYaupicture}, \ref{ThomasYauevidence}, the only formidable part is the \textbf{Harder-Narasimhan decomposition},  for which Joyce's LMCF provides the conjectural mechanism.

There are two primary applications of the Thomas-Yau-Joyce proposal to keep in mind: 
\begin{itemize}
\item  
The existence of special Lagrangians is important for geometric measure theory. A definition of stability conditions along the above lines is too tautological to be useful.

\item Defining special Lagrangian DT invariants is important for mirror symmetry (\cf section \ref{Mirrorsymmetry}). Knowing the existence of a Bridgeland stability condition on $D^b Fuk(X)$ is of great theoretic significance in view of Kontsevich and Soibelman's framework \cite{KontsevichSoibelman}\cite{KS2}, but without a more Floer theoretic characterization it would lack computability.

\end{itemize}

Thus even if Joyce's conjectures can be proved along the lines in \cite{Joyceconj}, it is still desirable to have a Floer theoretic characterization of the Bridgeland stability condition. We first revisit  Theorem \ref{Floertheoreticobstruction1} in the light of the Thomas-Yau-Joyce conjectural picture, but without assuming the automatic transversality and positivity condition.

\begin{conj}\label{Floertheoreticobstructionconj}
Suppose we have almost calibrated exact Lagrangian objects $L_1, L_2, L$, fitting into a \textbf{distinguished triangle} $L_1\to L\to L_2\to L_1[1]$, and satisfies the destabilizing condition
\[
\hat{\theta}_1= \arg \int_{L_1} \Omega > \hat{\theta}_2= \arg \int_{L_2} \Omega. 
\] 
Then the phase angle inequality (\ref{phaseobstruction}) follows. In particular, the derived category class of $L$ admits no special Lagrangian representative.	
\end{conj}

\begin{proof}
(Heuristic) Consider the Harder-Narasimhan decomposition (\ref{HarderNarasimhan}) of $L_1$:
\[
0=\mathcal{E}_0 \to \mathcal{E}_1\to \ldots \to \mathcal{E}_N=L_1,
\]
fitting into the distinguished triangles 
\[
\mathcal{E}_{i-1}\to \mathcal{E}_{i}\to L_{i}'\to \mathcal{E}_{i-1}[1],
\]
where $L_{i}'$ represents an object in $\mathcal{P}(\phi_i)$, with $\phi_1>\phi_2\ldots > \phi_N$. Since by assumption $L_1$ is almost calibrated, we have $L_1\in \mathcal{P}(-\frac{1}{2}<\phi<\frac{1}{2})$, hence $-\frac{1}{2}<\phi_i<\frac{1}{2}$. Since the central charges satisfy
\[
Z(L_1)= \sum_1^N Z(L_i'),\quad \arg Z(L_i)=\pi \phi_i,
\]
we must have $\hat{\theta}_1\leq \pi \phi_1$. The conjectural description of the Bridgeland stability condition requires that $L_i'$ has a special Lagrangian representative with constant Lagrangian phase $\pi \phi_i$. A weaker requirement which suffices for us is that there exists a representative $L_i'$ with Lagrangian angle function $\theta_{L_i'}$ satisfying the oscillation bound $|\theta_{L_i'}-\pi \phi_i|<\epsilon$ for any given $\epsilon>0$. It is expected that this flexibility allows one to assume sufficient smoothness on the Lagrangian.

By combining the distinguished triangles, we obtain a new distinguished triangle 
\[
L_1'\to L\to L''\to L_1'[1].
\]
Here $L_1', L, L''$ are all almost calibrated. Suppose for contradiction that $\sup_L \theta_L< \hat{\theta}_1$. Then $\sup_L \theta_L< \pi \phi_1$, and we can arrange $\sup_L\theta_L<\inf_{L_1'} \theta_{L_1'}$. 
The Floer degree formula (\ref{Floerdegree}) implies $CF^0(L_1', L)=0 $, and in particular $HF^0(L_1', L)=0$. The distinguished triangle splits: $L''\simeq L\oplus L_1'[1]$. Since $L''$ is almost calibrated, it lies in $\mathcal{P}(-\frac{1}{2}<\phi<\frac{1}{2})$, and so must $L_1'[1]$. But $L_1'\in \mathcal{P}(\phi_1)$ implies $L_1'[1]\in \mathcal{P}(\phi_1+1)$. Since $\phi_1+1> \frac{1}{2}$, we know $\mathcal{P}(\phi_1+1)\cap \mathcal{P}( -\frac{1}{2}<\phi<\frac{1}{2})=\emptyset$, contradiction. This proves $\sup_L \theta_L \geq \hat{\theta}_1$, subject to the conjectural existence of the Bridgeland stability condition.

A very similar argument, beginning with the Harder-Narasimhan decomposition of $L_2$, would show $\inf_L \theta_L\leq \hat{\theta}_2$.
\end{proof}

\begin{rmk}
In the above argument, once we achieved $CF^0(L_1', L)=0$, there is a different way to proceed. We reinterpret the distinguished triangle $L_1'\to L\to L''\to L_1'[1]$ as an isomorphism in $D^bFuk(X)$ between $L$ and a twisted complex built from $L_1'\cup L''$. This would give rise to a bordism current constructed from the universal family of holomorphic curves, with $\partial \mathcal{C}=L-L_1'-L''$. However, $CF^0(L_1',L)=0$ implies that no holomorphic curve contributing to $\mathcal{C}$ passes from $L_1'$ to $L$ in the clockwise direction of $\partial \Sigma$. The twisted complex structure on $L_1'\cup L''$ forbids the passage from $L_1'$ to $L''$ in the clockwise direction of $\partial \Sigma$. Thus if the holomorphic curve has any boundary portion on $L_1'$, its entire boundary would lie on $L_1'$, which cannot happen in the almost calibrated setting.

\end{rmk}

This motivates the following definition, whose precise meaning depends on the conjectural enlargement of the derived Fukaya category by incorporating singular Lagrangian objects: 

\begin{Def}\label{ThomasYausemistability}
Let $L$ be an almost calibrated exact Lagrangian brane representing a class in the (suitably enlarged) derived Fukaya category. Suppose for any almost calibrated exact Lagrangian objects $L_1,L_2$ fitting into a distinguished triangle $L_1\to L\to L_2\to L_1[1]$, we always have 
\[
\hat{\theta}_1=\arg \int_{L_1} \Omega \leq \hat{\theta}_2=\int_{L_2} \Omega, \quad \text{resp. } \hat{\theta}_1=\arg \int_{L_1} \Omega < \hat{\theta}_2=\int_{L_2} \Omega,
\] 
then we say $L$ is \textbf{Thomas-Yau semistable} (resp. \textbf{strictly stable}). If $L$ fails to be Thomas-Yau semistable, we say it is Thomas-Yau unstable. 
\end{Def}

We have attributed this definition to Thomas-Yau \cite{Thomas}\cite{ThomasYau}, since it is in their spirit that stability conditions should be Floer theoretic conditions to be tested on the distinguished triangles, and that one should restrict attention only to almost calibrated Lagrangians. We now argue that if Joyce's conjectural Bridgeland stability exists with its expected properties, then its semistable objects should agree with Thomas-Yau semistability.

\begin{conj}
An almost calibrated exact Lagrangian brane $L$ defines a semistable object in $D^bFuk(X)$ under Joyce's Bridgeland stability, if and only if it is Thomas-Yau semistable.
\end{conj}

\begin{proof}
(Heuristic)
If the derived category class of $L$ is semistable in Joyce's sense, then we can choose an optimal representative which is a special Lagrangian, or at least has phase oscillation arbitrarily small. By conjecture \ref{Floertheoreticobstructionconj}, we cannot have any destabilizing distinguished triangle, \ie $L$ is Thomas-Yau semistable.

Conversely, if $L$ is not a semistable object in Joyce's sense, then from its Harder-Narasimhan decomposition we can 
produce a destabilizing distinguished triangle $L_1\to L\to L_2\to L_1[1]$, with almost calibrated $L_1,L_2$, which violates Thomas-Yau semistability.
\end{proof}

Having discussed the semistable objects, it is interesting to see what Joyce's LMCF picture suggests about the Harder-Narasimhan decomposition.

\begin{conj}
Assume further that the K\"ahler metric on $X$ is Calabi-Yau.
Suppose $L$ is an almost calibrated exact Lagrangian brane in $D^b Fuk(X)$, with Harder-Narasimhan decomposition (\ref{HarderNarasimhan})
\[
0=\mathcal{E}_0 \to \mathcal{E}_1\to \ldots \to \mathcal{E}_N=L,
\]
fitting into the distinguished triangles 
\[
\mathcal{E}_{i-1}\to \mathcal{E}_{i}\to L_{i}\to \mathcal{E}_{i-1}[1],
\]
such that $L_i\in \mathcal{P}(\phi_i)$ with $\phi_1>\ldots >\phi_N$. We have $\hat{\theta}_i=\pi \phi_i=\arg \int_{L_i}\Omega$.  Then the phase angle inequality (\ref{phaseobstruction2}) and the volume lower bound (\ref{volumeobstruction2}) hold. 
\end{conj}

\begin{proof}
(Heuristic)
In Joyce's conjectural program, the Harder-Narasimhan decomposition is constructed by running the LMCF $(L_t)$ starting from the unobstructed Lagrangian $L$, and take the infinite time limit (\ref{infinitetimelimit}) to obtain the limiting special Lagrangians $L_1,\ldots, L_N$ with angles $\hat{\theta}_1>\hat{\theta}_2>\ldots >\hat{\theta}_N$, assuming $L_1,\ldots L_N$ have enough regularity to be admitted as objects of $D^b Fuk(X)$. It is expected that $L_1,\ldots L_N$ generate $L$ in $D^bFuk(X)$ via (\ref{HarderNarasimhan}).

A basic feature of LMCF in Calabi-Yau manifolds is that the Lagrangian angle satisfies a heat equation (\cf section \ref{LMCF}), so $\sup_{L_t} \theta$ is nonincreasing in time (resp. $\inf_{L_t} \theta$ is nondecreasing). Comparing the initial time with the infinite time limit,
this suggests $\sup \theta_L\geq \hat{\theta}_1$ and $\inf \theta_L\leq \hat{\theta}_N$.

Morever, if the ambient metric is Calabi-Yau, then LMCF is a special case of mean curvature flow, so the volume functional decreases in time. This monotonicity is not affected by the surgeries in Joyce's LMCF. Thus 
\[
\text{Vol}(L)\geq \sum_1^N \text{Vol}(L_i)= \sum_1^N |Z(L_i)|=\sum_1^N |\int_{L_i}\Omega|.
\]
Under the Calabi-Yau metric, the volume of the Lagrangian $L$ agrees with the $J$-volume:
\[
\text{Vol}(L)= \int_L e^{-i\theta}\Omega=\int_L |\Omega|.
\]
so (\ref{volumeobstruction2}) follows.
\end{proof}

\begin{rmk}
Analogously in the context of HYM connections, if a holomorphic bundle $E$ is unstable, then its Harder-Narasimhan decomposition provides a lower bound on the Yang-Mills energy of any Chern connection on $E$ compatible with $\bar{\partial}_E$, which improves the topological energy bound. This type of phenomenon is common in K\"ahler geometry, for instance it also happens in the context of K-stability. These topics are covered in the introduction of \cite{Donaldsonteststability}. 
\end{rmk}

In conclusion, Joyce's conjectural picture suggests that if a Lagrangian object is unstable, then it satisfies  certain angle and volume inequalities which quantitatively forbids it to be a special Lagrangian, and these obstructions detect the features of the Harder-Narasimhan decomposition. This should be compared with Theorem \ref{Floertheoreticobstruction2}, which contains the main features of the obstructions, but makes 
no a priori reference to the LMCF or special Lagrangian representatives. The cost is that Theorem \ref{Floertheoreticobstruction1} and \ref{Floertheoreticobstruction2} require the positivity condition as an extra hypothesis.

\subsubsection*{Almost calibrated case: categorical predictions of the Joyce picture}

A complete characterization of Bridgeland stability conditions on a triangulated category, known since the inception of the subject \cite[Prop 5.3]{Bridgeland}, is that  $\mathcal{P}(\phi_0< \phi\leq \phi_0+1)$ defines an abelian subcategory (`the heart of a bounded $t$-structure'), and the central charge function on this abelian subcategory satisfies the Harder-Narasimhan condition.

In the Thomas-Yau-Joyce picture, $\mathcal{P}(-\frac{1}{2}< \phi\leq \frac{1}{2})$ essentially is the same as the subcategory of almost calibrated Lagrangians, if we assume there is no special Lagrangian of phase $\frac{\pi}{2}$, which holds as long as the discrete set of values of $\Omega$-periods on $H_n(X,\Z)$ miss the phase angle $\frac{\pi}{2}$. Then this picture would predict almost calibrated Lagrangians to form an \textbf{abelian category}, which morever \textbf{generate} the entire derived Fukaya category using the shift operator. Morally, this is asserting that there are sufficiently many almost calibrated Lagrangians, which is evidently very deep since constructing geometric Lagrangian objects is known to be a difficult problem in symplectic topology. Another deep prediction of the existence of Bridgeland stability condition \cite[conjecture 3.6]{Joyceconj}, is that the derived Fukaya category $D^b Fuk(X)$ (after incorporating immersed and singular Lagrangians with local systems) is automatically \textbf{idempotent complete}, so agrees with $D^\pi Fuk(X)$. These predictions, if correct, are very interesting structural results on the Fukaya category, but at the moment they are controversial.

In Chapter \ref{Variational}, we will set up a variational framework to find special Lagrangian representatives of $D^bFuk(X)$ classes under the assumption of Thomas-Yau semistability. By restricting only to the subcategory of almost calibrated Lagrangians, our program evades these difficult structural claims on the entire $D^bFuk(X)$. It would thus not have the same strength as the Joyce program, nor is it subject to the same falsification criteria.

\subsection{Moduli integral formula for the Solomon functional}\label{ModuliintegralSolomonsection}

\subsubsection{Moduli integral formula for the Solomon functional}

Assuming automatic transversality,
we can rewrite the Solomon functional (\ref{Solomonfunctionalextension}) as a moduli space integral  in terms of the notations introduced in section \ref{Floertheoreticobstructions}. Let $u: \Sigma\to X$ be a holomorphic polygon, with first order deformation vector fields $v_1,\ldots v_{n-1}$, so we can define a holomorphic function $F$ via (\ref{holoF}). In clockwise order on $\partial \Sigma$, we encounter the degree one intersections on $L$, an intersection $p\in CF^0(L,L_0)$, the degree one self intersections on $L_0$, and an intersection $q\in CF^0(L_0,L)$. As before, we fix the additive constant by $F(q)=0$.

In the following calculation, we will use the complex orientation on $\Sigma$, and the \emph{counterclockwise orientation} on $\partial \Sigma$.
The Solomon functional contains a term
$
-\int_{\mathcal{C}} \lambda\wedge \text{Im}(e^{-i\hat{\theta}}\Omega).
$
Now
$
-\int_{\mathcal{C}} \lambda\wedge \Omega
$
can be expressed as an integral of the following $(n-1)$-form over the $(n-1)$-dimensional moduli spaces $\mathcal{M}$ of holomorphic curves: 
\[
\int_{\Sigma} \lambda\wedge \Omega(\cdot, v_1,\ldots v_{n-1})= \int_{\Sigma} \lambda\wedge dF.
\]
Notice that since $u:\Sigma\to X$ is a holomorphic curve and $\Omega$ is an $(n,0)$-form, adjusting $v_i$ by a vector field tangent to $\Sigma$ does not change this integrand, and all $v_i$ must hit $\Omega$ instead of the 1-form $\lambda$. After integration by part,
\[
\int_{\Sigma} \lambda\wedge dF= \int_{\Sigma} Fd\lambda- \int_{\partial \Sigma} F\lambda=   \int_{\Sigma} F\omega- \int_{\partial \Sigma} F\lambda.
\]

The Solomon functional contains another two terms
$
\int_L f_L \text{Im}(e^{-i\hat{\theta}}\Omega)
$ and $-\int_{L_0} f_{L_0} \text{Im}(e^{-i\hat{\theta}}\Omega)$.
Now $\int_L  f_L\Omega$ can be expressed as an integral of the following $(n-1)$-form over the moduli spaces $\mathcal{M}$:
\[
-\int_{\partial \Sigma\cap L} f_L \Omega(\cdot, v_1,\ldots v_{n-1})= -\int_{\partial \Sigma\cap L} f_L dF.
\]
The abused notation $\partial \Sigma\cap L$ means the part of $\partial \Sigma$ mapping to $L$ instead of $L_0$. Similarly $-\int_{L_0} f_{L_0} \Omega$ is the moduli space integral of the $(n-1)$-form
\[
-\int_{\partial \Sigma\cap L_0} f_{L_0} \Omega(\cdot, v_1,\ldots v_{n-1})= -\int_{\partial \Sigma\cap L_0} f_{L_0} dF.
\]
The extra minus sign comes from the fact that $\partial \mathcal{C}$ sweeps out the cycle $-L_0$ instead of $L_0$.

Combining all the three contributions, the Solomon functional is the moduli space integral with integrand
\[
\mathcal{I}=
\text{Im}\int_{\Sigma} e^{-i\hat{\theta}} F\omega - \text{Im} \int_{\partial \Sigma\cap L} e^{-i\hat{\theta}} d( f_L F) - \text{Im} \int_{\partial \Sigma\cap L_0} e^{-i\hat{\theta}} d( f_{L_0} F).
\]
The last two terms involve total derivatives, so can be integrated along boundary segments between the corner points, to yield
\begin{equation}\label{Solomoncornerterms}
\begin{split}
& -\text{Im} \int_{\partial \Sigma\cap L} e^{-i\hat{\theta}} d( f_L F) - \text{Im} \int_{\partial \Sigma\cap L_0} e^{-i\hat{\theta}} d( f_{L_0} F)
\\
&
= \text{Im} \sum_{ \text{all corners}}  e^{-i\hat{\theta}}   F f|^+_-  ,
\end{split}
\end{equation}
where $f|^+_- $ stands for the difference of the potentials $f_{L_+}-f_{L-}$ at a Lagrangian intersection point, such that  $\partial \Sigma$ moves from $L_+$ to $L_-$ in the clockwise direction. In the more general framework of Floer theory with Novikov coefficients, $f|^+_-$ have the interpretation as the Novikov exponents of these intersection points. The bounding cochain elements have $f|^+_-\geq 0$, while $p\in CF^0(L,L_0), q\in CF^0(L_0,L)$ may have negative Novikov exponents.

\begin{prop}(\textbf{Moduli space integral formula})
	The Solomon functional $\mathcal{S}(L)$ is the integral of the following complex valued volume form over the $(n-1)$-dimensional moduli spaces of holomorphic curves:
	\begin{equation}\label{Solomonfunctionalintegral}
	\mathcal{S}(L)=\int_{\mathcal{M}}\mathcal{I},\quad
	\mathcal{I}=
	\text{Im}\int_{\Sigma} e^{-i\hat{\theta}} F\omega +\text{Im} \sum_{ \text{all corners} }  e^{-i\hat{\theta}}   F f|^+_-.
	\end{equation}
\end{prop}

\begin{rmk}\label{normalizationFrmk}
The normalization $F(q)=0$ is convenient, but changing $F$ by a constant along $\Sigma$ does not affect $\mathcal{I}$, due to the energy identity \[
\int_{\Sigma}\omega+ \sum_{ \text{all corners} } f|^+_-=0.
\]
\end{rmk}

\begin{rmk}
We have focused the discussion on the holomorphic curves with boundary on both $L$ and $L_0$, which are the only curves relevant for the bordism current $\mathcal{C}$ in the almost calibrated case. In general we need also curves involving corners at $CF^{-1}(L,L)$ or $CF^{-1}(L_0,L_0)$, and the formula (\ref{Solomonfunctionalintegral}) takes into account all these contributions. 
\end{rmk}

\subsubsection{Change of reference Lagrangians formula revisited}

We now revisit Prop. \ref{SolomonchangeofreferenceProp} from the moduli space integral perspective, which we expect is better suited for generalization to compact Calabi-Yau settings. All transversality requirements of moduli spaces will be assumed, and in this sense the calculations below are formal. 

In Remark \ref{3Lag} we sketched that under the extra assumption $HF^{-1}(L_0,L_0)=0$, there is an $(n+2)$-dimensional universal family $\tilde{\mathcal{C}}$ over some $n$-dimensional moduli $\tilde{\mathcal{M}} $, such that the boundary of $\tilde{C}$  has three $(n+1)$-dimensional contributions, corresponding up to sign to the three bordism currents $\mathcal{C}_1,\mathcal{C}_2,\mathcal{C}_3$ between $L_0, L_0',L$, which are in turn the universal families over the $(n-1)$-dimensional moduli spaces $\mathcal{M}_i$ for $i=1,2,3$. Here $\mathcal{M}_i$ can be viewed as 
certain boundary strata of the compactification of $\tilde{\mathcal{M}}$. The integrand $\mathcal{I}$ naturally makes sense as an $(n-1)$-form on $\tilde{\mathcal{M}}$, and restricts naturally to $\mathcal{M}_i$. The change of reference formula (\ref{SolomonfunctionalchangeL0}) amounts to
\begin{equation}\label{SolomonchangeofreferenceLag2}
\int_{\mathcal{M}_3} \mathcal{I}= \int_{\mathcal{M}_1} \mathcal{I} + \int_{\mathcal{M}_2}  \mathcal{I}.
\end{equation}
Our strategy is to use \emph{Stokes formula on the moduli spaces}. As usual, the holonomy weighting factors will be suppressed in the moduli integral notations. Then (\ref{SolomonchangeofreferenceLag2}) reduces to the two claims:

\begin{claim}\label{dI=0}
The $n$-form 
$d\mathcal{I}=0$ over the moduli space $\tilde{\mathcal{M}}$. 
\end{claim}

\begin{claim}\label{Stokesboundary}
The Stokes boundary term is
\[
\int_{ \tilde{\mathcal{M}}  } d\mathcal{I}= \int_{\mathcal{M}_1} \mathcal{I} + \int_{\mathcal{M}_2}  \mathcal{I}- \int_{\mathcal{M}_3} \mathcal{I}.
\]
\end{claim}

We first explain Claim \ref{dI=0}. First, we calculate the derivatives of $F$. Let $y_1,\ldots y_n$ be local coordinates on $\tilde{\mathcal{M}}$, so $\frac{\partial}{\partial y_i}$ can be identified as first order deformations of holomorphic curves. The local coordiates on $\Sigma$ are denoted as $s,t$. We can write $F=\sum F_i (-1)^{i-1} dy_1\wedge\ldots \arc{dy_i}\ldots dy_n $,  such that along $\Sigma$,
\[
\partial_s F_i= (-1)^{i-1} \Omega(\frac{\partial}{\partial s} , \frac{\partial}{\partial y_1},\ldots \arc { \frac{\partial}{\partial y_i}}, \ldots \frac{\partial}{\partial y_n}), \quad \partial_t F_i= (-1)^{i-1} \Omega(\frac{\partial}{\partial t} , \frac{\partial}{\partial y_1},\ldots \arc { \frac{\partial}{\partial y_i}}, \ldots ).
\]
The holomorphic volume form satisfies $d\Omega=0$, whence
\[
\sum_{i=1}^n\partial_s \partial_i F_i= \partial_s( \Omega( \frac{\partial}{\partial y_1},\ldots, \frac{\partial}{\partial y_n})  ), \quad \sum_{i=1}^n\partial_t \partial_i F_i= \partial_t (\Omega( \frac{\partial}{\partial y_1},\ldots, \frac{\partial}{\partial y_n})  ).
\]
Notice $\Omega( \frac{\partial}{\partial y_1},\ldots, \frac{\partial}{\partial y_n}) $ vanishes at the corners due to the decay of the first order deformation vector fields, and comparing with the additive normalization convention on $F$, we find 
\begin{equation}\label{Fdivergence}
\sum_{i=1}^n  \partial_i F_i= \Omega( \frac{\partial}{\partial y_1},\ldots, \frac{\partial}{\partial y_n})  .
\end{equation}
In particular, at all the corner points $\sum_i \partial_i F_i=0$. The term $f|^+_-$ at the corners are independent of the moduli space parameters, so the only contribution to $d\mathcal{I}$ comes from the $\text{Im}\int_{\Sigma} e^{-i\hat{\theta}} F\omega$ term in formula (\ref{Solomonfunctionalintegral}).

We calculate 
\[
\sum_{i=1}^n \partial_i \int_{\Sigma} F_i\omega= \int_{\Sigma}\sum_i \partial_i F_i\omega + \int_{\Sigma}\sum_i  F_i \partial_i\omega.
\]
Here $\partial_i\omega$ is the Lie derivative of the symplectic form $\omega$ with respect to the vector field $\frac{\partial}{\partial y_i}$, which by Cartan's formula is
\[
\partial_i\omega = d (\omega(\frac{\partial}{\partial y_i},\cdot    )) + \iota_{  \frac{\partial}{\partial y_i}} d\omega=  d (\omega(\frac{\partial}{\partial y_i},\cdot    )).
\]
Thus
\[
 \int_{\Sigma}\sum_i  F_i \partial_i\omega= \int_{\partial \Sigma}\sum_i F_i \omega(\frac{\partial}{\partial y_i},\cdot    )- \int_{\Sigma}\sum_i  dF_i \wedge \omega(\frac{\partial}{\partial y_i},\cdot    ).
\]
Notice that $\partial\Sigma$ and $\frac{\partial}{\partial y_i}$ are both tangent to the Lagrangian boundary, so the $\partial \Sigma$ integrand vanishes. We are left with
\[
\begin{split}
& \sum_{i=1}^n \partial_i \int_{\Sigma} F_i\omega= \int_{\Sigma}\sum_i \partial_i F_i\omega- \int_{\Sigma}\sum_i  dF_i \wedge \omega(\frac{\partial}{\partial y_i},\cdot    )
\\
=& \int_{\Sigma} \Omega( \frac{\partial}{\partial y_1},\ldots, \frac{\partial}{\partial y_n}  ) \omega+\sum_i(-1)^{i-1}\omega(\frac{\partial}{\partial y_i},\cdot    )\wedge \Omega(\cdot, \frac{\partial}{\partial y_1},\ldots \arc { \frac{\partial}{\partial y_i}}, \ldots \frac{\partial}{\partial y_n}) .
\end{split}
\]
Here we used the definition of $F_i$ via $\partial_sF_i, \partial_t F_i$, and the formula (\ref{Fdivergence}) for $\sum_i \partial_i F_i$. We contract the identity $\omega\wedge \Omega=0$ with $\frac{\partial}{\partial y_1},\ldots, \frac{\partial}{\partial y_n} $. When two $\frac{\partial}{\partial y_i}$ hit $\omega$, the  $\Omega$ term will be contracted only $(n-2)$ times, which produces a $(2,0)$-form vanishing identically on the holomorphic curve $\Sigma$. When at most one $\frac{\partial}{\partial y_i}$ hits $\omega$, we obtain the above integrand. In effect, the integrand vanishes identically:
\[
 \sum_{i=1}^n \partial_i \int_{\Sigma} F_i\omega=0, 
\]
which then implies $d\mathcal{I}=0$.

We next explain Claim \ref{Stokesboundary}. In general, the compactified moduli space has many boundary strata corresponding to disc bubbling and disc splitting.

\begin{claim}\label{Stokesboundarycontributions}
Only the boundary strata corresponding to gluing holomorphic curves with virtual dimension $0$ and $n-1$, can have nonzero contributions to the Stokes boundary term.	
\end{claim}

To see this, we need to understand how $\mathcal{I}$ (and notably $F$) behaves near the boundary of the moduli space. Recall that when the holomorphic disc 
is degenerating to several disc components, then under transversality conditions, the cokernel of the extended linearized Cauchy-Riemann operator vanishes, and for small fixed gluing parameters, the kernel elements (\ie first order deformations) are up to small perturbation obtained by gluing the kernel elements from the degenerate disc components. The perturbation effect tends to zero as we approach the moduli space boundary. Now the kernel elements from different disc components have essentially disjoint supports, so unless we have at least $(n-1)$ kernel elements supported on one disc component such that $\Omega(\cdot, v_1,\ldots v_{n-1})$ does not vanish identically, we will have $dF=0$ for the moduli boundary strata, so that $F=\text{const}$ along $\Sigma$, whence $\mathcal{I}=0$ by Remark \ref{normalizationFrmk}. This shows Claim \ref{Stokesboundarycontributions}. We comment that this phenomenon is closely related to the fact that many moduli boundary strata do not contribute to the boundary of the bordism current $\mathcal{C}$ due to support reasons (\cf section \ref{LotayPacinirevisited}).

On the boundary strata, the only contributions to the integral $\mathcal{I}$ come from the $(n-1)$-dimensional moduli spaces. The role of the holomorphic curves of virtual dimension zero, is to provide the counting factors, in a manner entirely analogous to section  \ref{LotayPaciniimmersed}. Most contributions cancel out due to the Mauer-Cartan equation on the bounding cochains, and the closedness of the $HF^0$ generators. The remaining  contributions produce the RHS in Claim \ref{Stokesboundary}.

\subsubsection{First variation formula revisited}

We now explain how to semi-heuristically understand the first variation formula (\ref{Solomonfirstvariation}) as a consequence of Prop. \ref{SolomonchangeofreferenceProp}, from the perspective of the moduli space integral formula (\ref{Solomonfunctionalintegral}). We hope this viewpoint is better suited for generalization to compact Calabi-Yau settings.

Suppose we are given a 1-parameter exact isotopy of unobstructed exact immersed Lagrangians $L_t$, and we wish to calculate $\frac{d}{dt}\mathcal{S}_{L_0'}(L_t)$ at $t=0$. The change of reference Lagrangian formula (\cf Prop. \ref{SolomonchangeofreferenceProp}) allows us to replace $\mathcal{S}_{L_0'}(L_t)$ by $\mathcal{S}_{L_0}(L_t)$. The Lagrangian
$L_t$ for $|t|\ll 1$ is approximately the graph of $tdh$ in $T^*L_0$ (understood in an immersed sense), for the Hamiltonian function $h=h_t|_{t=0}$ on $L_0$. 
The holomorphic discs between $L_t$ and $L_0$ for $|t|\ll 1$ have \emph{small energy} of order $O(|t|)$, and are locally approximated by Morse trajectories of $h$. Write $X_h$ as the Hamiltonian vector field, namely $\omega(X_h, \cdot)=dh$, then 
\[
\frac{d}{dt}|_{t=0}\int_\Sigma F \omega= \int_{\partial\Sigma\cap L_0} F \omega(\cdot, X_h) =  -\int_{\partial\Sigma\cap L_0} F dh.
\]

Now we examine $\sum_{\text{all corners}} Ff|^+_-$ for very small $t$. The Lagrangian intersections come in two types:
\begin{itemize}
	\item The corners $p\in CF^0(L_t,L_0)$ and $q\in CF^0(L_0, L_t)$ correspond to the local extrema of the Hamiltonian $h$.
	\item Any self intersection $p_i$ between two local sheets $L_+,L_-$ of $L_0$ can be paired with a very nearby self intersection of $p_i^t$ between two sheets $L_+^t, L_-^t$ of $L_t$. The bounding cochain on $L_t$ is thus induced from the bounding cochain on $L_0$.
\end{itemize}

At the intersection points $p, q$,
\[
\begin{cases}
f|^+_-(q)= f_{L_0}(q)- f_{L_t}(q)=- th(q)+O(t^2),
\\
f|^+_-(p)= f_{L_t}(p)- f_{L_0}(p)= th(p)+O(t^2).
\end{cases}
\]
The self intersections are usually not important here, because the smallness of energy prevents their appearance on $\partial \Sigma$, unless $f|^+_-(p_i^t)=O(|t|)$, and $f|^+_-(p_i)=0$, which is a rather nongeneric situation. When the self intersections do appear, the evolution of the potential under exact isotopy gives
\[
(f_{L_+^t}-f_{L_-^t})(p_i^t)= (f_{L_+}-f_{L_-})(p_i)+ \int_0^t (h_{+,\tau}-h_{-,\tau})(p_i^\tau) d\tau= f|^+_-(p_i)+ t h|^+_-(p_i)+O(t^2),
\]
where $h_\pm$ keeps track of the hamiltonian on the different sheets of $L_0$. Thus
\[
f|^+_-(p_i^t)= f_{L_-^t}(p_i^t)- f_{L_+^t}(p_i^t)= - f|^+_-(p_i) - t h|^+_- (p_i)+O(t^2).
\]
Here we have a tricky sign reversal, because if $L_+$ and $L_-$ are clockwise ordered on $\partial \Sigma$, then $L_+^t$ and $L_-^t$ are counterclockwise ordered. In summary,
\[
\frac{d}{dt}|_{t=0}
\sum_{\text{all corners}} Ff|^+_- (t)
=\lim_{t\to 0}  \{ - Fh(q)+ Fh(p)-
\sum_{L_0-\text{self intersection corners}} Fh|^+_-(p_i) \} .
\]

Combining the above, and integrating by parts,
\[
\begin{split}
& \frac{d}{dt}|_{t=0}\left(\int_\Sigma F \omega+ \sum_{\text{all corners}} Ff|^+_- (t)\right)
\\
&
= \lim_{t\to 0} \{ - Fh(q)+ Fh(p)-
\sum_{L_0-\text{corners}} Fh|^+_-(p_i)- \int_{\partial\Sigma\cap L_0} F dh \}
\\
&=  \lim_{t\to 0} \{  \int_{\partial\Sigma\cap L_0} hdF \}.
\end{split}
\]
Observe that for very small $t$, as the holomorphic curves vary in the $(n-1)$-dimensional moduli spaces $\mathcal{M}$, under the \emph{counterclockwise} sign convention for $\partial \Sigma$, the boundary evaluation of $\partial \Sigma \cap L_0$ sweeps out the cycle $L_0$ (beware of the sign!), and any generic point on $L_0$ is swept out precisely once due to the Morse theory limiting description. Consequently, the moduli space integral
\[
\lim_{t\to 0} \int_{\mathcal{M}}  \{  \int_{\partial\Sigma\cap L_0} hdF \} = \int_{L_0} h\Omega,
\]  
hence
\[
\frac{d}{dt}|_{t=0}\int_{\mathcal{M}}  \left(\int_\Sigma F \omega+ \sum_{\text{all corners}} Ff|^+_- (t)\right)= \int_{L_0} h\Omega.
\]
By the moduli integral formula (\ref{Solomonfunctionalintegral}) of the Solomon functional,
\[
\frac{d}{dt}|_{t=0}\mathcal{S}(L_t)=
\frac{d}{dt}|_{t=0}\int_{\mathcal{M}} \mathcal{I}= \int_{L_0}  h\text{Im}(e^{-i\hat{\theta}}\Omega). 
\]
This recovers the first variation formula (\ref{Solomonfirstvariation}).

\subsubsection{Speculations on compact Calabi-Yau manifolds}\label{CompactCYspeculations}

Floer theoretic foundations  are much more complicated beyond the exact setting, and the foundations concerning the open-closed string map in the immersed Fukaya category setting are not fully written out in the literature. 
Nonetheless, due to the interest of the topic, we shall offer some speculations about how the Solomon functional formula (\ref{Solomonfunctionalintegral}) generalizes to the compact almost Calabi-Yau setting. Our local systems will have coefficients in $\R, \Q$,  \ie the parallel transport in the local system
have only Novikov exponent zero components. This convention is somewhat more restrictive than \cite{JoyceAkaho}\cite{Woodward1}\cite{Woodward2}.

\begin{rmk}
	In Joyce's LMCF, bounding cochains and local systems can be created ex nihilo during the flow, but in all the mechanisms the author is aware of, the flow preserves the above class of local systems.
\end{rmk}

First, we recall the role of  \emph{Novikov coefficients} in Floer theory. All Floer cochain spaces $CF^*$ are modules over the Novikov field \[
\Lambda=\{  \sum_i a_i T^{\lambda_i}: \quad \lambda_i\in \R, \lambda_1<\lambda_2<\ldots \to +\infty \},
\] 
and $a_i\in \Q $ or  $\R$ depending on the coefficient field choice.\footnote{We do not know if the Fukaya category can be defined over integers in general. One should not confuse the coefficients $a_i$ with the Novikov exponents $\lambda_i$. Typically $a_i$ are rational numbers related to  counting, while $\lambda_i$ are real numbers related to the energy.}
There are a few conventions to define $A_\infty$-operations. Let $L$ be a compact immersed Lagrangian with transverse self intersections.  In the Morse model \cite{Woodward1}, the self Floer cochain space
$CF^*(L,L)$ is generated by the Morse critical points on $L$, and the ordered self intersections (twisted by local system hom and orientation factors as usual). 
The Fukaya $A_\infty$-algebra is a collection of Novikov-multilinear operations
\[
m_k: CF^*(L,L)\otimes \ldots CF^*(L,L)\to CF^*(L,L)[2-k]
\]
defined by counting \emph{holomorphic treed discs} $u:\Sigma\to X$ (\cf \cite[Definition 3.1]{Woodward1}), weighted by the holonomy and orientation factors, and an \emph{energy factor} $T^{E(u)}$. Very roughly, the domain $\Sigma$ have surface parts (which consist of discs, and spheres attached to them), and tree parts connecting the disc boundaries. Then $u$ is a holomorphic map with Lagrangian boundary on the surface parts, and Morse gradient flowlines on the tree parts.
The role of $CF^*$ elements is to specify the limiting behaviour of the Morse flowlines, and the Lagrangian self intersections on $\partial \Sigma$. The energy $E(u)=\int_{\Sigma}\omega$ is the sum of $\int u^*\omega$ on all the surface parts of $\Sigma$.

A nontrivial fact is that (after complicated perturbation schemes, or virtual techniques) this gives rise to a \emph{curved $A_\infty$-algebra} structure \cite{Woodward1}. The most important new feature, absent in the exact case, is that the disc bubbling can occur at points of $L$, which are not necessarily self intersection points. The domain disc splits into two discs, attached at a boundary node. This phenomenon is compensated by considering two discs joined by a gradient flowline segment, whose length shrinks to zero, producing the same nodal discs in the degeneration limit. With the appropriate weights and orientations taken into account, these two effects would cancel algebraically.
On the other hand, the length parameter of the tree parts can tend to infinity, causing the Morse gradient flow line to break, a phenomenon which contributes to the boundary of the one dimensional moduli spaces, reflected algebraically in the $A_\infty$-relations.

Similar to the exact immersed case (\cf Appendix \ref{immersedFukaya}), the bounding cochains are $b\in CF^1(L,L)$ elements satisfying the nonnegative Novikov exponent requirement, and the Mauer-Cartan equation 
\[
m_0+ m_1(b)+ m_2(b,b)+\ldots =0.
\]
Generally speaking, the sum is infinite, but after truncating the Novikov series at any given high energy, only finitely many terms appear due to Gromov compactness, so the sum makes formal sense. As usual, the Lagrangian with bounding cochain structures are called unobstructed.

 The framework for setting up the Fukaya algebra of a single immersed Lagrangian, also assigns meanings to Floer cohomologies between two immersed Lagrangians with bounding cochain structures. Suppose $\alpha\in CF^0(L,L')$ and $\beta\in CF^0(L',L)$ represent elements in $HF^0(L,L')$ and $HF^0(L',L)$ whose cohomological compositions are the identities. The $\alpha, \beta$ are in generally represented by infinite series in the Novikov variable $T$, where some Novikov exponents may be negative, and may bot be bounded above, but at least they are bounded from below depending on $\alpha, \beta$. We now speculate that there is a bordism current $\mathcal{C}$ with $\partial \mathcal{C}=L-L'$, constructed from the universal families of treed holomorphic discs over the $(n-1)$-dimensional moduli spaces $\mathcal{M}$. The monomial summands of $\alpha,\beta$ and the bounding cochain elements prescribe the corners of the treed holomorphic discs, and monomials with different Novikov exponents are viewed as independent contributions to $\mathcal{M}$ and $\mathcal{C}$. Beyond the almost calibrated case, one would also need to incorporate degree $-1$ self intersections as usual. We think the moduli spaces that contribute to $\mathcal{C}$ would satisfy the Novikov exponent condition
 \begin{equation}\label{Novikovexponentcondition}
 \int_\Sigma \omega+ \sum_{\text{all corners}} \text{Novikov exponent}=0.
 \end{equation}
 Here the corners include the monomial summands of $\alpha, \beta$ (or the degree $-1$ self intersections as appropriate), and the bounding cochains at the degree one self intersections/Morse critical points of $L,L'$. Since all Novikov exponents at the bounding cochains  are non-negative (not so at $\alpha, \beta$, and the degree $-1$ self intersections!), this condition would impose an energy upper bound on the holomorphic treed discs depending on $\alpha,\beta$, whence \emph{only finitely many moduli spaces contribute to the bordism current}.

\begin{rmk}
The Novikov exponents correspond to $f|^+_-$ in the exact case. While in the exact case (\ref{Novikovexponentcondition}) is an automatic consequence of the energy identity (\ref{topologicalenergy2}), in general it is an extra condition on the moduli spaces. It sits well with the fact that the geometric unit has zero Novikov exponent.	
\end{rmk}

Now the moduli space integral formula (\ref{Solomonfunctionalintegral}) \emph{formally} makes sense almost verbatim, ignoring all virtual perturbation nuances. For first order deformations $v_1,\ldots v_{n-1}$ of the holomorphic treed discs, we can define $F$ on the domain $\Sigma$ via the 1-form $dF=\Omega(\cdot,v_1,\ldots v_{n-1})$. On the surface parts of $\Sigma$, we would obtain a holomorphic function $F$ by complex integrability as usual (which must be constant on the holomorphic sphere components by the Liouville theorem), while on the tree parts, there is no obstruction for the 1-form to be exact. Next, we replace the appearance of $f|^+_-$ in  (\ref{Solomonfunctionalintegral}) by the Novikov exponents of the monomial summands at the corners, to define the moduli integrand $\mathcal{I}$. The term $\int_{\Sigma} F\omega$ is understood to only involve integration on the surface parts of $\Sigma$. 
The Solomon functional $\mathcal{S}(L)$ still has the form $\int_{\mathcal{M}}\mathcal{I}$. Notice that adding a constant to $F$ would not change the moduli integrand $\mathcal{I}$, thanks to (\ref{Novikovexponentcondition}).

In the absence of the Lagrangian potential, the Novikov exponents of $\alpha, \beta$ are no longer canonically fixed. Suppose we replace $\alpha$ by $T^\mu \alpha$, and $\beta$ by $T^{-\mu}\beta$, for some $\mu\in \R$. This would affect the moduli integrand $\mathcal{I}$, by the amount \[
\mu \text{Im}\{ e^{-i\hat{\theta}}(F(p)- F(q))\},
\]
where $p,q$ stand for the components of $\alpha,\beta$.
By analogy with the exact case, we expect
\[
\text{Im}\left(e^{-i\hat{\theta} }
\int_{\mathcal{M}} F(p)- F(q) \right)=\text{Im}\left(e^{-i\hat{\theta}}  \int_L \Omega \right) =0,
\]
 whence $\mathcal{S}(L)$ is independent of $\mu$.

Once the foundations are in place, we expect

\begin{conj}
Fix a compact almost Calabi-Yau manifold $X$.
The Solomon functional is well defined for graded immersed unobstructed Lagrangians in the same $D^bFuk(X)$ class of a reference Lagrangian $L_0$, satisfying
\begin{itemize}
\item  The change of reference Lagrangian formula (\ref{SolomonfunctionalchangeL0}) holds,
\item  Gauge equivalent bounding cochains give rise to the same functional,
\item  Cohomologous choices of $HF^0$ generators $\alpha,\beta$ give rise to the same functional,
\item   The first variation formula (\ref{Solomonfirstvariation}) holds for any 1-parameter exact isotopy of unobstructed Lagrangians.
\end{itemize}

\end{conj}

The slogan is that Floer theory should 
fix the multivaluedness problem of the Solomon functional (\cf section \ref{Solomon}). As a more technical observation, once the change of reference Lagrangian formula (\ref{SolomonfunctionalchangeL0}) is established, one can remove the assumption for $L$ to be transverse to $L_0$, using a perturbation $L_0'$ of $L_0$.

\subsection{More applications of moduli space integrals}\label{Variantsandobservations}

We collect a number of further topics involving the moduli space integral technique. Sections \ref{LowerboundsSolomon} and \ref{Solomonboundedpartsection} are applications of the moduli integral formula (\ref{Solomonfunctionalintegral}) for the Solomon functional.

\subsubsection{Lotay-Pacini convexity}

Lotay and Pacini proved the convexity of their $J$-functional (\cf Prop. \ref{LotayPaciniconvexity}) through rather heavy calculations, so it is instructive to see that in the Calabi-Yau case, this result has a much simpler conceptual argument.

We interpret their `geodesic' as a bordism current $\mathcal{C}$ between two Lagrangians $L,L'$, constructed from universal families of holomorphic curves, such that automatic transversality and the positivity condition hold. In their highly idealized setting, only holomorphic strips $\Sigma\simeq \R_s\times [0,1]_t$ appear in the construction of $\mathcal{C}$.  We define the holomorphic function $F$ as usual. The 1-parameter family of totally real submanifolds is given by the constant $t$-coordinate slices $\mathcal{C}_t$ of $\mathcal{C}$, whose $J$-volume functional is expressible through moduli space integrals
\[
\text{Vol}_J(\mathcal{C}_t)=
\int_{\mathcal{C}_t} |\Omega|= \int_{\mathcal{M}}  \int_{\R_s\times \{t\}} |\frac{\partial F}{\partial s}|ds .
\]
Since $F$ is holomorphic, so is $\frac{\partial F}{\partial s}$, whence $|\frac{\partial F}{\partial s}|$ is subharmonic, which combined with the exponential decay at $s\to \pm \infty$ implies the convexity of the function in $t$
\[
\int_{\R_s\times \{t\}} |\frac{\partial F}{\partial s}|ds.
\]
Thus $\text{Vol}_J(\mathcal{C}_t)$ is convex  as a function of $t$, as Lotay and Pacini observed.

\subsubsection{Lower bound of the Solomon functional}\label{LowerboundsSolomon}

The theme of Chapter \ref{Variational} will be on the variational approach to find special Lagrangians by minimizing the Solomon functional in a fixed derived category class. 
As an important motivation, special Lagrangians are \emph{formal local minimizers} of the Solomon functional under Hamiltonian deformations (\cf section \ref{Solomon}). In fact we can do better \emph{under the automatic transversality and the positivity condition}:

\begin{prop}\label{specialLagminimizer}
(`\textbf{special Lagrangians are minimizers}')
Suppose $L_0$ is an exact immersed special Lagrangian of phase $\hat{\theta}\in (-\frac{\pi}{2}, \frac{\pi}{2})$, with unobstructed bounding cochain structure. Let $L$ be an almost calibrated, exact, immersed Lagrangian in the same $D^bFuk(X)$ class, which intersects $L_0$ transversely. Suppose the bordism current $\mathcal{C}$ with $\partial \mathcal{C}=L-L_0$ satisfies automatic transversality and the positivity condition. Then $\mathcal{S}(L)\geq \mathcal{S}(L_0)$.

\end{prop}

\begin{proof}
The incline angle of the tangent vector to $F(\partial\Sigma)\subset \C$ is equal to the Lagrangian angle modulo $\pi\Z$. Since $L_0$ is a special Lagrangian, along the $L_0$ boundary portion $\arg F= \hat{\theta}$. Thus
$\text{Im}  (e^{-i\hat{\theta}} F)=0 $ at $p,q$ and the self intersections on $L_0$.
The Solomon functional integrand simplifies to
\[
\text{Im}\int_{\Sigma} e^{-i\hat{\theta}} F\omega  +\sum_{ \text{$L$-self intersections on $\partial \Sigma $} }  \text{Im}(e^{-i\hat{\theta}}   F) f_L|^+_- .
\]
By the almost calibrated assumption on $L,L_0$, and the positivity condition, we obtain Claim \ref{Imagecurveclaim}, namely $F(\Sigma)$ lies above its $L_0$ boundary,
\[
\text{Im}(e^{-i\hat{\theta}} F  ) \geq 0 \quad \text{on } \Sigma.
\] 
Morever, the Novikov positivity requirement for the bounding cochain on $L$ says that 
$f_L|^+_-\geq 0$
at the degree one self intersections on $\partial \Sigma\cap L$. Thus the Solomon functional integrand is nonnegative, which implies $\mathcal{S}(L)\geq 0=\mathcal{S}(L_0)$.
\end{proof}

\begin{rmk}
Suppose we drop the positivity condition, then the key step $\text{Im}(e^{-i\hat{\theta}}F)\leq 0$ would break down, so the above proof of Prop. \ref{specialLagminimizer} would be invalidated. However, the conclusion may still be true (\cf section \ref{specialLagminimizerrevisited}).

\end{rmk}

\subsubsection{Bounded part of the Solomon functional}\label{Solomonboundedpartsection}

Let $L,L_0$ be both exact, immersed Lagrangians with unobstructed bounding cochain structures, lying in the same $D^bFuk(X)$ class, such that all intersections are transverse. Assume the bordism current $\mathcal{C}$ with $\partial \mathcal{C}=L-L_0$ satisfies automatic transversality and the positivity condition. We consider $L_0$ as a fixed reference Lagrangian, while $L$ can vary. We wish to find uniform a priori bound on certain parts of the Solomon functional, under natural conditions on $L$.

We shall assume:
\begin{itemize}
\item (\textbf{Quantitative almost calibratedness}) Both $L$ and $L_0$ have Lagrangian phase angles within $[-\frac{\pi}{2}+\epsilon, \frac{\pi}{2}-\epsilon]$ for some fixed small constant $\epsilon$.
	
\item (\textbf{Potential clustering}, \cf Lemma \ref{blockingconnectedcomponents}) The immersed Lagrangian $L$ can be represented by a twisted complex (\ref{twistedcomplex}) built from the immersed Lagrangians $L_1,\ldots L_N$, such that the oscillation of the Lagrangian potentials have uniform bounds
\[
\sup_{L_i} f_{L_i}- \inf_{L_i} f_{L_i} \leq A,
\]
while for any $i>j$,
\[
\sup_{L_j} f_{L_j}\leq \inf_{L_i} f_{L_i}.
\]
Without loss of generality, we also assume $\sup_{L_0} f_{L_0}- \inf_{L_0} f_{L_0}\leq A$ for the fixed Lagrangian $L_0$.

\end{itemize}

\begin{prop}\label{uniformenergybound}
(\textbf{Uniform energy bound})
Under the potential clustering assumption, all holomorphic polygons $u:\Sigma\to X$ with boundary on $L$ and $L_0$ contributing to $\mathcal{C}$ have uniformly bounded energy independent of $L$:
\[
E(u)=\int_{\Sigma} u^*\omega \leq A(N+1),
\]
and along $\partial \Sigma$ the degree one self intersections of $L_0, L_1,\ldots L_N$ arising from the bounding cochains satisfy a uniform bound
\[
\sum_i\sum_{b_i} f_{L_i}|^+_-(b_i) \leq A(N+1).
\]
\end{prop}

\begin{proof}
We consider holomorphic polygons whose boundary $\partial \Sigma$ encounters in the clockwise order intersections in  $p_{N-1}\in CF^1(L_N, L_{N-1}), \ldots$, $p_1\in CF^1(L_2,L_1)$, $p_0\in CF^0(L_1, L_0)$, $p_N\in CF^0(L_0, L_N)$, juxaposed possibly by more degree one self intersections $b_i$ of $L_i$. The notation here does not constrain the number of self intersections of $L_i$ that can occur on $\partial \Sigma$. The topological energy formula (\ref{topologicalenergy2}) expresses $E(u)$ in terms of the Lagrangian potentials at the intersections
\[
\begin{split}
& E(u)+\sum _i\sum_{b_i} f_{L_i}|^+_-(b_i) 
\\
 & = f_{L_N}(p_N)- f_{L_0}(p_N)+  \sum_{i=0}^{N-1} (f_{L_i}-f_{L_{i+1}}) (p_i)
\\
& = f_{L_0}(p_0)- f_{L_0}(p_{N}) +\sum_{i=1}^N (f_{L_i}(p_i)- f_{L_i}(p_{i-1}) ) 
\\
& \leq (N+1)A.
\end{split}
\]
By the Novikov positivity requirement of the bounding cochains $f_{L_i}|^+_-(b_i)\geq 0$, and the energy of the holomorphic curve is also positive, so they are individually bounded.

More generally, the polygons may miss some of the Lagrangians in $L_1,\ldots L_N$, but cannot reverse the order of the Lagrangians. This amounts to using a smaller effective value $N$, and the same argument implies the energy bound.
\end{proof}

We now consider the holomorphic function $F$ as before. Recall by Claim \ref{Monotonicityclaim} we have $0\leq \text{Re }(F)\leq \text{Re } F(p)$, where $p$ is the intersection point in $CF^0(L,L_0)$. In fact $F(\Sigma)$ must be contained in a triangular region determined by $\text{Re }F(p)$:

\begin{lem}
(\emph{Wedge region bound}) Under the quantitative almost calibrated hypothesis, we have $|\arg F| \leq \frac{\pi}{2}-\epsilon$, or equivalently
$
|\text{Im} (F)|\leq (\cot \epsilon) \text{Re }(F).
$
In particular $|F|\leq  \frac{1}{\sin \epsilon} \text{Re }F\leq   \frac{1}{\sin \epsilon} \text{Re }F(p).$

\end{lem}

\begin{proof}
The incline angle of the tangent vector of $F(\partial \Sigma)$ is equal to the Lagrangian angle mod $\pi \Z$. Together with Claim \ref{Monotonicityclaim} this implies $|\arg F| \leq \frac{\pi}{2}-\epsilon$ on the $\partial \Sigma$, whence the same bound holds on $\Sigma$ by the maximum principle for holomorphic functions.
\end{proof}

We now introduce an \textbf{elementary functional}
\begin{equation}
\bar{\mathcal{S}}(L)=\text{Im}\left( \sum_1^N (\sup_{L_i} f_{L_i}) e^{-i\hat{\theta}}\int_{L_i} \Omega\right) - (\sup_{L_0} f_{L_0}) \text{Im} ( e^{-i\hat{\theta}} \int_{L_0}\Omega).
\end{equation}
As in section \ref{Floertheoreticobstructions}, we introduce complex valued volume form $\tilde{\Omega}_{L_i}$ on the $(n-1)$-dimensional moduli spaces of holomorphic curves, whose core properties are
\begin{equation}
\text{Re } \tilde{\Omega}_{L_i}\geq 0, \quad \int_{\mathcal{M}} \tilde{\Omega}_{L_i}= \int_{L_i} \Omega,\quad i=0,1,\ldots N.
\end{equation}
Thus the elementary functional is also a moduli space integral, with integrand
\begin{equation}\label{Solomonelementary}
\text{Im} \sum_1^N ( e^{-i\hat{\theta}} \tilde{\Omega}_{L_j}) \sup f_{L_j} - \text{Im} (e^{-i\hat{\theta}}   \tilde{\Omega}_{L_0} ) \sup f_{L_0} .
\end{equation}
We decompose the Solomon functional into $\bar{\mathcal{S}}(L)$ and $\mathcal{S}(L)-\bar{\mathcal{S}}(L)$.

\begin{thm}\label{Solomonboundedpartthm}
(\textbf{Bounded part of the Solomon functional}) Under the quantitative almost calibratedness and the potential clustering assumption, and all the standing assumptions of this section, there is a uniform a priori bound independent of $L$,
\[
|\mathcal{S}(L)-\bar{\mathcal{S}}(L)| \leq
\frac{A(4N+2) }{\sin \epsilon} \int_{L_0} \text{Re }\Omega.
\]
\end{thm}

\begin{proof}
We analyze the moduli space integrand  (\ref{Solomonfunctionalintegral}) of the Solomon functional. Applying the uniform energy bound and the wedge region bound,
the first term is bounded by
\begin{equation*}
|\text{Im}\int_{\Sigma} e^{-i\hat{\theta}} F\omega | \leq \int_{\Sigma} |F|\omega \leq \frac{\text{Re}(F(p))}{\sin \epsilon} \int_{\Sigma} \omega \leq \frac{\text{Re}(F(p))}{\sin \epsilon} A(N+1).
\end{equation*}
More intrinsically $F(p)$ defines the complex valued volume form $\tilde{\Omega}_{L_0}$ on the moduli space, hence
\begin{equation}\label{Solomonboundedpart1}
|\text{Im}\int_{\Sigma} e^{-i\hat{\theta}} F\omega | \leq \frac{1}{\sin \epsilon} A(N+1) \text{Re}(\tilde{\Omega}_{L_0}  ).
\end{equation}

The other two terms in (\ref{Solomonfunctionalintegral}) are rewritten as a sum of contributions from intersection points in (\ref{Solomoncornerterms}). As in Prop. \ref{uniformenergybound}, we consider 
holomorphic polygons whose boundary $\partial \Sigma$ encounters in the clockwise order $p_{N-1}\in CF^1(L_N, L_{N-1}), \ldots$, $p_1\in CF^1(L_2,L_1)$, $p=p_0\in CF^0(L_1, L_0)$, $q=p_N\in CF^0(L_0, L_N)$ juxaposed possibly by more degree one self intersections $b_i$ of $L_i$. (The other cases, where $\partial \Sigma$ misses some Lagrangians, can be handled completely similarly.)
We first deal with these extra self intersections.
Using the wedge region bound, and the Novikov positivity requirement,
\[
|\text{Im}\sum_i\sum_{b_i} e^{-i\hat{\theta}} Ff_{L_i} |^+_-(b_i)| \leq  \sum_i\sum_{b_i} |F(b_i)| f_{L_i}|^+_-(b_i) 
\leq  \frac{\text{Re}(F(p))}{\sin \epsilon}\sum_i\sum_{b_i} f_{L_i}|^+_-(b_i).
\]
By Lemma \ref{uniformenergybound}, we have
\begin{equation}\label{Solomonboundedpart2}
|\text{Im}\sum_i\sum_{b_i} e^{-i\hat{\theta}} Ff_{L_i} |^+_-(b_i)| \leq \frac{1}{\sin \epsilon} A(N+1) \text{Re}(\tilde{\Omega}_{L_0}  ).
\end{equation}
We are left with the contributions of $p_0, p_1,\ldots p_N$ to (\ref{Solomoncornerterms}):
\[
\text{Im} \sum_0^{N-1}  e^{-i\hat{\theta}}   F (f_{L_{j+1}}- f_{L_j})(p_j).
\]
If we replace $f_{L_i}$ by its supremum value $\sup_{L_i}f_{L_i}$ for all $i=0,1,\ldots N$, the new expression would be
\[
\begin{split}
& \text{Im} \sum_0^{N-1}  e^{-i\hat{\theta}}   F(p_j) (\sup f_{L_{j+1}}-\sup f_{L_j})
\\
=& \text{Im} \sum_1^N  e^{-i\hat{\theta}}  \sup f_{L_j} ( F(p_{j-1}) - F(p_{j}) )- \text{Im} (e^{-i\hat{\theta}} F(p) \sup f_{L_0}   )
\end{split}
\]
which is more intrinsically the integrand (\ref{Solomonelementary}) of the elementary functional. Using the potential clustering assumption and the wedge region bound lemma,
the error of replacing the potentials by $\sup_{L_j}f_{L_j}$ can be bounded by
\begin{equation}\label{Solomonboundedpart3}
2A\sum_1^N |F(p_j)| \leq 2AN  \frac{\text{Re}(F(p))}{\sin \epsilon}=  \frac{2AN }{\sin \epsilon} \text{Re}(\tilde{\Omega}_{L_0}).
\end{equation}

Now  (\ref{Solomonboundedpart1})(\ref{Solomonboundedpart2})(\ref{Solomonboundedpart3}) are upper bounds on the three contributions to the difference between the Solomon functional integrand (\ref{Solomonfunctionalintegral}) and the elementary functional integrand. Their sum is bounded by
\[
\frac{A(4N+2) }{\sin \epsilon} \text{Re}(\tilde{\Omega}_{L_0}),
\]
so after integration on the moduli space,
\[
|\mathcal{S}(L)- \bar{\mathcal{S}}(L)|\leq \frac{A(4N+2) }{\sin \epsilon} \int_{L_0} \text{Re }\Omega
\]
as required.
\end{proof}

\begin{rmk}\label{potentialclusteringpromise}
	In section \ref{Quantitativealmostcalibrated} 	below
	we will deduce the potential clustering  and an upper bound on $N$ as  consequences of almost quantitative calibratedness, and very mild conditions on the ambient manifold $X$. In section \ref{AsymptoticSolomon} the boundedness of $|\mathcal{S}-\bar{\mathcal{S}}|$ will be essential for relating the asymptote of the Solomon functional to stability conditions. 
\end{rmk}

\begin{rmk}
In K\"ahler geometry, it is often useful to decompose natural functionals into two parts. For instance, the K-energy functional can be decomposed into an entropy part and a pluripotential part \cite[section 2.4]{ChenCheng}, which is important in the study of constant scalar curvature K\"ahler metrics. 

\end{rmk}

\begin{rmk}
We suggested in section \ref{Tunneling} that the Solomon functional is essentially the logarithm of the tunneling amplitude between Lagrangian branes. Pushing forth with this physics analogy, we may regard the elementary functional as a semiclassical approximation,\footnote{The elementary functional is proportional to the period integrals over the cycles $L_i$, which may be regarded as coming from integration over the moduli of constant maps. Such integrals are regarded as more classical then those involving nontrivial holomorphic curves.} and $\mathcal{S}-\bar{\mathcal{S}}$ as quantum fluctuation effects. Our main assertion then becomes that \emph{quantitative almost calibratedness with some extra hypotheses imply the a priori bound on the quantum fluctuation effects}. The author is not aware of previous suggestions in the physics literature, but Jake Solomon's formal Riemannian picture in section \ref{Solomon} may offer partial explanations for the relevance of the almost calibrated condition.

\end{rmk}

\subsubsection*{What if we relax the positivity condition?}

Suppose we drop the positivity condition on the bordism current, but keep all the other assumptions. Then the key difference is that
for holomorphic curves contributing negatively to $\partial \mathcal{C}$, we need to replace Claim \ref{Monotonicityclaim}  by Claim \ref{Montonicityclaim2}, and Claim \ref{Imagecurveclaim} by Claim \ref{Imagecurveclaim2}. Correspondingly, all appearance of $\text{Re }F(p)$ is replaced by its absolute value. Then the conclusion in Theorem \ref{Solomonboundedpartthm} is replaced by
\begin{equation}
|\mathcal{S}(L)-\bar{S}(L)| \leq \frac{A(4N+2) }{\sin \epsilon} \int_{\mathcal{M}} |\text{Re }\tilde{\Omega}_{L_0}|.
\end{equation}
The problem is that the RHS is no longer a manefestedly a priori bounded quantity.

\section{Continuity, LMCF and variational method}\label{Analysis}

We now proceed to the more analytic aspects of the problem of finding special Lagrangians. There are three principal methods for existence theorems in geometric analysis: continuity method, parabolic flows, and the calculus of variations. Generally speaking, continuity or flow methods are often closely related as the elliptic and parabolic cousins of each other, and allow one to work with a priori reasonably smooth objects, but finding a good continuity path or proving the long time existence of the flow may be difficult in a particular problem; the variational approach, on the other hand, operates with the space of a priori less regular objects to achieve some weak compactness, and then attempt to improve the smoothness via regularity theorems. Each approach contains substantial outstanding difficulties. We will try to maintain some equipoise, and compare the main difficulties in each approach. The sections on the flow and the continuity method borrow largely from various writings of Joyce with some new contents; the variational approach is essentially original, and will be presented in Chapter \ref{Variational}.


\subsection{Lagrangian mean curvature flow}\label{LMCF}

\subsubsection*{LMCF basics}

We now return to some analytic aspects of Joyce's proposal \cite{Joyceconj} related to the \textbf{Lagrangian mean curvature flow} (LMCF) inside a Calabi-Yau manifold. Recall a smooth  mean curvature flow means a family of immersions $\iota_t: L\to M$ parametrised by time $t\in [0,T)$, such that the velocity is equal to the mean curvature:
\begin{equation}
\partial_t \vec{x}= \vec{H}.
\end{equation}
The starting point of LMCF is an early observation of Smoczyk, which justifies the name:

\begin{prop}
	\cite[section 4.2]{Smoczyk}
Let $(L_t)$ be a smooth and compact mean curvature flow inside a K\"ahler-Einstein manifold, then the Lagrangian condition is preserved by the flow.
\end{prop}

\begin{rmk}
In more general K\"ahler settings, the Lagrangian condition is still preserved provided one couples the mean curvature flow to the K\"ahler-Ricci flow (\cf \cite{LotayPacini2}). Joyce's program may have natural extensions to the almost Calabi-Yau setting. Indeed the generalisation of the flow may even be advantageous for achieving certain genericity conditions, as in the work of Woodward and Palmer \cite{Woodward1}\cite{Woodward2}. 
\end{rmk}

Assuming $(L_t)$ is a smooth and compact LMCF inside a Calabi-Yau ambient manifold, then if the initial Lagrangian is graded, \ie the Lagrangian angle is well defined as a real valued function on the Lagrangian, then so is $L_t$. The grading is highly desirable, because of the foundational fact that the Lagrangian angle function $L_t\to \R$ satisfies the heat equation
\begin{equation}\label{heat}
(\partial_t- \Lap_{L_t})\theta_t=0,
\end{equation}
which among many other things, implies that $\sup_{L_t} \theta_t$ can only decrease in time, and $\inf_{L_t} \theta_t$ can only increase in time, and in particular the almost calibrated condition would be preserved by the flow. Inside a Calabi-Yau manifold, the mean curvature of $L_t$ is related to the Lagrangian angle $\theta_t$ by an appealing formula:
\begin{equation}
\vec{H}= J\nabla \theta_t,
\end{equation}
where $\nabla \theta$ stands for the gradient of $\theta$ along $L_t$. Along the LMCF
\[
\omega(\partial_t\vec{x}, \cdot)= \omega(\vec{H}, \cdot)= \omega(J\nabla \theta_t, \cdot)= - d\theta_t,
\]
so the Lagrangians evolve by the local Hamiltonian function $-\theta_t$ up to an additive constant.

Mean curvature flow in codimension greater than one does not satisfy the avoidance principle. As such embedded Lagrangians can become immersed during the flow, so the program should at least include immersed Lagrangians. Joyce further suggests that certain `stable Lagrangian singularities' should be admitted. For instance, inside Calabi-Yau 3-folds one should allow Lagrangians with local conical singularity modelled on the Harvey-Lawson $T^2$-cone \cite[Example 2.7]{Joyceconj}. The adjective `stable' here means that the flow should preserve this class of singularities at least for a short amount of time, even if one makes a generic perturbation of the initial data.

\subsubsection*{Finite time singularity, and prototypical bad behaviours}

The central difficulty of the subject is that \emph{finite time singularities are in general inevitable}, starting from complex dimension two. Indeed, a theorem of Neves \cite[Thm. 6.1]{Neves} says that for any embedded Lagrangian submanifold inside a Calabi-Yau surface, there exists a Lagrangian within the same Hamiltonian isotopy class, such that the LMCF with this initial data forms finite time singularity. \footnote{Whether the same holds for almost calibrated initial data is an interesting open problem.} There is also a good geometric reason why singularities must occur in Joyce's program: the Thomas-Yau uniqueness theorem applies to Lagrangians within the same derived Fukaya category class, which may include several Hamiltonian isotopy classes, at most one of which can have special Lagrangian representatives. In order for an initial Lagrangian in the wrong Hamiltonian isotopy class to find its way back to the right class along the LMCF, it must undergo a sequence of surgeries.

Now there is a substantial theory of weak solutions of mean curvature flows in the context of varifolds and currents, known as `Brakke flows' \cite{Brakke}, which exist under very general conditions. The problem is that such solutions are too weak to guarantee uniqueness of the flow, and the total mass of the varifold may jump down at discrete time. Even more fatally for our purpose, once the smoothness of the flow is dropped, the Lagrangian condition may not be preserved any more.  It is instructive to look at the prototypical bad behaviours:

\begin{eg}
Schoen and Wolfson \cite{SchoenWolfson} found area minimizers within certain Lagrangian isotopy classes, which are not minimal surfaces.\footnote{There is no contradiction: area minimisation among Lagrangians by no means guarantee area stationarity among submanifold.} The Brakke flow with such initial data further decreases mass in time, so must cease to be Lagrangian. However, these examples are not graded, so do not contradict Joyce's program. A possible lesson is that \emph{non-graded Lagrangians are bad}.
\end{eg}

\begin{eg}\label{Figure8}
Consider a figure eight curve inside $\R^2$, \footnote{Recall that curves in $\R^2$ are automatically Lagrangian.}whose two looms have unequal areas. Along the mean curvature flow (known as the `curve shortening flow' in this context) one loom shrinks first to zero size. At the moment of singularity, the Lagrangian angle at the self intersection point has a jump. From a more generalisable perspective, one notices that each loom encloses a holomorphic disc, and this singularity is associated with one holomoprhic disc shrinking to zero size and disappearing. The general lesson is that \emph{the shrinking down of small area holomorphic discs messes up the grading}, so it is desirable to exclude them if possible. \footnote{Indeed, one important ingredient in Neves's proof of singularity formation \cite{Neves} is the destruction of grading related to shrinking enclosed 2-dimensional areas. Although it is not explict in Neves's work, these areas seem related to holomorphic discs.}
\end{eg}

\begin{eg}
Consider any compact Lagrangian inside the unit ball of $\C^n$. By an easy maximum principle argument, during the flow $L_t$ remains inside the shrinking ball $\{ \sum |z_i|^2\leq 1-2n t \}$, so must develop a finite time singularity at some $t\leq \frac{1}{2n}$. From the Floer theoretic perspective, since such Lagrangians can always be displaced off itself by the Hamiltonian isotopy corresponding to translations in $\C^n$, its Floer cohomology is either obstructed or zero. As such, a compact Lagrangian supported in a small coordinate ball is invisible to the derived Fukaya category. From a different perspective, since such Lagrangians have zero homology class, they are excluded in the almost calibrated case.
\end{eg}

\subsubsection*{Joyce's LMCF proposal}

With this background in mind, one may better appreciate the upshot of Joyce's perspective: \emph{LMCF should be better behaved if the Lagrangians support unobstructed brane structures}, namely the bad singularities that would spell ruin in a more general context do not actually occur in his program. The moral reasons are:
\begin{itemize}
\item The Lagrangian branes are always assumed to be graded.
\item  Unobstructed Lagrangian branes cannot bound holomorphic curves with very small areas unless their Floer theoretic contributions exactly balance out, for otherwise certain positivity requirements in the Novikov ring will be violated.
\item
Assume the Lagrangian decomposes into two pieces, one of which is contained inside a small coordinate ball. Since this piece only contributes a zero object in $D^b Fuk(X)$, discarding this piece does not affect the $D^b Fuk(X)$ class of the Lagrangian.

\end{itemize}

 This gain comes at the burdensome cost of carrying the brane structure along the flow, which leads to somewhat counterintuitive prescriptions such as surgeries before the Lagrangian itself reaches a singularity, so that the unobstructed brane structure may not be lost prematurely. Joyce describes a number of singularities and surgeries that are expected to occur generically in his program:

\begin{itemize}
\item 
(\textbf{Openning up the neck})
The Lagrangian brane can develop new self intersection points, and may  flow from unobstructed to obstructed at $t=t_0$ through the shrinking of certain holomorphic curves with boundary on $L_t$, even through the underlying Lagrangian remains smooth. The Floer theoretic mechanism causing the obstruction (to do with positivity conditions in the Novikov ring) precisely ensures an angle condition at certain self intersection points of $L_{t_0}$, so that a Joyce-Lee-Tsui Lagrangian expander \cite{JoyceLeeTsui} can be glued into $L_{t_0}$ to continue the LMCF. Woodward and Palmer \cite{Woodward1}\cite{Woodward2} have performed substantial checks that suitable brane structures can be assigned to such surgeries so that $L_t$ remains unobstructed, and the Floer cohomology remains continuous throughout the surgery.

\item (\textbf{Neck pinching}) In some sense converse to the above process, the LMCF $L_t$ may contain a local region modelled on a Lawlor neck with small length scale parameters $\epsilon(t)$, which shrinks `slowly' in time, and $\epsilon(t)\to 0$ at time $t\to t_0$. \footnote{A gluing construction of neck pinching examples is in working progress with T. Collins. There is however a crucial sign difference concerning Lagrangian angles, between our work and the Joyce prediction, which may have rather disconcerting consequences for Joyce's program.} In some cases this may cause the domain of the immersed Lagrangian to become disconnected.

\item (\textbf{Collapsing zero objects}) After a number of surgeries, the Lagrangian $L_t$ may be decomposed into several disconnected pieces, some of which are zero objects in $D^bFuk(X)$, so in particular have homology class zero. A typical situation is that the zero objects are contained in  small coordinate balls.\footnote{Joyce suggests plausibly that Neves's example \cite{Neves}  exhibits this behaviour by splitting off a small Whitney sphere, although this is not proven.} We simply discard these pieces and continue the flow for the remaining pieces.

\item (\textbf{Stable singularities}) 
As mentioned above, one may need to include Lagrangians with certain local singularities, since generic perturbations cannot remove such singularities. How such singularities can form dynamically starting with smooth initial data is less clear, but Joyce offers some analogy with the setting of $U(1)$-invariant special Lagrangians in $\C^3$ \cite[Example 2.8]{Joyceconj}, where Harvey-Lawson $T^2$-cone singularities can appear and disappear in pairs within a 1-parameter family of deformations, and in particular smooth objects can be continuously deformed to such singular objects.

\end{itemize}

The Joyce program of LMCF with surgery contains a number of potentially counterintuitive phenomenon.

\begin{eg}\cite[Example 3.15]{Joyceconj}
After incorporating the surgery of the brane structures, Joyce's LMCF is no longer identical to the LMCF of the underlying Lagrangian. The most extreme case is to start with the union $L$ of two unobstructed special Lagrangians $L_1,L_2$ with phase angles $\theta_{L_1}<\theta_{L_2}$, 
with an intersection point $b\in CF^1(L_2,L_1)$ defining a closed morphism. We regard $L$ as an immersed Lagrangian with bounding cochain $b$. 
This fits into the distinguished triangle
\[
L_1\to L\to L_2\to L_1[1], 
\]
which is \emph{not} destabilizing for $L$. This configuration is stationary in ordinary LMCF. However, under Joyce's LMCF, the bounding cochain $b$ evolves in time, and loses positivity in the Novikov ring in finite time, after which one is supposed to `open up the neck' to continue the flow in a nontrivial fashion.

\end{eg}

\begin{eg}
\cite[section 3.4]{Joyceconj}
Joyce's LMCF in general needs to incorporate nontrivial rank one local systems. Even if the initial brane structure has trivial local system, it is possible for surgeries to create nontrivial local systems from the bounding cochain data at self intersection points. One may imagine such bounding cochain data to be a holonomy contribution concentrated at points, which can be converted into a smeared out holonomy contribution from a nontrivial local system.
\end{eg}

\begin{eg}
The `openning up the neck' surgery is governed by the Novikov positivity requirement of the bounding cochain, which depends on the choice of the bounding cochain, not just the underlying Lagrangian submanifold. The same underlying Lagrangian with different bounding cochains may therefore flow to different infinite time limits.
\end{eg}

In summary, the main difficulty of Joyce's LMCF program is that there is a huge gap between the general Brakke flow framework, and the kind of regularity control required for the long time existence of the LMCF. It would represent very substantial progress \footnote{Joyce \cite{Joyceconj} assesses the difficulty of his program in the Calabi-Yau 3-fold case to be comparable to Perelman's breakthrough on the Poincar\'e conjecture. Indeed, ruling out the cigar solution in the context of the Ricci flow is in itself already a major achievement of Perelman.} if one can classify possible singularity types under suitable genericity assumptions, say for  almost calibrated Lagrangians inside Calabi-Yau 3-folds. A large list of problems, from routine level up to the impossible, can be found in Joyce's excellent original paper \cite{Joyceconj}.

\subsubsection*{Infinite time limit and its difficulties}

Provided one can prove long time existence of LMCF, the total mass of $L_t$ will be uniformly bounded since it decreases during the flow. 
Under mild conditions  to ensure $L_t$ does not escape to spatial infinity (e.g. if the ambient Calabi-Yau manifold is compact), one can extract the infinite time subsequential limits of $L_t$ as currents. From the heat equation (\ref{heat}) on the Lagrangian angle $\theta$, 
\[
(\partial_t- \Lap_{L_t}) |\theta|^2= -2|\nabla \theta|^2= -2|\vec{H}|^2.
\]
If the Lagrangians remain sufficiently smooth, then an integration by part calculation shows
\[
\int_0^T \int_{L_t}
|\vec{H}|^2 dvol_{L_t}dt=
 \frac{1}{2} \left( \int_{L_0} |\theta|^2 dvol_{L_0}- \int_{L_T} |\theta|^2 dvol_{L_T} \right).
\]
Even if the volume mass can jump down at discrete time, such as during the collapsing of zero objects, we still expect
\begin{equation}\label{meancurvatureL2}
\int_0^T\int_{L_t}
|\vec{H}|^2 dvol_{L_t}dt \leq \frac{1}{2}\int_{L_0} |\theta|^2 dvol_{L_0}<\infty, \quad \forall T>0.
\end{equation}
In particular we can find a sequence of time $t_i\to \infty$, with
\[
\int_{L_{t_i}} |\nabla \theta|^2 dvol \to 0.
\]
This strongly suggests that the subsequential limit is a union of special Lagrangian currents with multiplicities.  

In the heursitic logic of Thomas-Yau-Joyce prgogram, the infinite time limit supposedly provides the Harder-Narasimhan decomposition. In general one cannot expect the special Lagrangian currents to be smooth, so this raises the question how to make sense of singular Lagrangians as representatives of $D^bFuk(X)$ classes, or whether we should use some weaker equivalence class.  Another interesting open problem is whether the limiting current is unique. In order to run the Thomas-Yau argument, one presumably also needs Floer theory for singular Lagrangians.

Joyce \cite{Joyceconj} already observed that it is not obvious how singular Lagrangians can carry brane structures, and it is logically possible for some Floer theoretic information to be lost in the infinite time limit. The suggestion is that hopefully the Lagrangian $L_t$ at large but finite time $t\gg 1$, can serve as a substitute for the infinite time limit, which presumably has better smoothness properties \cite{Joyceconj}. There is however no known justification (and probably false) that the surgeries terminate after some finite time, and $L_t$ decomposes into the union of several Lagrangian objects, in order to provide a Harder-Narasimhan decomposition.

We think Floer theory for singular Lagrangians is one of the foundational open questions necessary for an adequate solution of the Thomas-Yau conjecture. See section \ref{Floertheoryweakregularity} for further discussions.


\subsection{Continuity method}\label{Continuitymethod}

\subsubsection*{The continuity path}

The general idea of the continuity method is to work with a 1-parameter family of PDEs, and attempt to deform from an initial given solution, to a solution of the final PDE, provided the deformation encounters no obstruction, and satisfies suitable compactness properties. The hope that the continuity method \emph{may} be useful here, is based on the foundational fact that compact special Lagrangian submanifolds  inside almost Calabi-Yau manifolds have unobstructed deformation theory, before taking brane structures into account.

However, problems immediately ramp up once one attempts to set up a continuity path. The most na\"ive suggestion, based on the analogy with the HYM equation, is to prescribe the Lagrangian angle as a function on the domain of $L$. This however breaks the domain reparametrisation invariance of $\iota:L\to X$, and the author knows no satisfactory way to make general sense of this approach beyond graphical Lagrangians. Instead we 
fixed $\omega$, and allow $\Omega$ to vary in an infnite dimensional parameter space subject to the almost Calabi-Yau condition. In noncompact almost Calabi-Yau manifolds, we need to also keep the metric asymptote fixed at infinity. The continuity path is a generic 1-parameter family of $\Omega$. This setup strongly resemble the \textbf{wall crossing phenomenon} studied by Joyce \cite{Joycecounting} in the context of special Lagrangian enumerative invariants, and indeed the rest of this section liberally borrows from the ideas therein.

\subsubsection*{Two main obstacles}

There are two fundamental obstacles: 
\begin{itemize}
\item How can one find the \textbf{initial} special Lagrangian?
\item How can one guarantee compactness?
\end{itemize}

\subsubsection*{How to find an initial special Lagrangian}

The question about finding the initial special Lagrangian is specific to the continuity method, and does not appear in the LMCF approach. A natural suggestion is to look for special Lagrangians near certain degenerate limits, and our relaxation of the complex Monge-Amp\`ere equation ought to give much more flexibility. For instance, conifold degenerations are known to give rise to special Lagrangian spheres appearing as vanishing cycles \cite{HeinSun}.\footnote{While Hein and Sun's result is highly nontrivial, the entire difficulty goes into understanding the Calabi-Yau metric near the conifold point. If we are given the license to prescribe arbitrary K\"ahler metrics, the problem of finding special Lagrangian vanishing spheres near the conifold point becomes easy.}
 Another general source is to work near a suitable large complex structure limit, so that the K\"ahler metric can be made almost toric outside a small region, such that the torus fibres are much smaller compared to the characteristic length scale of the base. We can then attempt to find special Lagrangians via adiabatic limits, in close analogy with the standard procedure to find holomorphic curves via tropical degenerations \cite{Lin}. \footnote{The large complex structure limit is supposed to correspond to the large volume limit in the mirror, which is related to the $\mu$-stability, thus offering the hope of a mirror calculation of counting invariants.} The most accessible special Lagrangians in this approach, should be obtainable by small perturbations of the torus fibres. \footnote{The difficulty in \cite{LiSYZNA} to construct SYZ special Lagrangian fibrations again comes from the Calabi-Yau metrics. If one can freely prescribe K\"ahler metrics, then finding a special Lagrangian torus is not difficult.} The next candidate suggested by the Leray filtration of the torus fibration is already much harder.

\begin{Question}
Construct special Lagrangians whose toric projection to the base are small thickenings of certain 1-dimensional graphs.
\end{Question}

 One expects that locally along an edge these Lagrangians are  perturbations of $T^{n-1}\times \R$, with $T^{n-1}$ contained in the torus fibre direction, so that
 we obtain $(n-1)$ locally defined closed 1-forms $\int_{S^1 }\omega$ on the base
  corresponding to the cycles $S^1\subset T^{n-1}$, and the edge is to leading approximation given by requiring these 1-forms to vanish. The local model for the junction where three edges meet, \footnote{This is conceptually related to Matessi's `Lagrangian pair of pants' \cite{Matessi}.} may have the following topological description. In the $n=2$ case, we have a `pair of pants' inside $T^2\times \R^2$ with three asymptotic ends $S^1\times \R$; topologically this is the same as algebraic surface $\{ z_1+z_2=1  \}\subset \C^*\times \C^*$. In higher dimensions, we take a product of the pair of pants with $T^{n-2}$.

In the next order of perturbation, we expect the deformation of the Lagrangian in the base direction to be at least comparable to the length scale of the fibre, and presumably is fixed by the `special condition' $\text{Im}\Omega|_L=0$.

  \begin{rmk}
  From the viewpoint of the Thomas-Yau-Joyce picture, there is an additional problem to assign unobstructed brane structures to the initial special Lagrangian.
  \end{rmk}


\subsubsection*{Compactness and genericity}

The question about compactness largely reflects problems we already encountered in the LMCF approach. The essential issue is that without any further condition on the K\"ahler metric, special Lagrangians may be too singular, so that the McLean deformation theory for special Lagrangians may fail. The natural answer, closely related to the LMCF viewpoint, is that we should only work with \textbf{generic} K\"ahler structures, and 1-parameter families thereof.  According to the philosophy advocated by Joyce \cite{Joycecounting}, the importance of singularities are ranked according to their genericity. For special Lagrangians with first Betti number $k$, the moduli space of deformations is $k$-dimensional, so in a generic 1-parameter family of K\"ahler structures, one expects to encounter singularities with genericity index up to $k+1$, and those singularities of index $0, 1$ are the most important. \footnote{Understanding the moduli space of special Lagrangians requires singularities up to index $k+1$, but if we restrict to Lagrangian deformations with zero Lagrangian flux, then index $\leq 1$ may suffice in the optimistic view.} Before one can seriously pursue the rest of this strategy, it is necessary to have a classification of index $0,1$ special Lagrangian singularities in complex dimension $n$.

\begin{Question}
In complex dimension 3, classify all special Lagrangian singularities of index 0 and 1, namely all singularities that can occur in a generic 1-parameter family of special Lagrangians when $\omega$ is fixed and $\Omega$ varies. 
\end{Question}

\begin{eg}
The \textbf{Harvey Lawson $T^2$-cone} is a special Lagrangian cone inside $\C^3$ with link $T^2$, invariant under the diagonal $T^2\subset SU(3)$. Explicitly,
\[
L_{HL}= \{   (z_1, z_2, z_3)\in \C^3: |z_1|=|z_2|=|z_3|, \quad \text{Im}(z_1z_2z_3)=0, \quad \text{Re}(z_1z_2z_3)\geq 0         \}.
\]
Haskins \cite[Thm 1]{Haskins} proved that up to unitary transformations, this is the only strictly stable\footnote{Strict stability here is a condition on the Laplacian spectrum of the link. Unfortunately, the word `stable' is overloaded with many standard meanings in the literature.} special Lagrangian cone with smooth embedded link diffeomorphic to $T^2$.

The Harvey-Lawson cone admits three different 1-parameter deformations into smooth embedded special Lagrangians	$L_s^1, L_s^2, L_s^3$ for $s>0$. Here 
\[
L_s^1=\{  (z_1, z_2, z_3)\in \C^3: |z_1|^2-s=|z_2|^2=|z_3|^2, \quad \text{Im}(z_1z_2z_3)=0, \quad \text{Re}(z_1z_2z_3)\geq 0           \},
\]
and $L_s^2$, $L_s^3$ arise via cyclic permutations of $z_1, z_2,z_3$. Notably, there is a holomorphic disc $D_t^1$ of area $\pi s$ with boundary on $L_s^1$ (and similarly for $L_s^2, L_s^3$),
\[
D_t^1= \{  (z_1,0,0): |z_1|^2\leq s  \}.
\]
In particular, $L_s^a$ for $a=1,2,3$ cannot be exact Lagrangians, but have nonzero Lagrangian flux. 
 The $s\to 0$ limit corresponds to the holomorphic discs shrinking to zero area, or equivalently the Lagrangian flux tends to zero.

Now on a compact special Lagrangian inside an almost Calabi-Yau manifold, the Harvey-Lawson cone can arise as a local model for conical singularities. The gluing results of Joyce \cite[section 10]{JoyceSLsurvey}  shows that when certain homological conditions are satisfied, then there exist desingularisation families of special Lagrangians locally modelled on $L_s^a$, such that the singular special Lagrangians carrying the $T^2$-cone singularity arise in codimension one, so in this case the $T^2$-cone is an index one singularity in Joyce's sense.\footnote{Joyce's gluing result is quite subtle. Under certain homological conditions, the smoothing can be forbidden, in which case the $T^2$-cone is an index zero singularity. In other cases, due to some linear dependence of certain homology classes, two $T^2$-cone singularity may not behave independently, but together behave like an index one singularity. See \cite[section 10]{JoyceSLsurvey}. } This gluing result is not sensitive to varying $\Omega$. On the other hand, if one restricts to deformations with Lagrangian flux zero, which can be regarded as the analogue of exact isotopies in the mildly singular case, then an isolated local $T^2$-cone singularity cannot be desingularized, but instead keeps the singularity as it deforms.

\end{eg}

\subsubsection*{Wall crossing}

Some of the generic singularities in the LMCF are expected to have  elliptic analogues in the continuity method approach. We fix $\omega$ and consider a generic 1-parameter family of $\Omega$, and we follow the Lagrangian flux zero deformations of a given special Lagrangian.

\begin{itemize}
\item  Under exact isotopy, immersed Lagrangians may lose the unobstructed condition. One expects the surgery of the brane structure suggested by Joyce has an elliptic analogue, involving the same ingredients as  the `Maslov flow' studied by Woodward and Palmer \cite{Woodward1}\cite{Woodward2}.

\item  As already discussed in section \ref{Extensionwallcrossing}, the Lawlor neck is responsible for the gluing of two immersed special Lagrangians. This corresponds to the  `Lawlor neck pinching' singularity, as well as the `openning the neck' surgery in the LMCF.

\item 
(Stable singularity) Joyce suggests from his work on $U(1)$-invariant special Lagrangians in $\C^3$ \cite[Example 2.8]{Joyceconj}, that in a continuous 1-parameter family, isolated singular points of special Lagrangian 3-folds with local $T^2$-cone singularities can appear and disappear in pairs, by making the two $T^2$-cone singularities collide with each other and then smooth out. This is the main motivation for admitting the $T^2$-cone singularity in the LMCF, and it seems likely to be a generic singularity in the continuity method as well.

\end{itemize}

\begin{rmk}
The phenomenon of `collapsing zero object' in Joyce's LMCF has no analogue in the continuity approach, since the special Lagrangian condition forbids any homologically trivial component.
\end{rmk}


\subsubsection*{What's the role of the brane structure?}


As Bridgeland observed \cite[Thm 1.2]{Bridgeland}, the central charge map gives a local homeomorphism between the space of Bridgeland stability conditions, and the hom space from the numerical Grothendieck group to $\C$.
The Thomas-Yau-Joyce philosophy then suggests that for fixed $\omega$ and small deformations of $\Omega$, the stability condition only depends on the cohomology class $[\Omega]$.

A possible geometric interpretation consistent with the wall crossing picture, is that in a generic 1-parameter family of deformations of the almost Calabi-Yau structure fixing $\omega$ and $[\Omega]$, the special Lagrangian may undergo surgeries, such that the topology and the Hamiltonian isotopy class may change, but when equipped with the brane structures, the $D^b Fuk(X)$ class (or possibly some weaker equivalence class) remains constant. As in the LMCF approach, the role of the unobstructed brane structure is morally to prevent holomorphic discs from having very small area. However, the troubles caused by shrinking discs may be less severe in the continuity method than in the LMCF method, since the continuity method only deals with special Lagrangians, and cannot lose grading in a process like Example \ref{Figure8}. Instead, a significant amount of difficulty in the continuity method is absorbed into the problem of finding the initial special Lagrangians.

%





\subsubsection*{Comparison with LMCF}

To summarize the pros and cons compared to the LMCF approach, in the continuity method finding an initial special Lagrangian is a significant new difficulty. However, we no longer need to confront the extremely difficult task of long time existence for the flow. On a slightly more technical level, provided one has a classification for low index singularities, the \emph{genericity assumption} is likely to be easier to use in the continuity method than it is in the LMCF framework.

\section{Variational method}\label{Variational}

We begin by significantly narrowing the scope:
\begin{itemize}
\item We only consider exact Lagrangians inside Calabi-Yau Stein manifolds. The Stein assumption is meant to simplify Floer theory, at the cost of noncompactness of the ambient space. In particular $H_{n+1}(X,\Z)=0$ and $H_n(X,\Z)$ has no torsion, since Stein manifolds have the topological type of CW complexes of dimension $\leq n$.

 The complex Monge-Amp\`ere equation of the Calabi-Yau metric ensures special Lagrangians are area minimizing currents, at the cost of losing genericity. The author thinks the complex Monge-Amp\`ere equation is convenient but not completely indispensable to the Thomas-Yau conjecture, and some related discussions will be given at the end of section \ref{Compactnessregularity}.

\item We impose another mild condition on the Calabi-Yau metric, namely that the regularity scale grows to infinity asymptotically (\cf section \ref{Quantitativealmostcalibrated}). Its main goal is to prevent the Lagrangians from escaping to infinity.

\item We only consider graded Lagrangians whose phase angle function satisfies a quantitative almost calibrated condition $-\pi/2+\epsilon\leq \theta\leq \pi/2-\epsilon.$ An alternative characterization is that $\pm\text{Im}\Omega\leq (\cot \epsilon)\text{Re}\Omega$ when restricted to $L$, whence the quantitative almost calibrated condition is preserved under weak limits of currents.
\end{itemize}

The main question we wish to examine from a variational viewpoint is
\begin{Question}
Fix a $D^bFuk(X)$ class in the exact Calabi-Yau manifold $(X, \omega, \Omega)$, represented by some nontrivial unobstructed exact Lagrangian brane $L_0$ with phase function $-\pi/2+\epsilon< \theta< \pi/2-\epsilon$. Denote $\hat{\theta}=\arg \int_{L_0} \Omega\in (-\frac{\pi}{2}+\epsilon, \frac{\pi}{2}-\epsilon )$. When does there exist a possibly singular special Lagrangian representative $L$ of phase $\hat{\theta}$, in the same derived Fukaya category class  (or under some weaker equivalence relation)?

\end{Question}

\begin{rmk}
Making sense of the derived Fukaya category for Lagrangians with weak regularity is part of the question, which seems highly nontrivial. As a basic meta-principle, any two Lagrangian branes in the same derived category class must lie in the same homology class in $H_n(X)$, so the homology class is fixed. 
This still leaves open some ambiguity on the choice of local system (\cf Remark \ref{localsystem}), and the question of which singular Lagrangians to include (\cf section \ref{Variationalstrategy}, \ref{Floertheoryweakregularity}).
\end{rmk}

One possible interpretation of the \textbf{Thomas-Yau conjecture} is

\begin{conj}\label{ThomasYauexistence1}
(Thomas-Yau existence) Assume the derived Fukaya class of the quantitatively almost calibrated Lagrangian brane $L_0$ is Thomas-Yau semistable (\cf Definition \ref{ThomasYausemistability}). Then there is a special Lagrangian representative in the same $D^bFuk(X)$ class (or some weaker $S$-equivalence relation, see remark \ref{Sequivalence}).
\end{conj}

Our limited goal is not to prove this conjecture completely, but to clear up enough easier obstacles in order to pinpoint the deeper issues that need to be resolved. Whenever we make difficult claims that we are yet unable to prove, we will try to at least provide some heuristic reasons.

\subsection{Compactness and regularity}
\label{Compactnessregularity}

\subsubsection{Standard geometric measure theory}

We will use the standard language of geometric measure theory, see Federer \cite{Federer} or Morgan \cite{Morgan} for the terminologies.
The starting point of the variational approach is that there are foundational \textbf{compactness theorems} in geometric measure theory.

\begin{thm}\label{Federer}
(Federer-Fleming compactness theorem \cite{Federer}) Let $L_i$ be a sequence of $m$-dimensional integral currents in a complete Riemannian manifold $X$, all supported in a fixed bounded subset, with uniform bounds $\text{Mass}(L_i)\leq C$ and $\text{Mass}(\partial L_i)\leq C$. Then up to subsequence $L_i$ converges weakly in the current topology to an $m$-dimensional integral current $L$ with the same bounds.

\end{thm}

\begin{rmk}
While compactness in the current topology is elementary, the claim that the limit is also an integral current is  nontrivial, and can be viewed as a regularity result. The same holds with the Allard compactness theorem below. For our applications, we will always work with closed integral currents, namely $\partial L_i=0$, which implies $\partial L=0$ in the limit. To such currents one can associate a homology class.
\end{rmk}

\begin{rmk}
A more technical version of Federer-Fleming compactness replaces the current topology by the flat norm topology, which is a slightly stronger topology. The flat norm of an integral current $T$ is 
\[
\norm{T}_{flat}= \inf \{  \text{Mass}(A)+ \text{Mass}(B)| T=A+\partial B, \quad A, B\text{ are integral currents}       \},
\]
and the convergence $T_i\to T$ in this topology simply means $\norm{T-T_i}_{flat}\to 0$. 
\end{rmk}

\begin{thm}\label{Allard}
	(Allard compactness \cite{Allard}) Let $L_i$ be a sequence of $m$-dimensional integer rectifiable varifolds in a complete Riemannian manifold $X$, all supported in a fixed bounded subset, with a uniform volume upper bound $\text{Mass}(L_i)\leq C$ and a uniform bound on the first variation $\int_{L_i} |\vec{H}| \leq C$. Then up to subsequence, $L_i$ converges to an $m$-dimensional integer rectifiable varifold $L$ with the same bounds.
\end{thm}

\begin{rmk}\label{currentvsvarifold}
Federer-Fleming and Allard are somewhat complementary. Integral currents are a special kind of distribution valued forms, while varifolds are a special kind of measures on the real Grassmannian bundle $Gr(TX, m)$ over $X$ whose fibres parametrize $m$-dimensional planes in the tangent spaces of $X$. One key advantage of currents is that they know about orientations, while varifolds do not. The integral current $L$ recovers the underlying rectifiable subset $\text{supp}(L)$ with multiplicity, so can be canonically associated with a varifold $L^{var}$. 
 On the other hand, the natural topology on varifolds (\ie the topology as measures on $Gr(TX,m)$) remembers tangent plane information, which can be lost under the flat norm convergence of integral currents. Morever, assuming all the varifolds in the sequence are contained in a bounded region, then the total volume mass converges under varifold convergence, but not necessarily so under flat norm convergence. The intuition is that 
 \emph{morally the varifold topology detects one more derivative than the flat norm topology}. 
   This explains why Allard requires some integral control on the mean curvature, while Federer-Fleming does not. 
   
   We shall later use the informal terminology of \textbf{`varifold/current topology'} to refer to convergence simultaneously in the varifold topology and the flat norm topology on integral currents. 
\end{rmk}

\begin{eg}\label{snowflakeexample}
Inside $S^1_{2\pi}\times \R$ with the standard Euclidean metric, take $L_k$ as the graph over $S^1$ of the function $\frac{1}{k}\sin(kx)$. Then $L_k$ are Lagrangian currents, which converge to $S^1$ as currents, but due to the high oscillation, $\liminf Mass(L_k)> Mass(S^1)$, and $L_k$ do not converge to $S^1$ in the varifold sense. 
The Lagrangian angle of $L_k$ is prescribed by $\tan \theta= \cos(kx)$, which converges to zero in the current sense, but not strongly in $L^1$.
\end{eg}

One of the best \textbf{regularity theorems} in geometric measure theory is

\begin{thm}
(Almgren's big regularity theorem \cite{Almgren})	Let $L$ be a compactly supported $m$-dimensional closed integral current inside a complete Riemannian manifold, which minimizes the volume among all closed integral currents in the same homology class, then away from a closed subset $S\subset \text{supp}(L)$with  Hausdorff dimension at most $m-2$, the rectifiable subset $\text{supp}(L)\setminus S$ is a smooth submanifold.

\end{thm}

\begin{rmk}
Real codimension two singularity is the optimal result, as easily seen from the examples of singular algebraic curves in $\mathbb{CP}^2$, which are automatically area minimizers in their homology classes.
\end{rmk}

\begin{rmk}
Almgren's big regularity theorem is well known for its 
 monumental size of around 1000 pages. The recent works of Delellis et al. have somewhat simplified the proof, which still remains very nontrivial (\cf \cite{Delellis} for some introduction). 
\end{rmk}

A standard way to apply these theorems, for instance inside a compact ambient space, is to fix the homology class, and minimize the volume among all the integral currents therein. The compactness theorem guarantees the existence of an absolute volume minimizer, and the regularity theorem then improves its regularity to be more like submanifolds. This strategy is highly effective in producing minimal surfaces, but there is no useful criterion\footnote{If there is at least one special Lagrangian within the given homology class, then all absolute minimizers must be special Lagrangians, by an easy calibration argument. This however does not answer how to find the special Lagrangian in the first place.} to guarantee the volume minimizers to be special Lagrangians, which is why producing special Lagrangians is a highly nontrivial problem in geometric measure theory.

\subsubsection{Exact Lagrangians under weak regularity}\label{ExactLagweakregularity}

We need to ensure the class of Lagrangians in the variational setup is closed under the varifold/current topology.  A trivial observation is

\begin{lem}
Let $L_i$ be closed Lagrangian integral currents, and suppose $L_i\to L$ in the current topology, then the Lagrangian/quantitative almost calibratedness conditions pass to the limit.
\end{lem}

Let $L$ be a closed Lagrangian integral current, and $f_L$ be an $L^\infty$ function on $L$. 
We say the exact condition $\lambda= df_L$ holds in the weak sense, if for any compactly supported test $(n-1)$-form $\chi$,
\begin{equation}\label{exactweakLag}
	\int_L \lambda\wedge \chi=  -\int_L f_L d\chi.
\end{equation}
To make sense of the RHS, notice the rectifiability of $L$ allows the integration of the $L^\infty$-valued $n$-form $f_Ld\chi$. 
Equivalently, the normal current $f_L L$ has distributional derivative
$\chi\to \int_L \chi\wedge  \lambda$.

\begin{rmk}
	The examples of immersed Lagrangians show that we cannot require $f_L$ to have a continuous extension to $X$, so $L^\infty$-regularity is the best we can impose on $f_L$. 
\end{rmk}

\begin{lem}\label{limitexactness}
All Lagrangians are assumed to be contained in a fixed bounded region of $X$, homologous to $L_0$, and are quantitatively almost calibrated.
If $L_i$ is a sequence of exact Lagrangians with potential $f_{L_i}$, such that $f_{L_i}$ are uniformly bounded in $L^\infty$. Then
up to subsequence, there is a Lagrangian $L$ with potential $f_L$, such that $L_i\to L$ and
$f_{L_i}L_i\to f_LL$ as currents.  
\end{lem}

\begin{proof}
By Lemma \ref{Volumebound} the volume mass is uniformly upper bounded. By Federer-Fleming compactness, subsequentially $L_i\to L$ in the flat topology for some Lagrangian integral current $L$ homologous to $L_0$. This implies $\int_{L_k}g \text{Re}\Omega\to \int_L g\text{Re}\Omega$ for any $C^\infty$ test function $g$, even though $\liminf_i Mass(L_i)$ may be strictly greater than $Mass(L)$, as we do not assume varifold convergence.

We focus on a coordinate ball. The $n$-currents $f_{L_k}L_k$ can be viewed as a collection of ${2n\choose n}$ signed measures $g\mapsto \int_{L_k} f_{L_k}g dx_{i_1}\wedge\ldots dx_{i_n}$. Each of these measures are bounded by the measure
\[
g\mapsto (\sup \norm{f_{L_i} }_{L^\infty} ) \frac{C}{\sin \epsilon} \int_{L_k}g \text{Re}\Omega,
\]  
whose total mass is uniformly bounded for all $k$.  By the weak compactness of measures, subsequentially these signed measures converge, and the limiting signed measures have $L^\infty$ Radon-Nykodim derivatives $a_{i_1\ldots i_n}(x)$ with respect to the measure $g\mapsto \int_L g\text{Re}\Omega$:
\[
 \int_{L_k} f_{L_k}gdx_{i_1}\wedge\ldots dx_{i_n}
 \to \int_L g a_{i_1\ldots i_n} \text{Re}\Omega.
\]
Thus inside the coordinate ball, the currents $f_{L_k}L_k$ converge to $\lim_k f_{L_k}L_k$: 
\[
\eta\mapsto
\int_L \sum_{i_1<i_2\ldots<i_n} \eta(\partial_{i_1}\wedge \ldots \partial_{i_n}) a_{i_1\ldots i_n} \text{Re}\Omega,
\]
where $\eta$ is any test $n$-form.

Now $L$ is an integral current, so $\mathcal{H}^n$-a.e. $y\in \text{supp}(L)$ there is a well defined tangent space $T_yL$ and a local integer multiplicity $\Theta(y)$.
Recall a blow up limit of an $n$-current $N$ at a point $y\in X$ refers to a subsequential limit of the currents on $T_y X$ as $r\to 0$:
\[
\eta\mapsto \int_{N} \text{rescale}_{y,r}^* \eta,\quad \text{rescale}_{y,r}: x\mapsto \frac{x}{r} \text{ in the geodesic coordinates around $y$}. 
\]
For a.e $y\in \text{supp}(L)$, there is a unique blow up limit for the current $\lim_k f_{L_k}L_k$, which is
\[
\eta\mapsto \int_{T_yL} \sum  \eta(\partial_{i_1}\wedge \ldots \partial_{i_n}) a_{i_1\ldots i_n}(y)\Theta(y)\text{Re}\Omega,
\]
whose ${2n \choose n}$ component signed measures are just constant multiples of the Lebesgue measure on $T_yX$.

Observe that the weak formulation (\ref{exactweakLag}) passes to the limit:
\[
\int_L \lambda \wedge \chi=- (\lim_k f_{L_k}L_k)(d\chi).
\]
Thus the blow up limit of $\lim_k f_{L_k}L_k$ at a.e. $y\in \text{supp}(L)$ is in fact a closed current. Consequently, the polyvector
\[
\sum_{i_1<i_2\ldots<i_n}  a_{i_1\ldots i_n} \partial_{i_1}\wedge \ldots \partial_{i_n}
\]
must be a pure tensor lying in $\Lambda^n T_yL\subset \Lambda^n T_yX$. Hence
\[
\lim_k f_{L_k}L_k= f_L L
\]
for some $L^\infty$-\emph{function} $f_L$.
\end{proof}

\subsubsection*{Continuity of the Solomon functional}

Inside Stein manifolds, the Solomon functional can be defined for any Lagrangian $L$ with potential which is homologous to $L_0$, without further Floer theoretic inputs: the formula (\ref{Solomonfunctionalextension}) makes sense after choosing any bordism current $\mathcal{C}$ with $\partial \mathcal{C}=L-L_0$ in the sense of currents, and the choice does not matter.

\begin{lem}\label{Solomoncontinuity}
(\textbf{Continuity of the Solomon functional})
All Lagrangians are assumed to be contained in a fixed bounded region of $X$, homologous to $L_0$, and are quantitatively almost calibrated.	
Suppse $L_i$ is a sequence of  Lagrangian integral currents with potential $f_{L_i}$, such that $L_i\to L$ in the flat norm, and $f_{L_i}L_i$ converge to $f_L L$ as currents, then the Solomon functionals converge: $\mathcal{S}(L_i)\to \mathcal{S}(L)$.
\end{lem}

\begin{proof}
Since
$f_{L_i}L_i$ converges to $f_L L$ as currents,
\[
\int_{L_i} f_{L_i}\Omega \to  \int_{L} f_{L}\Omega.
\]
It suffices to justify 
$
\int_{\mathcal{C}_i} \lambda \wedge \Omega\to \int_{\mathcal{C}} \lambda \wedge \Omega,
$
where $\partial \mathcal{C}_i=L_i-L_0$, and $\partial \mathcal{C}=L-L_0$.

Now the flat norm convergence gives $L_i-L=A_i+\partial B_i$ for some integral currents $A_i, B_i$, with $Mass(A_i)+Mass(B_i)\to 0$. Since $[L_i]=[L]=[L_0]\in H_n(X)$, the homology class of $A_i$ is zero. A version of the isoperimetric theorem (\cf Prop. \ref{isoperimetricthmcomplete} below) then says $A_i=\partial B_i'$ with $Mass(B_i')\leq C Mass(A_i)^{(n+1)/n}\to 0$. Without loss of generality we absorb $B_i'$ into $B_i$. Then we can simply choose $\mathcal{C}_i=\mathcal{C}+ B_i$, which is legitimate since it satisfies $\partial \mathcal{C}_i=L_i-L_0$. The claim follows by \[
|\int_{B_i} \lambda \wedge \Omega|\leq C Mass(B_i)\to 0.
\]
\end{proof}


\subsubsection*{Robustness of potential clustering}


We revisit the potential clustering property (\cf section \ref{Solomonboundedpartsection}) from the geometric measure theory perspective.
Assume as always that the Lagrangian integral current $L$ is quantitatively almost calibrated, homologous to $L_0$, and equipped with Lagrangian potential $f_L$.
Given constants $N\in \mathbb{N}, A\in \R_+$, we say $L$ satisfies $(N,A)$-\textbf{potential clustering}, if \[
L=\sum_1^N L_i,\quad f_L L=\sum_1^N f_{L_i}L_i,
\]
for quantitatively almost calibrated, closed Lagrangian integral currents $L_i$ with potential $f_{L_i}$, contained inside the support of $L$, such that the oscillation of the Lagrangian potentials have uniform bounds
\[
\sup_{L_i} f_{L_i}- \inf_{L_i} f_{L_i} \leq A,
\]
while for any $i>j$,
\[
\sup_{L_j} f_{L_j}\leq \inf_{L_i} f_{L_i}.
\]
Without loss of generality, we assume $\sup_{L_0} f_{L_0}- \inf_{L_0} f_{L_0}\leq A$ for the fixed Lagrangian $L_0$.

\begin{rmk}
Here we allow $L_i$ to have overlapping supports. For instance, it is possible for $L_1=L_2$ as currents, but $f_{L_2}$ and $f_{L_1}$ differ by a constant.
\end{rmk}

We will later be interested in uniform upper bounds on $N, A$. For now, we observe the robustness under limits:

\begin{cor}\label{potentialoscillationrobust}
Fix the choice of $N,A$. Suppose we are given a sequence of Lagrangians $L^{(j)}$ with potential satisfying $(N,A)$-potential clustering, and assume $L_i^{(j)}\to L_i$ as currents for $1\leq i\leq N$, all $f_{L_i}^{(j)}$ have uniform $L^\infty$ bounds, and $f_{L_i^{(j)}}L^{(j)}\to f_{L_i} L_i$ as currents. Then the limit $L=\sum_1^N L_i$ with its potential $f_L$ also satisfies $(N,A)$-potential clustering.

\end{cor}

\begin{proof}
Notice a potential bound such as $f_L\geq c$ can be characterized by the positivity of the measure $g\mapsto \int_L (f_L-c) g \text{Re}\Omega$. This characterization is robust under current convergence, so
\[
\sup_{L_i} f_{L_i}\leq \limsup_j \sup_{L_i^{(j)}} f_{L_i^{(j)}},\quad \inf_{L_i} f_{L_i}\geq \liminf_j \inf_{L_i^{(j)}} f_{L_i^{(j)}},
\]
hence the potential clustering bounds pass to the limit.
\end{proof}

\subsubsection{Almgren's regularity in the almost Calabi-Yau setting?}

The main reason to impose the Calabi-Yau condition is so that any  special Lagrangian closed integral current is an absolute volume minimizer among all closed integral currents in the same homology class, by the calibration inequality (\ref{calibration}). If we believe Thomas-Yau conjecture to be valid more generally for almost Calabi-Yau ambient structures, with
\[
\frac{\omega^n}{n!}= (-1)^{n(n-1)} ( \frac{\sqrt{-1}}{2})^n e^{2\rho} \Omega \wedge \overline{\Omega},
\]
then we are naturally motivated to ask if Almgren's regularity extends to special Lagrangians in this setting:

\begin{Question}
Suppose $L$ is a compactly supported closed integral current inside an \emph{almost} Calabi-Yau manifold, which is a special Lagrangian in the sense of (\ref{specialLagrangian}). Does it imply the support of $L$ is smooth away from a Hausdorff codimension two subset?
\end{Question}

The following observations, left as easy exercises, are indications that almost Calabi-Yau manifolds behave  similarly as Calabi-Yau manifolds.

\begin{itemize}
	\item 
By a variant of the calibration inequality (\ref{calibration}), special Lagrangians minimize the weighted volume 
\[
\int_L e^{-\rho} dvol_L \geq |\int_L \Omega|
\]
within its homology class.

\item Under the smoothness assumption, the mean curvature of a Lagrangian submanifold with phase function $\theta$ satisfies the formula
\[
\vec{H}= -\nabla^\perp \rho+ J\nabla \theta,
\]
where $\nabla^\perp$ means the normal projection of the gradient, and $\nabla \theta$ is the derivative of $\theta$ along $L$. Thus for smooth special Lagrangians in a bounded region, we have the a priori bound $|\vec{H}|\leq C$.
\end{itemize}

\subsection{Quantitative almost calibratedness}\label{Quantitativealmostcalibrated}

One major advantage of the quantitative almost calibrated condition is that within a fixed homology class of Lagrangians, it guarantees an a priori volume upper bound (\cf Lemma \ref{Volumebound}). We shall explain that, under very mild asymptotic conditions on the ambient Calabi-Yau manifolds, it also guarantees that the Lagrangian remains within a bounded region. As such, the Federer-Fleming compactness applies automatically, and Allard compactness applies under the additional hypothesis of a uniform bound on $\int_L |\vec{H}|$. We also discuss a number of instructive but not particularly difficult consequences of quantitative almost calibratedness.

\begin{rmk}
The arguments in this section are adaptions of Neves \cite{Nevessurvey}. They can also be easily adapted to the almost Calabi-Yau setting under mild conditions on the volume density.	
\end{rmk}

 As a preliminary, we will say the regularity scale near a given point $P$ on the  Calabi-Yau manifold is at least $O(R)$, if $P$ is contained in a complex coordinate ball with Euclidean radius at least $R$, on which $\Omega=fdz_1\wedge \ldots dz_n$, and\footnote{We will actually only use the metric uniform equivalence. But for Calabi-Yau metrics, the higher derivative estimates are in any event implied by metric equivalence,  after shrinking the balls slightly, by Evans-Krylov theory.}
\begin{equation}\label{uniformellipticity}
C^{-1}\delta_{ij }\leq g_{i\bar{j}} \leq C\delta_{ij}, \quad |\partial^k g_{i\bar{j}} |\leq C(k) R^{-k}, \quad  |f-1|\leq \frac{1}{2}, \quad |\partial^k f|\leq CR^{-k}.
\end{equation}
From now on we will \emph{assume the regularity scale tends to infinity for $P\to \infty$}. This is a very mild condition on the Calabi-Yau manifold, for instance satisfied by asymptotically conical Calabi-Yaus.

\begin{rmk}
This asymptotic condition is essentially the weakest that can prevent the almost calibrated Lagrangian in a given homology class from escaping to infinity. For instance, if $L$ is a special Lagrangian in $X$, then $L\times S^1$ is a special Lagrangian in the product $X\times T^*S^1$, which can obviously be translated in the $\R$ direction to escape to infinity, albeit not preserving the $D^bFuk(X)$ class. One can still hope to obstruct escaping to infinity using Floer theory, but then incorporating singular Lagrangians requires more foundational work.
\end{rmk}

\subsubsection*{Isoperimetric inequality}

\begin{lem}(Isoperimetric inequality \cf \cite[Lem 3.10]{Nevessurvey})
Let $L$ be a closed Lagrangian integral current in $(X,\omega, \Omega)$, and consider a Euclidean coordinate ball $B_{2R}$ on which the regularity scale is at least $O(R)$. Assume the quantitative almost calibrated condition $\cos \theta\geq \sin \epsilon$. Then there is a universal constant $C$ depending only on $\epsilon, n$ and the metric uniform equivalence constant in (\ref{uniformellipticity}), so that
\[
Mass(A)^{(n-1)/n} \leq C Mass(\partial A)
\]
for all closed subsets $A$ of $\text{supp}(L)\cap B_R$ with rectifiable boundary. Here the Hausdorff measures are computed using the Calabi-Yau metric.

\end{lem}

\begin{proof}
The isoperimetric theorem \cite[Thm 6.1]{LeonSimon} guarantees the existence of an integral current $A'$ supported in $B_{R}$ such that $\partial A'=\partial A$ and for which
\[
Mass(A')^{(n-1)/n} \leq CMass(\partial A).
\]
 Notice the metric uniform equivalence means we do not need to be careful to distinguish the Hausdorff measure for the Euclidean metric and the Calabi-Yau metric. Let $T$ denote the cone over the current $A-A'$, then $\partial T=A-A'$, and thus by the quantitative calibrated condition,
 \[
 \begin{split}
  & Mass(A) \leq \frac{1}{\sin \epsilon} \int_A \text{Re} \Omega =  \frac{1}{\sin \epsilon} \int_{A'} \text{Re} \Omega +  \frac{1}{\sin \epsilon} \int_{\partial T} \text{Re} \Omega \\
 \leq & \frac{1}{\sin \epsilon} Mass(A') + \frac{1}{\sin \epsilon} \int_{ T} d\text{Re} \Omega \leq \frac{1}{\sin \epsilon} C Mass(\partial A)^{n/(n-1)},
 \end{split}
 \]
which is the isoperimetric inequality. 
\end{proof}

\begin{rmk}
The standard isoperimetric theorem inside the Euclidean space \cite[Thm 6.1]{LeonSimon} does not state $\text{supp}(A')\subset B_R$. Once we have found some $A'$ with $\partial A'=\partial A$ with mass control, we can pushforward by a Lipschitz retraction map $F:\R^{2n}\to B_1$:
\[
F(x)=\begin{cases}
x, \quad x\in B_R,\\
\frac{Rx}{|x|},\quad x\in \R^{2n}\setminus B_R.
\end{cases}
\]
Replacing $A'$ by $F_*A'$, the mass cannot increase, and $\partial F_*(A')=F_*(\partial A')=F_*(\partial A)=\partial A$, and we have ensured the support is contained in $B_R$.

\end{rmk}

The following version of the isoperimetric theorem should be well known to experts, but for lack of a reference we include a proof below.

\begin{prop}\label{isoperimetricthmcomplete}
(Isoperimetric theorem on complete manifolds) Let $X$ be a complete Riemannian manifold, and $\partial A$ be an $m$-dimensional exact integral current supported in a fixed bounded open subset $U\subset X$. Then there is an integral current $A'$ supported in a fixed large bounded subset of $X$, with $\partial A=\partial A'$ and 
\[
Mass(A')\leq \text{Const}\cdot \min\{    Mass(\partial A)^{(m+1)/m},  Mass(\partial A)\}.
\]

\end{prop}

\begin{proof}
We first isometrically embed $X$ into an ambient Euclidean space $\R^N$, so $A$ can be regarded as an integral current compactly supported in $X$. Fix a small number $\rho_0>0$ such that over $U$ the $\rho_0$-neighbourhood in $\R^N$ is isomorphic to the normal bundle, so there is a smooth retraction map $F$ back to $U$. The Lipschitz norm of $F$ is  approximately one.

Applying the deformation theorem for $\R^N$ \cite[section 5.3]{LeonSimon} to the current $\partial A$, with a parameter $\rho\leq\rho_0$ to be fixed, we can write 
\[
\partial A= P+ \partial R,
\]
where $P,R$ are integral currents inside $\R^N$, supported in the $O(\rho)$ neighbourhood of $\text{supp}(\partial A)$, with 
\[
Mass(R)\leq  C\rho Mass(\partial A),\quad Mass(P)\leq C Mass(\partial A), 
\]
where the constant $C$ depends only on $N,m$.
Morever, $P$ is an integral linear sum of $m$-dimensional faces in the standard grid decomposition of $\R^N$ with cube size $\rho$.
We now push forward via $F$:
\[
F_{*} P+ \partial F_*R = F_*(\partial A)= \partial A,
\]
since $\partial A\subset U\subset X$ is fixed by $F$. Note that $F_{*} P, F_*R$ both live inside $U$, and their mass bounds are essentially the same as $P,R$ respectively.

Suppose first that $Mass(\partial A)\ll 1$.  If $P$ is nonzero, then by the grid description of $P$,
\[
\rho^m \leq Mass(P) \leq CMass(\partial A)
\]
So by choosing $\rho=2(C Mass(\partial A))^{1/m}$ in the above, we force $P=0$, so $\partial A= \partial F_*R$, with mass bound $Mass(F_*R)\leq \text{const} Mass(\partial A)^{(m+1)/m}$, so it suffices to take $A'=F_*R$.

Now suppose $Mass(\partial A)\gtrsim 1$, then we choose $\rho=\rho_0$. Without loss of generality, we can replace $\partial A$ by $F_*P$, and pretend $R=0$. We know 
\begin{itemize}
\item $P$ is an integral linear combination of grid cube faces, where all the cubes lie in a bounded region of a fixed grid,
\item $F_*P$ is an exact current on $X$.
\end{itemize}
The set of all such $P$ form a finitely generated abelian group, which by classification is isomorphic to the direct sum of $\oplus_1^r \Z e_i$ and a finite abelian group. For any given element
\[
P=\sum a_i e_i+ \text{finite group part},
\]
the linear coefficients $a_i$ of $e_i$ for $i=1,\ldots r$ are bounded by $|a_i|\lesssim Mass(P)\lesssim Mass(\partial A)$. Each $e_i$ gives rise to an exact simplicial chain $F_*e_i$ inside $X$, which is the boundary of a finite mass integral current $Q_i$. Thus
\[
Mass(\sum_1^r a_iQ_i) \leq \sum_1^r |a_i| Mass(Q_i) \leq \text{const} \cdot Mass(\partial A).
\]     
The finite group part gives rise to another simplicical chain inside $X$ which is the boundary of some finite mass integral current. Thus we have produced an integral current $A'$ with $\partial A'=\partial A$, and mass bound
\[
Mass(A')\leq C(Mass(\partial A)) +C\leq \text{const }  Mass(\partial A).
\]
since we are in the $Mass(\partial A)\gtrsim 1$ case.
\end{proof}

\subsubsection*{Volume monotonicity and lower bound}

\begin{cor}\label{Volumelowerboundlemma}(Volume lower bound)
If $P$ is in the support of $L$, then there is a uniform lower bound on the volume of $L$ inside Eulidean coordinate balls of radius less than $R$:
\begin{equation}\label{volumelowerbound}
\text{Vol}_g(L\cap B(P, r)) \geq C^{-1}r^n, \quad \forall r\leq R.
\end{equation}
\end{cor}

\begin{proof}
Let $f(r)= \mathcal{H}^n(L\cap B(r))>0$, then $f$ is increasing in $r$, and for a.e. $0<r<R$, by the coarea formula,
\[
f'(r) = \int_{\partial B(r)\cap L} \frac{1}{|\nabla r|} d\mathcal{H}^{n-1} \geq C^{-1} \mathcal{H}^{n-1}(\partial B(r)\cap L ) \geq C^{-1} f(r)^{(n-1)/n}.
\]
The last inequality is the isoperimetric inequality. Thus
$
\frac{d}{dr} f^{1/n} \geq C^{-1},
$
whence we have the volume lower bound $f(r)^{1/n}\geq C^{-1}r$.
\end{proof}

\begin{rmk}
Volume lower bounds like (\ref{volumelowerbound}) are familiar in minimal surface theory, but usually require some integral bound on the mean curvature. Notably, here we need no such assumption; the quantitative almost calibrated condition only concerns the antiderivative $\theta$ of $\vec{H}= J\nabla \theta$. 
\end{rmk}

\subsubsection*{No escape to spatial infinity}

\begin{cor}\label{Noescape}
Assume near the infinity of $X$, the regularity scale grows to infinity.
Fix the homology class of the quantitatively almost calibrated Lagrangian $L$. Then $L$ is contained in a fixed compact subset of $X$.
\end{cor}

\begin{proof}
From Lemma \ref{Volumebound} there is an a priori volume bound $\text{Vol}(L)\leq C$. But if $P$ is in the support of $L$, then the regularity scale of $X$ near $P$ is bounded by
\[
R^n \leq C \text{Vol}(L\cap B(R)) \leq C,
\]
whence $P$ must remain in a fixed compact subset.
\end{proof}

\subsubsection*{Nontriviality of homology classes}

\begin{cor}(Nontriviality of homology classes)
There is a lower bound $\int_L \text{Re}\Omega \geq C^{-1}$ depending only on the ambient Calabi-Yau and the almost calibratedness constant $\epsilon$. Here we do not a priori specify the homology class of $L$.
\end{cor}

\begin{proof}
The asymptotic hypothesis on the regularity scale implies in particular that the regularity scale has a global lower bound on $X$. This gives a uniform lower bound on $\text{Vol}(L)$, which by the quantitative almost calibratedness gives a lower bound on the homological integral $\int_L\text{Re}\Omega $.
\end{proof}

The significance is that in a variational setup to find special Lagrangians among certain Lagrangian currents in the fixed homology class $[L_0]$, it may happen that the Lagrangian breaks into several connected components, each given by a closed Lagrangian current $L_i$ with  $\sum_i \int_{L_i} \text{Re}\Omega= \int_L\text{Re}{\Omega}$. The nontriviality of each homology class then puts an upper bound on the number of connected components. 
Morever, each $L_i$ would inherit a volume upper bound from $L$. Since all Lagrangians are contained in a fixed compact region, Poincar\'e duality easily gives an upper bound on the homology classes of $L_i$:
\[
\norm{ [L_i] }_{H_n(X)} \leq C\text{Vol}(L_i) \leq C. 
\]
Combining the above, \emph{only finitely many homology classes, multiplicities, and number of components can appear in a given variational problem}.

\subsubsection{Intrinsic distance bound and potential clustering}\label{Intrinsicdistancepotentialclustering}

Suppose $L$ is a \emph{smooth} immersed Lagrangian, then it inherits a Riemannian metric from the restriction of the Calabi-Yau metric, so we can speak of the intrinsic distance function on $L$.   Instead of extrinsic balls such as $B(P,r)$, we can talk about intrinsic balls $\hat{B}(P,r)$. A key distinction is that intrinsic distance does not need to extend to a continuous function on $X\times X$, the prototypical example being the union of two embedded Lagrangians, whose domains are disjoint, but whose images in $X$ intersect. Clearly, the intrinsic distance bounds extrinsic geodesic distance, so $\hat{B}(P,r)$ is always contained in an extrinsic geodesic ball of radius $r$, but the converse is far from true. The intrinsic distance between two distinct connected components would simply be infinity. One can think of intrinsic distance as a \emph{quantitative measurement of connectedness}.

As usual, the regularity scale on $X$ grows to infinity asymptotically by assumption, so the regularity scale has a global lower bound. The following lemma has the same proof as Cor. \ref{Volumelowerboundlemma}. (The essence of this argument also appears in Neves \cite[Lem 3.9]{Nevessurvey}).

\begin{lem}(Intrinsic ball volume lower bound)
Let $L$ be a smooth immersed compact Lagrangian in $X$, satisfying the quantitative almost calibrated condition. For any $P$ in the support of $L$, there is a uniform bound 
\[
\text{Vol}(\hat{B}(P,r)\cap L) \geq C^{-1}r^n, \quad r\leq 1. 
\]
\end{lem}

\begin{cor}
(Intrinsic diameter bound) Assume further that the smooth, quantitatively almost calibrated compact Lagrangian $L$ has \emph{connected} domain. Then within a fixed homology class, the intrinsic distance of $L$ has a uniform upper bound.
\end{cor}

\begin{proof}
Let $d=dist(P,P')\geq N=[d]$ be the intrinsic diameter of $L$. By connectedness, we can find $P_1,\ldots P_N$ with $dist(P, P_i)=i$. Now the intrinsic balls $\hat{B}(P_i, \frac{1}{2})$ are disjoint, but each takes up a nontrivial amount of volume $\geq C^{-1}$. Thus
\[
C^{-1}N \leq \text{Vol}(L)\leq \frac{1}{\sin \epsilon}\int_L \text{Re}(\Omega) \leq C,
\]
so there is an a priori bound on $N$, hence on $d$.
\end{proof}

\begin{cor}\label{Potentialoscillationbound}
(Potential oscillation bound) Assume $L$ is an exact, immersed, compact Lagrangian, satisfying the almost calibrated condition, such that the Lagrangian potential $f_L$ has connected range. Then the potential $f_L$ has an a priori bound
\[
\sup_L f_L- \inf_L f_L\leq C.
\]
\end{cor}

\begin{proof}
Given any $P, P'$ on $L$, we write the potential as a line integral of the Liouville 1-form
\[
f_L(P)-f_L(P')=\int_{P'}^{P} \lambda \leq \norm{\lambda}_{C^0} dist(P,P') .
\]
Thus the intrinsic ball volume lower bound implies that if $a$ lies in the range of $f_L$, then
\[
\text{Vol}( \{ |f_L-a|<\frac{1}{2} \}) \geq C^{-1}.
\]
The range of $f_L$ is by assumption a closed interval. If the interval has length $\geq N$, then we can find $N$ distinct values of $a$ with disjoint $\{ |f_L-a|<\frac{1}{2} \}$, so 
\[
C^{-1}N \leq \text{Vol}(L)\leq \frac{1}{\sin \epsilon}\int_L \text{Re}(\Omega) \leq C,
\]
This provides an a priori bound on $N$, hence on the potential oscillation.
\end{proof}

\begin{cor}
	Assume $L$ is an exact, immersed, compact Lagrangian, satisfying the almost calibrated condition, then it satisfies $(N,A)$-\textbf{potential clustering} (\cf section \ref{ExactLagweakregularity}) for some $N,A$ with uniform bounds. Morever, if $L$ carries an unobstructed brane structure, then the potential clustering property is consistent with the twisted complex interpretation. 
\end{cor}

\begin{proof}
By Lemma \ref{blockingconnectedcomponents}, a general immersed Lagrangian can be decomposed into a union of Lagrangians $L_i$ such that the ranges of the potentials $f_{L_i}$ are connected, and any bounding cochain structure naturally produces a twisted complex. The oscillation of each $L_i$ is bounded  in terms of the quantitative almost calibration condition, and the ambient features of $X$. Morever, the number of Lagrangian components is also a priori bounded.
\end{proof}

As a notable consequence, we obtain the \textbf{uniform energy bound} on the holomorphic curves (\cf Prop. \ref{uniformenergybound}).

\begin{rmk}
Although it is unclear how to make sense of the intrinsic distance on a general Lagrangian integral current, mildly singular Lagrangians (for instance with local conical singularities) do have a sensible notion of intrinsic distance, and the arguments in this section extend practically to all non-pathological examples, covering all Lagrangians that appear in Joyce's LMCF program. Furthermore, the potential clustering is robust under limits (\cf Cor. \ref{potentialoscillationrobust}). As such, we believe it holds for all Lagrangians relevant to our variational program (\cf the class $\mathcal{L}$ in section \ref{Variationalstrategy} below).	
\end{rmk}

\subsubsection{Bounded part of the Solomon functional revisited}\label{Solomonboundedpartrevisited}

In section \ref{Solomonboundedpartsection}, we discussed a Floer theoretic method to bound the difference between the Solomon functional and the elementary functional. The following alternative method is based on the homological nature of the Solomon functional  (\ref{Solomonfunctionalextension}), and has a more geometric measure theory flavour. Holomorphic curves do not feature in this approach, so the automatic transversality and the positivity conditions are not needed, and in fact the discussion works in the weak setting of varifolds and currents. It is vital to this approach that $H_{n+1}(X)=0$ for Stein manifolds, so no homological ambiguity can arise for the bordism current.

\begin{prop}\label{Solomonboundedparthomological}
Assume the Lagrangian $L$ with potential $f_L$ is quantitatively almost calibrated, homologous to $L_0$, and satisfies $(N,A)$-potential clustering. 
The elementary functional (\ref{Solomonelementary}) makes sense verbatim, with no smoothness assumptions.
Then
	$
	|\mathcal{S}(L)-\bar{\mathcal{S}}(L)|
	$
	has a uniform upper bound independent of $L$.
\end{prop}


\begin{proof}
We know $L-L_0$ is homologous to zero in $X$, and contained in a bounded subset of $X$ by Cor. \ref{Noescape}. 
By a version of the isoperimetric theorem (\cf Prop. \ref{isoperimetricthmcomplete}), we can find some compactly supported integral current $\mathcal{C}$ with $\partial \mathcal{C}=L-L_0$, with mass bound
	\[
	\text{Mass}(\mathcal{C}) \leq \text{const}\cdot \min\{  \text{Mass}(L-L_0)^{(n+1)/n}, \text{Mass}(L-L_0)   \} .
	\]
	This $\mathcal{C}$ has no relation to holomorphic curves.
	The quantitative almost calibratedness implies a volume bound on $L$ (\cf Lemma \ref{Volumebound}), hence $\text{Mass}(\mathcal{C})$ is a priori bounded. The homological nature of the Solomon functional  (\ref{Solomonfunctionalextension}) gives
	\[
	\mathcal{S}(L)= \int_L f_L \text{Im}(e^{-i\hat{\theta}} \Omega) -\int_{L_0} f_{L_0} \text{Im}(e^{-i\hat{\theta}} \Omega)- \text{Im} \int_{\mathcal{C}} \lambda\wedge e^{-i\hat{\theta}} \Omega.
	\]
	The mass bound then implies
	\[
	|\int_{\mathcal{C}} \lambda\wedge e^{-i\hat{\theta}} \Omega | \leq C \norm{\lambda}_{C^0} \norm{\Omega}_{C^0} \text{Mass}(\mathcal{C})\leq \text{const}.
	\]
	Here since $L$ is contained in a bounded region, the terms $\lambda$ and $\Omega$ are bounded. Finally, using the potential clustering bound,
	\[
	\begin{split}
	& | \int_L f_L \text{Im}(e^{-i\hat{\theta}} \Omega) - \int_{L_0} f_{L_0} \text{Im}(e^{-i\hat{\theta}} \Omega)-  \bar{\mathcal{S}}(L)| 
	\\
	& \leq  A (\int_L |\text{Im}(e^{-i\hat{\theta}}\Omega)| +\int_{L_0} |\text{Im}(e^{-i\hat{\theta}}\Omega)| )
	\\
	&  \leq  A (\text{Mass}(L) +\text{Mass}(L_0)  ) \leq  \frac{2A}{\sin \epsilon } \int_{L_0} \text{Re}\Omega.
	\end{split}
	\]
	Combining the above shows the a priori bound on $|\mathcal{S}(L)-\bar{\mathcal{S}}(L)|$.
\end{proof}

\subsection{Variational strategy}\label{Variationalstrategy}

The variational strategy to find special Lagrangians is the following:

\begin{itemize}
\item  Find a suitable subset $\mathcal{L}$ among all the quantitatively almost calibrated, exact Lagrangian integral currents homologous to $L_0$. The class $\mathcal{L}$ is closed in the varifold/current topology. It is very desirable to ensure Allard compactness and Federer-Fleming compactness both apply to $\mathcal{L}$.

\item Extend enough of Floer theory from the smooth setting to Lagrangian currents. Morally, the class $\mathcal{L}$ consists of those Lagrangians that can be equipped with unobstructed brane structures in some weak sense, all isomorphic to $L_0$ in $D^bFuk(X)$.

\item
When the Lagrangian is equipped with the potential $f_L$, the additive constant freedom of $f_L$ is a source of non-compactness, which affects  $\mathcal{S}$. We need to ultimately match up the asymptotic behaviour of $\mathcal{S}$ with the Floer theoretic obstructions. In other words, the role of stability conditions is to ensure the properness of the Solomon functional.

\item
Once the Solomon functional is proper, we will follow the direct minimization strategy to find its minimum. We need to justify that the minimum $L$ must be a special Lagrangian closed integral current, and then Almgren regularity will be able to ensure smoothness away from codimension two. Furthermore, we need a sufficiently robust version of the Thomas-Yau uniqueness argument to prove that the special Lagrangian representative is unique.

\end{itemize}

The class $\mathcal{L}$ is a balance between two requirements: the \emph{approximability by sufficiently smooth objects}, and the existence of \emph{sufficiently many competitors}. A moral definition of $\mathcal{L}$ is:
\begin{itemize}
\item Among all the quantitatively almost calibrated, exact Lagrangian integral currents homologous to $L_0$, we include all sufficiently smooth Lagrangians (eg. immersed, $T^2$-cones singularities, etc) which admit unobstructed brane structures isomorphic to $L_0$ in $D^bFuk(X)$.

\item Then take the closure under the varifold/current topology.
\end{itemize}

\begin{rmk}
Joyce's LMCF is expected to preserve the exactness, the quantitative almost calibrated condition, and the unobstructedness of the brane structure, so sufficiently smooth objects in $\mathcal{L}$ should remain in $\mathcal{L}$ under Joyce's LMCF. It is interesting to ask when the flow also preserves the positivity condition on the bordism current.
\end{rmk}

While at present several ingredients are missing, if this program can be carried through, it would prove the existence of special Lagrangians under the assumption of Thomas-Yau semistability (\cf Definition \ref{ThomasYausemistability}).

\subsubsection*{$L^p$-Smoothing property and Joyce's LMCF}

Allard compactness requires an a priori bound $\int_L |\vec{H}|\leq C$, which cannot be implied by the quantitative almost calibrated condition, since the mean curvature involves one more derivative than the Lagrangian angle. However, for the purpose of our variational strategy, it is enough to ensure any minimization sequence of $\mathcal{S}$ can be replaced by a sequence with $\int_{L_i} |\vec{H}|\leq C$.

\begin{conj}
($L^p$-\textbf{smoothing property}) There exists $p\geq 1$ and a uniform constant $C$, such that
for any  $L\in \mathcal{L}$, we can find $L'\in \mathcal{L}$ with $\mathcal{S}(L')\leq \mathcal{S}(L)$, and 
$\int_{L'} |\vec{H}|^p \leq C $. 
\end{conj}

\begin{rmk}
This is called a `smoothing property' because $L'$ quantitatively improves the regularity of $L$. It does not suggest $L'$ is smooth, and indeed we expect the special Lagrangians which minimize $\mathcal{S}$ may have codimension two singularity. Since the volume is a priori bounded, the H\"older inequality shows that the $L^p$-smoothing property is stronger for bigger $p$, and in particular $L^2$-smoothing implies $L^1$-smoothing.
\end{rmk}

We think the smoothing property may be quite deep, and our limited attempt here is to explain how it relates to Joyce's LMCF program, which suggests the smoothing property may hold with $p=2$. Recall the defining feature of the Solomon functional is its variation property under exact isotopies among unobstructed objects: 
\[
\delta \mathcal{S}= \int_L h \text{Im}(e^{-i\hat{\theta}}\Omega),
\]
which holds under sufficient smoothness assumptions.
Under a sufficiently smooth LMCF $(L_t)$ in a Calabi-Yau manifold, the Lagrangians evolve by the local Hamiltonian function $-\theta_t$  up to an inconsequential additive constant (\cf section \ref{LMCF}), so $\mathcal{S}$ evolves by
\begin{equation}\label{Solomonfunctionalevolution}
\begin{split}
& \partial_t \mathcal{S}= -\int_{L_t} (\theta_t- \hat{\theta}) \text{Im}(e^{-i\hat{\theta}}\Omega)= - \int_{L_t} (\theta_t- \hat{\theta}) \text{Im} (e^{i(\theta-\hat{\theta})} )dvol_{L_t}
\\
=& - \int_{L_t} (\theta_t-\hat{\theta}) \sin(\theta_t-\hat{\theta}) dvol_{L_t}.
\end{split}
\end{equation}
If $L_t$ is almost calibrated, then $-\pi<\theta-\hat{\theta}< \pi$, so $\partial_t \mathcal{S}\leq 0$. We conclude that \emph{the Solomon functional decreases in time along Joyce's LMCF under the almost calibrated assumption}, at least for the time between the surgeries. It is plausible $\mathcal{S}$ is either continuous or jumps downwards at the surgeries in Joyce's LMCF,\footnote{A somewhat analogous phenomenon in the Brakke flow is that the total volume mass is either continuous or can only jump downwards in time. The mass loss is typically related to the disappearance of a component of the evolving varifold, which is conceptually similar to `collapsing zero objects' in Joyce's LMCF. This is ruled out by the almost calibrated condition, so optimistically one can even hope for the continuity of the Solomon functional in the almost calibrated setting. } which would then imply the Solomon functional is monotone decreasing for all time.

Now recall that the heat equation on the Lagrangian angle implies an integral bound on the mean curvature (\ref{meancurvatureL2}). In particular, if the LMCF can be run for a \emph{definite amount of time} $T$, then there exists some $t\leq T$, with
\[
\int_{L_t} |\vec{H}|^2 dvol_{L_t} \leq T^{-1} \int_{L_{t=0}} \theta^2 dvol_{L_{t=0}} \leq CT^{-1},
\]
where crucially the a priori constant $C$ does not depend on any quantitative smoothness assumption on the initial Lagrangian, provided it is quantitatively almost calibrated. Such $L_t$ would be a good candidate for $L'$, subject to the hypothesis that Joyce's LMCF remains within the class of Lagrangians $\mathcal{L}$.

Morally the class $\mathcal{L}$ arises as varifold/current limits of those Lagrangians admissible in Joyce's program. Under the plausible assumption that Joyce's LMCF can be passed to the varifold/current limit, then the $L^2$-smoothing property can be well explained. The $\mathcal{S}(L')\leq \mathcal{S}(L)$ condition comes from the decrease of the Solomon functional along the flow, and the $\int_{L'} |\vec{H}|^2 dvol_{L'}\leq C $ condition would follow if Joyce's LMCF can be run for a uniform amount of time $T$. If $T$ can be taken  arbitrarily large, then we can demand further that the $L^2$ mean curvature is arbitrarily small.

\begin{rmk}\label{approximateLag}
In minimal surface theory, the ability to approximate an unknown object by objects with quantitative derivative controls, is frequently the key of the regularity theory.  Notable examples include the Lipschitz and harmonic approximations that lie at the core of De Giorgi's $\epsilon$-regularity theorem, and the center manifolds at the core of Almgren's big regularity theorem. An excellent survey is \cite{Delelliscentermfd}. While there are plenty of techniques for constructing area competitors in geometric measure theory, we lack useful ways to construct  competitors within the Lagrangian world. Developing such techniques is essential to the $L^p$-smoothing property, and possibly also to the Floer theoretic aspects of the variational program.
\end{rmk}

\subsection{Floer theory under weak regularity}\label{Floertheoryweakregularity}


The variational program needs to incorporate singular Lagrangians as objects of $D^bFuk(X)$, which naturally raises many Floer theoretic questions, such as: 
\begin{itemize}
\item Suppose a sequence of (exact, quantitatively almost calibrated, smooth) Lagrangians
converge in the varifold/current topology to some singular Lagrangian, then what Floer theoretic information can be passed to the limit?

\item What does it mean for two Lagrangian currents to lie in the same derived Fukaya category class?

\item  Does it still make sense to talk about Floer theoretic obstructions in the weak regularity setting?
\end{itemize}

In this section we will offer some general remarks and speculations about the nature of these problems, but will not solve them in any definitive way.

\begin{rmk}
There is a field called $C^0$-symplectic topology, which studies properties stable with respect to convergence of Lagrangians under $C^0$-Hamiltonian isotopies, especially spectral type invariants. This is morally related to our concerns here, but as far as the author understands, Floer theory for Lagrangian varifolds/currents is not yet explicitly treated in this field.

\end{rmk}

\subsubsection*{Floer theoretic difficulties}

If one wishes to build Floer theory for Lagrangian currents by mimicking the smooth case constructions, then one immediately runs into a large number of severe difficulties.

\begin{itemize}
\item   For exact embedded Lagrangians, the \textbf{self Floer cohomology} of a Lagrangian is isomorphic to the singular cohomology: $HF^*(L,L)\simeq H^*(L)$. Now in the light of Almgren's big regularity theorem, our best hope is that in the variational argument we only encounter codimension two singularities in the Lagrangian. We have no right to assume the topology of the Lagrangian is fixed in the variational framework. The homology groups $H_{n-m}(L)$ for $m\geq 1$ are highly unstable under varifold/current convergence if codimension two singularities can form, so for $m\geq 1$ we do not expect a direct geometric definition of $HF^m(L,L)$ for  Lagrangian currents, that possesses any reasonable continuity property under convergence.

\item  The standard way to set up Floer theory between two Lagrangians is to consider the \textbf{transverse intersection} points as the generators of the Floer complex, and counts of holomorphic strips as differentials between generators. This viewpoint depends heavily on the differential topology of the Lagrangians, which runs into troubles for Lagrangian currents, where tangent spaces only need to exist almost everywhere in a measure theoretic sense.

\item  Once Lagrangian intersections are not well behaved, we cannot define the \textbf{bounding cochains} supported at intersection points in the usual way.

\item Parallel transport along \textbf{local systems} may break down.

\item  It is unclear how to define (relative) \textbf{spin structures} on Lagrangian currents.

\item Standard Floer theory depends heavily on \textbf{transversality arguments} based on differential topology, which is lost on Lagrangian currents. 

\end{itemize}

In short, a direct geometric construction of the $A_\infty$ structure is unlikely for Lagrangian currents.

\subsubsection*{Formal limit perspective}

One natural idea is that we only develop Floer theory for sufficiently smooth Lagrangians (eg. immersed Lagrangians, isolated $T^2$-cones, etc), and formally treat Lagrangian currents using \textbf{approximation by smooth objects}.
Suppose $L_i$ are sufficiently smooth Lagrangian branes in the same $D^bFuk(X)$ class, and $L_i\to L$ in the varifold/current topology, and assume the brane structures provide a Cauchy sequence in some appropriate sense, then one \emph{formally} declare the Lagrangian current $L$ as carrying an object in the same $D^bFuk(X)$ class. A weak Lagrangian brane would then tautologically be an equivalence class of Cauchy sequences. The same Lagrangian current may in principle support many different formal brane structures, not necessarily all in the same derived category class.

In this perspective, weak Lagrangian branes are indirect constructions, whose properties amount to quantitative properties of sufficiently smooth Lagrangians that can be bounded in terms of a priori quantities such as the distance on the branes, the flat norm on the currents, the Hausdorff distance between the Lagrangians, etc.

\begin{Question}\label{formalbranecompactness}
Is there a notion of distance between two  Lagrangian branes $L,L'$ in the same $D^bFuk(X)$ class, that has \emph{precompactness property modulo gauge} under varifold/current topology, in the setting of exact, quantitative almost calibrated Lagrangians with bounded Lagrangian potential?
\end{Question}

One concrete notion of distance is as follows (\cf \cite[Definition 2.2]{FukayaGH}, see also \cite[section 5]{BiranCornea}).
We can look for the $\alpha, \beta$ representing generators in $HF^0(L,L')$ and $HF^0(L',L)$ with cohomological compositions equal to the identity; in the almost calibrated case $CF^{-1}(L,L')=0$, so $\alpha,\beta$ are unique up to scaling. Since all bounding cochains and $A_\infty$ products have non-negative Novikov exponents, and the sum of Novikov exponents add up to zero, we must have some negative Novikov exponent for $\alpha$ or $\beta$. In our context, the Novikov exponent amounts to $(f_L- f_{L'})(p)$ at $p\in CF^0(L,L')$ and $(f_{L'}-f_{L})(q)$ at $q\in CF^0(L',L)$.
The quantity
\[
-\min \{\text{Novikov exponents among all intersection points of $\alpha,\beta$}\}
\]
provides a candidate notion of distance $d(L,L')$ between Lagrangian branes. Notice this distance bounds the energy of the holomorphic discs with boundary on $L,L'$. Given three objects $L,L',L''$, by considering the composition of the generators, it is easy to deduce $d(L,L'')\leq d(L,L')+d(L',L'')$.

Does this notion of distance have any precompactness property? Namely, given a sequence of sufficiently smooth Lagrangian objects $L_i\in \mathcal{L}$, (eg. a minimizing sequence for the Solomon functional), and assuming the Lagrangian potentials are uniformly bounded, then up to making gauge equivalent choices of local systems and bounding cochains, when can we extract a Cauchy subsequence?

\begin{rmk}\label{Sequivalencermk1}
As an illustration of the subtlety, consider immersed Lagrangians $L$ built as the cone of $L_2\xrightarrow{\gamma} L_1[1]$. Replacing $\gamma$ by $c\gamma$ for $c>0$ results in new bounding cochain structures on $L_1\cup L_2$, but the distance between these brane structures is zero. The limit $c\to 0$ however belongs to a different $D^bFuk(X)$ class. This suggests our formulation of weak Lagrangian branes is probably not sufficient to distinguish between several derived category classes. 
\end{rmk}

One may also ask if the weak Lagrangian branes agree with ordinary Lagrangian branes in the case of smooth immersed Lagrangians:

\begin{Question}
Suppose $L_i$ is a sequence of immersed Lagrangian branes, all in the same $D^bFuk(X)$ class, and is a Cauchy sequence with respect to the distance on the branes. Suppose $L$ is an immersed Lagrangian, and $L_i\to L$ in the varifold/current topology. Then does there exist a suitable brane structure on $L$ so that $L_i\to L$ with respect to the distance on the branes?
\end{Question}

\subsubsection*{Geometric perspective: bordism currents and triangulated categories}

It is interesting to ask if \emph{any} Floer theoretic geometric construction may be performed on Lagrangian currents at all. 
While the $A_\infty$-category structure on the Fukaya category may not necessarily be robust under  varifold/current convergence of Lagrangians, only a subset of the structures are essential to the Thomas-Yau conjecture:
\begin{itemize}
	\item The notion of derived Fukaya category classes.

	\item  The notion of distinguished triangles, within the class of Lagrangians $\mathcal{L}$. This is the categorical shadow of the phenomenon that Lagrangians can be broken into several components under weak limits.
	
	\item  The central charge function. 
\end{itemize}

The central charge is of numerical nature, and is continuous under convergence in the current topology. A key feature of  $L,L'$ lying in the \emph{same derived category class} is that there is a bordism current $\mathcal{C}$ constructed from holomorphic curves, such that $\partial \mathcal{C}=L-L'$. Likewise for \emph{distinguished triangles} $L_1\to L\to L_2\to L_1[1]$ in the weak regularity setting, a key expected property is that there should be a bordism current between $L$ and $L_1+L_2$, constructed from families of holomorphic curves.

\begin{Question}\label{Gromovcompactnesscurrent}
	Given unobstructed (sufficiently smooth) exact Lagrangians $L_i, L_i'$ all in the same derived category class. Assume convergence $L_i\to L$ and $L_i'\to L'$ in the varifold/current topology. Can we assign an $(n+1)$-bordism current $\mathcal{C}$ between $L$ and $L'$, constructed from the moduli space of holomorphic curves with boundary on $L$ and $L'$?
\end{Question}

The basic idea is to take the bordism current $\mathcal{C}_i$ with $\partial \mathcal{C}_i=L_i-L_i'$, constructed from the universal family of holomorphic curves, and attempt to extract the limit as currents. This could be morally viewed as a version of \textbf{Gromov compactness} for families. As rather strong evidence, in the quantitatively almost calibrated setting we derived uniform energy bound for holomorphic curves contributing to $\mathcal{C}_i$, by proving the potential clustering property (\cf section \ref{Intrinsicdistancepotentialclustering}, and Prop. \ref{uniformenergybound}). If we work with \emph{Fukaya category over the integers}, the bordism currents $\mathcal{C}_i$ would be \emph{integral currents}, and we can hope to extract limit by some compactness argument. The problem is that we do not know $\mathcal{C}_i$ have uniform mass upper bounds. Morever, it is an interesting question how to formulate the parametrized family structure of the bordism current in the geometric measure theory language.

\begin{rmk}
Question \ref{Gromovcompactnesscurrent} is formulated without any smoothness assumption on $L$ and $L'$. In the light of the conjectural $L^p$-smoothing property (\cf section \ref{Variationalstrategy}), one may be able to assume some a priori $L^2$-bound on the mean curvature.
\end{rmk}

While Lagrangian intersections, bounding cochains, spin structures, local systems etc. do not make sense directly on Lagrangian currents, the bordism current has a chance to make sense, and encodes substantial information. For instance, the orientations of the moduli spaces reflect the spin structures, and the weighting factors for the moduli spaces encode the combined effect of bounding cochain elements and the parallel transport along the local system.

\begin{rmk}
	As mentioned in section \ref{Floertheoreticobstructions},
	the mere requirement for the Floer theoretic obstruction criterion (\ie the stability condition) to make sense for  Lagrangian currents is already very constraining. Most statements are simply impossible to make without concepts that need at least $C^1$-regularity, and the bordism currents between integration cycles are among the rare exceptions. This was one of the heuristic arguments in section \ref{Floertheoreticobstructions} that obstructions must come from bordism currents.
\end{rmk}

\begin{Question}
	How much of the triangulated category structure works for weak regularity exact Lagrangians? 
	How much of Floer theory can be developed upon the notion of bordism currents? Is it possible to encode weak Lagrangian branes \`a l\`a the formal limit perspective, in terms of bordism currents?
\end{Question}

We mentioned in Remark \ref{3Lag} that when more than two Lagrangians are present, Floer theory would also produce $(n+2)$-dimensional currents whose boundary exhibit homological relations between the $(n+1)$-dimensional bordism currents. Such `bordisms between bordisms' may encode further information about the triangulated category.

\subsubsection*{Previlleged role of $HF^0$}

We consider quantitative almost calibrated Lagrangians. We mentioned above that $HF^m(L,L)$ for $m\geq 1$ is problematic, by analogy with singular cohomology. On the other hand, $H^0(L)\simeq H_n(L)$ is much more \emph{robust} compared to higher cohomologies, in the sense that
the fundamental cycle of $L$ can deform in a \emph{continuous} way, under topological changes 
such as the shrinking of a codimension two cycle. Continuing with the analogy, we expect the geometric information in $HF^0$ behaves more continuously under current/varifold limits than the higher degree Floer groups.
This is compatible with the fact that the bordism current $\mathcal{C}$ between $L,L'$ encodes the compositions $\alpha\circ \beta=1_{L'}$ and $\beta\circ \alpha=1_L$, with $\alpha\in HF^0(L,L')$ and $\beta\in HF^0(L',L)$, and we expect bordism currents have some continuity properties under varifold/current convergence.

\begin{rmk}
This previlleged role of $HF^0$ is reflected in the usual Thomas-Yau argument (\cf section \ref{ThomasYauevidence}), which only makes use of $HF^0$, not the higher Floer cohomologies,  nor full set of higher $A_\infty$ products.
\end{rmk}

\begin{rmk}
In the passage from the Fukaya category to the derived category, the morphism space only retains $HF^0$, not the full $HF^*$. The ususal way the derived category remembers higher Floer cohomology, is via the shift operator $[m]$. However, in the Thomas-Yau-Joyce picture, working with the almost calibrated setting means conjecturally that we are picking out an abelian subcategory, which breaks the shift symmetry of the derived category. This gives a categorical explanation why $HF^0$ may behave very differently from the higher Floer groups.

\end{rmk}




\subsubsection*{Multiplicity issues}

The same underlying geometric Lagrangian can conceivably support many different objects in the Fukaya category. A possible source of this problem is a sequence of immersed Lagrangians $L_i$ converging to a \emph{multiple} of a Lagrangian current $L$. The underlying Lagrangian current contains only the support information and the multiplicity, which can be imagined as the number of sheets in $L_i$. Much geometric information, however, is not captured this way:
\begin{itemize}
\item Take two Lagrangians $L_1,L_2$ which are both $C^\infty$ close to a given immersed Lagrangian $L'$, but whose Lagrangian potentials differ by approximately a constant. In the limit $L_1\cup L_2\to 2L'$ as currents, but the \emph{potential information} is lost. On the other hand, the potential clustering property can restore this information.

\item Immersed Lagrangians may be nontrivial (branched) covers over other immersed Lagrangians. When this happens, the \emph{monodromy information} is not remembered by the underlying current. On the other hand, it is conceivable that some (generalized) local system data can restore this information.

\item Let $Q$ be a closed smooth manifold. Abouzaid \cite{Abouzaidloop} showed that the wrapped Fukaya category of the cotangent bundle $T^*Q$ is generated by any cotangent fibre $T_q^*Q$, and the wrapped Floer cochain complex of $T_q^*Q$ is $A_\infty$-equivalent to  $C_{-*}(\Omega_q Q)$ for the based loop space $\Omega_qQ$. In particular, for any (compact, embedded, exact) Lagrangian $L\subset T^*Q$, the Floer cohomologies $HW^*(T_q^*Q,L)$ and $HF^*(L,L)$ are representations of $H_{-*}(\Omega_q Q)$. This cotangent bundle case can be viewed as the local model of Lagrangians contained in a small neighbourhood of a given embedded Lagrangian.

\end{itemize}

It is  interesting to ask how much of such information can still make sense for Lagrangian currents.


\begin{rmk}
Multiple covers of Lagrangians may be related to the following problem of the Fukaya category. Given a class in the Grothendieck group of $D^b Fuk(X)$ represented by a Lagrangian, one may ask if the primitive of this class is also represented by a Lagrangian. Such questions are related to the idempotent closure problem of $D^b Fuk(X)$ in Joyce's program, which seems very delicate.
\end{rmk}

\begin{rmk}
Construction of special Lagrangian branched multiple covers over given special Lagrangians is currently studied by S. Donaldson \cite{Donaldsonmultivalued} and S. He among others.
\end{rmk}

\begin{rmk}
A holomorphic vector bundle analogue for multiply covered Lagrangians is the (multiple) extension of the bundle by itself, such as the $E'$ fitting into a short exact sequence $0\to E\to E'\to E\to 0$. In the HYM setting these are prototypical sources of semistable but not stable bundles, and it would not be surprising if similar phenomenon happens in the Thomas-Yau program. 
\end{rmk}

\subsection{Asymptotes of the Solomon functional}\label{AsymptoticSolomon}

We have emphasized that the Solomon functional depends not only on $L$, but also the potential $f_L$, and that the freedom of additive constants causes the space of $(L, f_L)$ to be noncompact, even though the space of Lagrangians $\mathcal{L}$ is more or less compact under the varifold/current topology.
We now wish to explain why the asymptotic behaviour of the Solomon functional should be controlled by Thomas-Yau semistability. The key tool is an a priori bound on the difference between the Solomon functional and the elementary functional, for which we gave sufficient conditions in section \ref{Solomonboundedpartsection} and \ref{Solomonboundedpartrevisited}.

In the setup of $(N,A)$-potential clustering (\cf Cor. \ref{potentialoscillationrobust}, section \ref{Intrinsicdistancepotentialclustering}),
we will rewrite the elementary functional $\bar{\mathcal{S}}$ (\cf (\ref{Solomonelementary})). Recall we have 
 a Lagrangian $L$ built from $L_1,\ldots L_N$; in the unobstructed immersed Lagrangian context, this structure comes from a twisted complex (\cf section \ref{Solomonboundedpartsection}).
We introduce the new Lagrangian currents
\[
\mathcal{E}_k= L_1+ L_2\ldots + L_k, \quad k=0, 1,2,\ldots N,
\]
which in the immersed context corresponds to the twisted complex (\ref{HNEk}). In particular $\mathcal{E}_N=L$, which is homologous to $L_0$. Thus
\begin{equation*}
\begin{split}
\bar{\mathcal{S}}= & \text{Im}\left( \sum_1^N (\sup_{L_i} f_{L_i}) e^{-i\hat{\theta}}(\int_{\mathcal{E}_i} \Omega-  \int_{\mathcal{E}_{i-1}} \Omega )   \right) - (\sup_{L_0} f_{L_0}) \text{Im} ( e^{-i\hat{\theta}} \int_{L_0}\Omega)
\\
= & \text{Im}\left( \sum_1^{N-1} (\sup_{L_i} f_{L_i}- \sup_{L_{i+1}} f_{L_{i+1}}) e^{-i\hat{\theta}}\int_{\mathcal{E}_i} \Omega    \right) + (\sup_{L_N} f_{L_N}-\sup_{L_0} f_{L_0}) \text{Im} ( e^{-i\hat{\theta}} \int_{L_0}\Omega).
\end{split}
\end{equation*}
But we chose in the beginning
$
\hat{\theta}=\arg \int_{L_0}\Omega.
$ Thus $\text{Im} ( e^{-i\hat{\theta}} \int_{L_0}\Omega)=0$, and
\begin{equation}\label{Solomonelementary2}
\bar{\mathcal{S}}= \sum_1^{N-1} (\sup_{L_i} f_{L_i}- \sup_{L_{i+1}} f_{L_{i+1}}) \text{Im}\left(  e^{-i\hat{\theta}}\int_{\mathcal{E}_i} \Omega    \right) .
\end{equation}
As part of the potential clustering property, we have
\begin{equation}
\sup_{L_1} f_{L_1}\leq \sup_{L_2} f_{L_2} \leq\ldots \leq\sup_{L_N} f_{L_N}.
\end{equation}

We arrive at the following key \textbf{dichotomy}:
\begin{itemize}
\item In the \textbf{unstable} case, there exists some $1\leq k\leq N-1$, such that 
\[
\text{Im}\left(  e^{-i\hat{\theta}}\int_{\mathcal{E}_k} \Omega    \right) >0 ,
\]
or equivalently
\begin{equation}
\arg \int_{\mathcal{E}_k} \Omega >\hat{\theta}.
\end{equation}
Notice $L$ fits into a distinguished triangle
\[
\mathcal{E}_k \to L\to \cup_{i\geq k+1} L_i\to \mathcal{E}_k[1],
\]
We explained in Theorem \ref{Floertheoreticobstruction1} under the extra hypotheses of automatic transversality and the positivity condition, that this leads to a  \textbf{Floer theoretic obstruction}. In Conjecture \ref{Floertheoreticobstructionconj} we heuristically argued that even without these extra hypotheses, the Floer theoretic obstruction should follow from the Thomas-Yau-Joyce program.

From a different perspective, we can add an arbitrarily large \emph{positive} number $a>0$ to the Lagrangian potential on $L_{k+1},\ldots L_N$. This is compatible with the Novikov positivity condition, so $L$ stays unobstructed, but $\bar{S}$ changes by an unbounded amount
\[
-a \text{Im}\left(  e^{-i\hat{\theta}}\int_{\mathcal{E}_i} \Omega    \right)<0. 
\]
We conclude that \emph{in the unstable case, the elementary functional is unbounded from below}.

\item In the \textbf{semistable} case, for any $L$ in the class $\mathcal{L}$ that can be written in the twisted complex form as above, we always have 
\begin{equation}\label{semistabilitySL}
\text{Im}\left(  e^{-i\hat{\theta}}\int_{\mathcal{E}_k} \Omega    \right) \leq 0,\quad \forall k=1,2,\ldots N-1.
\end{equation}
Then the elementary functional (\ref{Solomonelementary2}) is \textbf{nonnegative}.

In section \ref{Quantitativealmostcalibrated} we argued that since the homology class of $L$ is prescribed a priori, subject to the quantitative almost calibrated assumption, only finitely many possibilities of homology classes can arise for $L_i$ in any decomposition. Thus the stronger condition
\[
\text{Im}\left(  e^{-i\hat{\theta}}\int_{\mathcal{E}_k} \Omega    \right)<  0,\quad \forall k=1,2,\ldots N-1.
\]
would be equivalent to a uniform bound: for some small $c>0$,
\[
\text{Im}\left(  e^{-i\hat{\theta}}\int_{\mathcal{E}_k} \Omega    \right)\leq -c< 0,\quad \forall k=1,2,\ldots N-1.
\]
This holds when the class $\mathcal{L}$ is \textbf{stricly stable} (\cf Definition \ref{ThomasYausemistability}). Together with potential clustering, it implies
\[
\bar{\mathcal{S}}(L)\geq c (\sup_{L} f_L- \inf_L f_L+A).
\]
Thus if the Lagrangian potential oscillation becomes unbounded, then the elementary functional goes to positive infinity. The geometric intuition is the properness of the Solomon functional modulo a global additive constant for $f_L$.

\end{itemize}

Since the Solomon functional and the elementary functional only differ by a bounded amount, the above conclusions transfer to the Solomon functional. 
Thus \textbf{the Solomon functional is bounded below in the semistable case, and unbounded from below in the unstable case}. 
A key slogan here is that \emph{the asymptotic behaviour of the Solomon functional is governed by Floer theory}. This is analogous to the partially conjectural picture in the variational approach to the HYM equation, where the asymptotic behaviour of the Donaldson functional is governed by algebraic geometry (\cf section \ref{HYM}).

\begin{rmk}\label{ThomasYausemistabilityequivalence}
In Definition \ref{ThomasYausemistability}, the Thomas-Yau semistability makes use of distinguished triangles for all almost calibrated Lagrangian objects, not just those with $|\theta|\leq \frac{\pi}{2}-\epsilon$. This makes the Thomas-Yau semistability a priori stronger than the semistable situation of the above dichotomy. We expect from the Thomas-Yau-Joyce picture that both stability notions are actually equivalent under our initial assumption that there is a representative $L_0$ with $|\theta|< \frac{\pi}{2}-\epsilon$. But for our main purpose, that Thomas-Yau semistability implies the existence of special Lagrangians, we do not mind Thomas-Yau semistability being stronger than necessary.
\end{rmk}

\subsubsection{Thomas-Yau conjecture}\label{ThomasYauconjecturesubsection}

The following is our interpretation of the \textbf{Thomas-Yau existence conjecture}:

\begin{conj}\label{ThomasYauexistence}
Let $L_0$ be an exact, quantiatively almost calibrated, unobstructed Lagrangian object in $\mathcal{L}$.
Assuming Thomas-Yau semistability for $L_0$, then the following (equivalent) statements hold:
\begin{enumerate}
	\item There is a special Lagrangian representative in $\mathcal{L}$.	
	
	\item There is no distinguished triangle in $\mathcal{L}$ satisfying the destabilizing condition.
	
	\item The Solomon functional is bounded from below on $\mathcal{L}$.
	
	\item The Solomon functional has a minimizer in $\mathcal{L}$.	
\end{enumerate}
\end{conj}

Here is a glossary of the evidence presented previously. 
\begin{itemize}
\item $(2)$ is tautological from Thomas-Yau semistability (\cf Remark \ref{ThomasYausemistabilityequivalence}).
\item $(1)$ implies Thomas-Yau semistability: see the Floer theoretic obstructions Thm. \ref{Floertheoreticobstruction1}, Thm. \ref{Floertheoreticobstruction2}, Conj. \ref{Floertheoreticobstructionconj}, where we justified this for immersed Lagrangians under the automatic transversality and the positivity condition, or alternatively by assuming Joyce's program.

\item $(2)\iff (3)$: this is a consequence of $(N,A)$-potential clustering (\cf Cor. \ref{potentialoscillationrobust}, section \ref{Intrinsicdistancepotentialclustering}), the uniform bound for $\mathcal{S}-\bar{\mathcal{S}}$ (\cf section \ref{Solomonboundedpartrevisited}, \ref{Solomonboundedpartsection}), and the formula (\ref{Solomonelementary2}) for the elementary functional. 

\item $(1)\implies (4)$: see Prop. \ref{specialLagminimizer}, where we justified this under automatic transversality and the positivity condition.

\item $(4)\implies (3)$: obvious.

\end{itemize}

	




The rest of this section concerns $(2)\implies (4)$, and the next section concerns $(4)\implies (1)$. The arguments will rely on several unproven statements, which we consider plausible, but may involve rather significant difficulties or  substantial foundational work. Nevertheless, we think it is instructive to see heuristically how everything fits together.

\begin{conj}
In the semistable case, the Solomon functional has a minimizer.
\end{conj}

\begin{proof}
(Heuristic)  
First, we claim that for a minimizing sequence $L^{(k)}$ of the Solomon functional, without loss of generality the Lagrangian potential $f_L^{(k)}$ is a priori bounded:
\begin{equation}\label{potentialbound}
\sup_k\norm{f_L^{(k)}}_{L^\infty}\leq C.
\end{equation}
Consider the potential clustering setup. We can adjust the Lagrangian potentials on $L_i$ by constants separately, and as long as $\sup_{L_j}f_{L_j}\leq \inf_{L_i}f_{L_i}$ for $j<i$, this process will not affect the Novikov positivity requirement, so the Lagrangian branes should remain in $\mathcal{L}$. We view $\sup_{L_1}f_{L_1}, \sup_{L_2}f_{L_2}- \sup_{L_1}f_{L_1}, \ldots, \sup_{L_N}f_{L_N}-\sup_{L_{N-1}} f_{L_{N-1}}$ as independent constants. Adjusting all potentials by a common constant does not affect the Solomon functional, but allows us to set $\sup_{L_1}f_{L_1}=0$. Decreasing $\sup_{L_i}f_{L_{i}}- \sup_{L_{i-1}}f_{L_{i-1}}$ subject to the Novikov positivity requirement will \emph{decrease} the elementary functional (\ref{Solomonelementary2}), crucially because of the \emph{semistability condition} (\ref{semistabilitySL}). The part $\mathcal{S}-\bar{\mathcal{S}}$ is unchanged. 
Thus after this adjustment, the sequence is still minimizing for the Solomon functional. 
We can thus achieve $\sup_{L_{i-1}}f_{L_{i-1}}= \inf_{L_i}f_{L_i}$ for all $i$. By the potential clustering property, we then have (\ref{potentialbound}).

Next we need the \textbf{compactness} from geometric measure theory. As discussed in section \ref{Compactnessregularity} and \ref{Quantitativealmostcalibrated}, under quantitative almost calibratedness there is an a priori volume bound, and the Lagrangians all remain in a fixed bounded subset of $X$, 
 so Federer-Fleming compactness (\cf Theorem \ref{Federer}) holds automatically. The uniform potential bound (\ref{potentialbound}) would then justify that the weak limit is an almost calibrated Lagrangian current $L$ with bounded potential $f_L$ (\cf Lemma \ref{limitexactness}). The continuity of the Solomon functional (\cf Lemma \ref{Solomoncontinuity}) then shows $\mathcal{S}(L)=\inf_{\mathcal{L}} \mathcal{S}$.

In section \ref{Variationalstrategy} we presented the evidence for the conjectural $L^2$-smoothing property, which would allow us to assume a uniform a priori bound on the minimizing sequence
\[
\int_L |\vec{H}|\leq C.
\]
so we can use Allard compactness theorem \ref{Allard}. In effect, we can assume the minimizing sequence converges subsequentially both as currents and as varifolds. 
By assumption the class $\mathcal{L}$ is closed under the varifold/current topology of the Lagrangian, so the limit $L$ lies in $\mathcal{L}$, whence provides a \textbf{minimizer} in $\mathcal{L}$. 
\end{proof}

\begin{rmk}
If we demand $\mathcal{L}$ is closed under the flat topology of currents, without requiring varifold convergence, then we would not need the difficult $L^2$-smoothing property in the argument. However, this would allow the pathological behaviour in Example \ref{snowflakeexample}, which would increase the difficulty of Floer theory for weak regularity Lagrangians.
\end{rmk}

\begin{rmk}\label{Sequivalence}
For the geometric measure theoretic purpose of finding special Lagrangians, the existence of a minimizer as a \emph{Lagrangian current} $L$ is probably sufficient. However, for applications to the Fukaya category, it is highly desirable to know that $L$ carries a \emph{formal brane structure} (\cf section \ref{Floertheoryweakregularity}), which likely requires resolving Question \ref{formalbranecompactness}. Some analogies suggest the question may be subtle:
\begin{itemize}
\item In geometric invariant theory (GIT), there are niceties concerning semistable, polystable and stable objects. If we take a sequence of semistable objects in a fixed reductive group orbit, the limit may jump outside the orbit, so that the orbit does not admit a polystable representative. Several semistable orbits may be `$S$-equivalent', and each $S$-equivalence class contains a unique polystable orbit.

\item  In the gauge theory of holomorphic bundles, likewise a sequence of connections in the same complexified gauge orbit may jump outside the orbit in the limit; algebro-geometrically, this jumping of bundle structure is usually related to bundle extensions. 

\item One motivation for the Thomas-Yau program is to form the moduli space of (semi)stable Lagrangian branes. The Hausdorff property of the moduli space is a delicate question.
\end{itemize}

For these reasons, as well as Remark \ref{Sequivalencermk1},  we are not certain if the Lagrangian minimizer should be interpreted as a representative in the chosen $D^bFuk(X)$ class, or if several semistable $D^bFuk(X)$ classes should be identified under some suitable $S$-equivalence relation. We think this question requires further developments in Floer theory. The question is also reflected in the delicacy of the infinite time limit in Joyce's Bridgeland stability proposal.

\end{rmk}

\subsection{Minimizers and special Lagrangians}\label{MinimizerspecialLag}

\begin{conj}\label{specialLagminimizerconj}
A minimizer $L$ of the Solomon functional inside $\mathcal{L}$ is a special Lagrangian of phase $\hat{\theta}$.
\end{conj}

We will give several heuristic reasons. The essential issue is that there should be enough Lagrangian competitors within the class $\mathcal{L}$.

\subsubsection*{LMCF viewpoint}

In section \ref{Variationalstrategy} we discussed that the Solomon functional should be non-increasing under Joyce's LMCF. Suppose the flow extends weakly to Lagrangians in $\mathcal{L}$. The flow starting from a minimizer $L$ must have constant $\mathcal{S}(L_t)$, but the evolution (\ref{Solomonfunctionalevolution}) would then force $\theta=\hat{\theta}$, namely $L$ is a special Lagrangian, and the flow is in fact constant.

\subsubsection*{Hamiltonian variations}

If the Lagrangian angle of the minimizer satisfies $-\pi/2+\epsilon<\inf_L \theta \leq \sup_L \theta<\pi/2-\epsilon$, then we have a more elliptic argument. Given any compactly supported global $C^\infty$ Hamiltonian function $H$ on $X$, we can associate a 1-parameter family of symplectomorphisms $\phi_t$ by exponentiating the Hamiltonian vector field. Since $d\phi_t$ only moves the tangent planes by $O(|t|)$ for small $|t|\ll 1$, the Lagrangian angle of $\phi_t(L)$ is still within $(-\pi/2+\epsilon, \pi/2-\epsilon)$, namely the quantitatively almost calibrated condition is preserved.

Under global Hamiltonian deformations, the first variation of the Solomon functional is 
\[
\delta S(H)= \frac{d}{dt} \mathcal{S}( \phi_t(L)  )|_{t=0}= \int_L H \text{Im}(e^{-i\hat{\theta}}\Omega).
\]
We need another ingredient which is expected to hold once the Floer theory is sufficiently developed in the weak regularity setting:
\begin{itemize}
\item
The class of \emph{unobstructed} exact Lagrangian objects is preserved by Hamiltonian isotopies. As such $\phi_t(L)$ should remain inside the class $\mathcal{L}$.
\end{itemize}

These would imply that the minimizer $L$ satisfies
\[
\int_L H \text{Im}(e^{-i\hat{\theta}}\Omega)=0.
\]
for any compactly supported $C^\infty$ function on $X$. This means $\text{Im}(e^{-i\hat{\theta}}\Omega)=0$ as currents, which is equivalent to $\theta=\hat{\theta}$ under the almost calibrated setting.

\begin{rmk}
The assumption that $-\pi/2+\epsilon<\inf_L \theta \leq \sup_L \theta<\pi/2-\epsilon$ for the minimizer is not innocent, but represents a principal gap in our program to find special Lagrangian currents. The problem is that if on the minimizer $\sup_L \theta_L=\frac{\pi}{2}-\epsilon$, and a priori $L$ has no regularity assumption (eg. the Lagrangian angle may a priori be highly oscillatory), then we lack techniques to construct Lagrangian competitors which remain quantitatively almost calibrated.

\end{rmk}

\subsection{Thomas-Yau uniqueness revisited}

The Thomas-Yau uniqueness argument has a conceptually rather mysterious aspect: from local computations of Floer degrees, one arrives at the global conclusion that the two special Lagrangians share the same support. We shall now present a different argument, which is not completely rigorous, but unlike the standard arguments, it could potentially work on Lagrangians with mild singularities.

\begin{conj}\label{ThomasYauuniquenessweak}
	(\textbf{Thomas-Yau uniqueness in the weak setting})
	Suppose $L,L'$ are two special Lagrangian integral currents $L,L'$ with the same phase angle $\hat{\theta}$, equipped with suitable unobstructed brane structures, such that $L\simeq L'$ in $D^bFuk(X)$. Then $L=L'$ as currents.
\end{conj}

\begin{proof}
(Heuristic) In general, we expect there is an $(n+1)$-dimensional rectifiable current $\mathcal{C}$ with $\partial \mathcal{C}=L-L'$ constructed from universal families of holomorphic curves with boundary on $L$ and $L'$. 
The holomorphic curves $u:\Sigma\to X$ can appear in three types:
\begin{itemize}
	\item Automatically transverse holomorphic curves: there exist first order deformations $v_1,\ldots, v_{n-1}$ such that $dF=\Omega(\cdot,v_1,\ldots v_{n-1})$ does not vanish identically as a 1-form on $\Sigma$ (\cf section \ref{Automatictransversalitypositivity}). 
	
	\item Nonconstant holomorphic curves, which are not automatically transverse. We expect their boundary evaluation to be contained in a Hausdorff dimension $\leq n-1$ subset of $\text{supp}(L)\cup \text{supp}(L')$ (\cf section \ref{Automatictransversalitypositivity}).
	
	\item Constant holomorphic maps $u:\Sigma\to \text{supp}(L)\cap \text{supp}(L')$. These would only arise if $L$ and $L'$ have some overlapping support, so did not appear in our previous discussions. For dimensional reasons, these cannot contribute to the $(n+1)$-dimensional current $\mathcal{C}$.

	The key difference from the second case is that at interior points of $\text{supp}(L)\cap \text{supp}(L')$, there are $n$ linearly independent first order deformations, such that $v_1,\ldots v_n$ span $TL$ upon boundary evaluation. This behaviour can only be compatible with $dF=0$ for constant curves.
\end{itemize}

We now impose the special Lagrangian condition, and consider the automatically transverse case.
Along $\partial \Sigma$, the counterclockwise directional derivative of $F$ has argument equal to the constant Lagrangian angle $\hat{\theta}$ modulo $\pi\Z$. As such we expect $F(\partial \Sigma)$ to be contained in a line segment with incline angle $\hat{\theta}$. By the maximum principle on the holomorphic function $F$, the entire $F(\Sigma)\subset \C$ is contained in a line segment. However, the open mapping theorem in complex analysis then implies $F$ is constant, which \emph{rules out the automatically transverse curves}.

Now the only contributions to $\mathcal{C}$ would come from the nonconstant, not automatically transverse curves. This forces $\text{supp}(\partial \mathcal{C})\cap  (\text{supp}(L)\cup\text{supp}(L'))$ to be contained in a Hausdorff $(n-1)$-dimensional subset. However $\partial \mathcal{C}=L-L'$ as integral currents, so the $n$-dimensional current $L-L'$ has support dimension $\leq n-1$, which forces it to vanish. This shows $L=L'$. 
\end{proof}

\begin{Question}
When can we say furthermore that the formal brane structures on $L=L'$ are related by some gauge equivalence?
\end{Question}

\subsubsection{Special Lagrangians are minimizers}\label{specialLagminimizerrevisited}

We now revisit Prop. \ref{specialLagminimizer}. Our goal is to suggest that the automatic transversality, positivity condition, and even smoothness assumptions can be removed in Prop. \ref{specialLagminimizer}, at the cost of assuming the entire force of the Thomas-Yau conjecture, under the setting of this chapter.

\begin{conj}
If there exists a special Lagrangian $L$ in the class $\mathcal{L}$, then it is a minimizer of the Solomon functional. 
\end{conj}

\begin{proof}
(Heuristic)
The existence of a special Lagrangian representative should imply Thomas-Yau semistability (\cf Conjecture \ref{Floertheoreticobstructionconj}). By the Thomas-Yau existence conjecture \ref{ThomasYauexistence} this implies the Solomon functional has a minimizer $L'$, which must be a special Lagrangian. Then the Thomas-Yau uniqueness conjecture \ref{ThomasYauuniquenessweak} implies $L=L'$ as currents.
\end{proof}

\subsection{Comparison with Joyce's LMCF program}

We have already made extensive comparisons between the variational approach and Joyce's LMCF program, but it may help to summarize a few highlights.

\begin{itemize}
\item  Joyce's program is much more ambitious in that it tackles the entire derived Fukaya category, not just the almost calibrated Lagrangians. We feel the quantitative almost calibratedness is so pervasively used in the variational approach that it cannot be removed. Dropping the almost calibratedness will give rise to significantly more difficulties in Joyce's program: the collapsing of zero objects can then happen, and the Solomon functional no longer needs to decrease. Neves's example of finite time singularity \cite{Neves2} is a concrete manifestation of the difficulty. The almost calibrated condition is also natural from the viewpoint of the continuity method (\cf section \ref{Continuitymethod}), which deals with special Lagrangians inside varying ambient almost Calabi-Yau structures.

\item  Joyce does not specify the Bridgeland stability in a priori Floer theoretic terms. An a priori guess on the nature of the stability condition is central to the variational method. Even though our picture is largely conjectural, it seems to be the most precise description hitherto of how stability condition comes into the existence questions of special Lagrangians.

\item  Joyce primarily focuses on compact Calabi-Yaus, and mentions the exact case only as an easier analogue. 
We have focused on the exact case, although we feel some parts of our picture may extend to compact Calabi-Yaus, if one is prepared to overcome (even more) significant Floer theoretic technical hurdles. However, we do not know what would replace the a priori estimates on the Lagrangian potentials, and notably the potential clustering condition.

\item Joyce's LMCF involves objects with a priori higher regularity, even though its infinite time convergence behaviour may well require understanding weak regularity Lagrangians. The variational method requires working with varifold/current like objects throughout.

\item  Joyce's LMCF needs to make essential use of genericity conditions. This in particular requires extremely precise classification of all possible generic singularities in order to perform surgeries, a task that becomes overwhelmingly difficult for complex dimension $\geq 3$. Our variational program is less sensitive to such arguments. On the other hand, we still potentially need to understand some generic singularities, so that the class $\mathcal{L}$ contains enough competitors, to enable the proof of the $L^p$-smoothing property for some $p\geq 1$, and Conjecture \ref{specialLagminimizerconj}.


\item Although time and again we appealed to Joyce's LMCF to heuristically justify certain claims, it is only because we lack other ways of constructing Lagrangian competitors with sufficient control, and the basic logical framework of the variational approach is independent of the LMCF. It seems desirable (on account of the extraordinary difficulty of Joyce's program) to keep this logical independence manifest in the program to rigorize our variational proposal.

\item Joyce's program has a number of highly nontrivial categorical predictions discussed in section \ref{TowardsBridgeland}, such as the idempotent closedness of $D^b Fuk(X)$. Even if these predictions turn out to be false, it would not affect the validity of the variational method.


\end{itemize}

\section{Appendix on the Fukaya category}

This appendix is a brief reminder about Floer theory in the exact setting. There exist both excellent surveys on the  Fukaya category of embedded Lagrangians, such as Auroux \cite{Aurouxsurvey} and Smith \cite{Smithsurvey}, and many in depth treatments such as Seidel \cite{Seidelbook}, Akaho-Joyce \cite{JoyceAkaho} and FOOO \cite{FOOO}. Our very limited goal is to recall some key notions prevalent in the main text, and explain some basic intuitions, but we will not get into the more technical aspects, such as the details of perturbation schemes, which are treated carefully in these standard references.

For the Thomas-Yau-Joyce program, one also needs to incorporate immersed Lagrangians. The canonical reference is Akaho-Joyce \cite{JoyceAkaho} for a treatment using virtual techniques, and Woodward et al \cite{Woodward1}\cite{Woodward2} which avoids virtual counting by using stabilising divisors. The exact assumption affords some technical simplifications, for which a sketchy account is found in \cite[section 4.1]{Joyceconj}. Another technical treatment in the exact setting, not allowing certain teardrop curves, is in Alston-Bao \cite{AlstonBao}.

\subsection{Fukaya category for embedded exact Lagrangians}\label{Fukayaembedded}

\subsubsection*{Floer cohomology and $A_\infty$-structure with mod 2 coefficients}

Let $(X,\omega,J)$ be a Stein manifold, namely a K\"ahler manifold with $\omega=\sqrt{-1}\partial \bar{\partial} \phi$ for a plurisubharmonic exhaustion function $\phi$. In particular, $X$ is an exact symplectic manifold, meaning $\omega=d\lambda$, where $\lambda$ is the Liouville 1-form. All almost complex structure perturbations are assumed to agree with the fixed complex structure outside some compact set.


 Given two transversely intersecting exact embedded\footnote{In our terminology, embedded Lagrangians are always connected, while immersed Lagrangians can have disconnected domains.} Lagrangians $L, L'$ with potential $f_L, f_{L'}$, namely $df_L=\lambda|_L$ and $df_{L'}=\lambda|_{L'}$, and some extra brane data, one can associate an algebraic invariant called the \textbf{Floer cohomology}.
A general feature of Floer theory, is that the constructions depend on many auxiliary choices, but the invariants depend on only a small number of data, and should always be invariant under global Hamiltonian isotopies.

We assume $c_1(TX)=0$ and let $\Omega$ be a complex volume form on $TX$.
We shall always assume the Lagrangians to be \textbf{graded}, namely the phase function $\theta=\arg \Omega|_L: L\to S^1$ lifts to a real valued function. The grading is part of the brane data. 
Working first with $\Z_2$ coefficients, the Floer cohomology can be defined as the cohomology of a complex $(CF^*(L,L'), d)$. Here $CF^*(L,L')$ is generated by the transverse intersection points  $p\in L\cap L'$, whose \textbf{degrees} are given by
\begin{equation}\label{Floerdegree}
\mu_{L,L'}(p)= \frac{1}{\pi}( \sum \phi_i +\theta_L(p)- \theta_{L'}(p)   ) 
\end{equation}
where we put the tangent planes of $L,L'$ inside $T_pX\simeq \C^n$ into the standard form
\[
T_p L= \R^n \subset \C^n, \quad T_p L'= (e^{i\phi_1}, \ldots e^{i\phi_n})\R^n\subset \C^n, \quad 0<\phi_i<\pi.
\]
Notice if we reverse the role of $L, L'$, then we can regard $p\in CF^*(L', L)$, but this affects the degree by $\mu_{L,L'}(p)= n- \mu_{L',L}(p)$.
For alternative formulations of the degree in terms of Lagrangian Grassmannians, see \cite{Seideldegree}.

\begin{rmk}
The degree convention $\mu_{L,L'}$ here follows Joyce \cite{Joyceconj}, which corresponds to $\mu_{L',L}$ in  \cite{Seidelbook}\cite{Aurouxsurvey}\cite{Smithsurvey}. The advantage of this convention is its compatibility with the central charge formula $Z(L)=\int_L\Omega$ in the Bridgeland stability. If instead one uses the convention of \cite{Seidelbook}\cite{Aurouxsurvey}\cite{Smithsurvey}, then adding $\pi$ to the Lagrangian phase would correspond to the shift $[-1]$ in $D^bFuk(X)$, so the central charge would be $Z(L)=\overline{\int_L\Omega}$.
\end{rmk}

\begin{rmk}\label{degreealmostcalibrated}
For almost calibrated Lagrangians $-\pi<\theta_L-\theta_{L'}< \pi$, whence $-1< \mu_{L,L'}(p)< n+1$. Since the degrees are always integers, we must have $0\leq \mu_{L,L'}(p)\leq n$.  
\end{rmk}

We consider the moduli space $\mathcal{M}(p,q; J,[u])$ of finite energy holomorphic strips with ends at $p,q$ and boundary on $L,L'$, in the homotopy class $[u]\in \pi_2(X, L\cup L')$:
\[
\begin{split}
& u: \Sigma=\R\times [0,1]\to X,  \quad u(s,0)\in L, \quad u(s,1)\in L', \quad \lim_{s\to -\infty} u=p, \quad \lim_{s\to +\infty} u=q,
\\
& \partial_s u+ J(t,u) \partial_t u=0,\quad E(u)=\int u^*\omega=\iint |\partial u/\partial s|^2 dsdt<\infty.
\end{split}
\]
Index theory of the Cauchy-Riemann operator with Lagrangian boundary conditions implies this moduli space has virtual dimension $\deg q-\deg p$. The holomorphic strip equation is invariant under domain translation in the $\R$ direction. Using  generic domain dependent almost complex structures which are fixed outside a large compact set, one can achieve suitable \textbf{transversality} on the moduli spaces, and in particular $\mathcal{M}(p,q; J,[u])/\R$ are isolated points for $\deg q-\deg p=1$. A key advantage of the exact setting is that the energy can be computed a priori by the topological formula:
\begin{equation}\label{topologicalenergy}
\int_\Sigma u^*\omega= \int_{\partial\Sigma} \lambda= \int_{-\infty}^\infty df_L - \int_{-\infty}^\infty df_{L'}= (f_L- f_{L'})(q)- (f_L-f_{L'})(p).
\end{equation}
By \textbf{Gromov compactness}, the number of isolated points is finite, and only finitely many homotopy classes $[u]$ admit holomorphic strips. To save some notations, we sometimes write $\mathcal{M}(p,q)=\bigcup_{[u]\in \pi_2(M,L\cup L')} \mathcal{M}(p,q; J,[u])$.

\begin{rmk}
The role of convexity assumptions at the infinity of $X$ (such as the existence of a plurisubharmonic exhaustion function) is to ensure that for a finite given collection of Lagrangians, all holomorphic curves remain inside a fixed bounded region. This is needed to apply Gromov compactness.
\end{rmk}

\begin{rmk}
More generally, one can add a Hamiltonian term in the Cauchy-Riemann equation, and replace transverse intersection points by Hamiltonian chords. The Cauchy-Riemann equation then gets modified to the \textbf{Floer equation}
\begin{equation}\label{Floereqn}
\partial_s u+ J(t,u) (\partial_t u- X_H)=0,
\end{equation}
where $X_H$ is a Hamiltonian vector field. This perturbation is not needed for Floer theoretic transversality statements if $L$ and $L'$ are already transverse, but is an essential ingredient in showing the Hamiltonian invariance of Floer cohomology.
\end{rmk}

\begin{rmk}
We generally distinguish between the $s\to -\infty$ end, and the $s\to +\infty$ end. The main difference is the ordering of the Lagrangians at the intersection point. Here we are following the Joyce convention \cite{Joyceconj}, which is opposite to Auroux \cite{Aurouxsurvey}. This is dictated by compatibility with the degree formula (\ref{Floerdegree}). Similarly, later the $A_\infty$ product also requires the Lagrangian boundaries to be ordered clockwise, as opposed to the counterclockwise convention in Auroux \cite{Aurouxsurvey}.
\end{rmk}

The \textbf{Floer differential} $d: CF^k(L,L')\to CF^{k+1}(L,L')$ is 
$
d p= \sum n_{p,q} q,
$
where $n_{p,q}$ is the mod 2 count of
$\mathcal{M}(p,q)/\R$. The key fact of Floer theory is that $d^2=0$. For this, one considers the holomorphic strips between $p,r$ with $\deg r-\deg p=2$, modulo the translation invariance $\R$ direction. This moduli space $\mathcal{M}(p,r)/\R$ is one-dimensional. Generally in Floer theory, the boundary of the compactified moduli spaces comes from \textbf{disc breaking} and \textbf{disc and sphere bubbling}. The latter is ruled out for energy reasons by the exactness assumption, and the former gives
\[
\partial (\overline{ \mathcal{M}(p,r)/\R})= \bigcup_q \mathcal{M}(p,q)/\R \times \mathcal{M}(q,r)/\R.
\]
In terms of mod 2 counts, $\sum_q n_{p,q} n_{q,r}=0$, namely $d^2=0$. This fact allows one to take the cohomology, which is $HF^*(L,L')$. Although suppressed in this notation, the homotopy classes of discs are additive under disc breaking. This fact allows one to introduce some extra weighting factors involving energy and holonomy of local systems.

The general strategy to show the Floer cohomology is independent of the choices of almost complex structures and Hamiltonian perturbations, is to consider \textbf{continuity equations}, whose counts define chain maps at the level of $CF^*$, so descend to comparison maps between Floer cohomologies defined by different auxiliary data (\cf Auroux \cite[section 1.5]{Auroux}).

Floer cohomology admits rich algebraic structures, but the deeper structure is better set up at the chain level $CF^*(L,L')$. We temporarily avoid the issue of signs and self Floer cohomology.  
The Fukaya category can be seen as the generalization of Floer cohomology in two directions:
\begin{itemize}
\item We allow the interplay of many (transverse) Lagrangians. Each Lagrangian is labelled by an object in the Fukaya category. This labelling is the main difference between an algebra and a category.

\item The holomorphic strips are replaced by \textbf{holomorphic polygons}, with
boundary segments mapped to a clockwise ordered sequence of at least three Lagrangians $L_0, L_1,\ldots L_k$, and clockwise ordered boundary marked points $x_0, x_1,\ldots x_k$ mapped to the Lagrangian intersection points $q\in L_0\cap L_k, p_1\in L_0\cap L_1,\ldots p_k\in L_{k-1}\cap L_k$. We distinguish $x_0$ as the output, and regard $x_1,\ldots x_k$ as inputs. The marked points $x_0,x_1,x_2$ on the boundary of the domain disc are 
fixed, while the other $k-2$ marked points are
allowed to move freely preserving their cyclic ordering.

The  moduli spaces of such polygons (with suitably domain dependent perturbations) are denoted as $\mathcal{M}(p_1,\ldots p_k, q)$.
Moduli spaces with at least three marked points do not have the domain translation invariance, so there is no need to divide by $\R$.
\end{itemize} 

\begin{rmk}
From the viewpoint of gluing theory, it is convenient to regard the boundary marked points of the holomorphic polygons as punctures, where the Riemann surface structure is locally modelled on
strip like ends. The moduli of abstract holomorphic polygons with $k+1$ marked points has a compactification known as the Stasheff associahedron $\overline{\mathcal{R}}_{k+1}$. On account of the geometric picture of polygons in $\R^2$, we often refer to the strip like ends as corners.
\end{rmk}

The energy formula (\ref{topologicalenergy}) generalizes to the holomorphic polygon case:
\begin{equation}\label{topologicalenergy2}
\int_{\Sigma} u^*\omega=(f_{L_0}-f_{L_k})(q)- \sum_1^k (f_{L_{i-1}}-f_{L_{i}})(p_i) .
\end{equation}
The virtual dimension formula is
\begin{equation}
\text{vdim} \mathcal{M}(p_1,\ldots p_k, q)= \deg q-\sum_1^k \deg p_i +k-2.
\end{equation}
Here $\deg q-\sum_1^k \deg p_i$ comes from the index theory of the Cauchy-Riemann operator, and $k-2$ comes from the freedom to move the marked points on the boundary. Under suitable domain dependent perturbation schemes, in this exact setting one can ensure transversality, so that the moduli space is smooth. For setting up the Fukaya category, the zero dimensional moduli spaces are particularly important, since counting points give rise to operations, and 1-dimensional moduli spaces are important for producing $A_\infty$-relations. In the main text, we have also given considerable attention to $(n-1)$-dimensional moduli spaces, since these are relevant for producing bordism currents.

Within the exact setting, disc and sphere bubbling is impossible. After compactification,
the moduli space of holomorphic polygons can have two kinds of boundaries, due to two kinds of disc breaking:
\begin{itemize}
\item (Disc breaking at the corners) The disc may break at $L_i\cap L_{i+1}$. The polygons near the breaking limit are obtained from gluing polygons with corners mapped to $p_1, \ldots, p_{l-1}, r, p_{l+1},\ldots, q$, and strips with boundary on $L_{l-1}, L_l$ and two ends mapped to $p_l, r$. (Of course, disc breaking can also happen at the outgoing corner $q$.)

\item (Disc splitting at the edges) 
When there are at least 4 Lagrangians, the domain disc can split into two discs with $l+1\geq 3$ and $k+1-l\geq  3$ marked points. The edges of one disc map to $L_0, \ldots L_j, L_{j+l}, \ldots L_k$, with cyclically marked points mapping to $p_1,\ldots p_j, r\in L_j\cap L_{l+j}, p_{l+j+1},\ldots p_k$ and $q$.  The edges of the other disc map to $L_j, \ldots L_{l+j}$, with marked points mapping to $p_{j+1},\ldots p_{l+j}$ and $r\in L_j\cap L_{l+j}$.


\end{itemize}

When disc breaking and disc splitting are taken into account, the moduli spaces can be compactified into $\overline{\mathcal{M}}(p_1,\ldots p_k,q)$. We then have
\begin{equation}\label{moduliboundary}
\begin{split}
\partial \overline{\mathcal{M}}(p_1,\ldots p_k,q) =&\bigcup \overline{\mathcal{M}(p_l,r)/\R}\times \overline{\mathcal{M}}(p_1,\ldots p_{l-1}, r, p_{l+1}, \ldots p_k, q)\\ & \cup \bigcup \overline{\mathcal{M}}(p_1,\ldots, p_j, r, p_{l+j+1} ,\ldots p_k, q)\times \overline{\mathcal{M}}(p_{j+1}, \ldots p_{l+j}, r).
\end{split}
\end{equation}
In particular, the (virtual) dimensions of both sides are equal, which constrains $\deg r$.

\begin{rmk}
More generally, disc breaking and disc splitting can happen in a bubble tree fashion. Such multiple splitting/breaking do not concern us, because under sufficient transversality conditions, they occur only in codimension at least two  in the moduli space. To set up Fukaya categories in the exact setting, only zero and one dimensional moduli spaces are needed, so the multiple bubble trees do not occur. When we make use of higher dimensional moduli spaces in the main text, the bubble trees do occur, but the codimension two condition means the deeper boundary strata do not contribute to the boundary of the bordism current, in the sense of currents. 	
\end{rmk}

The \textbf{$A_\infty$-structure} is the algebraization of the disc breaking/splitting phenomenon. It consists of multilinear maps
\[
m_k: CF^*(L_{k-1}, L_k)\otimes\ldots  CF^*(L_1, L_2)\otimes CF^*(L_0, L_1)\to CF^*(L_0, L_k)[ 2-k ]
\]
satisfying the $A_\infty$-relation
\[
\sum_{l=1}^k \sum_{j=0}^{k-l} \pm m_{k+1-l} (p_k, \ldots , p_{j+l+1}, m_l(p_{j+l}, \ldots p_{j+1}), p_j,\ldots p_1)=0.
\]
Here $m_1$ is the degree one Floer differential $d$, and for $k\geq 2$ the operation $m_k$ is defined by counting holomorphic polygons in moduli spaces of virtual dimension zero,
\[
m_k(p_k,\ldots, p_1)=\sum_q \# \mathcal{M}(p_1,\ldots p_k, q) q.
\]
Virtual dimension zero requires $\deg q=\sum_1^k \deg p_i+ 2-k$, which explains the degree of $m_k$. The $A_\infty$-relation is the direct translation of (\ref{moduliboundary}), with disc breaking at corners contributing the $m_1$ terms, and disc splitting contributing the other terms.

The first few $A_\infty$-relations have clear geometric meanings:
\begin{itemize}
\item The Floer differential squares to zero.

\item The Floer product $m_2$ satisfies the Leibniz rule. As such $m_2$ descends to a product structure on the mod 2 coefficient Floer cohomology $HF^*(L_1,L_2)\otimes HF^*(L_0,L_1)\to HF^*(L_0, L_2)$.

\item The $m_2$ is associative up to a homotopy given by the $m_3$ terms. In particular the Floer product is associative on cohomology.
\end{itemize}

The higher structures naturally lead to the Fukaya category of embedded exact Lagrangians. This requires some discussion on signs, brane structures, and self Floer cohomologies.

\subsubsection*{Self Floer cohomology}

It is desirable to take Floer cohomology of $L$ with itself. 
One major feature of $HF^*(L,L)$ is that it contains \textbf{units}, at least at cohomological level.

One challenge to implement \textbf{self Floer cohomology} is that $L$ is not transverse to itself, so the Cauchy-Riemann equation needs perturbation. There are many frameworks to address this problem, and one idea dating back to Floer is to use the Hamiltonian invariance of Floer cohomology, to think of self Floer cohomology via $HF^*(L,L)\simeq HF^*(L,\phi_{\epsilon H}(L))$ where $\phi_{\epsilon H}$ is the time one flow of the small generic Hamiltonian $\epsilon H$ \cite[section 1.6]{Aurouxsurvey}. For $\epsilon\ll 1$, the Lagrangian $\phi_{\epsilon H}(L)$ can be identified as a graph over $L$ inside $T^*L$, the transverse intersection $L\cap \phi_{\epsilon H}(L)$ are the critical points of $H|_L$, and  a suitable setup of the Floer trajectories (\ref{Floereqn}) can be identified as Morse flowlines of $H$. Thus $HF^*(L,L)$ is isomorphic to the Morse cohomology of $L$, so $HF^*(L,L)\simeq H^*(L)$. In the exact case, the ring structure on $HF^*(L,L)$ defined from perturbed holomorphic triangles agrees with the cup product ring structure on $H^*(L)$. 
The unit can be represented by the Morse generator of $H^0(L)$, or more non-perturbatively via the Piunikhin-Salamon-Schwarz map.

It takes some effort to promote the self Floer cohomology to the Fukaya category framework, and ensure the consistency in the perturbation schemes (\cf Auroux \cite[section 2.1]{Aurouxsurvey} for a sketch and Seidel \cite{Seidelbook} for details). In applications it is often more convenient to avoid Hamiltonian perturbations as much as possible.

\begin{eg}(Floer products involving the identity)\label{FLoerproduct}
We wish to heuristically explain a special case relevant to Joyce-Imagi-Santos (\cf section \ref{Lawlor}), concerning the geometric interpretation of the Floer product mod 2
\[
HF^0(L',L)\otimes HF^0(L,L')\to HF^0(L,L).
\]
Here $L, L'$ are assumed to be transverse.
Hamiltonian invariance means we can alternatively think of
\[
HF^0(L', \phi_{\epsilon H} (L))\otimes HF^0(L,L')\to HF^0(L, \phi_{\epsilon H}( L)).
\]
This is defined by the count of holomorphic triangles with input corners at $CF^0(L, L')$, $CF^0(L', \phi_{\epsilon H} (L))$, and an output corner at $CF^0(L, \phi_{\epsilon H}( L))$. We may assume the Morse function $H|_L$ has only one maximum point $r$ on $L$, which represents the unit of $HF^*(L,L)$. When $\epsilon\to 0$, then $L$ and $\phi_{\epsilon H}(L)$ coincide, and the holomorphic triangles become holomorphic strips with ends at $CF^0(L, L')$, $CF^0(L', L)$ (alternatively seen as a degree $n$ output) and passing through the point $r\in L$. This last incidence condition is independent of the position of $r$ on $L$, since we can choose $H$ to have its maximum at any generic prescribed point. Notice in this strip interpretation, there is no longer any Hamiltonian perturbation. This interpretation featured in Lemma \ref{JoyceImagicountinglemma}.	
\end{eg}

\subsubsection*{Sign issues and brane structures}

To go beyond mod 2 coefficients, we need to orient moduli spaces. Good references can be found in Seidel's book \cite{Seidelbook} and Abouzaid \cite[Appendix]{Abouzaidloop}. All Lagrangians are assumed to be graded, with second Stiefel-Whitney class equal to the restriction of a fixed class in $H^2(X,\Z/2)$, and we equip the Lagrangians with \textbf{relative spin structures}. At any transverse Lagrangian intersection point $p\in L_+\cap L_-$, there is a unique up to homotopy path $\Lambda_p$ of Lagrangian planes  in $T_pX\simeq \C^n$ with graded lift interpolating $TL_+$ and $TL_-$. We fix a relative spin structure on $\Lambda_p$, compatible with the relative spin structure on $L_\pm$. We can associate a vector space $o_p$ as the determinant line of the Cauchy-Riemann operator $D_p$ on the upper half plane with boundary data $\Lambda_p$. The dual of $o_p$ is denoted $o_p^\vee$, namely $o_p\otimes o_p^\vee\simeq\R$ canonically. The orientation line $|o_y|$ is the free abelian group generated by the two possible orientations of $o_y$ with the relation that their sum vanishes. Furthermore, we equip the Lagrangians $L$ with (rank one) \textbf{local systems} $E$, and write the Floer cochain complex as the graded vector space
\[
CF^*(L, L')= \oplus_{p\in L\cap L'} \text{Hom}(E|_p, E'|_p  )\otimes |o_p|.
\]

\begin{rmk}\label{localsystem}
There are some variants on the coefficient ring/field of the local system. The simplest case is the trivial local system, in which case we simply delete the $\Hom$ factor.
Other popular choices have parallel transport in $\mathbb{Q}^*, \R^*, \C^*$, or the units in the Novikov ring.\footnote{The $U(1)$-local systems are popular in the physics literature, but appear rarely in Floer theory.} Different choices could in principle lead to slightly different versions of the derived Fukaya category. The smaller the coefficient ring/field, the more stringent is the notion of derived isomorphism of objects. For the purpose of extending the Solomon functional (\cf section \ref{Solomonfunctionalextension}) to be real valued, we require all coefficients to be at least contained in $\R$, so we will usually work simultaneously with $\R$, $\Q$ and $\Z$ local systems. On the other hand, it is claimed in \cite[Remark 4.5]{Woodward1} that in the exact setting the immersed Fukaya algebras can be defined \emph{over the integers}. The specific advantage of working over integers, as discussed in the main text, is primarily that the bordism current $\mathcal{C}$ between Lagrangians is then an \emph{integral current}, rather than $\R$-linear combinations of integral currents.



\end{rmk}

Given a holomorphic polygon $u: \Sigma\to X$, with inputs $x_1,\ldots x_k$ and output $x_0$ mapping to $p_1,\ldots p_k$ and $q$, the det line of the linearized Cauchy-Riemann operator $D_u$ can be computed from gluing kernel and cokernels:
\[
\det (D_u\# D_{p_k}\ldots \# D_{p_1} )\simeq \det (D_u)\otimes o_{p_k}\ldots \otimes o_{p_1}.
\]
The role of the \emph{relative spin structure}, is to specify a homotopically unique choice of isotopy between the glued operator $D_u\# D_{x_k}\ldots \# D_{x_1} $ and $D_{x_0}$ (\ie an \emph{isotopy between Lagrangian boundary conditions}), hence a preferred isomorphism 
\[
\det D_u\simeq o_{q}\otimes o_{p_1}^\vee \otimes \ldots o_{p_k}^\vee.
\]
The tangent space of the moduli space of holomorphic polygons $\mathcal{M}(p_1,\ldots p_k, q) $ involves not only the linearized Cauchy-Riemann operator, but also the variation of the complex structure of the domain of the polygon, controlled by the Stasheff associahedron $\overline{\mathcal{R}}_{k+1}$. Denote $\lambda^{top}(V)$ as the top wedge product of a vector space $V$. Then there are preferred isomorphisms depending on the relative spin structure choice
\begin{equation}\label{orientationisomorphism}
\lambda^{top}( T\mathcal{M}(p_1,\ldots p_k, q)  )\simeq \lambda^{top}(\mathcal{R}_{k+1} ) \otimes o_{q}\otimes o_{p_1}^\vee \otimes \ldots o_{p_k}^\vee.
\end{equation}

\begin{rmk}\label{orientationdual}
Fixing an orientation on $L$, then $CF^0(L,L')$ is naturally dual to $CF^n(L',L)$. The local system factor $\Hom(E,E')$ is naturally dual to $\Hom(E',E)$. Given a Lagrangian path $\Lambda_p$ associated to a Lagrangian intersection $p$, the reverse path is also associated with a determinant line bundle, which can be identified with
\[
o_p^\vee \otimes \lambda(TL) ,
\]
since the two half planes with Lagrangian boundaries can be glued to a disk, such that the det line of the Cauchy-Riemann operator is canonically isomorphic to $\lambda(TL)$.

\end{rmk}

For $k\geq 2$, when the moduli spaces are zero dimensional, so carry canonical orientations, then a universal orientation choice for $\lambda^{top}(\mathcal{R}_{k+1} )$ determines an operator 
\[
|c_u|: |o_{p_k}|\otimes \ldots \otimes |o_{p_1}|\to |o_{q}|.
\]
In our degree conventions the corners $x_0, x_1,x_2,\ldots x_k$ on the domain disc boundary are ordered clockwise, so a natural orientation on the Stasheff associahedron can be obtained by fixing $x_0,x_1,x_2$ and allowing the other corner points to move in the clockwise orientation.
The parallel transports along the local systems contribute another factor 
\[
\Hom(E_{k-1}, E_k)|_{p_k}\otimes \ldots \Hom(E_1, E_2)|_{p_2}\otimes \Hom(E_0, E_1)|_{p_1}\to \Hom(E_0, E_k)|_q.
\]
Each pseudoholomorphic polygon contributes to the operation
\[
m_k: 
CF^*(L_{k-1}, L_k)\otimes \ldots CF^*(L_0,L_1)\to CF^*(L_0,L_k)[2-k]
\]
via the tensor product of the orientation factor $|c_u|$ and the local system factor, multiplied by another sign factor depending only on the degrees (\cf \cite[equation (12.24)]{Seidelbook})
\[
(-1)^{\deg p_1+ 2\deg p_2+\ldots k\deg p_k}.
\]


In the case of holomorphic strips, we have a natural isomorphism
\begin{equation}\label{orientationstrip}
T\mathcal{M}(p,q )= \R(-\partial_s)\oplus T( \mathcal{M}(p,q )/\R  ).
\end{equation}
where $-\partial_s$ is the translation vector field pointing towards the input point. When $\mathcal{M}(p,q)/\R$ consists of isolated points, it carries canonical orientations, whence by (\ref{orientationisomorphism}) we obtain 
\[
|c_u|: |o_{p}|\to |o_q|.
\]
The local system parallel transport produces another factor
\[
\Hom(E_0,E_1)|_{p}\to \Hom(E_0, E_1)|_q.
\]
Each pseudoholomorphic strip contributes to the Floer differential 
\[
d: CF^*(L_0,L_1)\to CF^{*+1}(L_0,L_1)
\]
by the product of these two factors. We write
\[
m_1: CF^*(L_0,L_1)\to CF^{*+1}(L_0,L_1),\quad 
m_1= (-1)^{\deg p_1} d.
\]
When the signs and local system weighting factors are taken into account, the $A_\infty$-relation reads
\begin{equation}
\sum_{l=1}^k \sum_{j=0}^{k-l} (-1)^\dagger m_{k+1-l} (p_k, \ldots , p_{j+l+1}, m_l(p_{j+l}, \ldots p_{j+1}), p_j,\ldots p_1)=0.
\end{equation}
where
 $\dagger=j+\deg p_1+\ldots +\deg p_j$. The \textbf{Fukaya category} for the compact embedded Lagrangians comprises of the following data:
\begin{itemize}
	\item The objects are embedded Lagrangians (with additional brane data, such as grading, Lagrangian potential, orientation, relative spin structure, and local system).
	\item  The morphisms $\Hom^*(L,L')$ are the vector spaces $CF^*(L,L')$ (where $L$ can coincide with $L'$).
	
	\item  The $A_\infty$-composition maps are the multilinear maps $m_k$ satisfying the $A_\infty$ relations. 
\end{itemize} 
The Fukaya category is an example of an $A_\infty$-category.


In particular, the Floer differential squares to zero, so we can define the Floer cohomology groups $HF^*(L,L')$ for embedded exact Lagrangian branes. The Floer product on cohomology is given by
\[
[p_2]\circ [p_1] =(-1)^{\deg p_1} m_2(p_2,p_1)
\]
which is associative.

\begin{eg}\label{FLoerproduct2}(Floer products involving the identity, continued)
In the context of Example \ref{FLoerproduct}, the orientation isomorphism (\ref{orientationisomorphism}) for the holomorphic triangle is determined by whether the isotopy of the Lagrangian boundary conditions respects the relative spin structure. Since the relative spin structure on $\phi_{\epsilon H} (L)$ is induced from $L$, this problem is equivalent to the corresponding isotopy problem for the limiting holomophic strip. The holonomy factor of the local systems for the holomorphic triangle, is also reduced to that of the limiting strip.

In the simplest case when we are given closed elements $\alpha\in CF^0(L,L'), \beta\in CF^0(L',L)$ each involving only one intersection point, the local systems are trivial, and only one holomorphic curve contributes to the Floer product, then $\beta\circ \alpha= 1_L\in HF^0(L,L)$ means that for the holomorphic strip from $\alpha$ to $\beta$ passing through a generically chosen point $r\in L$, the Lagrangian boundary condition on the disk obtained by gluing $TL,TL'$ and the two Lagrangian paths at the two strip like ends, can be contracted to constant, respecting the prescribed relative spin structures on $L,L'$ and the two ends. More generally, many intersections points and holomorphic strips may contribute to the Floer product, and $\beta\circ \alpha= 1_L\in HF^0(L,L)$ means a weighted signed count of holomophic strips is equal to one.

Under sufficient transversality assumptions, we can form the $(n-1)$ dimensional moduli space of holomorphic strips from $\alpha$ to $\beta$, and thereby produce an $(n+1)$-dimensional universal family $\mathcal{C}$, as in the main text section \ref{LotayPacinirevisited}. Using the relative spin structures on $L,L'$ and the Lagrangian paths associated with the ends, we use (\ref{orientationisomorphism})(\ref{orientationstrip}) and Remark \ref{orientationdual} to induce a canonical orientation on the moduli space from $\beta\otimes \alpha\in CF^0(L',L)\otimes CF^0(L,L')$. Using the complex orientation on the holomorphic curve $\Sigma$, and inserting an extra minus sign, we obtain an orientation on $\mathcal{C}$. This tricky minus sign accounts for the difference between the counterclockwise orientation of $\partial \Sigma$ compatible with the complex orientation, and the clockwise orientation of $\partial \Sigma$ compatible on the $L$-boundary with the translation vector field $-\partial_s$. Putting everything together, $\beta\circ \alpha= 1_L\in HF^0(L,L)$ means in the sense of weighted counts, that $\partial \mathcal{C}$ passes once through a generic point $r\in L$ in the same orientation as $\lambda(TL)$. In other words, the $L$-boundary evaluation of $\partial \mathcal{C}$ \emph{sweeps out the oriented cycle} $L$.

The same argument says that if $\alpha\circ \beta=1_{L'}\in HF^0(L',L')$, then the moduli space of holomorphic strips from $\beta$ to $\alpha$ produces a universal family $\mathcal{C}'$, whose $L'$-boundary evaluation map sweeps out the oriented cycle $L'$. The subtle point is that due to the reversal of the $\R$-translation vector fields, $\mathcal{C}'$ has the reverse orientation as $\mathcal{C}$. Therefore, the $L'$-boundary evaluation of $\partial \mathcal{C}$ sweeps out the oriented cycle $-L'$ instead of $L'$. Here ends the example.

\end{eg}

\subsubsection*{Twisted complexes, distinguished triangles, derived category}

A fundamental problem of the embedded Fukaya category is that it lacks enough geometric objects. Morally, Fukaya category is a construction that inputs the symplectic geometry of Lagrangian branes, and outputs the representation theory of an $A_\infty$-category. Now the general feature of $A_\infty$-module categories is that one can take cones and idempotent summands, two properties which are useful for classifying such categories, and desirable for mirror symmetry. The problem is that cones and idempotent summands are not obviously represented by embedded Lagrangian objects under the Yoneda embedding. The common solution is to sideline this issue by the formal algebraic construction of twisted complexes and idempotent completions. This is not quite adequate for the Thomas-Yau conjecture. However, we will discuss how the introduction of immersed Lagrangian objects geometrizes the twisted complexes (\cf section \ref{immersedFukaya}). The geometric meaning of idempotents is an open problem.

The formal algebraic constructions are well explained in \cite[section 3]{Aurouxsurvey} and \cite[section 4]{Smithsurvey}, to which we refer the reader for more details. Given objects $L_1,\ldots L_N$ of the Fukaya category $\mathcal{A}$, a \textbf{twisted complex} $(L,b_L)$ consists of
\begin{itemize}
\item The formal shifted direct sum $L=\oplus_1^N L_i[k_i]$ with $k_i\in \Z$ formally keeping track of degrees (the geometric meaning of the shift $[1]$ is to add a constant $\pi$ to the Lagrangian phase, which reverses the orientation of the Lagrangian, with a corresponding twist to the spin structure),

\item and a strictly triangular differential $b_L\in \End(L)$, \ie a collection of maps $b_{ij} \in \Hom^{k_j-k_i+1} (L_i, L_j)$ for $i>j$,\footnote{In most symplectic references such as \cite{Aurouxsurvey} the morphisms $b_{ij}$ go in the opposite direction $i<j$. This just amounts to reversing the ordering of $L_1,\ldots ,L_N$. We find our reversed convention a little more convenient for the Harder-Narasimhan decomposition.}
\end{itemize}
satisfying the equation
\[
\sum_{k\geq 1} m_k(b_L,\ldots,  b_L)=0, \quad \ie
\]
\[
\sum_{k\geq 1} \sum_{ i=i_0>\ldots >i_k=j} m_k(b_{i_{k-1}i_k}, \ldots, b_{i_0 i_1})=0.
\]
Notice the strict triangularity implies the sum is finite. One can define morphisms between these twisted complexes, and assign  $A_\infty$-structures to make twisted complexes into an $A_\infty$-category $Tw \mathcal{A}$, into which $\mathcal{A}$ naturally embeds fully faithfully. Using the $A_\infty$-structure, it makes sense to talk about closed morphisms and cohomologies, similar to the construction of Floer cohomology.

Given twisted complexes $A=(L,b_L), B=(L',b_{L'}) \in Tw\mathcal{A}$, and a closed morphism $f\in \Hom^0( (L,b_L), (L', b_{L'}  ) )$, the abstract \textbf{mapping cone} of $f$ is the twisted complex
\[
\text{Cone}(f)= \left( L[1]\oplus L',  \left( \begin{matrix}
b_L & 0\\
f & b_{L'}
\end{matrix} \right)    \right).
\]
Generally, a mapping cone of $f$ is an object of $Tw \mathcal{A}$ quasi-isomorphic to $\text{Cone}(f)$. This gives rise to a \textbf{distinguished triangle} $A\to B\to \text{Cone}(f)\xrightarrow{ [1]}A$. This illustrates the advantage of introducing twisted complexes: $Tw\mathcal{A}$ is a \textbf{triangulated category}.

The cohomological category of $Tw \mathcal{A}$ is commonly denoted $D^b Fuk(X)$. This has the same objects as $Tw \mathcal{A}$, but the Floer cochain spaces are replaced by their $H^0$, namely we remember the Floer cohomology.

Under the Yoneda embedding, $Tw\mathcal{A}$ embedds into its module category. The idempotent closure $Tw^\pi \mathcal{A}$ is obtained by formally adding the direct summands of the Yoneda image of twisted complexes in $Tw\mathcal{A}$. The cohomological category of $Tw^\pi \mathcal{A}$ is commonly denoted $D^\pi Fuk(X)$. In the variant setting of compact $X$, it is usually $D^\pi Fuk(X)$ instead of $D^b Fuk(X)$ that shows up in mirror symmetry, since the derived category of coherent sheaves is automatically idempotent closed.

\begin{rmk}
Once immersed Lagrangians are admitted as objects of Fukaya categories, the twisted complexes are largely redundant.
Joyce \cite[conjecture 3.6]{Joyceconj} claims that by including immersed and singular Lagrangians with rank one local systems, then $D^b Fuk(X)$ is automatically idempotent closed, so there is no difference between $D^b Fuk(X)$ and $D^\pi Fuk(X)$. However, it is highly nonobvious why direct summands are Yoneda represented by geometric Lagrangian objects,\footnote{There exist some wild speculations, such as incorporating coisotropic branes into the Fukaya category in order to have more geometric objects.} so this claim is regarded by many experts as a weakness of Joyce's proposal. For this reason, in our more restrictive proposal we stick with the more geometric $D^b Fuk(X)$ (including immersed and singular objects, but not formal idempotent summands) in favour of $D^\pi Fuk(X)$, and the idempotent closure problem does not falsify our program.

\end{rmk}






\subsection{Immersed exact Lagrangians}\label{immersedFukaya}

According to Joyce's LMCF program,  immersed Lagrangians are a necessary part of any Fukaya category adequate for the Thomas-Yau conjecture. As far as the author is aware, only immersed Floer cohomology \cite{JoyceAkaho}, rather than  the full categorical framework, has been written down in the literature,  although in the exact setting this is commonly believed to be a relatively routine matter, as sketched in \cite[section 4.1]{Joyceconj}. Our limited goal is to highlight the main difference with the embedded case, namely the issues of \textbf{obstructions and bounding cochains}. Once these two issues are taken into account, what works in the embedded case will also work in the immersed case.

\subsubsection*{Teardrop curves and obstructions}

The assumptions on $(X,\omega)$ are as in the previous section.
Immersed Lagrangians are immersions $\iota: L\to X$ with $\omega|_L=0$, and all self intersections are transverse. The domain of $L$ is allowed to be disconnected, so the union of finitely many transversely intersecting embedded Lagrangians are examples of immersed Lagrangians. Each \textbf{self intersection} point of two local sheets $L_+, L_-$ corresponds to two different points $x_\pm$ on the domain of $L$. It is important to distinguish $x_+$ and $x_-$, because for the boundary of the holomorphic curve to pass through $x_+$ in the clockwise direction means crossing from $L_+$ to $L_-$, and $x_-$ signifies the opposite crossing.

We say $L$ is exact, if there is a function $f_L$ on the domain of $L$, such that $df_L$ agrees with the Liouville 1-form restricted to $L$. For energy reasons, this forbids nontrivial holomorphic disks with boundary on $L$ which never change local sheets at any boundary point. The caveat is that the relative homology class $[\omega]\in H_2(X,L)$ may still be nonzero. The brane structures on $L$ are as in the embedded case. The construction of $CF^*(L,L)$ depends on the approach, but a common feature is that it includes
\[
CF^*_{self}(L,L)=\bigoplus_{\text{self intersection $p$}} CF^*(L_+,L_-)\oplus CF^*(L_-,L_+).
\]
generated by the local system factor $\text{Hom}(E_+, E_-)|_p$ (resp. $\text{Hom}(E_-,E_+)|_p$) tensored with the orientation line. 


The Gromov compactness discussion is largely similar to the embedded case.
A new phenomenon is the \textbf{teardrop curves}, namely the holomorphic curves with boundary on $L$ and a single output corner at a self intersection point $r\in CF^*(L_+,L_-)$. Of particular importance is the case with $\mu_{L_+, L_-}(r)=2$. The number $2$ is intuitively explained by the 2 degrees of freedom of the domain M\"obius transforms fixing the corner point $Aut(D^2,1)$, modulo which such teardrop curves occur in dimension zero moduli spaces.

Now if we attempt to run the usual argument for $d^2=0$ in Floer cohomology, we would consider the moduli space of holomorphic strips between $p,q$ with $\deg q-\deg p=2$, modulo the translation $\R$. However, in addition to the usual strip breaking, the holomorphic strips can also break into a holomorphic triangle with inputs $p,r$ and output $q$, and a teardrop curve with corner at $r$. In summary, teardrop curves with corner at a degree 2  intersection point \textbf{obstruct} Floer cohomology.

The automorphism group $Aut(D^2, 1)$ forbids the na\"ive domain dependent perturbation schemes, which in turn causes transversality problems. In the literature there are two approaches to solve this problem: Joyce and Akaho \cite{JoyceAkaho} use virtual perturbation techniques for bordered Riemann surfaces, while Woodward et al. \cite{Woodward1}\cite{Woodward2} circumvent the virtual perturbations by utilizing stabilising divisors. Both approaches assign curved $A_\infty$ algebra structures  $(m_0,m_1,\ldots)$ to the Floer cochain spaces $CF^*(L,L)$ of immersed Lagrangians. In the exact setting, the $m_0\in CF_{self}^2(L,L)$ term amounts to a count of teardrop curves with corner at degree 2 self intersection points, with weighting factors coming from the holonomy of the local system.
Since in the main text the emphasis is on the automatic transversality assumption, we shall not dwell on the details of perturbation schemes, but only identify a few simplifications in the exact setting.

\begin{rmk}
The rough idea of Woodward et al. is to introduce interior marked points, constrained to lie on a Donaldson divisor $D$ disjoint from the Lagrangians. The virtual dimension is not affected by these divisor constraints, since 
each interior marked point increases it by 2, while each divisor constraint decreases it by 2.
 One needs to arrange $D$ to be of sufficiently high degree, so that each nontrivial pseudoholomorphic disk with boundary on the Lagrangians has at least one intersection with $D$. On a teardrop curve, imposing the divisor constraint at interior marked points kills the domain automorphisms $Aut(D^2, 1)$, so one can then introduce domain dependent perturbation of almost complex structures compatible with $D$ to achieve sufficient transversality to make sense of counts. The appealing feature of this approach, is that adding marked points does not alter the geometric interpretation of the holomorphic curves, so stays closer to geometry than the virtual approach.

 The framework of Woodward et al. \cite{Woodward1}\cite{Woodward2} is not restricted to exact settings, and works also for compact symplectic manifolds with rational $[\omega]\in H^2(X)$. Producing the Donaldson divisor with the intersection properties is easier if $[\omega]\in H_2(X,L)$ is a rational class, although the methods in \cite[section 3.1]{CharestWoodward} allows one to largely relax this assumption.

In exact manifolds, as mentioned in \cite[Remark 4.5]{Woodward1}, one can avoid the spherical components of the treed disks.  In the exact Lagrangian setting, the only bubbling happens at the self intersection points. These afford significant simplifications to the construction, and allows one to think of the treed disks in \cite{CharestWoodward}\cite{CharestWoodward2} \cite{Woodward1}\cite{Woodward2} in terms of a tree of holomorphic polygons connected at the self intersection points. By avoiding the troublesome sphere bubbles, one can also relax the restriction of moduli spaces of dimension at most one.

\end{rmk}

\begin{rmk}
A very technical aspect of Akaho-Joyce \cite{JoyceAkaho} is that the $A_\infty$ structure is not constructed directly, but through a sequence of approximations involving energy cutoff scales. In the exact setting, the topological energy formula implies a priori energy bounds, so this complication would not arise.
\end{rmk}

\subsubsection*{Cancellation of obstructions}

To make sense of Floer cohomology one needs to cancel the obstructions by introducing \textbf{bounding cochains} $b\in CF_{self}^1(L,L)$, which represents  a formal sum of  $b_p\in \text{Hom}(E_+,E_-)|_p\otimes |o_p|$ associated to degree one intersection points $p\in CF^1(L_+, L_-)$. We require
\begin{itemize}
\item The \textbf{Novikov positivity} condition $f_{L_+}(p)\geq f_{L_-}(p)$ for each of the intersection points appearing in $b$.

\item The \textbf{Mauer-Cartan equation}
\begin{equation}
m_0^b=
m_0+ m_1(b)+ m_2(b, b)+ \ldots =0\in CF_{self}^2(L,L).
\end{equation}
\end{itemize}

Geometrically, the coefficients of $q\in CF_{self}^2(L,L)$ in the $m_k(b,\ldots b)$ term represent the zero dimensional counts of holomorphic polygons with the inputs at the summands $b_p$ of $b$, and the output at $q$, weighted by the holonomy and orientation factors. Using the Novikov positivity requirement of the bounding cochain, the topological energy formula (\ref{topologicalenergy2}) for the polygon then implies
\[
\int_{\Sigma} \omega\leq f_{L_-}(q)- f_{L_+}(q),
\]
where the boundary of $\Sigma$  passes from $L_+$ to $L_-$ at $q$ in the clockwise direction. By Gromov compactness, this uniform energy bound implies there are only  \emph{finitely many}  terms involved in the Mauer-Cartan equation. When such a bounding cochain $b$ exists, we say $(L,b)$ defines an unobstructed Lagrangian brane. In this case, both the Akaho-Joyce and the Woodward-Palmer approaches assign \textbf{self Floer cohomology groups} $HF^*((L,b),(L,b))$, defined as the cohomology of a degree one operator
\[
m_1^b: CF^*(L,L)\to CF^{*+1}(L,L),\quad m_1^b(x)=\sum_{k,k'\geq 0} m_{k+k'+1}(\underbrace{b,\ldots b}_{k'}, x,\underbrace{b,\ldots, b}_{k}).
\]
This cohomology is invariant under global Hamiltonian deformations. Two bounding cochains $b,b'$ on $L$ are said to be \textbf{gauge equivalent}, if there is $h\in CF^0(L,L)$ satisfying the Novikov positivity condition, such that 
\[
b-b'= \sum_{k,k'\geq 0} m_{k+k'+1}(b',\ldots b', h, b,\ldots b).
\]
Gauge equivalent bounding cochains give rise to isomorphic Floer cohomology.

\begin{rmk}
In the embedded case, there are no self intersections, so the Mauer-Cartan equation is vacuous, and the Lagrangian is automatically unobstructed, with zero bounding cochain. The unobstructed condition is not automatic in general for immersed Lagrangians, and a significant aspect of the Joyce program in \cite{Joyceconj} is that unobstructed Lagrangians ought to be better behaved in the LMCF.	
\end{rmk}

Now suppose $(L,b)$ and $(L',b')$ are two unobstructed Lagrangian branes, intersecting transversally avoiding the self intersections of $L$ and $L'$. Then we can define the Floer cohomology $HF^*((L,b),(L',b'))$. The Floer cochain space $CF^*(L,L')$ is the same as in the embedded case, generated by the local system factor tensored with the orientation factor, associated to the transverse intersection points. The Floer differential is
\[
m_1^{b,b'}(p)= \sum_{k,k'\geq 0} m_{k+k'+1}(b',\ldots b', p,b,\ldots b),
\]
where the sum has $k'$ insertions of $b'$, and $k$ insertions of $b$. The coefficient of $q\in CF^{*+1}(L,L')$ are morally defined by the weighted count of holomorphic polygons with boundary marked points mapping to the summands of $b,\ldots p, b',\ldots, q$, arranged in clockwise order. A similar a priori energy bound argument shows the sum is finite.

It is instructive to see why $(m_1^{b,b'})^2=0$. We consider the breaking of one dimensional moduli spaces, associated with $p,r\in CF^*(L,L')$ with $\deg r-\deg p=2$. There are several mechanisms for disc bubbling and disc splittings:
\begin{itemize}
\item The polygon breaks into two parts, connected at a nodal point mapping to some $q\in CF^*(L,L')$ with $\deg q-\deg p=1$. The sum of all such contributions give rise to $\langle m_1^{b,b'}(p), q\rangle \langle m_1^{b,b'}(q), r\rangle$, and summing over $q, r$ produces $(m_1^{b,b'})^2(p)$.

\item The polygon bubbles off a teardrop curve at a self intersection point $q$ of degree 2 on either $L$ or $L'$.

\item The polygon splits into two parts, connected at a node mapping to a degree 2 self intersection point $q$ on either $L$ or $L'$.

\end{itemize}

The combined effect of the last two contributions, is a sum of the weighted counts of polygons with boundary mapping to $b,b,\ldots p, b',\ldots q, b'\ldots r$ multiplied by the coefficient of $q$ in $m_0^{b'}=m_0+m_1(b')+\ldots\in CF^2_{self}(L',L')$  in the case of $q\in CF^2_{self}(L',L')$
(the case with $q\in CF^2_{self}(L,L)$ gives an entirely similar contribution related to $m_0^b\in CF^2_{self}(L,L)$). By the unobstructed assumption $m_0^b=0$ and $m_0^{b'}=0$, so these contributions vanish. But the grand sum of all contributions from all boundaries of the moduli spaces should be zero, which implies $(m_1^{b,b'})^2=0$.

The generalization to many Lagrangians is a matter of bookkeeping. We have the $A_\infty$ compositions
\[
m_k^{b_0,\ldots b_k}: CF^*(L_{k-1},L_k)\otimes \ldots CF^*(L_0,L_1)\to CF^*(L_0,L_k)[2-k],
\]
\begin{equation}\label{Ainftytwisted}
m_k^{b_0,\ldots b_k}(p_k,\ldots p_1)=\sum m_l( b_k,\ldots b_k, p_k, b_{k-1},\ldots, p_{k-1},\ldots, p_1, b_0,\ldots b_0  ).
\end{equation}
In particular, this induces a product structure on Floer cohomology $HF^*(L_1,L_2)\otimes HF^*(L_0,L_1)\to HF^*(L_0,L_2)$ (with bounding cochains suppressed in the notation),
\[
[\beta]\circ[\alpha]= (-1)^{\deg \alpha} m_2^{b_0,b_1,b_2}(\beta,\alpha).
\]
We say two unobstructed Lagrangian branes $L,L'$ are isomorphic in $D^bFuk(X)$, if there exist $[\alpha]\in HF^0(L,L')$ and $[\beta]\in HF^0(L',L)$, such that their compositions are the cohomological units: $[\beta]\circ [\alpha]=1_L\in HF^0(L,L)$ and $[\alpha]\circ [\beta]=1_{L'}\in HF^0(L',L')$.

\subsubsection*{The union of several components}

In our convention an immersed Lagrangian can have several components. Of particular interest is the case where $L$ is the union of transverse immersed Lagrangians $L_1,\ldots L_N$ with bounding cochains $b_1,\ldots , b_N$ respectively, and we have morphisms $b_{ij}\in CF^1((L_i, b_i), (L_j, b_j))$ for $i>j$. The key assumption here is that the morphisms only go in one direction from $L_i$ to $L_j$, not vice versa. We assume that $b=\sum b_i+ \sum_{i>j} b_{ij}$ is a bounding cochain for the immersed Lagrangian $L$, and in particular all intersection points in $b_{ij}$ satisfy the Novikov positivity condition $f_{L_i}\geq f_{L_j}$. We can write  out the Mauer-Cartan equation
\[
m_0+ m_1(b)+ m_2(b,b)+\ldots =0
\]
in component form: for any $i>j$,
\[
\sum_l \sum_{k\leq l} \sum_{ i=i_0>\ldots >i_k=j} m_l(b_{i_k}, \ldots b_{i_k}, b_{i_{k-1}i_k},    \ldots, b_{i_1}, \ldots b_{i_1}, b_{i_0 i_1}, b_{i_0},\ldots, b_{i_0})=0.
\]
The key observation is that this is precisely how one would define \textbf{twisted complexes} built on $L_1,\ldots L_N$, in the presence of the bounding cochains $b_1,\ldots b_N$ and the data $b_{ij}$, when no further degree shifts are involved (\cf the exact setting in section \ref{Fukayaembedded}). In this sense, we say that `immersed Lagrangians geometrises twisted complexes'. In other words, if the unobstructed immersed Lagrangians are admitted into the Fukaya category, then there is no need to formally add twisted complexes.


\begin{lem}(Blocking together connected components based on potential clustering)\label{blockingconnectedcomponents}
Assume $L$ is the finite union of transversely intersecting immersed Lagrangians, with a bounding cochain $b$. Then $L$ can be decomposed as a twisted complex built from some $L_1,\ldots L_N$, such that $\inf_{L_i} f_{L_i}> \sup_{L_j} f_{L_j}$ whenever $i>j$, and the Lagrangian potential $f_{L_i}$ has connected range for each $L_i$ .
\end{lem}

\begin{proof}
The decomposition can continue as long as there exists a real number $c$, such that 
the Lagrangian components can be partitioned into two types, with Lagrangian potential strictly smaller than $c$ (resp. greater than $c$). As long as $\inf_{L_i} f_{L_i}> \sup_{L_j} f_{L_j}$ whenever $i>j$, the Novikov positivity condition on the Lagrangian intersection points would imply that the entries $b_{ij}\in CF^1(L_i,L_j)$ of $b$ can only go in the direction $i>j$ and not vice versa, so the immersed Lagrangian $L$ is necessarily of the twisted complex form. This algorithm stops in finitely many steps since there are only finitely many components involved.
\end{proof}


	

\subsubsection*{Orientation signs on bordism currents}

In section \ref{LotayPacinirevisited}, \ref{LotayPaciniimmersed} we encountered the $(n-1)$-dimensional moduli spaces such as $\mathcal{M}(b,\ldots, b, \alpha, b',\ldots ,\beta)$ and $\mathcal{M}(b,\ldots, b, \gamma)$. The special case of holmorphic strips was already mentioned in Example \ref{FLoerproduct2}.

We now consider the moduli $\mathcal{M}(p_1,\ldots p_k)$ of polygons with at least 3 corners $p_1,\ldots p_k$, all regarded as inputs, arranged in clockwise order on $\partial \Sigma$, each carrying the local system factors $\Hom(E_+,E_-)|_{p_i}$ and the orientation factors $|o_{p_i}|$. The clockwise composition of the local system hom factors and the parallel transport along $\partial \Sigma$, produces a holonomy factor around $\partial \Sigma$, which is a number in $\Q,\R,\Z$ depending on the coefficient ring choice. Using (\ref{orientationisomorphism}) and Remark \ref{orientationdual}, as well as the clockwise orientation convention on the Stasheff associahedron, we acquire a (na\"ive) orientation on $T\mathcal{M}(p_1,\ldots p_k)$. 
To assign \textbf{orientation} and \textbf{weighting factors} to $\mathcal{M}(p_1,\ldots p_k)$, we take the product of the holonomy factor, the na\"ive orientation on $T\mathcal{M}(p_1,\ldots p_k)$, and another universal sign factor
\[
(-1)^{\deg p_1+2\deg p_2+\ldots + k\deg p_k} (-1)^{\deg p_k}.
\]
The appearance of this universal sign adjustment is a familiar convention in the open-closed map, \cf \cite[eqn 5.24]{Abouzaidgeneration}. The notation $\mathcal{M}$ is a shorthand for the weighted sum of all the $(n-1)$-dimensional moduli spaces involved in the construction of the bordism current.

 We equip the domain $\Sigma$ with the complex orientation, and together with an \emph{extra minus sign}, the orientation on $\mathcal{M}$ induces the orientation on $\mathcal{C}$. This minus sign arises for the same reason as in Example \ref{FLoerproduct2}, namely the discrepancy between our clockwise convention on $\partial \Sigma$, with the standard complex orientation on $\Sigma$.

\end{document}